\documentclass[11pt]{amsart}

\usepackage{graphicx}
\usepackage{framed}

\usepackage[utf8]{inputenc}	

\usepackage[all]{xy}
\usepackage{pb-diagram, pb-xy}

\usepackage{amssymb}
\usepackage{amsthm}

\usepackage{hyperref}

\usepackage{calc}

\usepackage[framemethod=tikz]{mdframed}

\usepackage{enumitem}
\setlist[itemize]{noitemsep}

\newtheoremstyle{saetze} 
    {5pt}                    
    {5pt}                    
    {\itshape}                   
    {12pt}                           
    {\bfseries}                   
    {.}                          
    {.5em}                       
    {}  

\theoremstyle{saetze}
\newtheorem{thm}{Theorem}[section]
\newtheorem{lem}[thm]{Lemma}
\newtheorem{cor}[thm]{Corollary}
\newtheorem{prop}[thm]{Proposition}

\newtheorem{conj}[thm]{Conjecture}

\theoremstyle{definition}
\newtheorem{definition}[thm]{Definition}
\newtheorem{example}[thm]{Example}
\newtheorem{remark}[thm]{Remark}

\newcommand{\bbZ}{{\mathbb Z}}

\newcommand{\g}{{\mathfrak{g}}}
\newcommand{\gl}{{\mathfrak{gl}}}

\newcommand{\frakg}{{\mathfrak g}}

\newcommand{\calA}{{\mathcal A}}

\newcommand{\calC}{{\mathcal C}}

\newcommand{\calF}{{\mathcal F}}

\newcommand{\calN}{{\mathcal N}}

\newcommand{\calR}{{\mathcal R}}
\newcommand{\calS}{{\mathcal S}}
\newcommand{\calT}{{\mathcal T}}

\newcommand{\Rep}{{\mathrm{Rep}}}

\newcommand{\id}{{\mathit{id}\hspace{-0.1em}}}

\newcommand{\one}{{\mathbf{1}}}

\newcommand{\Z}{{\mathbb{Z}}}

\newcommand{\NN}{{\mathcal{N}}}

\def\N{\mathbb{N}}
\newcommand{\A}{{\mathbb A}}

\DeclareMathOperator{\sdim}{sdim}

\DeclareMathSizes{12}{11}{8}{5} 
\usepackage{xypic,dsfont}

\makeatletter
\patchcmd{\l@section}{1.0em}{\z@}{}{}
\makeatother

\hypersetup{
    colorlinks,
    citecolor=red,
    filecolor=black,
    linkcolor=blue,
    urlcolor=black
}

\long\def\comment#1{}

\DeclareRobustCommand{\gobblefive}[5]{}

\makeatletter
\def\part{\@startsection{part}{1}%
\z@{.7\linespacing\@plus\linespacing}{.5\linespacing}%
{\large\scshape\centering}}
\makeatother

\begin{document}

\title[Classical tensor categories]{On classical tensor categories attached to the irreducible representations of the General 
Linear Supergroups $GL(n\vert n)$}
\author{{\rm Th. Heidersdorf, R. Weissauer}}

\address{T. H.: Max-Planck Institut f\"ur Mathematik, Bonn}
\address{T.H.: Mathematisches Institut Universit\"at Bonn}
\email{heidersdorf.thorsten@gmail.com} 
\address{R. W.: Mathematisches Institut, Universit\"at Heidelberg}
\email{weissaue@mathi.uni-heidelberg.de}

\date{}

\begin{abstract} We study the quotient of $\mathcal{T}_n = Rep(GL(n|n))$ by the tensor ideal of negligible morphisms. If we consider the full subcategory $\mathcal{T}_n^+$ of $\mathcal{T}_n$ of indecomposable summands in iterated tensor products of irreducible representations up to parity shifts, its quotient is a semisimple tannakian category $Rep(H_n)$ where $H_n$ is a pro-reductive algebraic group. We determine the $H_n$ and the groups $H_{\lambda}$ corresponding to the tannakian subcategory in $Rep(H_n)$ generated by an irreducible representation $L(\lambda)$. This gives structural information about the tensor category $Rep(GL(n|n))$, including the decomposition law of a tensor product of irreducible representations up to summands of superdimension zero. Some results are conditional on a hypothesis on $2$-torsion in $\pi_0(H_n)$.
\end{abstract}


\maketitle
\vspace{-1.3cm}

\setcounter{tocdepth}{1}

\tableofcontents



\section{ Introduction}


\subsection{Semisimple quotients} The categories of finite dimensional representations $\mathcal{T}_{m|n}$ of the general linear supergroups
$GL(m\vert n)$ over an algebraically closed field $k$ of characteristic zero are abelian tensor categories, where representations in  $\mathcal{T}_{m|n}$ are always understood to be algebraic.  
However, contrary to the classical case of the general linear groups $GL(n)$ these categories are not semisimple. Whereas the tensor product $V \otimes V$, $V \simeq k^{m|n}$, is completely reducible, this is no longer true for the tensor product $\A = V \otimes V^{\vee}$. Indeed $\A$ defines the indecomposable adjoint representation of $GL(n\vert n)$ in the case $m=n$, hence admits a trivial one dimensional subrepresentation defined by the center and a trivial one dimensional quotient representation defined by the supertrace. In contrast to the classical case the supertrace is trivial on the center, and $\A$ is indecomposable with three irreducible Jordan-Hoelder factors $1,S^1,1$ with the superdimensions $1, -2, 1$ respectively defined by the filtration ${\frak z} \subseteq \mathfrak{sl}(n\vert n) \subseteq \mathfrak{gl}(n\vert n)$, where ${\frak z}$ denotes the center of $\mathfrak{gl}(n\vert n)$.

\medskip
Although the irreducible representations
of $GL(m\vert n)$ can be classified by highest weights similarly to the classical case,
this implies that the tensor product of irreducible representations is in general far from being
completely reducible. In fact Weyl's unitary trick fails in the superlinear setting. While the structure of $\mathcal{T}_{m|n}$ as an abelian category is now well understood \cite{Brundan-Stroppel-4}, its monoidal structure remains mysterious.

\medskip
The perspective of this article is that, in order
to restore parts of the classical picture, two finite dimensional representations $M$ and $M'$ of $GL(m|n)$
should not be distinguished, if there exists an isomorphism $$M \oplus N \cong M' \oplus N'$$
where $N$ and $N'$ are negligible modules. Here we use the notion that a finite dimensional module is said to be negligible if it is a direct sum of indecomposable modules
whose superdimensions are zero. A typical example of a negligible module 
is the indecomposable adjoint representation $\A$ in the case $m=n$. To define this precisely we divide our category $\mathcal{T}_{m|n}$ by the tensor ideal $\mathcal{N}$ \cite{Andre-Kahn} of negligible morphisms. The quotient is a semisimple abelian tensor category. By a fundamental result of Deligne it is equivalent to the representation category of a pro-reductive supergroup $G^{red}$ \cite{Heidersdorf-semisimple}.  

\medskip
Taking the quotient of a non-semisimple tensor category by objects of categorial dimension 0 has been studied in a number of different cases. A well-known example is the quotient of the category of tilting modules by the negligible modules (of quantum dimension 0) in the representation category of the Lusztig quantum group $U_q(\mathfrak{g})$ where $\mathfrak{g}$ is a semisimple Lie algebra over $k$ \cite{Andersen-Paradowski} \cite{Bakalov-Kirillov}. The modular categories so obtained have been studied extensively in their applications to the 3-manifold invariants of Reshetikhin-Turaev. In \cite{Jannsen} Jannsen proved that the category of numerical motives as defined via algebraic correspondences modulo numerical equivalence is an abelian semisimple category. It was noted by Andr\'e and Kahn \cite{Andre-Kahn} that taking numerical equivalence amounts to taking the quotient by the negligible morphisms. Jannsen's theorem has been generalized to a categorical setting by \cite{Andre-Kahn}. In particular they study quotients of tannakian categories by the ideal of negligible morphisms. Recently Etingof and Ostrik \cite{Etingof-Ostrik} studied semisimplifactions with an emphasis on finite tensor categories.

\medskip
A general study of $Rep(G)/\mathcal{N}$, where $G$ is a supergroup scheme, was initiated in \cite{Heidersdorf-semisimple} where in particular the reductive group $G^{red}$ given by $Rep(G^{red}) \simeq Rep(GL(m|1))/\mathcal{N}$ was determined. This example is rather special since $Rep(GL(m|1))$ has tame representation type. One can always assume $m\geq n$. Note for $m\geq n \geq 2$  the problem of classifying irreducible representations of $G^{red}$ is wild \cite{Heidersdorf-semisimple}. Therefore one should not study the entire quotient $\mathcal{T}_{m|n}/\mathcal{N}$, but rather pass to a suitably small tensor subcategory in $\mathcal{T}_{m|n}$. In our case we consider for this the Karoubi envelope of the irreducible objects $Rep(G)$, which for convenience are only considered
up to suitable parity shift. The image $\overline{\mathcal{T}}_{m|n}$ of this subcategory in $Rep(G)/\mathcal{N}$
defines a tannakian category, and the aim of this paper is to determine its Tannaka
group $H_n$ in the cases $G=Gl(n\vert n)$. Indeed, as we show in \cite{HW-ss}, the computation of the corresponding Tannaka group $H_{m|n}$ in the case $G=Gl(m\vert n)$ can be reduced to the case $m=n$ besides an additional factor $GL(m-n)$ that appears in the Tannaka group $H_{m|n}$ which arises from the Tannakian group of the Tannakian 
subcategory generated by the covariant tensor representations.  
In this sense the major complications of the constructions arise in the case $m=n$. 
So, for simplicity, we restrict ourselves to the case $m=n$ and postpone the additional
combinatorial arguments that are necessary for the general
case $m>n$ to the paper \cite{HW-ss}.

\subsection{The Tannaka category $\overline{\mathcal{T}}_n$} 
For convenience we now write ${\mathcal{T}}_n$ instead of ${\mathcal{T}}_{n,n}$.
To define the karoubienne hull  of the irreducible representations 
in $\mathcal T_n$, we work with the tensor
subcategory ${\mathcal T}_n^+$ generated by the irreducible representations of nonnegative superdimension in the following sense. Up to parity shift 
the irreducible representations $L=L(\lambda)$ of ${\mathcal T}_n$ are parametrized up to isomorphy by their highest weights $\lambda$. We define an equivalence relation
on the set of highest weights, such that $\lambda$ and $\lambda'$ are called equivalent
if $L(\lambda)$ or its Tannaka dual $L(\lambda)^\vee$ is isomorphic
to $Ber^r \otimes L(\lambda')$ for some power $Ber^r, r\in \mathbb Z$ of the Berezin determinant representation $Ber$  in ${\mathcal T_n}$.   Let $Y^+_0(n)$ denote  
the set of equivalence classes $\lambda/\!\sim$ of maximal
atypical weights. The cases where $L(\lambda)^\vee \cong Ber^r \otimes L(\lambda)$ holds
are called (SD)-cases if $\dim(L(\lambda)) >1$. The remaining cases
are called (NSD)-cases.
Notice that an irreducible representation $L(\lambda)$
of $GL(n\vert n)$ can be replaced by a  parity shift $X_\lambda$ of $L(\lambda)$ so that the superdimension $\sdim(X_\lambda)$ 
becomes $\geq 0$. For maximal atypical representations $L(\lambda)=[\lambda_1,...,\lambda_n]$ this is well defined since  the superdimension is not zero. But of course this is ambiguous for irreducible representations
of $GL(n\vert n)$ with $\sdim(L)=0$, i.e. for the irreducible representations that are not maximal atypical. But these representations are negligible in the sense above. Thus we may consider only objects that are retracts of iterated tensor products of direct sums of maximal atypical irreducible representations $X_\lambda$
of $GL(n\vert n)$ satisfying $\sdim(X_\lambda) > 0$. The tensor category thus obtained will be
baptized $\mathcal{T}_n^+$.  The full tensor subcategory $\mathcal{T}_n^+$ of ${\mathcal{T}}_n$ has more amenable properties than the full category ${\mathcal{T}}_n =Rep(GL(n\vert n))$. To motivate this, let us compare it with the tensor
category of finite dimensional algebraic representations $Rep(G)$ of an arbitrary algebraic group $G$ over $k$. 
In this situation the tensor subcategory generated by irreducible representations is semisimple\footnote{In \cite[p.231, 1.22ff]{Kr-Wei} it was forgotten to mention the important passage to the tensor subcategory generated by simple objects. The corresponding  statement is false without it as kindly pointed out by Y. {Andr\'e}.} and can be identified with the tensor category of the maximal reductive quotient of $G$. The tensor category 
$\mathcal{T}_n^+$ however  is not a semisimple tensor category in general. To make it semisimple we proceed as follows:

\medskip
Let $\overline{\mathcal{T}}_{n}$ denote the quotient category of $\mathcal{T}_n^+$ obtained
by killing the negligible morphisms in the tensor ideal ${\mathcal N}$ and hence in particular all neglegible objects, i.e. $$\overline{\mathcal{T}}_{n} \cong
\mathcal{T}_n^+/\mathcal{N}\ .$$ In order to analyze these categories, we work inductively using the cohomological tensor functors $DS: \mathcal{T}_n \to \mathcal{T}_{n-1}$ of \cite{Heidersdorf-Weissauer-tensor}. We show in lemma \ref{DS-restricts} that $DS$ induces a tensor functor $\eta_n: \mathcal{T}_n^+ \to \mathcal{T}_{n-1}^+$.

\begin{thm} The categories $\overline{\mathcal{T}}_{n}$ are semisimple
tannakian categories. A fibre functor $\omega$ is provided by the
composite of functors $\eta_m$ for $1\leq m\leq n$. Their Tannaka groups $H_n$ are projective limits of reductive algebraic groups over $k$ such that there is an equivalence of tensor categories
$$\overline{\mathcal{T}}_{n} \ \cong\ Rep(H_n)\ .$$
The functor $\eta_n$ induces a closed embedding of affine group schemes $H_{n-1} \hookrightarrow H_n$ over $k$ such that $\eta_n: Rep(H_n) \to Rep(H_{n-1})$, defined by $DS$ on objects, can be identified  with the restriction functor for this group scheme embedding. 
\end{thm}

The Tannaka group $H_n$ is subgroup of the product of Tannaka groups $H_\lambda$, where $\lambda$ runs over the set of equivalence classes $\lambda/\sim$ of maximal atypical highest weights $\lambda$. 
By definition $H_\lambda$ is the Tannaka group of
the tannakian subcategory ${\mathcal T}_\lambda$ that is generated by the simple object $X_\lambda$ in $\overline{\mathcal T}_n$. For the twisted Berezin $X_\lambda=B$ the group $H_\lambda$
is isomorphic to the multiplicative group ${\mathbb G}_m$ whose characters correspond to
the irreducible one dimensional representations in ${\mathcal T}_n^+$, the powers 
$B^r$ of $B$. In general
$H_\lambda$ may be considered as an algebraic subgroup  of the general linear
group $GL(V_\lambda)$ of the finite dimensional $k$-vectorspace $V_\lambda=\omega(X_\lambda)$ defined by the fibre functor $\omega$. 
Note that $\dim(V_\lambda) = \sdim(X_\lambda)$ and this value 
is bounded by $n!$. The restriction of the determinant character of $GL(V_\lambda)$
will be denoted $\det_\lambda$. In theorem \ref{thm1.4} we show that $\det_\lambda$
is a power $B^{\ell(\lambda)}$ of the character defined by $B$.

\begin{thm} The character $\det_\lambda$ of the group
 $H_\lambda$ is represented by the image of $\det(X_{\lambda}) = \Lambda^{\sdim(X_{\lambda})}(X_{\lambda})$ in the Tannaka category 
$\overline{\mathcal T}_n$. In ${\mathcal T}_n^+$ one has
$$  \det(X_{\lambda}) \, \cong \, B^{\ell(\lambda)} \, \oplus \, N $$
for some negligible object $N$ with the integer 
$$  \ell(\lambda) \ =\  n^{-1} \sdim(X_\lambda)\, D(\lambda) \ ,        \ $$
where $D(\lambda)$ is explicitely described by the weight $\lambda$ in section \ref{picardgroup}.
\end{thm}

In the (SD)-cases the isomorphism $L(\lambda)^\vee \cong Ber^r \otimes L(\lambda)$ 
defines a nondegenerate pairing $X_\lambda \otimes X_\lambda \to B^r$. It induces 
a nondegenerate $k$-bilinear pairing on $V_\lambda$ such that $H_\lambda$
becomes a subgroup of its group of similitudes. In appendix \ref{sec:pairings} we determine the
parity of this pairing.

\begin{thm}
In the (SD)-cases the parity $\varepsilon(X_\lambda)$
of the invariant pairing $\langle .,.\rangle$ on $V_\lambda$ defining $H_\lambda$  is 
$$   \varepsilon(X_\lambda) = \varepsilon(\lambda_{basic}) \ $$
where $\varepsilon(\lambda_{basic})= (-1)^{\sum_{i=1}^n \mu_i}$ if
$L(\lambda_{basic}) \cong [\mu_1,...,\mu_n]$ 
is the irreducible basic representation associated to $L(\lambda)$.
\end{thm}

For each maximal atypical weight $\lambda$ we define characters
$$   \mu_\lambda: H_\lambda  \to {\mathbb G}_m  $$
of the Tannaka groups $H_\lambda$ as follows: 
First suppose that $X_\lambda$
is not isomorphic to a power of $B$.
Then, in the (NSD)-cases $\mu_\lambda$ is defined to be  
$\det_\lambda$. 
In the (SD)-cases $\mu_\lambda$ is defined
as the restriction of the similitude character to $H_\lambda$. The similitude
character of $H_\lambda$ is defined by an object in ${\mathcal T}_\lambda$ that
is isomorphic to the image of $B^r$. To make these characters well defined
notice the following:
For the twisted Berezin the associated Tannaka group $H_\lambda$
is isomorphic to the multiplicative group ${\mathbb G}_m$, whose characters correspond to
the irreducible one-dimensional representations in ${\mathcal T}_n^+$, the powers 
$B^r$ of $B$. Any tensor functor
between tannakian categories (compatible with the fibre functors) induces a group homorphism between the Tannaka groups that is uniquely determined up to conjugacy by the functor. This observation, applied to the tannakian subcategories
of ${\mathcal T}_\lambda$ that are generated by $\det(X_\lambda)$ resp. $B^r$, 
together with the fact that inner automorphisms of $\mathbb G_m$ are trivial,  shows that
all characters $\mu_\lambda$ are uniquely defined once an 
isomorphism $\mu_\lambda$ for $X_\lambda \cong B$ between the Tannaka group $H_B:=H_\lambda$ 
and the multiplicative group ${\mathbb G}_m$ has been chosen.
We fix such an isomorphism, denoted $\mu_B$ in the following. 

\medskip
This being said, a conjectural description of the structure of $H_n$
can be given as follows:

\begin{conj} \label{MainConj} (1) $H_n$ is the 
subgroup of the infinite product $$ \prod_{\lambda/ \sim\ \in\, Y^+_0(n)} H_\lambda \ ,$$
defined by the elements $ h=(h_\lambda)_{\lambda/\sim}$ that satisfy 
$   \mu_\lambda(h_\lambda) = \mu_B(h_B) $ for all $\lambda$. 

\medskip\noindent
(2) For the (NSD)-cases the group  $H_\lambda$, considered as a subgroup of $GL(V_\lambda)$,  is equal to $GL(V_\lambda)$ if $\mu_\lambda\neq 1$ holds and is equal to the subgroup $SL(V_\lambda)$ otherwise. Recall, $\mu_\lambda=\det_\lambda$ holds in that case.

\medskip\noindent
(3) Otherwise $H_\lambda$, considered as a
subgroup of the group of similitudes of the pairing $\langle .,. \rangle$ on $V_\lambda$,   
is equal to the connected component of the similitude group if $\mu_\lambda\neq 1$ holds and is otherwise equal to the kernel of the similtude homomorphism $\mu_\lambda$ on this connected component. 
\end{conj}

The parity shift of the twisted Berezin $B=\Pi^{n}(Ber)$   is an invertible object of the Tannaka category $\overline{\mathcal T}_n$,
i.e. an object $I$ of superdimension $1$ such that $I \otimes I^\vee \cong {\bf 1}$
holds. Conjecture \ref{MainConj} implies that the group $Pic(\overline{\mathcal T}_n)$ of isomorphism classes of invertible objects in $\overline{\mathcal T}_n$ is generated
by the twisted Berezin $B$. For this notice that in \ref{MainConj} (1) the kernel of the projection from  $H_n\subset \prod_{\lambda/\sim} H_\lambda$ in the product to the $L(\lambda)=B$-component is a connected semisimple profine groupscheme by \ref{MainConj} (2) and (3), hence only admits trivial characters.
We also prove the converse.

\begin{thm} A necessary and sufficient condition that
Conjecture \ref{MainConj} holds is that $Pic(\overline{\mathcal T}_n)$
is generated by the twisted Berezin $B$.
\end{thm}

So the main obstacle for the proof conjecture \ref{MainConj} turns out to be the structure of $Pic({\mathcal T}_\lambda)$. Conjecture \ref{MainConj} is also equivalent to the assertion that  exceptional 
(SD)-cases in the sense of the next theorem do not occur.

\begin{thm}\label{regul} For all maximal atypical highest weights
$\lambda$ there is a homomorphism $\nu: Pic({\mathcal T}_\lambda)\to \mathbb Z$
whose kernel is a two-torsion group $\mu_2^k$ of rank $k=k(\lambda)$ where
$0\leq k\leq 2$. In the (so called) regular case where $k(\lambda)=0$, the group
Tannaka group $H_\lambda$ of ${\mathcal T}_\lambda$ is the one described
in conjecture \ref{MainConj}.  
\end{thm}
 
\begin{cor}
The Picard group $Pic(\overline{\mathcal T}_n)$ is a direct product of $\mathbb Z$
and a 2-power torsion group.
\end{cor}


\subsection{The exceptional cases} The nonregular cases in the sense of the last theorem \ref{regul}
will be called exceptional cases. For these exceptional $\lambda$ we show that there exists a subgroup $\tilde H_\lambda$ of $H_\lambda$ of index two 
$$   0 \to   \tilde H_\lambda  \to H_\lambda  \to \mu_2 \to 0  \ .$$
The restriction of the irreducible representation $V_\lambda$ 
of $H_\lambda$ to the subgroup $\tilde H_\lambda$ decomposes into
a direct sum $W_\lambda \oplus W_\lambda^\vee$ of two irreducible faithful
nonisomorphic representations of $\tilde H_\lambda$ on orthogonal Lagrangian subspaces $W_\lambda$ and $W_\lambda^\vee$ of the metric space $(V_\lambda,\langle . ,\rangle)$; see theorem \ref{ThisIsMain} and section \ref{exc-sd}.
In this way we can view $\tilde H_\lambda$ as a subgroup of $GL(W_\lambda)$, and we show that the following holds
$$    SL(W_\lambda) \subseteq \tilde H_\lambda \subseteq GL(W_\lambda) \ .$$
Finally let $G_\lambda$ denote the derived group of the connected component
$H_\lambda^0$ of $H_\lambda$, and let similar $G_n$ denote the derived group of the connected component $H_n^0$. 
For the derived connected subgroup $G_n$ of $H_n$ we prove the following result.

\begin{thm}\label{weakversion} (1)  $G_n$ is isomorphic to the 
infinite product $$ \prod_{\lambda/ \sim\ \in\, Y^+_0(n)} G_\lambda \ .$$
The groups $G_\lambda$ are isomorphic to $SL(V_\lambda)$, $SO(V_\lambda)$,
$Sp(V_\lambda)$ in the (NSD) resp. the even or odd regular (SD)-cases and 
they are isomorphic to $SL(W_\lambda)$, for a Lagrangian subspace $W_\lambda$
of $V_\lambda$, in the exceptional (SD)-cases. 

\medskip
(2) If $\lambda$ is not an exceptional (SD)-case, the groups $H_\lambda$ are as
described in conjecture \ref{MainConj}. Furthermore the analog of conjecture \ref{MainConj} holds for the Tannaka group generated by the simple objects $X_\lambda$ for which
which $\lambda$ does not belong to an exceptional (SD)-case.
\end{thm}

In view of theorem \ref{weakversion} our conjecture \ref{MainConj} is hence equivalent to the conjectural nonexistence of exceptional (SD)-cases.

\medskip
Reformulating these statements for the category of representations of $GL(n\vert n)$, what we have achieved is

\begin{itemize}

\item a (partial) description of the decomposition law of tensor products of irreducible representations into indecomposable modules up to negligible indecomposable summands; and 

\item a classification (in terms of the highest weights of $H_{\lambda}$ and $H_{\mu}$) of the indecomposable modules of non-vanishing superdimension in iterated tensor products of $L(\lambda)$ and $L(\mu)$.

\end{itemize} 

To determine this decomposition it suffices to know the Clebsch-Gordan coefficients
for the classical simple groups of type $A,C,D$. Notice
that $\dim(V_\lambda)$ is always even in the (SD)-cases, hence simple groups
of type $B$ do not occur.
Furthermore the superdimensions of the indecomposable summands are just the  dimensions of the corresponding irreducible summands of the tensor products in $Rep_k(H_n)$.
Without this,  to work out any such decomposition is rather elaborate. For the case $n=2$ see \cite{Heidersdorf-Weissauer-GL-2-2}. In fact the knowledge of the Jordan-H\"older factors usually gives too little
information on the indecomposable objects itself. In the (NSD) and the odd (SD)-case it is enough for these two applications to know the connected derived group $G_{\lambda}$ since the restriction of any irreducible representation of $H_{\lambda}$ to $G_{\lambda}$ stays irreducible. Therefore these results hold unconditionally in these cases. In the even (SD)-case we need the finer (but conjectural) results of section \ref{sec:conj-pic} to see that $H_{\lambda}$ is connected. We refer the reader to example \ref{example-sp(6)} and section \ref{sec:physics} for some examples.

\subsection{Relation to physics}

Part of the motivation for our computations of the Tannaka groups $H_n$ comes from  the real supergroups $G=SU(2,2\vert N)$ which are covering groups of
the super conformal groups $SO(2,4\vert N)$. 
The complexification ${\frak g}$ of $Lie(G)$ is isomorphic to the complex Lie superalgebras $\mathfrak{sl}(4\vert N)$. The finite dimensional representations of $\frak g$ are hence related to those 
of the Lie superalgebras $\mathfrak{gl}(n\vert n)$ for $n\leq 4$. 
Since the complexification ${\frak g}$ defines complex supervector fields
on four dimensional Minkowski superspace $M$ and compactifications
of it, these Lie superalgebras play a role in string theory and 
the AdS/CFT correspondence. 

\medskip\noindent
For supersymmetric fields $\psi$ on $M$ with values in a finite dimensional representation $L$ of ${\frak g}$, the Feynman integrals of conformal theories 
are computed from tensor contractions and superintegration. These can be considered as contractions between tensor products of fields. The computations will require the analysis of higher tensor products $L^{\otimes r}\otimes (L^\vee)^{\otimes s}$ and generalizations of Fierz rules. The results will strongly depend on the rules of the underlying tensor categories ${\mathcal T}_n$.  Since it seems reasonable
to consider not only fields of superdimension zero, but besides the constant representation also such with values in maximal atypical representations $L$ of $\frak g$, our study of tensor categories may be a little step into this direction. To look at examples for $n\leq 4$ we now replace the groups $H_\lambda$ by their compact inner forms $H_\lambda^c$
and ${\mathbb G}_m$ by $U(1)$, to make things look more familiar to
physicists.  Indeed notice that the tensor categories $Rep_k(H_\lambda)$ and $Rep_{\mathbb C}(H_\lambda^c)$ are equivalent.

\medskip  
For $\mathfrak{gl}(n\vert n)$ and $n\leq 4$ there only exist   finitely many 
isomorphism classes of quotient groups $H_\lambda$ of the tannakian groups $H_n$ of $\overline{\mathcal T}=Rep(H_n)$. Besides $U(1)$ that corresponds to the twisted Berezin, the smallest such groups are $SU(2)$ and $SU(3)$  
related to representations denoted $L=S^1$ and $L=S^2$. 
If we pass to $\mathfrak{sl}(n\vert n)$, for $n\leq 3$ these are the only groups except for $Sp^c(6)$ in the case $n=3$. For the more involved discussion of the case $n=4$ we refer to section \ref{sec:physics}  and example \ref{exn3}. 
One may ask whether the appearence of the groups $U(1), SU(2), SU(3)$ here is a mere accident, or whether there does exist some connection with the symmetry groups arising in the standard model of elementary particle physics? Looking for relations between internal symmetries and supersymmetry has a long history going back to the Coleman-Mandula theorem \cite{CM}, so this may be of interest. A possible relation could be the following:

\medskip
If in such a theory, for some mysterious physical reasons, the tensor product contributions to the Feynman integrals from direct summands of $L^{\otimes r}\otimes (L^\vee)^{\otimes s}$ of superdimension zero would be relatively small in a certain energy range due to supersymmetry cancellations, 
then to first order they would be  negligible.
Hence a physical observer might come up with the impression that 
the underlying rules of symmetry are imposed by the invariant theory of the quotient tensor category $\overline{\mathcal T}_n\! =\! Rep(H_n)$ instead of $\mathcal T_n$, i.e. the tensor categories  that are obtained by ignoring negligible indecomposable summands of superdimension zero. 
Thus $H_n$ would appear as an internal symmetry group of the theory
in an approximate sense. 
Of course, any such speculation is highly tentative for various reasons: Computations along such lines will probably be extremely involved. Fields with values in maximal atypical representations $V$ may produce ghosts in the associated 
infinite dimensional representations of ${\frak g}$. In other words, such field theories may a priori not be superunitary and it is unclear whether the passage to the cohomology groups for operators like $DS$ or the Dirac operator $H_D$ \cite{Heidersdorf-Weissauer-tensor}, breaking the conformal symmetry, would suffice to get rid of ghosts.

\subsection{Structure of the article} \label{intro-structure} Our main tool are the cohomological tensor functors $DS: \mathcal{T}_n \to \mathcal{T}_{n-1}$ of \cite{Heidersdorf-Weissauer-tensor}. In the main theorem of \cite[Theorem 16.1]{Heidersdorf-Weissauer-tensor} we calculate $DS(L(\lambda))$. In particular $DS(L(\lambda))$ is semisimple and multiplicity free. We show in lemma \ref{DS-restricts} that $DS$ induces a tensor functor $DS: \mathcal{T}_n^+ \to \mathcal{T}_{n-1}^+$ and by lemma \ref{DS-respects-negligibles} one can construct a tensor functor on the quotient categories \[ \eta: \mathcal{T}_n^+/\mathcal{N} \to \mathcal{T}_{n-1}^+/\mathcal{N}. \]  This seemingly minor observation is one of the crucial points of the proof since it allows us to determine the groups $H_n$ and $G_n$ inductively. We also stress that it is not clear whether $DS$ naturally induces a functor between the quotients $\mathcal{T}_n^+/\mathcal{N}$ and $\mathcal{T}_{n-1}^+/\mathcal{N}$ on the level of morphisms. In fact, if one enlarges $\mathcal{T}^+$ to the larger category $\mathcal{T}^{ev}$ of section \ref{sec:determinant}, $DS$ does not preserve negligible morphisms. The $DS$ functor however agrees with the functor $\eta$ on objects. The quotient $ \mathcal{T}_n^+/\mathcal{N}$ is equivalent to the representation category $Rep(H_n)$ of finite-dimensional representations of a pro-reductive group. By the deep and powerful theorem \ref{deligne-theorem} of Deligne the induced $DS$ functor determines an embedding of algebraic groups $H_{n-1} \hookrightarrow H_n$ and the functor $DS$ is the restriction functor with respect to this embedding. 

\medskip
Hence the main theorem of \cite{Heidersdorf-Weissauer-tensor} tells us the branching laws for the representation $V_{\lambda}$ with respect to the embedding $H_{n-1} \hookrightarrow H_n$. Our strategy is to determine the groups $H_n$ or $G_n$ inductively using the functor $DS$. For $n = 2$ we need the explicit results of \cite{Heidersdorf-Weissauer-GL-2-2} to give us the fusion rule between two irreducible representations and we describe the corresponding Tannaka group in lemma \ref{H_2-for-GL-2-2}. Starting from the special case $n=2$ we can proceed by induction on $n$. For this we use the embedding $H_{n-1} \hookrightarrow H_n$ along with the known branching laws and the classification of small representations due to Andreev, Elashvili and Vinberg \cite{Andreev-Vinberg-Elashvili} which allows to determine inductively the connected derived groups $G_n = (H_n^0)_{der}$ for $n\geq 3$; see section \ref{proof-derived}. The passage to the connected derived group means that we have to deal with the possible decomposition of $V_{\lambda}$ when restricted to $G_n$. In order to determine $G_n$ we first determine the connected derived groups $G_{\lambda}$ corresponding to the tensor subcategory generated by the image of $L(\lambda)$ in $\overline{\mathcal{T}}_n$ in theorem \ref{Tannakagroup}. Roughly speaking the strategy of the proof is the following: We use the inductively known situation for $G_{n-1}$ to show that for sufficiently large $n$  the rank and the dimension of $G_{\lambda}$ is large compared to the dimension of $V_{\lambda}$, i.e. $V_{\lambda}$ or any of its irreducible constituents in the restriction to $G_{\lambda}$ should be \textit{small} in the sense of \cite{Andreev-Vinberg-Elashvili}. We refer to section \ref{proof-derived} for more details on the proof.

The final sections are devoted to the determination of $H_{\lambda}$ and $Rep(H_n)$. We determine the groups $H_{\lambda}$ in section \ref{sec:conjecture}. We split this determination into three cases: NSD, regular SD and exceptional SD. The crucial tool here is the determination of the determinant $det(X_{\lambda})$. This determinant is computed in the later sections \ref{picardgroup} and \ref{sec:determinant}.

In section \ref{sec:conj-pic} we conjecture a stronger structure theorem, namely that there are no exceptional SD-cases. We describe various conditions which are equivalent to this statement. We end the article with some low-rank cases and a discussion of cases of potential physical relevance.

\medskip
We have outsourced a large number of technical (but necessary) results to the appendices \ref{equivalences}, \ref{sec:pairings} and \ref{technical-lemmas} as to not distract the reader too much from the structure of the arguments. Appendix \ref{appendix-1} discusses some evidence for our conjectures.


\subsection{Outlook}
For the general $\mathcal{T}^+_{m|n}$-case (where $m \geq n$) recall that every maximal atypical block in $\mathcal{T}_{m|n}$ is equivalent to the principal block of $\mathcal{T}_{n|n}$. We fix the standard block equivalence due to Serganova and denote the image of an irreducible representation $L(\lambda)$ under this equivalence by $L(\lambda^0)$. 

\begin{thm} Suppose that $\sdim (L(\lambda)) > 0$. Then $H_{\lambda} \cong GL(m-n) \times H_{\lambda^0}$ and $L(\lambda)$ corresponds to the representation $L_{\Gamma} \boxtimes V_{\lambda^0}$ of $H_{\lambda}$. Here $L_{\Gamma}$ is an irreducible representation of $GL(m-n)$ which only depends on the block $\Gamma$ (the \emph{core} of $\Gamma$ as defined by Serganova).
\end{thm}

We prove this in \cite{HW-ss}. The problem of determining the semisimplification of $Rep(G)$ can be asked for any basic supergroup. We expect that the strategy employed in this paper (induction on the rank via the $DS$ functor) serves as a blueprint for other cases, but one can certainly not expect uniform proofs or results. In fact the semisimplicity of $DS$ is now known in the $OSp$-case \cite{GH}, but fails for the $P(n)$-case \cite{Entova-Serganova-kw-p}. Note also that even in the $OSp$-case no exact analog of the structure theorem can hold since the supergroup $OSp(1|2n)$ will appear as a Tannaka group. 

  
\subsection*{Acknowledgements}
We thank the referee for an exceptionally thorough and helpful review of earlier versions of this article. The research of T.H. was partially funded by the Deutsche Forschungsgemeinschaft (DFG, German Research
Foundation) under Germany's Excellence Strategy – EXC-2047/1 – 390685813.
 
\subsection*{About the journal version} The present version is almost identical to the version that appeared in Selecta (Sel. Math., New Ser. 29, No. 3, Paper No. 34, 101 p. (2023)). The Arxiv version contains the additional appendices \ref{appendix-depth} and \label{counter-2} with sample calculations. 
  





\section{ The superlinear groups}\label{2}

Let $k$ be an algebraically closed field of characteristic zero. We adopt the notations of \cite{Heidersdorf-Weissauer-tensor}. With $GL(m|n)$ we denote the general linear supergroup and by $\mathfrak{g} = \mathfrak{gl}(m|n)$ its Lie superalgebra. A representation $\rho$ of $GL(m|n)$ is a representation of $\mathfrak{g}$ such that its restriction to $\frakg_{\bar{0}}$ comes from an algebraic representation of $G_{\bar{0}} = GL(m) \times GL(n)$. We denote by $\mathcal{T} = \mathcal{T}_{m|n}$ the category of all finite dimensional representations with parity preserving morphisms.

\subsection{The category ${\calR}$} Fix the morphism $\varepsilon: \mathbb Z/2\mathbb Z \to G_{\overline 0}=GL(n)\times GL(n)$ which maps $-1$ to the element 
$diag(E_n,-E_n)\in GL(n)\times GL(n)$ denoted $\epsilon_n$. Notice that
$Ad(\epsilon_n)$ induces the parity morphism on the Lie superalgebra $\mathfrak{gl}(n|n)$ of $G$. 
We define the abelian subcategory
$\calR = sRep(G,\varepsilon)$ of $\mathcal{T}$ as the full subcategory of all objects $(V,\rho)$ in $\mathcal{T}$  
with the property $  p_V = \rho(\epsilon_n)$; here $p_V$ denotes the parity morphism of $V$ and $\rho$ denotes the underlying homomorphism $\rho: GL(n)\times GL(n) 
\to GL(V)$ of algebraic groups over $k$.
The subcategory ${\calR}$ is stable under the dualities ${}^\vee$ and $^*$. 
For $G=GL(n\vert n)$ we usually write $\mathcal{T}_n$ instead of $\mathcal{T}$, and ${\calR}_n$ instead of $\calR$. The irreducible representations in $\calR_n$ are parametrized by their highest weight with respect to the Borel subalgebra of upper triangular matrices. A weight $\lambda=(\lambda_1,...,\lambda_n \ | \ \lambda_{n+1}, \cdots, \lambda_{2n})$ of an irreducible representation in $\calR_n$ satisfies $\lambda_1 \geq \ldots \lambda_n$, $\lambda_{n+1} \geq \ldots \lambda_{2n}$ with integer entries. The Berezin determinant of the supergroup $G=G_n$
defines a one dimensional representation $Ber$. Its weight is
is given by $\lambda_i=1$ and $\lambda_{n+i}=-1$ for $i=1,..,n$. For each representation $M \in \mathcal{R}_n$ we also have its parity shifted version $\Pi(M)$ in $\mathcal{T}_n$. Since we only consider parity preserving morphisms, these two are not isomorphic. In particular the irreducible representations in $\mathcal{T}_{n}$ are given by the $\{L(\lambda), \Pi L(\lambda) \ | \ \lambda \in X^+ \}$. The whole category $\mathcal{T}_n$ decomposes as  $\mathcal{T}_{n} = \calR_{n} \oplus \Pi \calR_{n}$ \cite[Corollary 4.44]{Brundan-Kazhdan}. For maximal atypical $\lambda$ exactly one of $L(\lambda), \Pi L(\lambda))$ has positive superdimension. We call this irreducible module $X_{\lambda}$ and $B = \Pi^n(Ber)$ in $\mathcal{T}_n^+$, for $Ber=[1,...,1]$, the twisted Berezin.

\subsection{Kac objects} We put $\mathfrak{p}_{\pm} = \g_{(0)} \oplus \g_{(\pm1)}$ for the usual $\Z$-grading $\g = \g_{(-1)} \oplus \g_{(0)} \oplus \g_{(1)}$. We consider a simple $\g_{(0)}$-module as a $\mathfrak{p}_{\pm}$-module in which $\g_{(1)}$ respectively $\g_{(-1)}$ acts trivially. We then define the Kac module $V(\lambda)$ and the anti-Kac module $V'(\lambda)$ via \[ V(\lambda)  = Ind_{\mathfrak{p}_+}^{\g} L_0(\lambda) \ , \ V'(\lambda)  = Ind_{\mathfrak{p}_-}^{\g} L_0(\lambda) \] where $L_0(\lambda)$ is the simple $\g_{(0)}$-module with highest weight $\lambda$. The Kac modules are universal highest weight modules. $V(\lambda)$ has a unique maximal submodule $I(\lambda)$ and $L(\lambda) = V(\lambda)/I(\lambda)$ \cite[Proposition 2.4]{Kac-Rep}. We denote by $\calC^+$ the tensor ideal of modules with a filtration by Kac modules in $\calR_n$ and by $\calC^-$ the tensor ideal of modules with a filtration by anti-Kac modules in $\calR_n$. 

\subsection{Equivalence classes of weights} Two irreducible representations $M$, $N$ in $\mathcal{T}$ are said to be
equivalent $M \sim N$, if either $M \cong Ber^r \otimes N$ or $M^\vee \cong Ber^r \otimes N$ holds for some $r\in \mathbb Z$. This obviously defines an equivalence relation on the set of isomorphism classes of irreducible representations of $T$. A self-equivalence of $M$ is given by an isomorphism $f: M \cong Ber^r \otimes M$ (which implies $r=0$ and $f$ to be a scalar multiple of the identity) respectively an isomorphism $f: M^\vee \cong Ber^r \otimes M$.
If it exists, such an isomorphism uniquely determines $r$ and is unique up to a scalar
and we say $M$ is of type (SD). Otherwise we say $M$ is of type (NSD). 
The isomorphism $f$ can be viewed as a nondegenerate $G$-equivariant bilinear form
$$   M \otimes M  \to  Ber^r  \ ,$$
which is either symmetric or alternating. So we distinguish between the cases ($\text{SD}_{\pm}$). 



\section{ Weight and cup diagrams}\label{3}

\subsection{Weight diagrams and cups} Consider a weight \[ \lambda=(\lambda_1,...,\lambda_n | \lambda_{n+1}, \cdots, \lambda_{2n}).\] Then $\lambda_1 \geq ... \geq \lambda_n$ and $\lambda_{n+1} \geq ... \geq \lambda_{2n}$ are integers, and every $\lambda\in {\mathbb Z}^{2n}$ satisfying these inequalities occurs as the highest weight of an irreducible representation $L(\lambda)$. The set of highest weights will be denoted by $X^+=X^+(n)$. Following \cite{Brundan-Stroppel-4} to each highest weight $\lambda\in X^+(n)$  we associate  two subsets of cardinality $n$ of the numberline $\mathbb Z$
\begin{align*} I_\times(\lambda)\ & =\ \{ \lambda_1  , \lambda_2 - 1, ... , \lambda_n - n +1 \} \\
 I_\circ(\lambda)\ & = \ \{ 1 - n - \lambda_{n+1}  , 2 - n - \lambda_{n+2} , ... ,  - \lambda_{2n}  \}. \end{align*}

We now define a labeling of the numberline $\mathbb Z$.
The integers in $ I_\times(\lambda) \cap I_\circ(\lambda) $ are labeled by $\vee$, the remaining ones in $I_\times(\lambda)$ resp. $I_\circ(\lambda)$ are labeled by $\times$ respectively $\circ$. All other integers are labeled by $\wedge$. 
This labeling of the numberline uniquely characterizes the weight vector $\lambda$. If the label $\vee$ occurs $r$ times in the labeling, then $r=atyp(\lambda)$ is called the {\it degree of atypicality} of $\lambda$. Notice $0 \leq r \leq n$, and for $r=n$ the weight $\lambda$ is called
{\it maximal atypical}. A weight is maximally atypical if and only if $\lambda_i = - \lambda_{2n-i+1}$ for $i=1,\ldots,n$ in which case we write $$L(\lambda) = [\lambda_1,\ldots,\lambda_n]\  .$$

To each weight diagram we associate a cup diagram as in \cite{Brundan-Stroppel-1} \cite{Heidersdorf-Weissauer-tensor}. The outer cups in a cup diagram define the sectors of the weight as in \cite{Heidersdorf-Weissauer-tensor}. We number the sectors from left to right $S_1$, $S_2$, $\ldots$, $S_k$.

\begin{example}\label{ex-cup-diag} Consider the (maximal atypical) irreducible representation $[7,7,4,2,2,2]$ of $GL(6|6)$. Its associated weight and cup diagram have two sectors:

\begin{center}
\medskip
 
 \scalebox{0.7}{
\begin{tikzpicture}
\foreach \x in {7,6,2,-1,-2,-3} 
     \draw[very thick] (\x-.1, .1) -- (\x,-0.1) -- (\x +.1, .1);
\foreach \x in {-4,0,1,3,4,5,8,9,10,11} 
     \draw[very thick] (\x-.1, -.1) -- (\x,0.1) -- (\x +.1, -.1);
%

\draw (-4,-0.5) node {-4};
\draw (-3,-0.5) node {-3};
\draw (-2,-0.5) node {-2};
\draw (-1,-0.5) node {-1};
\draw (0,-0.5) node {0};
\draw (1,-0.5) node {1};
\draw (2,-0.5) node {2};
\draw (3,-0.5) node {3};
\draw (4,-0.5) node {4};
\draw (5,-0.5) node {5};
\draw (6,-0.5) node {6};
\draw (7,-0.5) node {7};
\draw (8,-0.5) node {8};
\draw (9,-0.5) node {9};
\draw (10,-0.5) node {10};
\draw (11,-0.5) node {11};


\draw[very thick] [-,black,out=90, in=90](7,+0.2) to (8,+0.2);
\draw[very thick] [-,black,out=90, in=90](6,+0.2) to (9,+0.2);
\draw[very thick] [-,black,out=90, in=90](2,+0.2) to (3,+0.2);
\draw[very thick] [-,black,out=90, in=90](-1,+0.2) to (0,+0.2);
\draw[very thick] [-,black,out=90, in=90](-2,+0.2) to (1,+0.2);
\draw[very thick] [-,black,out=90, in=90](-3,+0.2) to (4,+0.2);


\end{tikzpicture} }
\smallskip

\text{The cup diagram of $\lambda$}
\end{center}

\end{example}

\subsection{Important invariants} The segment and sector structure of a weight diagram is completely encoded by the positions of the $\vee$'s. Hence any finite subset of $\bbZ$ defines a unique weight diagram in a given block. We associate to a maximal atypical highest weight the following  invariants: 

\begin{itemize}
\item the type (SD) resp. (NSD),
\item the number $k=k(\lambda)$ of sectors of $\lambda$, 
\item the sectors $S_\nu=(I_\nu,K_\nu)$ from left to right (for $\nu=1,...,k$),
\item the ranks $r_\nu = r(S_\nu)$, so that $\# I_\nu = 2r_\nu$, 
\item the distances $d_\nu$ between the sectors (for $\nu=1,...,k-1$), 
\item and the total
shift factor $d_0=\lambda_n + n-1$. 
\end{itemize}
If convenient, $k$ sometimes may also denote the number of segments, 
but hopefully no confusion will arise from this. 

\medskip
A maximally atypical weight $[\lambda]$ is called basic if $(\lambda_1,...,\lambda_n)$
defines a decreasing sequence $\lambda_1 \geq \cdots \geq \lambda_{n-1} \geq \lambda_n=0$
with the property $n-i \geq \lambda_i$ for all $i=1,...,n$. The total number
of such {\it basic weights} in $X^+(n)$ is the Catalan number $C_n$. Reflecting the graph of such a sequence
$[\lambda]$ at the diagonal, one obtains another basic weight $[\lambda]^*$. By \cite[Lemma 21.4]{Heidersdorf-Weissauer-tensor} a basic weight $\lambda$ is of  type (SD) if and only if $[\lambda]^* = [\lambda]$ holds. 
To every maximal atypical highest 
weight $\lambda$ is attached a unique maximal atypical highest 
weight $\lambda_{basic}$ 
$$  \lambda \mapsto \lambda_{basic} \ $$
having the same invariants as $\lambda$, except that
$d_1=\cdots = d_{k-1}=0$ holds for $\lambda_{basic}$ and the leftmost $\vee$ is at the vertex $-n+1$. 

\begin{example} In Example \ref{ex-cup-diag} the weight $[7,7,4,2,2,2]$ is of type (NSD). It has two $k=2$ sectors of rank $r_1 = 4$ and $r_2=2$ with shift factor $d_0=7$ and $d_1  = 1$. Its associated basic weight is 

\begin{center}
\medskip
 
 \scalebox{0.7}{
\begin{tikzpicture}
\foreach \x in {5,6,2,-1,-2,-3} 
     \draw[very thick] (\x-.1, .1) -- (\x,-0.1) -- (\x +.1, .1);
\foreach \x in {-4,0,1,3,4,7,8,9,10} 
     \draw[very thick] (\x-.1, -.1) -- (\x,0.1) -- (\x +.1, -.1);
%

\draw (-4,-0.5) node {-5};
\draw (-3,-0.5) node {-4};
\draw (-2,-0.5) node {-3};
\draw (-1,-0.5) node {-2};
\draw (0,-0.5) node {-1};
\draw (1,-0.5) node {0};
\draw (2,-0.5) node {1};
\draw (3,-0.5) node {2};
\draw (4,-0.5) node {3};
\draw (5,-0.5) node {4};
\draw (6,-0.5) node {5};
\draw (7,-0.5) node {6};
\draw (8,-0.5) node {7};
\draw (9,-0.5) node {8};
\draw (10,-0.5) node {9};


\draw[very thick] [-,black,out=90, in=90](5,+0.2) to (8,+0.2);
\draw[very thick] [-,black,out=90, in=90](6,+0.2) to (7,+0.2);
\draw[very thick] [-,black,out=90, in=90](2,+0.2) to (3,+0.2);
\draw[very thick] [-,black,out=90, in=90](-1,+0.2) to (0,+0.2);
\draw[very thick] [-,black,out=90, in=90](-2,+0.2) to (1,+0.2);
\draw[very thick] [-,black,out=90, in=90](-3,+0.2) to (4,+0.2);


\end{tikzpicture} }
\smallskip

\text{The cup diagram of $\lambda_{basic}$}
\end{center}

\end{example}




\section{ Cohomological tensor functors.}\label{sec:main} 

\subsection{The Duflo-Serganova functor} \label{ds-functor} We attach to every irreducible representation a sign. If $L(\lambda)$ is maximally atypical in $\calR_n$ we put $\varepsilon(L(\lambda)) = (-1)^{p(\lambda)}$ for the parity $p(\lambda) = \sum_{i=1}^n \lambda_i$. For the general case see \cite{Heidersdorf-Weissauer-tensor}. Now for $\varepsilon$ define the full subcategories $\calR_n(\varepsilon)$. These consists of all objects in $\calR_n$ whose irreducible constituents $L$ have sign $\varepsilon(L) = \varepsilon$. Then by \cite[Corollary 15.1]{Heidersdorf-Weissauer-tensor} the categories $\calR_n(\varepsilon)$ are semisimple categories.  

\medskip
Note that $\sdim(X)\geq 0$ holds for all irreducible objects $X\in \calR_n(\varepsilon)$ in case $\varepsilon(X)=1$ and also
for all irreducible objects $X\in \Pi\calR_n(\varepsilon)$ in case $\varepsilon(X)=-1$.
For each irreducible representation $L(\lambda)$ with $sdim(X_\lambda) \neq 0$ let 
$$  X_\lambda = \Pi^{p(\lambda)}( L(\lambda)) $$
denote the parity shift of $L(\lambda)$ that satisfies $sdim(X_\lambda)\geq 0$.
In the case of the Berezin representation $Ber =[1,....,1]$ we also
write $B$ for this parity shift. Notice, $B=Ber$ if $n$ is even and
$B=\Pi(Ber)$ if $n$ is odd.

\medskip

We recall some constructions from the article \cite{Heidersdorf-Weissauer-tensor}. Fix the following element $x\in \mathfrak{g}_1$, \[ x = \begin{pmatrix} 0 & y \\ 0 & 0 \end{pmatrix} \text{ for } \ y = \begin{pmatrix} 0 & 0 & \ldots & 0 \\ 0 & 0 & \ldots & 0 \\ \ldots & & \ldots &  \\ 1 & 0  & 0 & 0 \\ \end{pmatrix}. \]   Since $x$ is an odd element with $[x,x]=0$, we get $$2 \cdot \rho(x)^2 =[\rho(x),\rho(x)] =\rho([x,x]) =0 $$ for any representation
$(V,\rho)$ of $GL(n|n)$ in ${\calR}_n$. Notice $d= \rho(x)$ supercommutes with $\rho(GL(n-1|n-1))$. Then we define the cohomological tensor functor $DS$ as  \[ DS =  DS_{n,n-1}: \mathcal{T}_n \to \mathcal{T}_{n-1} \]
via  $DS_{n,n-1}(V,\rho)= V_x:=Kern(\rho(x))/Im(\rho(x))$. 

In fact $DS(V)$ has a natural $\Z$-grading and decomposes into a direct sum of $GL(n-1|n-1)$-modules
$$   DS(V,\rho) \ = \ \bigoplus_{\ell \in\mathbb Z}  \ \Pi^\ell(H^\ell(V)) \ ,$$ for certain cohomology groups $H^{\ell}(V)$.
If we want to emphasize the $\mathbb Z$-grading, we also
write this in the form
\[  DS(V,\rho) \ = \ \bigoplus_{\ell \in\mathbb Z}  \ H^\ell(V)[-\ell].\]

\begin{thm} \label{mainthm} \cite[Theorem 16.1]{Heidersdorf-Weissauer-tensor}  Suppose $L(\lambda)\in \calR_n$ is an irreducible
atypical representation, so that $\lambda$ corresponds to
a cup diagram $$ \bigcup_{j=1}^r \ \ [a_j,b_j] $$ with $r$ sectors
$[a_j,b_j]$ for $j=1,...,r$. Then $$DS(L(\lambda)) \ \cong\ \bigoplus_{i=1}^r  \ \Pi^{n_i} L(\lambda_i)$$ is the direct sum of irreducible atypical
representations $L(\lambda_i)$ in $\calR_{n-1}$ with shift $n_i \equiv p(\lambda)
- p(\lambda_i)$ modulo 2. The representation $L(\lambda_i)$ is uniquely defined by
the property that its cup diagram is $$ [a_i +1, b_i-1] \ \ \ \cup \ \ \bigcup_{j=1, j\neq i}^r \ \ [a_j,b_j] \ ,$$ the union of the sectors $[a_j,b_j]$ for $1\leq j\neq i \leq r$ and (the sectors occuring in) the segment $[a_i+1,b_i-1]$.
\end{thm}

\noindent  
In particular $DS(L(\lambda))$ is semisimple and multiplicity free.

\begin{example} Consider again the (maximal atypical) irreducible representation $[7,7,4,2,2,2]$ of $GL(6|6)$ of Example \ref{ex-cup-diag}. The parity is $\varepsilon(\lambda) = 1$. Applying $DS$ gives 2 irreducible representations. The representation $[\lambda_1] = [7,7,4,2,2]$ is associated to the derivative of the first sector

\begin{center}
\medskip
 
 \scalebox{0.7}{
\begin{tikzpicture}
\foreach \x in {7,6,2,-1,-2} 
     \draw[very thick] (\x-.1, .1) -- (\x,-0.1) -- (\x +.1, .1);
\foreach \x in {-4,-3,1,3,4,5,8,9,10,11} 
     \draw[very thick] (\x-.1, -.1) -- (\x,0.1) -- (\x +.1, -.1);
%

\draw (-4,-0.5) node {-4};
\draw (-3,-0.5) node {-3};
\draw (-2,-0.5) node {-2};
\draw (-1,-0.5) node {-1};
\draw (0,-0.5) node {0};
\draw (1,-0.5) node {1};
\draw (2,-0.5) node {2};
\draw (3,-0.5) node {3};
\draw (4,-0.5) node {4};
\draw (5,-0.5) node {5};
\draw (6,-0.5) node {6};
\draw (7,-0.5) node {7};
\draw (8,-0.5) node {8};
\draw (9,-0.5) node {9};
\draw (10,-0.5) node {10};
\draw (11,-0.5) node {11};


\draw[very thick] [-,black,out=90, in=90](7,+0.2) to (8,+0.2);
\draw[very thick] [-,black,out=90, in=90](6,+0.2) to (9,+0.2);
\draw[very thick] [-,black,out=90, in=90](2,+0.2) to (3,+0.2);
\draw[very thick] [-,black,out=90, in=90](-1,+0.2) to (0,+0.2);
\draw[very thick] [-,black,out=90, in=90](-2,+0.2) to (1,+0.2);


\end{tikzpicture} }
\smallskip

\text{The cup diagram of $L(\lambda_1)$}
\end{center} 

Then the parity is $\varepsilon(\lambda_1) = 1 = \varepsilon(\lambda)$. The second irreducible representation is $\Pi [7,3,1,1,1]$ (note the parity shift since $\varepsilon(\lambda_2) \neq \varepsilon(\lambda)$) with cup diagram 

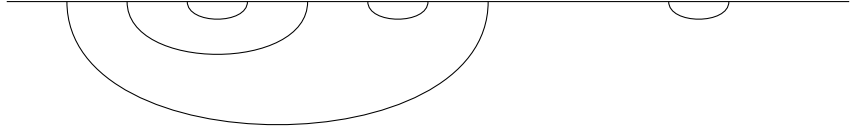
\begin{center}
 
 \scalebox{0.7}{
\begin{tikzpicture}
\foreach \x in {7,2,-1,-2,-3} 
     \draw[very thick] (\x-.1, .1) -- (\x,-0.1) -- (\x +.1, .1);
\foreach \x in {-4,0,1,3,4,5,6,8,9,10,11} 
     \draw[very thick] (\x-.1, -.1) -- (\x,0.1) -- (\x +.1, -.1);
%

\draw (-4,-0.5) node {-4};
\draw (-3,-0.5) node {-3};
\draw (-2,-0.5) node {-2};
\draw (-1,-0.5) node {-1};
\draw (0,-0.5) node {0};
\draw (1,-0.5) node {1};
\draw (2,-0.5) node {2};
\draw (3,-0.5) node {3};
\draw (4,-0.5) node {4};
\draw (5,-0.5) node {5};
\draw (6,-0.5) node {6};
\draw (7,-0.5) node {7};
\draw (8,-0.5) node {8};
\draw (9,-0.5) node {9};
\draw (10,-0.5) node {10};
\draw (11,-0.5) node {11};


\draw[very thick] [-,black,out=90, in=90](7,+0.2) to (8,+0.2);
\draw[very thick] [-,black,out=90, in=90](2,+0.2) to (3,+0.2);
\draw[very thick] [-,black,out=90, in=90](-1,+0.2) to (0,+0.2);
\draw[very thick] [-,black,out=90, in=90](-2,+0.2) to (1,+0.2);
\draw[very thick] [-,black,out=90, in=90](-3,+0.2) to (4,+0.2);


\end{tikzpicture} }
\smallskip

\text{The cup diagram of $L(\lambda_2)$}
\end{center} 

All in all $DS[7,7,4,2,2,2] \cong [7,7,4,2,2] \oplus \Pi [7,3,1,1,1]$.

\end{example}

\subsection{The Hilbert polynomial}\label{sec:hilbert} Similarly to $DS$ we can define the tensor functors $DS_{n,n-m}: \mathcal{T}_n \to T_{n-m}$ by replacing the $x$ in the definition of $DS$ by an $x$ with $m$ $1$'s on the antidiagonal. These functors admit again a $\Z$-grading. In particular we can consider the functor $DS_{n,0}: \mathcal{T}_n \to T_0=svec_k$ with its decomposition 
$DS_{n,0}(X) = \bigoplus_{\ell\in\mathbb Z} D_{n,0}^\ell(X)[-\ell]$ for objects $X$ in $\mathcal{T}_n$ and
objects $D_{n,0}^\ell(X)$ in $svec_k$ where $D_{n,0}^\ell(X)[-\ell]$ is the object $\Pi^{\ell}D_{n,0}^\ell(X)$ concentrated in degree $\ell$ with respect to the $\Z$-gradation of $DS_{n,0}(X)$.
For $X\in \mathcal{T}_n$ we define the Laurent polynomial
$$ \omega(X,t) = \sum_{\ell\in\mathbb Z} \sdim(DS_{n,0}^\ell(X)) \cdot t^\ell \ $$
 as the Hilbert polynomial of the graded module $DS^\bullet_{n,0}(X)= \bigoplus_{\ell\in\mathbb Z} DS_{n,0}^\ell(X)$.  Since $\sdim(W[-\ell])=(-1)^\ell \sdim(W)$ and $X= \bigoplus DS_{n,0}^\ell(X)[-\ell]$ holds, the formula  $$\sdim(X) = \omega(X,-1)$$ follows. For $X= Ber_n^i$ 
$$    \omega(Ber_n^i,t) \ = \ t^{ni}  \ .$$ For more details we refer the reader to \cite[section 25]{Heidersdorf-Weissauer-tensor}.

\subsection{The Dirac functor} \label{sec-Dirac} 

In \cite[Section 5]{Heidersdorf-Weissauer-tensor} we also consider the Dirac operator
$$   D = \partial + \overline\partial \ .$$  Here $\overline x = x^T$ denotes the supertranspose of $x$ and $\overline\partial= i \rho(\overline x)$. Let $H=diag(0_{n-1},1,1,0_{n-1})$ and $M:= V^H$. Then we show that
$$  H_D(V) = Kern(D: M \to M)/Im(D: M \to M) \ $$
defines a symmetric monoidal functor $\mathcal{T}_n \to \mathcal{T}_{n-s}$ where $s$ is the rank of $x$. It follows from Lemma \ref{imp} that $H_D$ agrees with $DS$ on the subcategory $\mathcal{T}_n^+$.




\bigskip\noindent

\section{ Tannakian arguments} \label{sec:tannakian-arguments}

\subsection{The category $\mathcal{T}_n^+$}
Let $\mathcal{T}_n^+$ denote the Karoubian envelope of the simple nonnegative
representations, i.e the full subcategory of $\mathcal{T}_n$, whose objects consist of all retracts of iterated tensor products of irreducible representations in $\mathcal{T}_n$ that are not maximal atypical and of  maximal atypical irreducible representations $X_\lambda$ in $\calR_n(+1) \oplus \Pi\calR_n(-1)$, defined as at the begining of section \ref{ds-functor}. 
Obviously $\mathcal{T}_n^+$ is a symmetric monoidal  idempotent complete $k$-linear category closed under the $*$-involution. It contains all irreducible
objects of $\mathcal{T}_n$ up to a parity shift.  It contains the standard representation $V$ and its dual $V^\vee$, and hence contains all mixed tensors \cite{Heidersdorf-mixed-tensors}. Furthermore all objects $X$ in $\mathcal{T}_n^+$ satisfy condition $\tt T$ (see section 6 in \cite{Heidersdorf-Weissauer-tensor}) and $\mathcal{T}_n^+$ is rigid. For this it suffices for irreducible $X\in \mathcal{T}_n^+$ that $X^\vee \in \mathcal{T}_n^+$. This is obvious since $X^\vee$ is irreducible with $\sdim(X^\vee) = \sdim(X) \geq 0$, and hence $X^\vee\in \mathcal{T}_n^+$.

\subsection{Conventions on tensor categories} We use the same definition of a tensor category  as is used in \cite{EGNO} except that we do not require the category to be abelian. Tensor functors are additive as in \cite{EGNO} but need not be exact. Our definition of a tensor functor therefore agrees with the one used in \cite{Deligne-Milne}.

\subsection{The ideal of negligible morphisms} An ideal in a $k$-linear category $\calA$ is for any two objects $X,Y$ the specification of a $k$-submodule $\mathcal{I}(X,Y)$ of $Hom_{\calA}(X,Y)$, such that $g \mathcal{I}(X',Y)f \ \subseteq \mathcal{I}(X,Y')$ holds for all pairs of morphisms $f \in Hom_{\calA}(X,X')$, $ g \in Hom_{\calA}(Y,Y')$. Let $\mathcal{I}$ be an ideal in $\calA$. By definition $\calA/\mathcal{I}$ is the category with the same objects as $\calA$ and with \[ Hom_{\calA/\mathcal{I}} (X,Y ) = Hom_{\calA}(X,Y)/\mathcal{I}(X,Y) \ .\] An ideal in a tensor category is a tensor ideal if it is stable under $\one_C \otimes -$ and $- \otimes \one_C$ for all $C \in \calA$. Let $Tr$ be the trace. For any two objects $A, B$ we define $\mathcal{N}(A,B) \subset Hom(A,B)$ by \[ \mathcal{N}(A,B) = \{ f \in Hom(A,B) \ | \ \forall g \in Hom(B,A), \ Tr(g \circ f ) = 0 \}. \] The collection of all $\mathcal{N}(A,B)$ defines a tensor ideal $\mathcal{N}$ of $\calA$ \cite{Andre-Kahn}.

\medskip

Let $\calA$ be a super tannakian category. An indecomposable object will be called \textit{negligible}, if its image in $\calA/\calN$ is the zero object. By \cite{Heidersdorf-semisimple} an object is negligible if and only if its categorial dimension is zero.

\begin{example} An irreducible representation has superdimension zero if and only if it is not maximal atypical, see section \ref{3}. The standard representation $V \simeq k^{n|n}$ has superdimension zero and therefore also the indecomposable adjoint representation $\mathbb{A} = V \otimes V^{\vee}$.
\end{example}

Any super tannakian category is equivalent (over an algebraically closed field) to the representation category of a supergroup scheme by \cite{Deligne-tensorielles}. In that case the categorial dimension is the superdimension of a module. If $\calA$ is a super tannakian category over $k$, the quotient of $\calA$ by the ideal $\calN$ of negligible morphisms is again a super tannakian category by \cite{Andre-Kahn}, \cite{Heidersdorf-semisimple}. More generally, for any pseudo-abelian full subcategory $\tilde{\calA}$ in $\calA$ closed under tensor products, duals and containing the identity element the following holds: 

\begin{lem} The quotient category $\tilde{\calA}/\calN_{\tilde{\mathcal{A}}}$ is a semisimple super tannakian category.
\end{lem}

\medskip{\it Proof}. The quotient is a $k$-linear semisimple rigid tensor category by \cite[Theorem 1 a)]{Andre-Kahn-erratum}. The quotient is idempotent complete by lifting of idempotents (or see \cite[2.3.4 b)]{Andre-Kahn} and by \cite[2.1.2]{Andre-Kahn} a $k$-linear pseudoabelian category is abelian. The Schur finiteness \cite{Deligne-tensorielles} \cite{Heidersdorf-semisimple} is inherited from $\calA$ to $\tilde{\calA}/\calN$. \qed   

\medskip
This in particular applies  to the situation where $\tilde{\calA}$ is the full subcategory of objects which are retracts of iterated tensor products of a fixed set of objects in $\calA$.
In particular for $\tilde{\calA} = \mathcal{T}_n^+$ and $\calA=\mathcal{T}_n$ this implies

%

\begin{cor}
The tensor functor $\mathcal{T}_n^+ \to \mathcal{T}_n^+/\calN$ maps $\mathcal{T}_n^+$
to a semisimple super tannakian category $\overline{\mathcal{T}}_{n}:=\mathcal{T}_n^+/\calN$.
\end{cor}

\begin{prop} \label{thm:tannakagroup} The category $\overline{\mathcal{T}}_{n}$ is a tannakian category, i.e. 
there exists a pro-reductive algebraic $k$-groups $H_n$ such that the
category $\overline{\mathcal{T}}_{n}$ is equivalent as a tensor category 
to the category $Rep_k(H_n)$ of finite dimensional $k$-representations of $H_n$
$$ \overline{\mathcal{T}}_{n}  \sim Rep_k(H_n) \ .  $$  
\end{prop}

{\it Proof}. By a result of Deligne \cite[Theorem 7.1]{Deligne-Festschrift} it suffices to show that
for all objects $X$ in $\mathcal{T}_n^+$ we have $\sdim(X)\geq 0$.
We prove this by induction on $n$. Suppose we know this 
assertion for $\mathcal{T}_{n-1}$ already. Then 
all objects of $ \mathcal{T}_{n-1}^+$ have superdimension $\geq 0$
(for the induction start $n=0$ our assertion is obvious).
Since the tensor functor $DS: \mathcal{T}_n \to \mathcal{T}_{n-1}$ preserves
superdimensions, it suffices for the induction step that $DS$
maps $\mathcal{T}_n^+$ to $\mathcal{T}_{n-1}^+$. 

\begin{lem}\label{DS-restricts} The functors $DS_{n,n-m}: \mathcal{T}_n \to \mathcal{T}_{n-m}$
and $\omega_{n,n-m}: \mathcal{T}_n \to \mathcal{T}_{n-m}$ restrict to functors from $\mathcal{T}_n^+$
to $\mathcal{T}_{n-m}^+$. In particular
$$ DS: \mathcal{T}_n^+ \to \mathcal{T}_{n-1}^+  \ .$$
\end{lem}
 
\begin{proof} Since $DS_{n,n-m}$ and $\omega_{n-m}$
preserve tensor products and idempotents, it 
suffices by the definition of $\mathcal{T}_n^+$ that $DS_{n-m}(X), \omega_{n-m}(X) \in \mathcal{T}_{n-m}^+$ holds for all irreducible
objects $X$ in $\mathcal{T}_n^+$. 
Now theorem \ref{mainthm} implies $DS(X) \in \mathcal{T}_{n-1}^+$ since any irreducible
representation $X$ maps to a semisimple representation $DS(X)$ and for maximal atypical $X 
\in \mathcal{T}_{n-1}^+$ all summands of $DS(X)$ are in $\mathcal{T}_{n-1}^+$.   
This proves the claim for $DS(X)$, $X$ irreducible. But then also for
$DS_{n,n-m}(X)$, $X$ irreducible, since then again $DS_{n,n-m}(X)$ is semisimple
by proposition 8.1 in \cite{Heidersdorf-Weissauer-tensor}. The same then also holds for $\omega_{n,n-m}(X) = H_{\overline\partial}(DS_{n-m}(X))$ by loc.cit. 
\end{proof}

\begin{cor}\label{ds-negligible}
Under $DS$ negligible objects in $\mathcal{T}_n^+$ map 
to negligible objects in $\mathcal{T}_{n-1}^+$. 
\end{cor}

{\it Proof}.  We have shown $\sdim(Y)\geq 0$ for all objects $Y$ in $\mathcal{T}_{n-1}^+$. Therefore $\sdim(DS(X))= \sdim(X)=0$ implies
$sdim(Y_i)=0$ for all indecomposable summands $Y_i$ of $Y=DS(X)$, since
$sdim(Y_i)\geq 0$. \qed

\begin{remark} Since irreducible objects $L$ satisfy condition {\tt T} in the sense that $\overline\partial$ is trivial on $DS_{n,n-m}(L)$ \cite[proposition 8.5]{Heidersdorf-Weissauer-tensor}, and since
condition {\tt T} is inherited by tensor products and retracts, all objects in
$\mathcal{T}_n^+$ satisfy condition {\tt T}. Hence \cite[proposition 8.5]{Heidersdorf-Weissauer-tensor} implies the following lemma.
\end{remark}

\begin{lem} \label{imp}
On the category $\mathcal{T}^+_n$ the functor $H_D(.)$ is naturally equivalent to the
functor $DS: \mathcal{T}_n^+ \to \mathcal{T}_{n-1}^+$. Similarly the functors $\omega_{n,n-m}(.):  \mathcal{T}_n^+ \to \mathcal{T}_{n-1}^+$ are naturally equivalent to $DS_{n,n-m}(.)$.
\end{lem}

\begin{cor} \label{kernel-of-DS} $DS(X)=0$ in $\mathcal{T}_{n-1}^+$ if and only if $X$ is a projective object in $\mathcal{T}_n$.
\end{cor} 

{\it Proof}. Any negligible maximal atypical object in $\mathcal{T}_n^+$ (i.e. a negligible object in the principal block) maps under $DS$
to a negligible maximal atypical object in $\mathcal{T}_{n-1}^+$.
Furthermore $DS(X)=0$ for $X$ in $\mathcal{T}_n^+$ implies that $X$ is an anti-Kac object.
If $X\neq 0$, then $X^*$ is a Kac object in $\mathcal{T}_n^+$. Hence $H_D(X^*)=0$. Since $X^*\in \mathcal{T}_n^+$ satisfies
condition {\tt T}, this implies $DS(X^*)=0$ and hence $X^*$ is a Kac and anti-Kac object. The corollary follows since $\mathcal{C}^+ \cap \mathcal{C}^- = Proj$.

\begin{cor} If $X \in \mathcal{T}_n^+$ and $X$ is a Kac or anti-Kac object, then $X \in Proj$.
\end{cor}

Even though negligible objects map to negligible objects by Corollary \ref{ds-negligible}, it is highly non-trivial negligible morphisms map to negligible morphism and that we therefore get an inducted functor between the quotient categories.

\begin{lem}\label{DS-respects-negligibles}
The functor $DS: \mathcal{T}_n^+ \to \mathcal{T}_{n-1}^+$ gives rise to a $k$-linear exact tensor functor
between the quotient categories
$$  \eta:  \overline{\mathcal{T}}_{n} \to \overline{\mathcal{T}}_{n-1} \ .$$ In particular the iterated tensor functor $\omega = \eta \circ\ldots \eta:\overline{\mathcal{T}}_n \to vec_k$ defines a fibre functor for $\overline{\mathcal{T}}_n$.
\end{lem}

\medskip
{\it Proof}. We define the ideal $\mathcal{I}^0$ via \[ \mathcal{I}^0(X,Y) = \{ f:X \to Y \ | \ f \text{ factorizes over a negligible object.} \} \] 
Obviously $\mathcal{I}^0$ is a tensor ideal for $\mathcal{T}_n^+$. As for any tensor ideal $\mathcal{I}^0 \subset \mathcal{N}$ 
the quotient $\mathcal{T}_n^+/\mathcal{I}^0 =: \mathcal{A}_n^+$ becomes a rigid tensor category and
$\mathcal{T}_n^+\to \mathcal{T}_n^+/\mathcal{I}^0 = \mathcal{A}_n^+$ a tensor functor. Under this tensor functor
an indecomposable object $X$  in  $\mathcal{T}_n^+$ maps to zero in the quotient $\mathcal{A}_n^+$ if and only if $\sdim(X) = 0$. Furthermore, since the tensor functor $DS$ maps negligible objects
of $\mathcal{T}_n^+$ to negligible objects of $\mathcal{T}_{n-1}^+$, the functor $DS$
induces a $k$-linear tensor functor $DS': \mathcal{A}_n^+ \to \mathcal{A}_{n-1}^+$. The category $\mathcal{A}_n^+$ is pseudoabelian since we have idempotent lifting in the sense of \cite[Theorem 5.2]{Linckelmann} due to the finite dimensionality of the Hom spaces. By the  definition of $\mathcal{A}_n^+$ and $\mathcal{T}_n^+$, the dimension of each object in $\mathcal{A}_n^+$ is a natural number and, contrary 
to  $\mathcal{T}_n^+$, it does not contain any nonzero object that maps to an element isomorphic to  zero under the quotient functor $\mathcal{A}_n^+ \to \mathcal{A}_n^+/\mathcal{N}$. Therefore $\mathcal{A}_n^+$ satisfies conditions d) and g) in \cite[Theorem 8.2.4]{Andre-Kahn}. By \cite[Theorem 8.2.4 (i),(ii)]{Andre-Kahn} this implies 
that $\mathcal{N}( \mathcal{A}_n^+)$ 
equals the radical $\mathcal{R}( \mathcal{A}_n^+)$ of $ \mathcal{A}_n^+$; note that $\mathcal{N}( \mathcal{A}_n^+) = \mathcal{N}( \mathcal{T}_n^+)/\mathcal{I}^0$ and that 
${\mathcal N}(A,A)$ is a nilpotent ideal in $End(A)$ for any $A$ in $\mathcal{A}_n^+$ by assertion b) of \cite[Theorem 8.2.4 (i),(ii)]{Andre-Kahn}. 
Since  $\mathcal{N}$ always is a tensor ideal, $\mathcal{R}( \mathcal{A}_n^+)$ in particular  is a tensor ideal.  This allows to apply \cite[Theorem 13.2.1]{Andre-Kahn} to construct a monoidal section $s_n: \mathcal{A}^+_n/\mathcal{N}(\mathcal{A}_n^+) \to \mathcal{A}^+_n$ for the tensor functor $\pi_n: \mathcal{A}^+_n \to \mathcal{A}^+_n/\mathcal{N}(\mathcal{A}_n^+)$.
The composite tensor functor $$\eta:=\pi_{n-1}\circ DS' \circ s_n$$
defines a $k$-linear tensor functor 
$$  \eta:  \overline{\mathcal{T}}_{n}  \to  \overline{\mathcal{T}}_{n-1} \ .$$

Since $DS'$ is additive and $\overline{\mathcal{T}}_{n}$ is semisimple, $\eta$ is additive and
hence exact. \qed

\begin{remark} 1) 
The $k$-linear tensor functor $\pi_{n-1}\circ DS': \mathcal{A}_n^+ \to \overline{\mathcal{T}}_{n-1}$
defines the tensor ideal ${\mathcal K}_n$ of $\mathcal{A}_n^+$ of morphisms annihilated by $\pi_{n-1}\circ DS'$. Obviously $ {\mathcal K}_n \subseteq {\mathcal N}$.

2) Let $S$ be the image of a simple object in $\mathcal{A}_n^+$. Since $ \calN(\mathcal{A}_n^+)= \calR(\mathcal{A}_{n}^+)$, some given morphism $f\in Hom_{\mathcal{A}_{n}^+}(S,A)$ is in 
$ \calN(\mathcal{A}_{n}^+)(S,A)$ if and only if for all $ g \in Hom_{\mathcal{A}_{n}^+}(S,A)$ the composite
$g\circ f$ is zero \cite[Lemma 1.4.9]{Andre-Kahn} (note that the endomorphisms of $S$ in $\mathcal{A}_{n}^+$ are in $k\cdot id$, hence  \cite[Lemma 1.4.9]{Andre-Kahn} can be applied.).

3) By \cite[Theorem 13.2.1]{Andre-Kahn} the section $s_n$ is unique up to isomorphism. Therefore the functor $\eta$ so constructed is unique up to isomorphism.

\end{remark}

\begin{remark} 
We do not know whether $DS(\calN(\mathcal{T}_n^+)) \subseteq \calN(\mathcal{T}_{n-1}^+)$ holds. If this were true for all $n$, then
also $DS_{n,n-i}(\calN(\mathcal{T}_n^+)) \subseteq \calN(\mathcal{T}_{n-i}^+)$ would hold. We consider this a fundamental question in the theory. For $n=1$ observe that $\mathcal{A}_1^+ = \mathcal{T}_1^+/\mathcal{N}$. Indeed $\mathcal{T}_1^+$ has only one proper tensor ideal $\mathcal{N} = \mathcal{I}^0$ as can be easily seen by looking at the maximal atypical objects $Ber^i$ and $P(Ber^j)$ in $\mathcal{T}_1^+$. The tensor ideal $\mathcal{I}^0$ could be different from $\mathcal{N}$ for $n \geq 2$. With respect to the partial ordering on the set of tensor ideals given by inclusion, $\mathcal{I}^0$ is the minimal element in the fibre of the decategorification map of the thick ideal of indecomposable objects of superdimension 0 \cite[Theorem 4.1.3]{Coulembier}. The negligible morphisms are the largest tensor ideal in this fibre.  
\end{remark}

\begin{example} Note that it is important here to work in $\mathcal{T}_n^+$ since for example $DS(K(\one))$ (where $K(\one)$ is the Kac-module of the trivial representation) splits into a direct sum of maximal atypical irreducible modules (see \cite[Section 10]{Heidersdorf-Weissauer-tensor}). Hence the identity morphism of $K(\one)$ does not map to a negligible morphism. There are even counter examples in the smaller category $\mathcal{T}^{ev}$ of section \ref{sec:determinant}. In $\mathcal{T}_{1}$ consider the indecomposable ZigZag module (see \cite{Heidersdorf-semisimple}) with socle $Ber^{-1}$ and $Ber$ and top $\one$. The inclusion of $Ber^{-1}$ induces an isomorphism when taking $DS$-cohomology. On the other hand the inclusion is negligible. Note that $Ber^{-1}$ is odd. The ZigZag module can be obtained as $\one[1]$ in the stable category $\mathcal{K}$. Since $\one$ is even, $\one[1]$ is odd. Hence their parity shifts define even objects in $\mathcal{T}_1^{ev}$.
\end{example}

\subsection{$DS$ as a restriction functor} Recall from \cite[Theorem 8.17]{Deligne-Festschrift} the following fundamental theorem on $k$-linear tensor categories: Suppose $\calA_1, \calA_2$ are $k$-linear abelian rigid symmetric monoidal tensor categories
with $k \cong End_{\calA_i}({\bf 1})$ as in loc. cit. Assume that all objects of $\calA_i$
have finite length and all $Hom$-groups have finite $k$-dimension. Assume that $k$ is a perfect field so that $\calA_1 \otimes \calA_2$ is again  $k$-linear abelian rigid symmetric monoidal tensor categories
with $k \cong End_{\calA_i}({\bf 1})$ as in \cite[8.1]{Deligne-Festschrift}. Suppose
$$ \eta: \calA_1 \to \calA_2 $$
is an {\it exact tensor functor}. Then $\eta$ is faithful \cite[Proposition 1.19]{Deligne-Milne}.

\begin{thm} \label{deligne-theorem} \cite[Theorem 8.17]{Deligne-Festschrift}
Under the assumptions above there exists
a morphism $$ \pi(\calA_2) \to \eta(\pi(\calA_1)) $$
as in \cite[8.15.2]{Deligne-Festschrift} such that $\eta$ induces a tensor equivalence
between the category $\calA_1$ and the tensor category of objects in $\calA_2$
equipped with an action of $\eta(\pi(\calA_1))$, so that the natural action of
$\pi(\calA_2))$ is obtained via the morphism $ \pi(\calA_2) \to \eta(\pi(\calA_1)) $.
\end{thm}

\medskip
Suppose $\omega: \calA_2 \to Vec_k$ is fiber functor of $\calA_2$, i.e. 
$\omega$ is an exact faithful tensor functor. Then $\calA_2$
is a Tannakian
category and $\calA_2 \cong Rep_k(H)$ as a tensor category.  
If $\calA_2 = Rep_k(H)$ is a Tannakian category for some affine group $H$ over $k$,
then $\pi(\calA_2)= H$ by \cite[Example 8.14 (ii)]{Deligne-Festschrift}. More precisely, an $\calA_2$-group is the same as an affine $k$-group equipped with an $H$-action, and here $H$ acts on itself by conjugation. The forgetful functor $\omega$ of $Rep_k(G)$ to $Vec_k$ is a fiber functor. 
By applying this fiber functor we obtain a fiber functor  $\omega\circ \eta: \calA_1 \to Vec_k$ for the tensor category $\calA_1$. In particular $\calA_1$ becomes a Tannakian category with
Tannaka group $H'= \omega\circ \eta(\pi(\calA_1))$. Furthermore, by
applying $\eta$ to the morphism $ \pi(\calA_2) \to \eta(\pi(\calA_2)) $ in $\calA_2$,
we get a morphism $\omega(\pi(\calA_2)) \to (\omega\circ\eta)(\pi(\calA_1))$ in
the category of $k$-vectorspaces, which defines a group homomorphism
$$   f: H' \to H $$
of affine $k$-groups inducing a pullback functor $$Rep(H') \to Rep(H) \ ,$$ that gives
back the functor $\eta: \calA_1 \to \calA_2$ via the equivalences $\calA_1=Rep_k(H')$
and $\calA_2 =Rep_k(H)$ obtained from the fiber functors. 

\begin{lem}\cite[Proposition 2.21(b)]{Deligne-Milne} \label{inj}
The morphism $f:H' \to H$ thus obtained is a closed immersion if and only if 
every object $Y$ of  $\calA_2$ is isomorphic to a subquotient of an object of the
form $\eta(X), X \in \calA_1$.   
\end{lem}

\medskip
The statements above will now be applied for the tensor functor 
$$   \eta: \calA_1 \to \calA_2 $$
obtained from $DS$ between the quotient categories $\calA_1= \mathcal{T}_n^+/\calN$ and 
$\calA_2= \mathcal{T}_{n-1}^+/\calN$. Notice that the assumptions above on $k$ and $\calA_i$
are satisfied so that $\calA_2$ is a tannakian category with fiber functor $\omega$
giving an equivalence of tensor categories $\calA_2 = Rep_k(H_{n-1})$. 
Obviously $\eta$ induces an {\it exact tensor functor} between the quotient  categories,
since $DS$ is additive, maps negligible objects of $\mathcal{T}_n^+$ into negligible objects of $\mathcal{T}_{n-1}^+$ and since the categories $\calA_i$ are semisimple. 
As in our case $k$ is algebraically closed, we know that up to an isomorphism the group $H_n$ only depends on $\calA_1$ but not on the choice of a fiber functor. As explained above, this defines a homomorphism of affine $k$-groups (uniquely defined up to conjugacy)
$  f: H_{n-1} \longrightarrow H_n $.

\begin{thm} The homomorphism $f:H_{n-1} \to H_n$ is injective
and the functor $\eta: Rep_k(H_{n})\to Rep_k(H_{n-1})$ induced  
by $DS: \mathcal{T}_n^+ \to \mathcal{T}_{n-1}^+$ can be identified with the restriction
functor for the homomorphism $f$.
\end{thm}

\medskip
{\it Proof}. By lemma \ref{inj} it suffices that every indecomposable
$Y$ in $\mathcal{T}_{n-1}^+$ with $\sdim(Y) >0$
is a subobject of an object $DS(X), X\in \mathcal{T}_{n}^+$. By assumption $Y$ is a retract of a tensor product of irreducible modules $L_i \in \mathcal{T}_{n-1}^+$. So it suffices that
each $L_i$ is a subobject of some object $DS(X_i), X_i\in \mathcal{T}_{n}^+$.
We can assume that $Y$ is not negligible and irreducible, hence maximal atypical 
and $Y=\Pi^{r}L(\lambda)$ for some $r$. Then $L(\lambda) = [\lambda] =[\lambda_1,...,\lambda_{n-1}]$. By a twist
with Berezin we may assume that $\lambda_{n-1}\geq 0$. Then
we define $[\tilde\lambda] = [\lambda_1,...,\lambda_{n-1},0]$ so that
for $X = \Pi^rL(\tilde\lambda)$ we get by theorem \ref{mainthm} and \cite[Lemma 10.2]{Heidersdorf-Weissauer-tensor} the assertion $DS(X) = Y \oplus $ other summands. Notice that by construction $X = \Pi^rL(\tilde\lambda)$ is in $\mathcal{T}_n^+$.  But this proves our claim. \qed


\medskip
In other words, the description of the functor $DS$ on irreducible objects in $\mathcal{T}_n$
given by theorem \ref{mainthm} can be interpreted as branching rules for the inclusion
$$ f: \ H_{n-1} \hookrightarrow H_n\ .$$
We will show later how this fact gives information on the groups $H_n$.

\subsection{Enriched morphism} Using the $\Z$-grading of $DS$ (see section \ref{ds-functor}), we can define an extra structure on the tower of Tannaka groups. This extra structure will not be used in the later determination of the Tannaka groups. The collection of cohomology functors $H^i: \calR_n \to \calR_{n-1}$ for $i\in \mathbb Z$
defines a tensor functor
$$  H^\bullet: \calR_n \to Gr^\bullet(\calR_{n-1}) \ $$
to the category of $\mathbb Z$-graded objects in $\calR_{n-1}$.
Using the parity shift functor $\Pi$, this functor can be extended to a tensor functor
$$  H^\bullet: \mathcal{T}_n^+ \to Gr^\bullet(\mathcal{T}_{n-1}^+) \ ,$$
which induces a corresponding tensor functor on the level
of the quotient categories
$$  H^\bullet: \overline{\mathcal{T}}_{n}=\mathcal{T}_n^+/\calN \to Gr^\bullet(\mathcal{T}_{n-1}^+/\calN)) = Gr^\bullet(\overline{\mathcal{T}}_{n-1}) \ .$$
Using the language of tannakian categories this induces an 'enriched'
group homomorphism
$$  f^\bullet:\  H_{n-1} \times \mathbb G_m   \to H_{n} \ .$$
Its restriction to the subgroup $1 \times H_{n-1}$ is the homomorphism $f$ 
from above.

\subsection{The involution $\tau$}
Note that the category $\mathcal{T}_n^+$ is closed under $\vee$ and $*$ and hence is
equipped with the tensor equivalence $\tau: X \mapsto (X^\vee)^*$.
This tensor equivalence induces a tensor equivalence of $\overline{\mathcal{T}}_{n}= \mathcal{T} _n^+/\calN$
and hence an automorphism $\tau=\tau_n$ (unique up to conjugacy) of the group $H_n$. Since all objects of $\mathcal{T}_n^+$ satisfy property $\tt T$ \cite[Section 6]{Heidersdorf-Weissauer-tensor},
the involution $*$ commutes with $DS$. Since this also holds for  the Tannaka duality, we get a compatibility
$$  (H_{n-1},\tau_{n-1}) \hookrightarrow (H_n,\tau_n) \ .$$

\subsection{Characteristic polynomial} By iteration the morphisms $f^\bullet$ successively 
define homomorphisms $ H_{n-i}\times (\mathbb G_m)^{i} \to H_n$ and therefore
we get a homomorphism in the case $i=n$ $$h: (\mathbb G_m)^n \to H_n \ .$$
This allows to define a characteristic polynomial, defined by the restriction $h^*(V_X)$ of the representation $V_X=\omega(X)$
of $H$ to the torus $(\mathbb G_m)^n$
$$  h_X(t) = \sum_{\chi}  \dim(h^*(V_X)_\chi)\cdot t^\chi  $$
where $\chi$ runs over the characters $\chi=(\nu_1,...,\nu_n) \in 
\mathbb Z^n = \mathbb X^*((\mathbb G_m)^n)$. It is easy to see that
$\omega(X,t)=h_X(t,...,t)$ (see section \ref{sec:hilbert}).



\section{ The structure of the derived connected groups $G_n$}\label{derived-group}


\subsection{Setup and Notations} 
The Tannaka group generated by the object $X_\lambda=\Pi^{p(\lambda)}L(\lambda)$
for $p(\lambda) = \sum_{i=1}^{n} \lambda_i$
will be denoted $H_\lambda$ and we define 
$$  G_\lambda := (H_\lambda^0)_{der} \subseteq H_{\lambda}^0 \subseteq H_\lambda \ .$$
Finally define $V_\lambda \in Rep(H_\lambda)$ as the irreducible finite dimensional faithful
representation (or the underlying vector space) of $H_\lambda$ corresponding to $X_{\lambda}$, i.e. the representation $\omega(X_{\lambda})$ for the fibre functor $\omega$ defined in section \ref{sec:tannakian-arguments}.
 

\medskip
\textit{(SD)-Types}.
Now assume for $L(\lambda)$ that $$\varphi: L(\lambda)^{\vee} \cong L(\lambda) \otimes Ber^{-r}$$ holds for some $r \in \mathbb{Z}$ and some isomorphism $\varphi$
in $\mathcal{T}_{\lambda}$. The evaluation morphism $eval: L(\lambda)^{\vee} \otimes L(\lambda) \to \one$  
and the isomorphism $L(\lambda)^{\vee} \cong L(\lambda) \otimes Ber^{-r}$
gives rise to a nondegenerate pairing $$ L(\lambda) \otimes L(\lambda) \to Ber^r \ $$ 
As a nondegenerate pairing
of the simple object $L(\lambda)$ it is even either or odd,
where the parity $\varepsilon_\lambda \in \{\pm 1\} $ is given by $\varphi^\vee =
\varepsilon_\lambda \cdot \varphi$
 (see Appendix \ref{sec:pairings}, where this sign will be computed). If we replace $L_\lambda$ by the parity shift
$X_\lambda = \Pi^{p(\lambda)}(L_\lambda)$, our pairing on $L(\lambda)$ induces a pairing 
$ X_\lambda \otimes X_\lambda \to Ber^r $
on $X_\lambda$. For this notice
that $\Pi^2(L\otimes L) \cong \Pi(L) \otimes \Pi(L)$ holds and the underlying vectorspaces  of $L(\lambda)$ and $X_\lambda$ coincide. Notice that $nr$ is even
by lemma \ref{lemma-2}  so that $B^r = Ber^r$ always holds. The 
resulting pairing
$$  X_\lambda \otimes X_\lambda \to B^r $$
will be denoted $<\cdot,\cdot>$.
Since a parity shift switches symmetric and antisymmetric pairings \cite{Scheunert},
the parity $\varepsilon(X_\lambda)$ of the induced pairing 
$<\cdot,\cdot>$ for $X_\lambda$ is
$$ \varepsilon(X_\lambda) = (-1)^{p(\lambda)} \varepsilon_\lambda \ .$$
This sign $\varepsilon(X_\lambda)$ will only depend on $\lambda_{basic}$ by 
lemma \ref{duality-type}. 
Since $<\cdot,\cdot>$ is uniquely defined
up to a nonvanishing constant, we will fix the pairing once for all.

\medskip\noindent
The nondegenerate pairing $<\cdot,\cdot>$ on $X_\lambda$ induces  a  nondegenerate pairing on $V_\lambda$ of the same parity $\varepsilon(X_\lambda)$, and will also be denoted  $<\cdot,\cdot>$ by abuse of notation. Indeed, since $<\cdot,\cdot>$ is defined in terms
of Tannaka duality, the evaluation morphism $eval$ and the isomorphism $\varphi$, this follows by functoriality.  

The Tannaka group $H_{\lambda}$ of $\mathcal{T}_{\lambda} = <\! X_\lambda\!>$ acts faithfully on $V_\lambda$ such that
 $  <h v_1,h v_2> \ = \ \mu(h) \cdot <v_1,v_2> $
holds for all $h\in H_\lambda$ and the similitude character $\mu: GSp(V_\lambda)\to  \mathbb{G}_m$  of the pairing $<\cdot ,\cdot>$ on $V_\lambda$. 

\medskip\noindent
For a  vector space $V_{\lambda}$ over an algebraically closed field with a nondegenerate symmetric or antisymmetric pairing $\langle .,.\rangle$ let  
$$ G\bigl(V_{\lambda},\langle .,.\rangle\bigr)  = \{ g \in GL(V_{\lambda}) \ | \ <gv,gw> = \mu(g)<v,w>, \  \forall v,w \in V_{\lambda} \} $$ 
be  the similitude group with its similitude character $\mu: G(V_{\lambda},\langle .,.\rangle) \to k^*$. The sign character $sgn: G(V_{\lambda},\langle .,.\rangle) \to \mu_2$ is defined by
$$ sgn(g) = \frac{det(g)}{\mu(g)^m} \ .$$
Notice that $\dim(V_{\lambda}) = 2m$ by lemma \ref{weird1} will always be even unless
$X(\lambda)$ has dimension one and hence is a power of $B$.
In the symmetric resp. antisymmetric case this above similtude group defines the orthogonal similitude group $GO(V_{\lambda})$ resp.
the symplectic similitude group $GSp(V_{\lambda})$. 

 In the $GSp$-case  $sgn$ is trivial and the kernel of $\mu$ is the connected symplectic group $Sp(V_{\lambda})$. In the $GO$-case the kernel of $\mu$ is the orthogonal
 group $O(V_{\lambda})$, and the kernel of $sgn$ on  $O(V_{\lambda})$
is the connected group  $SO(V_{\lambda})$.
 The kernel of $sgn$ on $GO(V_{\lambda})$ is the connected subgroup $GSO(V_{\lambda})$.

\medskip\noindent
The Tannaka group $H_{\lambda}$ of the Tannaka category $\mathcal{T}_{\lambda} = <\! X_\lambda\!>$ generated by $X_\lambda$ acts faithfully on $V_\lambda= \omega(X_\lambda)$ such that
 $  <h v_1,h v_2> \ = \ \mu(h) \cdot <v_1,v_2> $
holds for all $h\in H_\lambda$. Hence
$H_\lambda$ is a subgroup of $G\bigl(V_{\lambda},\langle .,.\rangle\bigr)$.

\medskip
{\it A priori bounds}. We distinguish two cases: Either $X_\lambda$ is a weakly selfdual object (SD), i.e. $X_\lambda^\vee \cong
B^r \otimes X_\lambda$ for some $r$; or alternatively $X_\lambda$ is not weakly selfdual (NSD). 
Summarizing we obtain the following bounds for the groups $H_{\lambda}$. We have 
$   H_\lambda \subseteq GL(V_\lambda) $
in the case (NSD) and
$   H_\lambda \subseteq GO(V_\lambda)$ resp. $ H_\lambda \subseteq GSp(V_\lambda) $ for even resp. odd $\varepsilon(X_\lambda)$ in the (SD)-cases. If 
$X_\lambda$ is properly self dual in the sense $X_\lambda^\vee \cong X_\lambda$,
these simlitude groups can be replaced by their subgroups $O(V_\lambda)$ resp.
$Sp(V_\lambda)$.

\subsection{The structure of $G_{\lambda}$} Recall that two maximal atypical weights $\lambda, \ \mu$ are equivalent $\lambda \sim \mu$ if there exists $r\in \Z$ such that 
$L(\lambda) \cong Ber^r \otimes L(\mu)$ or $L(\lambda)^\vee \cong Ber^r \otimes L(\mu)$ holds. Another way to express this is to consider the restriction of the representations $L(\lambda)$
and $L(\mu)$ to the Lie superalgebra $\mathfrak{sl}(n|n)$. These restrictions remain irreducible
and $\lambda \sim \mu$ holds if and only if $L(\lambda) \cong L(\mu)$ or $L(\lambda)\cong L(\mu)^\vee$ as representations of $\mathfrak{sl}(n|n)$. Let $X^+(n)$ be the set of dominant weights and 
let $Y^+(n)$ be the set of equivalence classes of dominant weights. Similarly let $X^+_0(n)$ denote
the class of maximal atypical dominant weights and $Y^+_0(n)$ the set of corresponding equivalence classes. 
If we write $\lambda \in Y^+_0(n)$, we mean that $\lambda \in X^+_0(n)$ 
is some representative of the class in $Y^+_0(n)$ defined by $\lambda$. 

Let $\lambda$ be of SD-type. Then there exists $r\in \mathbb Z$ such that
$L(\lambda) \cong Ber^r \otimes L(\lambda)^\vee$. Hence there exists an equivariant
nondegenerate pairing
$$  X_\lambda \times X_\lambda \longrightarrow B^r \ .$$
This pairing is either symmetric (even) or antisymmetric (odd). We then say $X_{\lambda}$ is even and put $\varepsilon(X_\lambda)=1$, or odd and $ \varepsilon(X_\lambda)=-1$. The next lemma is proven in appendix \ref{sec:pairings}.

\begin{lem} \label{duality-type} For all irreducible objects $X_\lambda$ of SD-type in $\mathcal{T}_n^+$ we have
$$  \varepsilon(X_\lambda) = \varepsilon(X_{\lambda_{basic}}) = (-1)^{p(\lambda_{basic})}.$$
\end{lem}

\begin{thm} \label{Tannakagroup}
$G_\lambda = SL(V_\lambda)$ if $X_{\lambda}$ is (NSD). If $X_{\lambda}$ is (SD) and  $V_{\lambda}|_{G_{\lambda'}}$ is irreducible, then $G_\lambda = SO(V_\lambda)$
respectively $G_\lambda = Sp(V_\lambda)$ according to whether $X_\lambda$ is even respectively odd. If $X_{\lambda}$ is (SD) and  $V_{\lambda}|_{G_{\lambda'}}$ is not irreducible, then $G_\lambda \cong SL(W)$ for $V_{\lambda}|_{G_{\lambda'}} \cong W \oplus W^{\vee}$. 
\end{thm}

This theorem is proven in sections \ref{sec:overview} - \ref{proof-derived}. Many examples can be found in section \ref{sec:ind-start}. We conjecture that a stronger version is true: $V_{\lambda}$ should always stay irreducible. We refer to section \ref{sec:conjecture} for a discussion of this case.

\begin{remark} The (NSD) case is the generic case for $n \geq 4$. Since $SL(V_{\lambda}) \cong G_{\lambda} \subset GL(V_{\lambda})$, all representations of $H_{\lambda}$ stay irreducible upon restriction to $G_{\lambda}$. Hence the derived group sees already the entire tensor product decomposition into indecomposable representations up to superdimension zero. 
The same remark is true for a selfdual weight of symplectic type. In the orthogonal case we could have a decomposition of an irreducible representation of $H_{\lambda}$ into two irreducible representations of $G_{\lambda}$ since $O(V_{\lambda})$ and $GO(V_{\lambda})$ have two connected components. 
\end{remark}

\begin{example} The smallest case for which $V_{\lambda}$ could decompose when restricted to $G_{\lambda}$ is the case $[\lambda] = [3,2,1,0] \in \mathcal{T}_4^+$ with sector structure

\begin{center}
\medskip
 
 \scalebox{0.7}{
\begin{tikzpicture}
\foreach \x in {-3,-1,1,3} 
     \draw[very thick] (\x-.1, .1) -- (\x,-0.1) -- (\x +.1, .1);
\foreach \x in {-2,0,2,4} 
     \draw[very thick] (\x-.1, -.1) -- (\x,0.1) -- (\x +.1, -.1);
%

\draw (-3,-0.5) node {-3};
\draw (-2,-0.5) node {-2};
\draw (-1,-0.5) node {-1};
\draw (0,-0.5) node {0};
\draw (1,-0.5) node {1};
\draw (2,-0.5) node {2};
\draw (3,-0.5) node {3};
\draw (4,-0.5) node {4};


\draw[very thick] [-,black,out=90, in=90](-3,+0.2) to (-2,+0.2);
\draw[very thick] [-,black,out=90, in=90](-1,+0.2) to (0,+0.2);
\draw[very thick] [-,black,out=90, in=90](1,+0.2) to (2,+0.2);
\draw[very thick] [-,black,out=90, in=90](3,+0.2) to (4,+0.2);


\end{tikzpicture} }
\smallskip

\text{The cup diagram of $L(\lambda)$}
\end{center}

Then $DS[3,2,1,0]$ decomposes into four irreducible representations \[ L_1 = [3,2,1], L_2 = [3,2,-1], L_3 = [3,0,-1], L_4 = [1,0,-1 ]\ . \]

Since $L_1 = Ber^{2}L_4$ and $L_2 \cong L_3^{\vee}$ we have two equivalence classes \[\{ L_1,L_4 \}, \{ L_2,L_3 \}.\] In fact $ V_{\lambda_1} \cong V_{\lambda_4} \cong st(SO(6))$
and $ V_{\lambda_2} \cong st(SL(6)), V_{\lambda_3} \cong st(SL(6))^{\vee}$. 

If $V_{[3,2,1,0]}$ does not decompose under restriction to $G_{[3,2,1,0]}$, then $G_{\lambda} \cong SO(24)$ and $V_{\lambda} \cong st(SO(24))$. If it decomposes $V_{\lambda} = W \oplus W^{\vee}$, then $G_{\lambda} \cong SL(12)$ and $W \cong st(SL(12))$. Since $W \nsim W^{\vee}$ this implies that the embedding $SO(6) \times SL(6) \to SL(12)$ gives the branching rules \begin{align*} W & \mapsto st(SL(6)) \oplus st(SO(6)) \\ W^{\vee} & \mapsto st(SL(6))^{\vee} \oplus st(SO(6)).\end{align*}

\end{example}

\subsection{The structure theorem on $G_n$} We now determine $G_n$.

\begin{lem}
Suppose a tannakian category $\calR$ with Tannaka group $H$ is $\otimes$-generated as a tannakian category
by the union of two subsets $V'$ and $V''$. Let $H'$ and $H''$ be the Tannaka groups of the
tannakian subcategories generated by $V'$ respectively $V''$. Then there exists
an embedding $H \hookrightarrow H' \times H''$ so that the composition with the projections
is surjective. 
\end{lem}

\begin{proof} For arbitrary $V', \ V''$ (not necessarily finite) we get $\mathcal{T}(V') \hookrightarrow T(V' \cup V'')$, $\mathcal{T}(V'') \hookrightarrow \mathcal{T}(V \cup V")$ for the tensor categories $\mathcal{T}$ generated by $V'$, $V"$, $V' \cup V"$ respectively. This gives natural epimorphisms $\pi': H \to H'$ and $\pi'': H\to H''$
which induce a morphism $i: H \to H'\times H''$ so that the composition
with the projections are $\pi'$ and $\pi''$. It remains to show that $i$ is injective. The morphism $i$ is injective since $g \in H$ is trivial if it acts trivially on all generators in $V' \cup V''$ of $\mathcal{T}(V' \cup V'')$.
\end{proof}

\medskip
We remark that the inclusion $H \hookrightarrow H' \times H''$ induces an 
inclusion $H^0 \hookrightarrow (H')^0 \times (H'')^0$ of the Zariski connected components
and hence an inclusion of the corresponding adjoint groups $H^0_{ad} := (H^0)_{ad}$
$$ H^0_{ad} \hookrightarrow (H')^0_{ad} \times (H'')^0_{ad} \ ,$$
and, abbreviating
$ H^0_{der} \hookrightarrow (H')^0_{der} \times (H'')^0_{der}$,
similarly for the derived groups $G:=H^0_{der} := (H^0)_{der}$
$$ G \hookrightarrow G' \times G'' \ .$$

\medskip\noindent

We also need the following variant of Goursat's lemma.

\begin{lem} \label{product}
Suppose $H$ is a connected reductive subgroup of the product $A\times B$ of two semisimple affine algebraic $k$-groups $A$ and $B$, so that the projections to $A$ and $B$ are surjective. Then \begin{enumerate}
\item If $A$ and $B$ are connected simple $k$-groups, then either $H_{ad}= A_{ad}\times B_{ad}$
or $H_{ad}\cong A_{ad} \cong B_{ad}$.   
\item $H\cong A\times B$, if $A$ and $B$ are of adjoint type without common factor.  
\item If $A$ and $B$ are connected, $H \cong A \times B$ if and only if $H_{ad} \cong A_{ad} \times B_{ad}$.
\item Suppose $A$ is a connected semisimple group and $B$ is a connected simple group. Let $H$ be a proper subgroup $H$ of $A\times B$, that surjects onto $A$ and $B$ for the projections. Then there exists a simple normal subgroup $C$ of $A$, such that the image $H/C$ of $H$ in $(A/C) \times B$ is a proper subgroup of $(A/C) \times B$, if $A $ is not a simple group.
\end{enumerate}
\end{lem}

\medskip
{\it Proof}. (1)-(3) are obvious. Part (4) can be reduced to the case of adjoint groups by part (3).
So we may assume that $B$ and $A$ are groups of adjoint type. We now use the following
fact. Any semisimple $A$ group of adjoint type is isomorphic to the product $\prod_{i=1}^r A_i$
of its simple subgroups $A_i$. Its factors are the normal simple subgroups of $A$. These factors and hence this product decomposition is unique up to a permutation of the factors. 
Any nontrivial algebraic homomorphism of $A$ to a simple group $B$ is obtained as projection
of $A$ onto some factor $A_i$ of the product decomposition composed with an injective homomorphism $A_i \to B$.
Since $H \subseteq A \times B$ projects onto the first factor $A$ and  $B$ is simple, and since $H$ is a proper subgroup of the connected semisimple group $A\times B$, the kernel of the projection
$p_A: H\to A$ is a finite normal and hence central subgroup of $H$. It injects into the center of $B$, hence is trivial. Thus $p_A:H\to A$ is an isomorphism so that $H$ defines the graph of a group homomorphism $A \to B$. Since
$A$ is of adjoint type and therefore a product of simple groups $A \cong \prod_{i=1}^r A_i$,
the kernel of the homomorphism $A\to B$ must be of the form $\prod_{i\neq j} A_i$. Unless $A$ is simple, for $C=A_i$ and any $i\neq j$ assertion (4) becomes obvious.  \qed

\begin{cor} \label{thm:two-factors} Let $\lambda$ and $\mu$ be two maximal atypical weights and denote by $G_{\lambda,\mu}$ the connected derived group of the Tannaka group $H_{\lambda,\mu}$ corresponding to the subcategory in $\overline{\mathcal{T}}_n$ generated by $L(\lambda)$ and $L(\mu)$. If $\lambda$ is not equivalent to $\mu$, \[ G_{\lambda,\mu} \cong G_{\lambda} \times G_{\mu}.\]
\end{cor}

\begin{proof} If $G_{\lambda}$ and $G_{\mu}$ are not isomorphic, lemma \ref{product} implies the claim. Otherwise $G_{\lambda,\mu} \cong G_{\lambda} \cong G_{\mu}$ (special case of lemma \ref{product}.1). We assume by induction that the statement holds for smaller $n$. By theorem \ref{asymmetry} there exists either $L(\lambda_i)$ which is not equivalent to any $L(\mu_j)$ or there exists $L(\mu_j)$ which is not equivalent to any $L(\lambda_i)$ - a contradiction since then the branching of $V_{\lambda}$ and $V_{\mu}$ to $G_{n-1}$ would not be the same. For $n=2$ we give an adhoc argument in section \ref{sec:ind-start}.\end{proof}

%
%

\begin{thm} {\bf Structure Theorem for $G_n$}. \label{derived-str-thm} The connected derived group $G_n$ of the Tannaka group
$H_n$ of the category $\mathcal{T}_n^+$ is isomorphic to the product
$$    G_n \ \cong \ \prod_{\lambda \in Y^+_0(n)}  G_\lambda \ . $$

\end{thm}

\begin{proof} This follows essentially from theorem
\ref{Tannakagroup}, where the structure of the individual groups
$G_\lambda$ was determined. Using lemma \ref{product}, one reduces the statement of the theorem to 
a situation that involves only two
inequivalent weights $\lambda$ and $\mu$:  By part (3) of lemma \ref{product} we may replace the derived groups by the adjoint groups. Then the assertion follows from part (4) of the lemma by induction on the number of factors reducing the assertion to the case of two groups $G_{\lambda}$, $G_{\mu}$ dealt with in corollary \ref{thm:two-factors}.
\end{proof}

\begin{example} Consider the tensor product of two inequivalent representations $L(\lambda)$ and $L(\mu)$ of non-vanishing superdimension. Then \[ L(\lambda) \otimes L(\mu) = I \ \mod \ \calN \] for an indecomposable representation $I$. Indeed $L(\lambda)$ and $L(\mu)$ correspond to representations of the derived connected Tannaka groups $G_{\lambda}$ and $G_{\mu}$. Since $G_{\lambda}$ and $G_{\mu}$ are disjoint groups in $G_n$, tensoring with $L(\lambda)$ and $L(\mu)$ corresponds to taking the external tensor product of these representations.
\end{example}



\section{ Proof of the structure theorem: Overview} \label{sec:overview}

We now determine $G_{\lambda}$ inductively using the $k$-linear exact tensor functor
between the quotient categories of the representation categories
\[  \eta:  \overline{\mathcal{T}}_{n} \to \overline{\mathcal{T}}_{n-1} \]
constructed in lemma \ref{DS-respects-negligibles} with the help of $DS: \mathcal{T}_n^+ \to \mathcal{T}_{n-1}^+$. By the main theorem on $DS$ (theorem \ref{mainthm}), the restriction of $V_{\lambda} = \omega(X_{\lambda})$ to the subgroup $H_{n-1}$ is a multiplicity free representation. We assume by induction that theorem \ref{Tannakagroup} and theorem \ref{derived-str-thm} hold for $H_{n-1}$ and $G_{n-1}$. 

\medskip\noindent
We have inclusions
$$   G_{\lambda'} \hookrightarrow G_\lambda \hookrightarrow H_\lambda^0 \hookrightarrow H_\lambda $$
where $G_{\lambda'}$ denotes the image of the natural map $(H_{n-1}^0)_{der} \to G_\lambda = (H_\lambda^0)_{der}$. The restriction of $V_\lambda$ to $G_{\lambda'}$ decomposes
$$ V_\lambda \ \cong \ \bigoplus_{i=1}^k \ V_{\lambda_i}$$ where the  $V_{\lambda_i}$ are the irreducible representations in the category $\overline{\mathcal{T}}_{n-1}^+$ corresponding to the irreducible constituents $L(\lambda_i), i=1,..,k$, of $DS(L(\lambda))$. By induction we obtain 
 $$G_{\lambda'}  \cong  \prod_{\lambda_i /\sim} G_{\lambda_i} $$ where the $G_{\lambda_i}$ are described in theorem \ref{Tannakagroup}.

\medskip
In a first step we discuss the situation in the $n=2$ and the $n=3$ case as well as the Tannaka groups $G_{\lambda}$ for $L(\lambda) =Ber^r \otimes [i,0,\ldots,0]$, $r,i \in \Z$. The $n=2$-case is needed for the start of the inductive determination of $G_n$. In this case we can use the known tensor product decomposition between irreducible modules in $\mathcal{T}_2$ to determine $G_2$ and $H_2$. In order to get a clear induction scheme in the proof of the structure theorem, we need to rule out certain exceptional cases which can only occur for $n \leq 3$ and for the modules $Ber^r \otimes [i,0,\ldots,0]$.  This will allow us to assume $n \geq 4$ in section \ref{proof-derived}.

\medskip\noindent
In the next step we show that $G_{\lambda}$ is simple. By induction all the $V_{\lambda_i}$ are standard representations for simple groups of type $A,B,C,D$ or $V_{\lambda_i}|_{G_{\lambda_i}} = W \oplus W^{\vee}$ for $G_{\lambda_i} \cong SL(W)$. The representation $V_{\lambda}$ decomposes under restriction to $G_{\lambda}$ in the form $W_1 \oplus \ldots \oplus W_s$ (we later show that $s$ is at most 2). If we restrict these $W_{\nu}$ to $G_{\lambda'}$, they are meager representation of $G_{\lambda'}$ in the sense of definition \ref{def-meager}. The crucial lemma \ref{meager-lemma} shows then that $G_{\lambda}$ is simple. This allows us to use the classification of small representations due to Andreev-Elashvili-Vinberg. 

\medskip\noindent
Our aim is then to show that the dimension of the subgroup $G_{\lambda'}$ is large compared to the dimension of $V_{\lambda}$ (given by the superdimension formula for $L(\lambda)$ in \cite{Heidersdorf-Weissauer-tensor}) as in lemma \ref{small} or corollary \ref{verysmall}. A large rank and a large dimension of $G_{\lambda'}$ implies that the rank and the dimension of $G_{\lambda}$ must be large, forcing $V_{\lambda}$ to be a small representation of $G_{\lambda}$ in the sense of lemma \ref{small} and corollary \ref{verysmall}. If we additionally know that $G_{\lambda}$ is simple and that also $r(G_{\lambda}) \geq \frac{1}{2}(\dim(V_{\lambda})-1)$ , corollary \ref{verysmall} will immediately imply that $G_{\lambda}$ is of type $SL(V_{\lambda}), \ SO(V_{\lambda})$ or $Sp(V_{\lambda})$. However the strong rank estimate will not always hold and we will be in the less restrictive situation of lemma \ref{small}.

\medskip\noindent
Here the (NSD) and the (SD) case differ considerably. In the (NSD) case each irreducible representation $V_{\lambda_i}$ (corresponding to $L(\lambda_i)$ in $DS(L(\lambda))$) gives a distinct direct factor in the product $G_{\lambda'}  \cong  \prod_{\lambda_i /\sim} G_{\lambda_i}$ since all irreducible representations of $DS(L(\lambda))$ are inequivalent in the (NSD) case by lemma \ref{duality}. The dimension estimate for $G_{\lambda}$ so obtained then implies that $V_{\lambda}$ is a small representation. In the (SD) case however two representations $V_{\lambda_i}, V_{\lambda_j}$ will contribute the same direct factor $G_{\lambda_i} \simeq G_{\lambda_j}$ if $\lambda_i \sim \lambda_j$. This decreases the dimension and rank estimate of the subgroup $G_{\lambda'}$ in $G_{\lambda}$ and therefore of $G_{\lambda}$. 

\medskip\noindent
To finish the proof we need to understand the restriction of $V_{\lambda}$ to $G_{\lambda}$. The group of connected components acts transitively on the irreducible constituents $V_{\lambda} = W_1 \oplus \ldots \oplus W_s$ of the restriction to $H_{\lambda}^0$ and $G_{\lambda}$. Using that the decomposition of $V_{\lambda}$ to $H_{n-1}$ is multiplicity free in a weak sense (obtained from an analysis of the derivatives of $L(\lambda)$ in section \ref{equivalences}), we show finally in section \ref{mackey-clifford}, using Clifford-Mackey theory, that $V_{\lambda}$ can decompose into at most $s=2$ irreducible representations of $G_{\lambda}$.



\section{ Small representations}\label{sec:small}

Our aim is to understand the Tannaka groups associated to an irreducible representation by means of the restriction functor $DS : \mathcal{T}_n^+ \to \mathcal{T}_{n-1}^+$. We have a formula for the superdimension of an irreducible representation  \cite{Heidersdorf-Weissauer-tensor} and we know inductively the ranks and dimensions of the groups arising for $k <n$. This gives strong restrictions about the groups in the $\mathcal{T}_n^+$-case due to the following list of small representations.

\medskip{\it List of small representations}. For a simple connected algebraic group $H$ and a nontrivial irreducible representation $V$ of $H$
the following holds \cite{Andreev-Vinberg-Elashvili}

\begin{lem} \label{small} $\dim(V) = \dim(H)$ implies that $V$ is isomorphic to the adjoint
representation of $H$. Furthermore, except for a finite number of
exceptional cases, $\dim(V) < \dim(H)$ implies that $V$ belongs to the regular cases
\begin{enumerate}[label=\textbf{R.\arabic*}]
\item $V\cong st, S^2(st), \Lambda^2(st)$ or their duals in the $A_r$-case, 
\item $V=st$ (the standard representation) in the $B_r,D_r$-case, 
\item $V\cong st$ in the $C_r$-case,
\item $V \hookrightarrow \Lambda^2(st)$ in the $C_r$-case 
\end{enumerate}
where the list of exceptional cases is
\begin{enumerate}[label=\textbf{E.\arabic*}]
\item $\dim(V)=20,35,56$ for $V=\Lambda^3(st)$ and $A_r$ in the cases $r=5,6,7$.
\item $\dim(V)=4,8,16,32,64$ for the spin representations of $B_r$ in the cases $r=2,3,4,5,6$.
\item $\dim(V)=8,8,16,16,32,32,64,64$ for the two spin representations of $D_r$ in the cases
$r=4,5,6,7$.
\item $\dim(V)=27,27$ for $E_6$ with $\dim(E_6) =78$ (standard representation and its dual).
\item $\dim(V)=56$ for $E_7$ with $\dim(E_7)=133$.
\item $\dim(V)=7$ for $G_2$ with $\dim(G_2)=14$.
\item $\dim(V)=26$ for $F_4$ with $\dim(F_4)=52$.
\end{enumerate}
In particular $\dim(V) \geq r+2$ holds, except for $G=A_r$ in the cases $V\cong st$ or $V\cong st^\vee$.
\end{lem}

\begin{cor} \label{verysmall}
Let $V$ be an irreducible representation of a simple connected group $H$
such that $4 \leq \dim(V) < \dim(H)$ and \[2r(H) \geq \dim(V) -1\] holds. 
Then $H$ is of type $A_r, B_r, C_r,D_r$ and $V=st$ the standard representation 
of this group of dimension $r+1,\ 2r+1, \ 2r, \ 2r$ for $r\geq 3,\ 2,\ 2,\ 2$ respectively,
or $H=D_4$ and $V$ is one of the two 8-dimensional spin representations.
\end{cor}

From the classification in lemma \ref{small} one also obtains

\begin{lem}\label{trivial}
For a simple connected grous $H$ with an  irreducible root system of rank $r$   we have $\dim(H) \geq r(2r-1)$  except for $H\cong SL(n)$ with $\dim(H) = r(r+2)$. Furthermore $r\leq \dim(V)$ holds  for any nontrivial irreducible representation $V$ of $H$.
\end{lem}




\section{ The cases $n=2, 3$ and the $S^i$-case} \label{sec:ind-start}

In the next sections we determine the group $G_n$ and the groups $G_{\lambda}$. Since we will determine these groups inductively starting from $n=2$, we need to start with this case. We also discuss the $n=3$ case separately since we have to rule out some exceptional low rank examples in the classification of \cite{Andreev-Vinberg-Elashvili} in section \ref{sec:small}.

\medskip
{\it Warm-up}. Suppose $n=1$. Then $H_1$ is the multiplicative group
$\mathbb G_m$. Indeed the irreducible representations of it correspond
to the irreducible modules $\Pi^i Ber^i$ for $i\in \mathbb Z$.

\subsection{The case $n=2$} For $S^i = L([i,0])$ and $i\geq 1$ let denote
$$X_i:=\Pi^i([i,0])\ .$$  
Then $X_i^\vee \cong B^{1-i} \otimes X_i$, hence $X_1^{\vee} \cong X_1$.
We use from \cite{Heidersdorf-Weissauer-GL-2-2} the fusion rule \[  [i,0] \otimes [j,0] \ = \ \text{indecomposable}\ \oplus \ \delta_i^j \cdot  Ber^{i-1}   \oplus \text{negligible} \]
for $1\leq i \leq j$ together with $Ber^r \otimes [i,0] \cong [r+i,r]$ for all $r\in \mathbb Z$.

\begin{lem} If $H_{X_i}$ denotes the Tannaka group of $X_i$, then \[ H_{X_i} \simeq \begin{cases} SL(2) \quad & i = 1\\ GL(2) & i \geq 2.\end{cases} \]
\end{lem}

{\it Proof.} Since $H_1 \hookrightarrow H_2\twoheadrightarrow H_{X_i} $
can be computed from $DS$ we see that $H_1$ injects into $H = H_{X_i}$ and the two dimensional irreducible 
representation $V = V_{X_i}$  of $H_{X_i}$ attached to $X_i$ decomposes into
$$   V\vert_{H_1} \ = \ det^{-1} \oplus det^i \ .  $$
corresponding to $DS(X_i) \ = \ Ber^{-1} \oplus Ber^i$. 
If $H^0_{X_i} \cong \mathbb G_m$, the finite group $\pi_0(H)$ acts on
$H^0$. By Mackey's theorem the stabilizer of the character $Ber^{-1}$ has index two in $H_{X_i}$ and acts by a character
on $V$. Since the only automorphisms of $\mathbb G_m$ are the identity and the inversion,
this would imply $i=1$. Hence $V\otimes V$ would restrict to $\mathbb G_m$ with
at least three irreducible constituents $det^{-2} \oplus det^2$ (corresponding to $Ber^{-2} \oplus Ber^2$) and a two dimensional 
module $W$ with an action of $\pi_0(H)$ such that a subgroup of index two acts by a character.
But $X^\vee_1 \cong X_1$ implies that $V$ is self dual, and hence $W$ contains the trivial representation. This contradicts the fusion rule from above. Hence $H^0 \neq \mathbb G_m$
and the same argument as above shows that $H^0$ can not be a torus. Hence the
rank $r$ of each irreducible component of the Dynkin diagram of $(H^0_{der})_{sc}$ is $r\geq 1$
and hence $\dim(H)\geq 3$. 
By lemma \ref{trivial} we know $r\leq \dim(V)=2$ and accordingly $\dim(H)=3$ by 
lemma \ref{small}. Therefore $(H^0_{der})=SL(2)$ and $V\vert_{H_{der}^0}$ is the irreducible
standard representation. Since $H$ acts faithful on $V$
$$  SL(2) \subseteq H \subseteq GL(2) \ .$$
Now we use $V^\vee \cong Ber^{i-1} \otimes V$, which implies $H = GL(2)$ for $i>1$. Indeed $\Lambda^2(V)$ is the character $Ber^{i-1}$ by the fusion rules above.
For $i=1$ the isomorphism $V^\vee \cong V$ implies that $det(V)$ is trivial on
$H$, hence 
$$H=SL(2)$$ in the case $i=1$. \qed

\subsection{The $H_2$-case} We discuss the Tannaka group  generated by
all irreducible representations. First consider the Tannaka group $H$ of $\langle X_i , X_j
\rangle_\otimes$ for some pair $j > i$. The derived groups of the Tannaka groups $H'$ resp. $H''$ of 
$\langle X_i 
\rangle_\otimes$ and $\langle  X_j
\rangle_\otimes$
are $SL(2)$.

\medskip\noindent
We claim that $H_{der} \cong H'_{der} \times H''_{der}$.
If this were not the case, then $H_{der} \cong SL(2)$ (special case of lemma \ref{product}.1).
But then the tensor product $X_i \otimes X_j$ considered as a representation of $H$
corresponds to the tensor product of two standard representation and hence is a reducible
representation with two irreducible factors. However this contradicts the fusion rules stated above.
This implies $H_{der} \cong SL(2) \times SL(2)$ and hence $H_{ad} \cong H'_{ad} \times H''_{ad}$.

\medskip\noindent
Now consider the Tannaka group $H$ of $\langle X_{i_1},...,X_{i_k}\rangle_\otimes $ for $k > 2$. We claim
that $H$ is connected and that it is the product
$$  H_{der} \ \cong \ \prod_{\nu=1}^k H_{der}(X_{i_\nu}) $$
of the derived Tannaka groups of the $\langle X_{i_\nu}\rangle_\otimes $.
This is an immediate consequence of lemma \ref{product}

\medskip\noindent
The fusion rule $S^i \otimes S^i \cong Ber^{i-1} \oplus \text{indecomposable} \oplus \text{negligible}$ implies $\Lambda^2(X_i) \cong B^{i-1} \oplus \text{negligible}$. In particular the image of $B^{i-1}$ is contained in $Rep(H_{X_i})$ and generates a subgroup of form $\mathbb{G}_m$. So the Tannaka group $H_2$ of the category $\mathcal{T}_2^+/\mathcal{N}$ 
sits in an exact sequence 
$$    0 \to \lim_{k}  \prod_{\nu=0}^{k-1} SL(2) \to H_2 \to \mathbb G_m \to 0 \ .$$
The derived group of $H_2$ is the projective limit of groups $SL(2)$ with a copy
for each irreducible object $X_{\nu+1}$ for $\nu=0,1,2,3,...$.
The structure of the extension is now easily recovered from the 
following decription:

\begin{lem}\label{H_2-for-GL-2-2} $  H_2 \ \subset \  \prod_{\nu=0}^\infty GL(2) $
is the subgroup defined by all elements $g=\prod_{\nu=0}^\infty g_\nu $ in the product with the property $det(g_\nu) = det(g_1)^\nu$. The automorphism $\tau_2$ is inner.
\end{lem}

\medskip
We usually write $GL(2)_\nu$ for the $\nu$-th factor of the product $\prod_{\nu=0}^\infty GL(2) $.
Using the description of the last lemma, the torus $H_1\cong \mathbb G_m$ embeds into
$H_2$ as follows
$$  H_1 \ni t \mapsto \prod_{\nu=0}^\infty diag(t^{\nu+1},t^{-1}) \in H_2 \subset \prod_{\nu=0}^\infty GL(2)_\nu \ .$$
Defining $det(g)=det(g_1)$ for $g=\prod_{\nu=0}^\infty g_\nu$ in $H_2$, the representation of the quotient group $\mathbb G_m$ of $H_2$ defined
by the  Berezin determinant $Ber \in \mathcal{T}_2$,
corresponds to the character $det(g)$ of the group $H_2$.



\bigskip\noindent
We continue with two special cases: The $S^i$-case for any $n$, and the case $G_3$.

\subsection{The $S^i$-case} Consider the modules $X_i = \Pi^i([i,0,0])$ in $\mathcal{T}_3^+$.
They have superdimension $3$ for $i\geq 2$. Let $H$ (or sometimes $H_{X_i}$) denote the associated Tannaka group
and $V$ the associated irreducible representation of $H$. 

\begin{lem} We have $H_{X_1} = SL(2)$ and $G_{X_i} \simeq SL(3)$ for any $i \geq 2$ and $H_{X_i} \simeq GL(3)$ for any $i \geq 3$.
\end{lem}

{\it Proof}. The natural map $H_2 \to H_3 \to H$ allows to consider $V$ as a 
representation of $H_2$, and as such we get
$$  V\vert_{H_2} \ \cong \ {\det}^{-1}\ \oplus\  X_i \ $$
for $i\geq 2$ (here $X_i$ on the right is the irreducible 2-dimensional
standard representation of $GL(2)_{i-1}$, restricted to $H_2$).
Hence $\dim(A) \geq 3$ for at least one simple factor $A$ of $H^0$ and every irreducible summand $W$ of $V\vert_A$ has dimension $\leq \dim(A)$.
By lemma \ref{small} therefore $W$ either has dimension $3$ and $A_{sc}=SL(3)$, $W=st$
or $W=st^\vee$, 
or $A_{sc} =SL(2)$ and $W=S^2(st)$. If $H^0_{der}$ is not simple, we replace it by its simply connected cover and write $(H^0_{der})_{sc} = A_{sc} \times A'$ (where $A'$ is a product of simple groups). The 
representation $V$ is then an external tensor product $$ V = W \boxtimes W' $$
of irreducible representations $W, W'$ of $A_{sc}$ and $A'$. Since $V$ is a faithful representation of $H$, the lift of $V$ (again denoted $V$) to $(H^0_{der})_{sc}$ has finite kernel. Since it has finite kernel, $\dim(W)>1, \ \dim(W') > 1$ holds. Hence $\dim(W) = 3$ implies $(H^0_{der})=A$ and $V\vert_{H^0}$ and $V\vert_{H^0_{der}}$ remain irreducible by dimension reasons. 
If $A_{sc} =SL(2)$ and $W=S^2(st)$,
the image of $H_2$ surjects onto $H_{der}$. This contradicts the fact that $V$ is irreducible
but $V\vert_{H_2}$ decomposes, and excludes the case $A_{sc}=SL(2)$. Hence
$$  H^0_{der} \ \cong \ SL(3) \ .$$ 
Since $H$ acts faithfully on $V$, we also have
$  H \subseteq GL(V)=GL(3)$.
The restriction of $V$ to $H_2$ has determinant
$det^{-1} \cdot \det(X_i) \cong det^{-1}det^{i-1} = det^{i-2}$.
Hence 
$$   H \ \cong \ GL(3) $$
for all $i\geq 3$. \qed


\medskip
For $j > i \geq 2$ let $H$ denote the Tannaka
group of $\langle X_i, X_j \rangle_\otimes$ and $H',H''$ the connected components of the Tannaka
groups of $\langle X_i \rangle_\otimes$ resp. $\langle X_j \rangle_\otimes$.
Then we claim
$$     H^0_{der} \ \cong \ H'_{der} \times H''_{der} \ ,$$ 
since otherwise $H'_{der} \cong H''_{der}$ by lemma \ref{product}.1.
But this is impossible since then the morphisms $H_2 \to H_3 \to H$
would induce the same morphisms $(H_2)_{der} \to H_{der} \to H'_{der}$
and   $(H_2)_{der} \to H_{der} \to H''_{der}$, which contradicts theorem
\ref{mainthm}. Indeed the factor $SL(2)_{i-1}$ maps nontrivially to $H'_{der}$
but trivially to $H''_{der}$. Since $H$ acts faithfully on the
representation associated  to the object $X_i \oplus X_j$ on the other hand
$H \subseteq GL(\omega(X_i)) \times GL(\omega(X_j))$.

\medskip\noindent
The same arguments enable us to determine the connected derived groups for any $n \geq 3$: 

\begin{lem} The Tannaka group $H$ of the modules $X_i=\Pi^i([i,0,..,0])$ in $\mathcal{T}_n^+$
satisfies $H^0_{der} \cong SL(n)$ and $H \subseteq GL(n)$ for all $i\geq n-1$,
and $H=GL(n)$ for all $i\geq n$. For $i < n-1$ we get $H^0_{der} \cong SL(\sdim(X_i))$. 
\end{lem}

{\it Proof.} Indeed we have in $H^0_{der}$ a simple component $A$ of semisimple rank $r \geq n-1$ by induction.  Obviously $A$ contains $SL(n-1)$ 
and cannot be of Dynkin type $A_r$ unless $A=SL(n)$ by lemma \ref{small}. 

Notice that
$\dim(A) \geq r(2r-1) \geq (n-1)(2n-3) > n$  or $\dim(A) \geq r(r+2) \geq (n-1)(2n) > n$, for $n\geq 3$ 
by lemma \ref{trivial}. 
The restriction of $V$ decomposes into irreducible
summands $W,W',...$ of dimension $\dim(W)\leq n$, and the dimension of all these representations is $\leq  r$. So the possible representations are listed
in lemma \ref{small}. None of them has dimension $\leq r+1$ except for
the case where $A$ is of type $A_r$ and $V\cong st$ or $V\cong st^\vee$.\qed

\subsection{The $n=3$-case}  We analyse the remaining $n=3$-cases.

\begin{lem}
The derived connected group $G_3=(H_3)^0_{der}$ of $H_3$ is 
$$ G_3 \ \cong \ \prod_{\lambda} G_\lambda \ ,$$  where $\lambda$ runs over 
all $\lambda=[\lambda_1,\lambda_2,0]$ with integers $\lambda_1,\lambda_2$ such that $$0\leq 2\lambda_2 \leq \lambda_1 \ $$
and $G_\lambda\cong 1, SL(2),SL(3),Sp(6),SL(6)$ according to whether
$\lambda$ is $0, [1,0,0]$ or $[2+\nu,0,0]$, for $\nu\geq 0$, or $\lambda = [2\lambda_2,\lambda_2,0]$, for
$\lambda_2> 0$, or $0<2 \lambda_2 <  \lambda_1$. 
\end{lem}

\begin{remark} We discuss the general case in the next section assuming $n \geq 4$. The assumption $n \geq 4$ is only relevant because we want to have a uniform behaviour regarding derivatives. Essentially all the arguments regarding simplicity of $G_{\lambda}$ and Clifford-Mackey theory apply to the $n=3$ case at hand. In the proof we discuss $[2,1,0]$ in detail and sketch the key inputs for the other cases. 
\end{remark}

{\it Proof}. Let us consider $X=X_\lambda$ for $L(\lambda)=[210]$. The associated irreducible
representation  Tannaka group $H = H_X$
admits an alternating pairing,
hence $H_X$ is contained in the symplectic group 
of this pairing $$ H_X \ \subseteq\ Sp(6)\ .$$
We claim that $H^0_{der}$ is simple. If not, we replace it by its simply connected cover and write it as a product \[ (H^0_{der})_{sc} = G_1 \times G_2.\] The faithful representation $V_X$ of $H_X$ has finite kernel when seen lifted to a representation of $(H^0_{der})_{sc}$. Therefore $V_{\lambda}$ as a representation of $(H^0_{der})_{sc}$ is of the form $V_1 \boxtimes V_2$ with $\dim(V_i) > 1$. The representation $V_{\lambda}$ restricts to the subgroup $SL(2) \times SL(2)  = G_{\lambda'}$ as \[ V_{\lambda}|_{G_{\lambda'}} \cong 2 \cdot (st \boxtimes \one) \oplus (\one \boxtimes st).\] This is easily seen using \[ DS(\Pi [2,1,0]) \cong \Pi [2,1] \oplus \Pi [2,-1] \oplus \Pi [0,-1].\] Since $\Pi [2,1]  \cong Ber^{-2} \otimes [0,-1]$ they both give a copy of the standard representation of the same $SL(2)$. Hence the restriction of $V_{\lambda}$ to the first $SL(2)$-factor is of the form \[ V_{\lambda}|_{SL(2)} \cong 2 \cdot st\oplus 2 \cdot \one \] and \[ V_{\lambda}|_{SL(2)} \cong st\oplus 4 \cdot \one \] for the second $SL(2)$-factor. Now consider the restriction to any of the two $SL(2)$-factors $$    V\vert_{SL(2)} =  V_1\vert_{SL(2)} \otimes V_2\vert_{SL(2)}. $$ Since $dim(V_1) =2$ and $dim(V_2) = 3$, their restriction to $SL(2)$ is either $st$ or $2 \cdot \one$ for $V_1$ and $st \oplus \one$ or $3 \cdot \one$ for $V_2$. The Clebsch-Gordan rule for $SL(2)$ shows that $V\vert_{SL(2)} =  V_1\vert_{SL(2)} \otimes V_2\vert_{SL(2)}$ is not possible, hence $H^0_{der}$ must be simple.  
The image of $H_2$ in $H$
contains two copies of $SL(2)$. Since $H^0_{der}$ is not $SL(2) \times SL(2)$, we get $\dim(H^0_{der}) \geq 7$ and the representation $V$ is small. Since $V_{\lambda}$ restricted to the subgroup $SL(2) \times SL(2)$ has 3 summands of dimension 2 each, the restriction to $H^0_{der}$ can decompose into at most 3 summands: either $V_{\lambda}$ stays irreducible, or decomposes in the form $W \oplus W^{\vee}$ or in the form $W_1 \oplus W_2 \oplus W_3$ with $\dim(W_i) = 2$. But the latter implies $W_i \cong st$ for the standard representation of $SL(2)$. This would mean $\dim(H^0_{der}) \leq 6$, a contradiction. The case $W \oplus W^{\vee}$ cannot happen either since the restriction of $W \oplus W^{\vee}$ to $SL(2) \times SL(2)$ would have an even number of summands. Therefore $V_{\lambda}|_{H^0_{der}}$ is irreducible. Since it is selfdual irreducible of dimension 6 and carries a symplectic pairing, we conclude from lemma \ref{small} or lemma \ref{verysmall} 
that $H^0_{der}=Sp(6)$ and $V$ is the standard representation.
But then 
$$   H_X \ \cong \ Sp(6)  \ .$$

\medskip
Similarly consider $X=\Pi(Ber^{1-b}\otimes [2b,b,0])$ for $b>1$.
Then $X^\vee \cong X$. Then either $H \subseteq O(6)$ or
$H\subseteq Sp(6)$ for $H=H_X$. The image of $H_2$ in $H$ contains
$SL(2)^2$. Hence $\dim(H_{der}^0) \geq 6$ and $r\geq 2$.
Furthermore  $H_{der}^0 \not\cong
SL(3)$.
If $r=2$, then we get a contradiction by Mackey's lemma.
Hence $r\geq 3$ and the restriction of the 6-dimensional representation 
$V=\omega(X)$ of $H$ to $H_{der}^0$ remains irreducible.
By the upper bound obtained from duality therefore the semisimple 
rank is $r=3$. Hence $V$ is a small irreducible representation of
$H_{der}^0$ of dimension 6. Hence by lemma \ref{small} 
we get $H_{der}^0 = SO(V)$ resp. $Sp(V)$, since $H_{der}^0 \not\cong
SL(3)$. In the second case then $H=Sp(6)$. In the first case
it remains to determine whether $H = SO(6)$ or $H=O(6)$.   
 
\medskip
Finally the case $X=X_\lambda$ where $L(\lambda)=[a,b,0]$ for $a>b>0$ and 
$a\neq 2b$. In this case $X^\vee \not\cong Ber^\nu \otimes X$
for all $\nu\in \mathbb Z$. The image of $H_2$ 
in $H=H_X$ contains $SL(2)^3$, hence the restriction of $V=\omega(X)$
to $H_{der}^0$ remains again irreducible and defines a small representation
of dimension 6. This now implies $H_{der}^0 = SL(6)$, since $X$ is not weakly selfdual which excludes the cases $Sp(6)$ and $SO(6)$. On the other hand we know that $det(V)$ is 
nontrivial on the image of $H_1$, and hence 
$$    H_X \ \cong \ GL(6)  \ .$$
The structure of $G_3$ follws from theorem \ref{derived-str-thm}. \qed

\begin{example} \label{example-sp(6)} For $\Pi [2,1,0]$ the associated Tannaka group is $H_X = Sp(6)$. Furthermore $X$ corresponds to the standard representation of $Sp(6)$ and decomposes accordingly. Hence \[ X \otimes X \ = \ I_1 \oplus I_2 \oplus I_3 \ \  \mod    \calN \] with the indecomposable representations $I_i \in \calR_3$ corresponding to the irreducible $Sp(6)$ representations $L(2,0,0)$, $L(1,1,0)$ and $L(0,0,0)$. Now consider the tensor product $I_1 \otimes I_1$. For $I_1$ corresponding to $L(2,0,0)$ it decomposes as \begin{align*} I_1  \otimes I_1 \ = \  \bigoplus_{i=1}^{6}\  J_i \mod \calN \end{align*} with the 6 indecomposable representations $J_i$ corresponding to the 6 irreducible $Sp(6)$-representations in the decomposition $$ L(2,0,0)^{\otimes 2} \ = \  L(4,0,0) \oplus L(3,1,0) \oplus L(2,2,0) \oplus L(2,0,0)  \oplus L(1,1,0) \oplus \one \ .$$ In this way we obtain the tensor product decomposition up to superdimension $0$ for any summand of nonvanishing superdimension in such an iterated tensor product. Furthermore these indecomposable summands are parametrized by the irreducible representation of $Sp(6)$. Although $n = 3$ and the weight $[2,1,0]$ are small, we found it  hardly possible to achieve this result by a brute force calculation.
\end{example}





\section{Tannakian induction: Proof of the structure theorem} \label{proof-derived}

\subsection{Restriction to the connected derived group} 

Recall that $H_\lambda$ denotes the Tannaka group of the tensor category
generated by $X_\lambda$ and $V_\lambda =\omega(X_\lambda)$ is a faithful
representation of $H_\lambda$. We have inclusions
$$   G_{\lambda'} \hookrightarrow G_\lambda \hookrightarrow H_\lambda^0 \hookrightarrow H_\lambda $$
where $G_{\lambda'}$ denotes the image of the natural map $(H_{n-1}^0)_{der} \to G_\lambda = (H_\lambda^0)_{der}$. Similarly we denote by $H_{\lambda'}$ the image of $H_n$ in $H_{\lambda}$. The restriction of $V_\lambda$ to $H_{n-1}$ (or $H_{\lambda'}$) decomposes
$$ V_\lambda \ \cong \ \bigoplus_{i=1}^k \ V_{\lambda_i}$$ where $V_{\lambda_i}$ are the irreducible representations in the category $Rep(H_{n-1})$ corresponding to the irreducible constituents $L(\lambda_i), i=1,..,k$ of $DS(L(\lambda))$.
To describe $G_{\lambda'}$ we use the structure theorem for $\mathcal{T}_{n-1}^+$ (induction assumption). Therefore it suffices to
group the highest weight $\lambda_i$ for $i=1,..,k$ into equivalence classes.
Using the structure theorem for the category $\mathcal{T}_{n-1}^+$ and theorem \ref{mainthm}, we then
obtain 
 $$G_{\lambda'}  \cong  \prod_{\lambda_i /\sim} G_{\lambda_i} $$ 
Again using the structure theorem for $G_{n-1}$, each $V_{\lambda_i}$ is either irreducible on $G_{\lambda_i}$ or it decomposes in the form $W_i \oplus W_i^{\vee}$ and $G_{\lambda_i} \cong SL(W)$. The groups $G_{\lambda_i}$ are independent in case (NSD). For (SD) the only dependencies between them come from the equalities $G_{\lambda_{k+1-i}} = G_{\lambda_i}$ for $i=1,...,k$ by section \ref{equivalences}. Using these strong conditions let us consider $V_\lambda$ as a representation
of $H_\lambda^0$. Since an irreducible representation of $H_\lambda^0$ is an
irreducible representation of its derived group $G_\lambda$, the decomposition of $V_\lambda$ into irreducible representation for the restriction to $H_\lambda^0$ resp. $G_\lambda$
coincide. Let 
$$  V_\lambda \ = \ \bigoplus_{\nu=1}^s  \ W_\nu $$
denote this decomposition. We then
restrict each $W_{\nu}$ to $G_{\lambda'}$.

\[ \xymatrix@C=1em@R=1em{ V_{\lambda}  \ar@{|->}[dd] & H_{\lambda} \ar@{-}[dd] \ar@{-}[dr] & & \\ & & G_{\lambda} \ar@{-}[dd] & \bigoplus_{\nu=1}^s W_{\nu} \ar@{|->}[dd] \\ \bigoplus_{i=1}^k V_{\lambda_i}  & H_{\lambda'} \ar@{-}[dr] & & \\ & & G_{\lambda'}  & \bigoplus_{l=1}^t W'_l} \]

By induction each $W'_l$ can be seen as the standard representation or its dual of a simple group of type $A,B,C,D$.

\subsection{Meager representations} If we use by induction the structure theorem for $G_{n-1}$, we see that the representations $W_i$ in $V_{\lambda}|_{G_{\lambda}}$ are meager in the sense below. We analyze in this section the implications of $W_i$ to be meager.

\begin{definition} A finite dimensional representation $V$ of a reductive group $H$ will be called small if $\dim(V) < \dim(H)$ holds.
\end{definition}

\begin{definition} \label{def-meager} A representation $V$ of a semisimple connected group $G$ will be called meager, if every irreducible constituent $W$ of $V$ factorizes over a simple quotient group of $G$ and is isomorphic to the standard representation of this simple
quotient group or isomorphic to the dual of the standard representation for a simple quotient group of Dynkin type $A,B,C,D$.
\end{definition}

If a representation $V$ of $H$ is small resp. meager, 
any subrepresentation of $V$ is small resp. meager.

\medskip\noindent
We now relax the notation and write $G$ instead of $(H^0_{\lambda})_{der}$ and $G'$ instead of $(H_{\lambda'})^0_{sc}$, the simply connected cover of $G_{\lambda'} = (H_{\lambda'}^0)_{der}$. Then there exists a homomorphism $\varphi:G' \to G$ with finite kernel. We show later in theorem \ref{meager-applies} that except for some special cases ($n \leq 3$ or Berezin twists of $S^i$) the situation will be as in the assumptions of the next lemma \ref{meager-lemma}.

\medskip\noindent
So suppose $G'$ is a semisimple connected simply connected group and $V$ is a faithful meager representation
of $G$.  Each irreducible constituent of $V$ then factorizes 
over one of the projections $p_\mu: G' \to G'_\mu$ where $G' \cong \prod G'_{\mu}$. We then say that the corresponding constituent is of type
$\mu$.

\begin{lem} \label{meager-lemma} Suppose $V$ is an irreducible faithful representation of the semisimple connected
group $G$ of dimension $\geq 2$. Suppose $G'$ is a connected semisimple group and $\varphi: G'\to G$
is a homomorphism with finite kernel such that \begin{enumerate}
\item The restriction $\varphi^*(V)$ of $V$ to $G'$ is meager and for fixed $\mu$ every (nontrivial) irreducible constituents of type $\mu$ in the restriction of $V$ to $G'$ has multiplicity at most 2.
\item If an irreducible constituent $W'$ of $V|_{G'}$ occurs with multiplicity 2 for a type $\mu$ in $V\vert_{G'}$ (such a $\mu$ is called an exceptional type), then either 
\begin{enumerate}[label=(\roman*)]
\item $W'$ is the standard representation of $G_{\mu} \cong SL(2)$, or
\item there is a unique type, say $\mu = \mu_2$, such that the restriction of $V$ to $G_{\mu}'$ is equal to either $ 2W\oplus 2W^\vee$ as a representation of the quotient
$SL(W)$ of $G'$ for $\dim(W) \geq 3$ or equal to $W \oplus W^{\vee}$ for $\dim(W) = 2$ or 
\item there is a unique type, say $\mu=\mu_0$, with $G'_\mu\cong Sp(W')$ or $(G'_\mu)_{sc} = Spin(W)$ such that the standard representation $st$ of $G'_\mu$ occurs twice.
\end{enumerate}
\item No irreducible constituent  of the restriction of $V\vert_{G'}$ is a trivial representation of $G'$. 
\item The semisimple group $G'$ has at most one simple factor isomorphic to $SL(2)$. The index, if it occurs, will be denoted $\mu_1$.
\end{enumerate}
 Under these assumptions
$G$ is a simple group or $G'$ is a product of exceptional types in the sense of (2).
\end{lem}

\begin{remark} For the connection to our case see Theorem \ref{meager-applies}. The cases (2)(ii) and (2)(iii) can appear for weakly selfdual weights, see Appendix \ref{sec:mult} for the possible $\lambda$. It is crucial here that we can assume $n \geq 4$.
\end{remark}

{\it Proof}. We may replace $G$ and $G'$ by their simply connected coverings without changing
our assumptions, so that we can assume
that $G$ and $G'=\prod_\mu G'_\mu$ 
decompose into a product of simple groups. Then $V$ is not faithful any more, but has finite kernel. 
The restriction of the meager representation $V$ to $G'$ decomposes into the sum
$\bigoplus_\mu J_\mu$ of representations $J_\mu$ such that
$J_\mu$ is trivial on $\prod_{\lambda\neq \mu} G'_\lambda$
$$  V\vert_{G'} =  \bigoplus_\mu J_\mu \ ,$$
hence $J_\mu$ can be considered as a representation of the factor $G'_\mu$
of $G'$. Furthermore $J_\mu$ is either an irreducible representation of $G'_\mu$,
 or the direct sum $J_\mu \cong W\oplus W^\vee$ (as a  representation
of $G'_{\mu}\cong SL(W)$) by the assumption 1) and 2) or there exists a unique type $\mu$
of Dynkin type $B,C,D$ where $J_\mu = st \oplus st$ for the standard representation
$st$ of this group $G'_\mu$.

\medskip\noindent
If the semisimple connected $G$ is not simple,
$G=G_1 \times G_2$ is a product of semisimple groups and the irreducible
representation $V$ is an external tensor product $$ V = V_1 \boxtimes V_2 $$
of irreducible representations $V_1, V_2$ of $G_1$ resp. $G_2$. Since $V$ has finite kernel and $G$ is connected,
$\dim(V_i)>1$ holds. For each factor $G'_\mu \hookrightarrow G'=\prod_\mu G'_\mu$  consider the composed map
$$   G'_\mu \to G_1 \times G_2 \ .$$
This map has finite kernel.

\medskip\noindent
We claim that there exists at least one index $\mu$ such that 
both compositions $G'_\mu \to G_i$ with the projections $G\to G_i$ ($i=1,2$)
are nontrivial except when $G'$ has only exceptional types. To prove the claim, suppose $G'_\mu\to G_2$ would be the trivial map. 
Then the restriction of $V$ to $G'_\mu\subseteq G'$ is $V\vert G'_\mu =
\dim(V_2) \cdot V_1\vert_{G'_\mu}$. Hence $\dim(V_2) \leq2$, since otherwise we get a contradiction to assumption (1) of the lemma. $V_1\vert_{G'_\mu}$ also contains
at least one nontrivial irreducible constituent by assumption (3), and this constituent 
can occur by assumption (1) at most with multiplicity two in $V\vert_{G'}$.
If then $\dim(V_2)=2$, then
there must exist a nontrivial irreducible constituent $I_\mu \subseteq V_1\vert_{G'_\mu}$ of $G'_\mu$ 
by assumption (3). Hence if $\dim(V_2) = 2$, $V\vert_{G'_\mu} $ contains $I_\mu \oplus
I_\mu$ both of some type $\mu$ and we are in an exceptional type (see (2)).

\medskip\noindent
We assume now that $\{\mu\}$ is not an exceptional type.
We may therefore choose $\mu$ so that both $G'_\mu \to G_i$ are both nontrivial.  Then
$$    V\vert_{G'_\mu} =  V_1\vert_{G'_\mu} \otimes V_2\vert_{G'_\mu}  $$
is the tensor product of two nontrivial representations  
$V_1\vert_{G'_\mu} $ and $V_2\vert_{G'_\mu}$ of $G'_\mu$. 
Since $V\vert_{G'}$ is a meager representation of $G'$,
all irreducible constituents of the restriction of $V\vert_{G'}$ to $G'_\mu$ are trivial representations
of $G'_\mu$ except for at  most two of them (see assumption (1)), which are standard representations up to duality. Since $V_i$ are irreducible representations
of $G$ (recall $V \cong V_1 \boxtimes V_2$) and $V$ has finite kernel, the restriction of $V$ to $G'_\nu$ has finite kernel. Hence both of the representations $V_i\vert_{G'_\mu}$ have finite kernel, hence contain an irreducible nontrivial representation of $G'_\mu$. Otherwise the restriction $V\vert_{G'_\mu}$ would be trivial contradicting that $G'_\mu \to G_i$
have finite kernel for both $i=1,2$ and $V_i$ both have finite kernel on $G_i$.
For every nontrivial irreducible representations $I_1 \subseteq    V_1\vert_{G'_\mu}$
and $I_2 \subseteq    V_2\vert_{G'_\mu}$ of $G'_\mu$ the representation
$$   I_1 \otimes I_2  $$
only contains trivial representations and standard representations $st$ up to duality  by assumption (2).
Since the trivial representation occurs at most once in the tensor product of two irreducible representations, this implies
$I_1\otimes I_2 \subseteq J_\mu \oplus 1 \subseteq st \oplus st^\vee \oplus 1$ (note that $\mu$ is not exceptional).
Hence
$\dim(I_1)\dim(I_2)\leq 1 + 2\cdot \dim(st) < 1 + 2\cdot \dim(st) + \dim(st)^2$.
Hence $\min(\dim(I_\nu)) <  1 + \dim(st)$.  In particular, the corresponding representation with minimal dimension,
say $I_1$, has dimension $\leq \dim(st)$ and hence $I_1$ is a small representation of $G'_\mu$. Since it is small, it belongs to the list of lemma \ref{verysmall}. Therefore $I_1$ is 
the standard representation of $G'_\mu$ or its dual, unless the group $G'_\mu$ is of Dynkin type
$D_4$ and $I_1$ is a spin representation. In the first case, 
considering highest weights it is clear that
$st \otimes I_2 \subseteq st\oplus st^\vee \oplus 1$ is impossible.
In the remaining orthogonal case $G'_\mu$ of Dynkin type $D_4$,
the representation $I_1 \otimes I_2$ must have dimension $\geq 8^2$. But this
contradicts $\dim(I_1)\dim(I_2)\leq 1 + 2\cdot \dim(st) = 1 + 8 + 8 = 17$, and finally 
proves our assertion.\qed

\begin{cor} \label{weakly-mult} In the situation of lemma \ref{meager-lemma}, the restriction of the representation
$V$ to the group $G'$ is multiplicity free unless $G'$ contains an exceptional type (in which case the irreducible constituent has multiplicity 2). If $G'$ has at least one non-exceptional type, then the restriction contains at least one constituent with multiplicity 1.
\end{cor}

{\it Proof}. If the restriction of $V$ to $G'$ contains an irreducible summand $I$ of $G'$
with multiplicity $\geq 2$, then the restriction of $I$ at least under one map $G'_\mu \to G$ contains a nontrivial constituent of $G'_\mu$ with multiplicity $>1$. Hence the restriction of $I$ contains $J_\mu$ by the assumption 1) and 2) of the main lemma
such that $J_\mu \cong I_\mu\oplus I_\mu$ and we are in an exceptional type.\qed 

\begin{definition} Let $G, \ G'$ be semisimple connected groups and $\varphi: G'\to G$ a homomorphism with finite kernel. The restriction of the irreducible representation $V$ of $G$ to $G'$ is called \textit{weakly multiplicity-free} if at least one irreducible constituent has multiplicity 1.
\end{definition} 

\subsection{Mackey-Clifford theory} \label{mackey-clifford}

Let $H$ be a reductive group and $H^0$ its connected component.
We assume that $G$ is the  connected derived group of $H^0$. 
Let $V$ be a finite dimensional irreducible
faithful representation of $H$ and let
$$ V\vert_{H^0} =  W_1 \oplus \cdots \oplus W_s $$
be the decomposition of $V$ into irreducible summands $(W_\nu,\rho_\nu)$ after restriction  to $H^0$.
The restriction of each $W_\nu$ to $G$ remains irreducible (this follows from Schur's lemma and the fact that the image of
$H^0$ in $GL(W_\nu)$ is generated by the image of $G$ in $GL(W_\nu)$ and the image of the 
connected component of the center of $H^0$, whose image is in the center of $GL(W)$).
By Clifford theory \cite{Clifford} $\pi_0(H)=H/H^0$ acts on the isotypic components
$m_{\mu} W_\nu$ permuting them transitively; i.e. $\rho_\nu(g) = \rho_1(h g h^{-1})$
for certain $h \in H$. Here we define the isotypic part of an irreducible $W_{\nu}$ to be the sum of all subrepresentations of $V|_{H^0}$ which are isomorphic to $W_{\nu}$. Since $\pi_0(H)$ acts transitively on these isotypic components, the multiplicity $m = m_{\mu}$ of each isotypic part is the same. Let us write $$ V\vert_{H^0} =  m \cdot (W_1 \oplus \cdots \oplus W_{\tilde{s}}).$$

Representations $(W_\nu,\rho_\nu)$ from different isotypic parts 
are pairwise nonisomorphic representations of $H^0$ (in our application later this also remains true
for the restriction to $G$ by the $G'$-multiplicity arguments). But 
$\rho_1(h_1 g h_1^{-1})\cong \rho_1(h_2 g h_2^{-1})$ as representations of $g\in H^0$ (or $g\in G$)
holds if $h_1^{-1}h_2 \in H^0$ (resp. $h=h_1^{-1}h_2 \in H^0$). Therefore the automorphism $int_h: H^0 \to H^0$ acts trivially and the isotypic components $m W_{\nu}$ are permuted transitively by $Out(H^0) = Aut(H^0)/Inn(H^0)$. If a finite group   acts transitvely on a set $X$, this implies that the cardinality of the set divides the order of the group. Therefore $\tilde{s} \leq | Out(H^0) |$. Furthermore $\dim W_i = \frac{1}{m}\frac{1}{\tilde{s}} \dim V$. 


\medskip\noindent
If $H = H_{\lambda}$ is the Tannaka group of an irreducible maximal atypical
module $L(\lambda) \in \mathcal{T}_n^+$ and $V = V_{\lambda} =\omega(L(\lambda))$ is the associated 
irreducible representation of $H$ and $W_1,...,W_s$ are the irreducible constituents
of the restriction of $V$ to $H^0$,  then the following theorem holds.

\begin{thm} \label{meager-applies} Suppose that $L(\lambda)$ is not a Berezin twist of $S^i$ for some $i$ or its dual, and suppose $n\geq 4$. Then for $G=(H^0_{\lambda})_{der}$ and $G' = G_{\lambda'}$ the irreducible representations $W_1,...,W_s$ of $G$ satisfy the conditions of lemma \ref{meager-lemma} and $G'$ has at least one non-exceptional type $\mu$. In particular $G$ is a connected simple algebraic group and 
$V$ is a weakly multiplicity free representation of $H^0$. 
\end{thm}

\begin{remark} See also Appendix \ref{sec:mult} for an overview.
\end{remark}

\begin{proof} The irreducibility and faithfulness is a tannakian consequence of the definitions. 
We claim that condition 1) and 2) follow from induction on $n$ and the classification of 
similar and selfdual derivatives $\lambda_i$ of $\lambda$ in section \ref{equivalences}. 

\medskip\noindent
If we restrict $V_{\lambda_i}$ to $G_{\lambda'}$, the induction assumption implies that the restriction is either irreducible (the regular case) or $V_{\lambda_i}$ decomposes as $W \oplus W^{\vee}$ for the group $SL(W)$ (the exceptional case). The exceptional case can only happen if $\lambda_i$ is of type (SD).

\medskip\noindent
If $\lambda$ is (NSD), then equivalence classes of its derivatives consist of one element by proposition \ref{duality}. At most one $\lambda_i$ is of (SD) type. If $\lambda$ is (SD), equivalence classes can consist of one or two elements by corollary \ref{size-equiv}. At most two derivatives $\lambda_{\nu}, \lambda_{\mu}$ can be of type (SD) by lemma \ref{selfdual-derivative}. In case $\lambda_i$ is (NSD), the restriction of $V_{\lambda_i}$ to $G_{\lambda'}$ remains irreducible.

\medskip\noindent
Let us then assume that we are in the case where $\lambda_{\nu}$ is not equivalent to any other derivative. Then $\lambda_{\nu}$ belongs to one the three cases (1), (2), (3) treated in the proof of lemma\ref{selfdual-derivative}. In the regular cases $V_{\lambda_i}$ remains irreducible. If $\lambda_i$ is exceptional, it decomposes as $W \oplus W^{\vee}$ for the group $SL(W)$. Then the restriction of $W_{\nu}$ to $G_{\lambda'}$ has multiplicity 1 unless $\dim(W) = 2$ or $\dim(W) = 1$ since $W^{\vee} \ncong W$ for $\dim(W) \geq 3$. These two cases would lead to an irreducible constituent of multiplicity 2 (trivial representation or standard representation of $SL(2)$).

\medskip\noindent
Assume therefore that $\lambda_{\nu} \sim \lambda_{\mu}$ for $\nu \neq \mu$. Then $\lambda$ is of ladder type and $\{\nu,\mu\} = \{1,k\}$ such that $G_{\nu} = G_{\mu}$ is of (SD)-type and either symplectic or orthogonal regular or exceptional (see lemma \ref{selfdual-derivative}). Here we use that $n \geq 4$.

\medskip\noindent
In the exceptional case the standard representation $W$ and its dual $W^{\vee}$ of $SL(W)$ appear with multiplicity 2 in the restriction of $V$ to $G_{\lambda'}$. We don't have $W \cong W^{\vee}$ unless $\dim(W_{\mu}) = 2$ which is impossible for $n \geq 4$ (it would mean $\dim(V_{\lambda_1}) = 4$, but $\dim(V_{\lambda_1}) = (n-1)!$ since $\lambda$ is of ladder type).

\medskip\noindent
In the regular case (SD)-case we have $G_1 = G_k = SO(V_{\lambda_1})$ or $Sp(V_{\lambda_1})$. Both $V_{\lambda_1}$ and $V_{\lambda_k}$ remain irreducible after restriction, hence the multiplicity is again 2.

The uniqueness assertion about the types in (2)(ii) and (2)(iii) follows since case (iii) can only occur for $\lambda_{\nu} \sim \lambda_{\mu}$ of (SD) type, and there are at most two such derivatives. The multiplicity 2 assertion for (2)(ii) holds since at most one selfdual derivative can have dimension 4 or 2 (see proof of condition 3 and 4).

\medskip\noindent
Condition 3) is seen as follows: The trivial representation of $G'$ is attached to a derivative $\lambda_\mu$ of $\lambda$ only if $L(\lambda)$ isomorphic to $S^i \otimes Ber^j$ for some $i\geq 1$ and some $j\in\mathbb Z$ by lemma \ref{trivial-occurence}. 
Concerning condition 4): A factor $G'_\mu$ of $G'$ of rank 1 (i.e. with derived group $SL(2)$) is attached to some derivative $\lambda_\mu$ of $\lambda$
only if $L(\lambda)=S^1$ or $\lambda$ has only two sectors, one sector $S$ of rank 1 and
the other sector $S'$ corresponds to $S^1$ on the level $n-1$. In other words $\partial S S'$ resp. $S' \partial S$
gives $S^1$ and the corresponding group $SL(2)$, but not the other derivative unless $n\leq 3$. 

\begin{center}
\medskip
 
 \scalebox{0.7}{
\begin{tikzpicture}
\foreach \x in {-3,-2,-1,1,7} 
     \draw[very thick] (\x-.1, .1) -- (\x,-0.1) -- (\x +.1, .1);
\foreach \x in {0,2,3,4,5,6,8} 
     \draw[very thick] (\x-.1, -.1) -- (\x,0.1) -- (\x +.1, -.1);
%
\draw (-3,-0.5) node {-3};
\draw (-2,-0.5) node {-2};
\draw (-1,-0.5) node {-1};
\draw (0,-0.5) node {0};
\draw (1,-0.5) node {1};
\draw (2,-0.5) node {2};
\draw (3,-0.5) node {3};
\draw (4,-0.5) node {4};
\draw (5,-0.5) node {5};
\draw (6,-0.5) node {6};
\draw (7,-0.5) node {7};
\draw (8,-0.5) node {8};


\draw[very thick] [-,black,out=90, in=90](-3,+0.2) to (4,+0.2);
\draw[very thick] [-,black,out=90, in=90](-2,+0.2) to (3,+0.2);
\draw[very thick] [-,black,out=90, in=90](-1,+0.2) to (0,+0.2);
\draw[very thick] [-,black,out=90, in=90](1,+0.2) to (2,+0.2);
\draw[very thick] [-,black,out=90, in=90](7,+0.2) to (8,+0.2);


\end{tikzpicture} }
\smallskip

\text{The first derivative is $S^1$ for $n=4$.}
\end{center}

Hence by our assumptions, the group $G'$ has at most one simple factor $SL(2)$. If an irreducible constituent of the restriction of $V$ to $G'$ has multiplicity 2, it comes from a derivative of type (SD). Hence if all types of $G'$ are exceptional, all derivatives of $L(\lambda)$ would have to be selfdual. This can only happen for $n \leq 3$ by the analysis in section \ref{equivalences}. Hence lemma \ref{meager-lemma} and corollary \ref{weakly-mult} imply the last statement.
\end{proof}

\begin{thm} The simple group $G$ is of type $A,B,C,D$ and $W_1\vert_G$ is either the standard representation of $G$
or its dual.
\end{thm}

\begin{proof} We suppose that $L(\lambda)$ is not a Berezin twist of $S^i$ for some $i$ and suppose $n\geq 4$. We distinguish the cases $NSD$ and $SD$. In the NSD-case we claim  that we have
$$  r(G_\lambda) \geq (\dim(V_\lambda)-1)/2  \  $$
and that for $n\geq 4$ and $\dim(V_\lambda) \geq 4$  \[  \dim(G_\lambda) > \dim(V_\lambda)   \] holds (note that $\dim(V_\lambda) \leq 3$
for $n\geq 4$ implies $k=1$ and $\dim(V_\lambda)=\dim(V_{\lambda_1})$). For all $i=1,..,k$ the superdimension formula of \cite{Weissauer-GL}\cite[Section 16]{Heidersdorf-Weissauer-tensor} implies by lemma \ref{sdim-compared} that 
$$\dim(V_\lambda) \leq n \cdot \dim(V_{\lambda_i})/r_i  \ $$
where $r_i =r(V_{\lambda_i})\geq 1$ is the rank of $\lambda_i$. 
Obviously $ \dim(G_{\lambda_i}) \leq \dim(G_{\lambda})$.

\medskip\noindent
Since we excluded the $S^i$-case, no $V_{\lambda_i}$ has dimension $1$ by lemma \ref{trivial-occurence}. At most one of the representations $V_{\lambda_i}$ is selfdual by lemma \ref{selfdual-derivative}. We make a case distinction on whether there exists one $V_{\lambda_i}$ that splits in the form $W'_i \oplus (W'_i)^{\vee}$ upon restriction to $G_{\lambda'}$ or not. In the latter case 
we know $r(G_{\lambda_i}) \geq \frac{1}{2} \dim(V_{\lambda_i})
$ by theorem \ref{Tannakagroup} and the induction assumption. 
Now by proposition \ref{duality} and the assumption (NSD) all $\lambda_i$ in the derivative of $\lambda$ are inequivalent for $i\neq j$. 
Hence we get
$$ r(G_\lambda) \geq \sum_{i} r(G_{\lambda_i}) \geq \sum_i \frac{1}{2}\dim(V_{\lambda_i}) 
\geq \frac{1}{2}(\dim(V_\lambda)) \ .$$
Since $\dim(G_{\lambda_i}) \geq 3 r(G_{\lambda_i})$, this implies
$ \dim(G_\lambda) \geq \frac{3}{2} (\dim(V_\lambda) - 1) 
 $ and hence
$ \dim(G_\lambda) > \dim(V_\lambda) $ (note that we have at least one $SL$ factor $G_{\lambda_i}$ for which $r(G_{\lambda_i}) > \frac{1}{2} \dim(V_{\lambda_i})$). If $V_{\lambda}$ splits $V_{\lambda} = W_1 \oplus \ldots \oplus W_s$ we may replace $V_{\lambda}$ by any $W_\nu$ for an even better estimate. Therefore lemma \ref{verysmall} implies that $V_{\lambda}$ (or $W_{\nu}$) is the standard representation or its dual of a simple group of type $A,B,C,D$. If $V_{\lambda}$ stays irreducible, then we obtain $G_{\lambda} \cong SL(V_{\lambda})$ since $V_{\lambda}$ is not selfdual.


\medskip\noindent
If $V_{\lambda_i}$ splits, $G_{\lambda_i} \cong SL(W_i)$ for $V_{\lambda} \cong W_i \oplus W_i^{\vee}$ by induction assumption. If the dimension of $V_{\lambda_i}$ is $2d_i$, we then have $r(G_{\lambda_i}) = d_i - 1$ and therefore have to replace the estimate $r(G_{\lambda_i}) \geq \frac{1}{2} \dim(V_{\lambda_i})$ by the estimate $r(G_{\lambda_i}) \geq \frac{1}{2} (\dim(V_{\lambda_i} -2))$. Since $V_{\lambda_i}$ can only decompose if it is of type SD, $L(\lambda)$ has more than one sector. All the other $k-1 \geq 1$ derivatives $L(\lambda_i)$ are of type NSD and define inequivalent $SL(V_{\lambda_i})$. For each of these we obtain $r(G_{\lambda_j}) = \dim V_{\lambda_j} - 1$. Summing up we obtain \[ r(G_\lambda) \geq \sum_{i} r(G_{\lambda_i}) 
\geq  \frac{1}{2}(\dim(V_{\lambda_i}) - 2) + \sum_{j \neq i} \dim(V_{\lambda_j}) - 1 \ .\] This implies again the necessary estimates to apply lemma \ref{verysmall}.

\medskip\noindent
We now consider the SD-case. If $V_{\lambda}$ decomposes \[ V_{\lambda}|_{G_{\lambda}} \cong W_1 \oplus \ldots \oplus W_s \] then we can assume by reindexing that $\dim(W_i) = \frac{1}{s}\dim(V_{\lambda})$. Note that $\dim(W_1) > 1$ follows from the induction assumption. 

\medskip\noindent
In the SD case we proceed as follows: We first show that $V_{\lambda}$ or $W_1$ is small. Since we cannot prove the strong rank estimates for $r(G_{\lambda})$ as in the NSD case, we work through the list of exceptional cases in lemma \ref{small}.

\medskip\noindent

The list of superdimensions in the n=4 and n=5 case in sections \ref{sec:physics} \ref{small-superdimensions} along with the induction assumption shows in these cases that $V_{\lambda}$ is small. Therefore we can assume $n \geq 5$. We use the known formulas $\dim(SL(n)) = n^2 - 1$, $\dim SO(n) = \frac{n(n-1)}{2}$ and $\dim(Sp(2n)) = n (2n+1)$.

\medskip\noindent
We recall from the analysis in lemma \ref{selfdual-derivative} that $L(\lambda)$ can only have more than one selfdual derivative if it is completely unnested, i.e. it has $n$ sectors of cardinality 2. In this case it has 2 selfdual derivatives coming from the left and rightmost sectors and, if $n$ is odd, another derivative coming from the middle sectors. If $\lambda$ is not of this form, then the unique weakly selfdual derivative comes from the middle sector (of arbitrary rank).

\medskip\noindent

We want to show $\dim(G_{\lambda}) > \dim(V_{\lambda})$. By induction $G_{\lambda_i}$ is either $SO(V_{\lambda_i})$, $Sp(V_{\lambda_i})$, $SL(V_{\lambda_i})$ or $SL(W_i)$ for $V_{\lambda_i} = W'_i \oplus (W'_i)^{\vee}$. We estimate the dimension of $G_{\lambda_i}$ via $\sum \dim(G_{\lambda_i})$. We claim that we can assume that we have more than one sector because otherwise $\dim(V_{\lambda}) = \dim(V_{\lambda_1})$ implies that $V_{\lambda}$ is small using the induction assumption. If $V_{\lambda_1}$ is an irreducible representation of $G_{\lambda'}$ the claim is clear by induction assumption. If it splits $V_{\lambda_1} = W'_1 \oplus (W'_1)^{\vee}$, then $\dim( V_{\lambda_1}) < \dim(SL(W'_1))$ provided $\sdim(L(\lambda_1)) \geq 3$. Now $\sdim(L(\lambda_1)) = 2$ can only happen for $L(\lambda_i) \cong Ber^{r} \otimes S^1$ for some $r$ (and then $V_{\lambda_1}$ is an irreducible representation of $G_{\lambda'}$). We therefore assume $k >1$. The worst estimate for the dimension is obtained if all $V_{\lambda_i}$ split as $W'_i \oplus (W'_i)^{\vee}$ and therefore $G_{\lambda_i} \cong SL(W_i)$. This case can only happen if either $n=2$ or $n=3$. For $n \geq 4$ the lowest estimate for the dimension of $G_{\lambda}$ occurs if $\lambda$ is completely unnested with 2 selfdual derivatives coming from the left and right sector and we have $n/2$ equivalence classes of derivatives (or $\left \lfloor{n/2}\right \rfloor + 1$ for odd $n$). The left and right sector then contribute a single $SL(W'_1) = SL(W'_k)$ and if $n$ is even for all other derivatives  $G_{\lambda_i} \cong SL(V_{\lambda_i})$ with $V_{\lambda_i} \sim V_{\lambda_{k-i}}$ and therefore same connected derived Tannaka group. If $n =2l+1$ is odd the middle sector can contribute another derivative of type SD with Tannaka group $SL(W'_{l+1})$. The dimension estimate works as in the case above and we therefore ignore this case.

\medskip\noindent
We show now that $\dim( G_{\lambda}) > \dim(V_{\lambda})$ provided we have two SD derivatives coming from the left- and rightmost sector. Denote by $d_i$ the dimension of $V_{\lambda_i}$. For $i=1,k$ it is even $d_1 = 2d'_1 = 2d'_k$ by lemma \ref{weird1}. We then obtain for the dimension of $G_{\lambda'}$ \begin{align*} \dim( G_{\lambda'}) = \frac{1}{2} ((d'_1)^2 -1 + (d'_k)^2 - 1) + \frac{1}{2} \sum_{j \neq 1,k} d_j^2 - 1.\end{align*} It is enough to show $2 \dim V_{\lambda_i} < \dim G_{\lambda_i}$ for each $i$. The smallest possible superdimensions for a selfdual irreducible representation are $2,4,12,\ldots$. The $\dim = 2$ case can only happen for $L(\lambda_i) \cong Ber^{r} \otimes S^1$ which is not possible by assumption. Hence $d'_1 \geq 3$. This case occurs for $[2,1,0]$ for $n=3$, $[2,2,0,0]$ for $n=4$ and all their counterparts for larger $n$ by appending zeros to the weight (e.g. $[2,1,0,0]$). These are not derivatives of a selfdual representation $L(\lambda)$ unless $L(\lambda)$ has one sector (which we excluded). Therefore we can assume $d'_1 \geq 6$. Then \[2 \ \dim(V_{\lambda_1}) = 4 d'_1 < (d'_1)^2 -1 =  \dim(G_{\lambda_1}).\] 

\medskip\noindent 
For the NSD derivatives we can exclude the case $d_i = 2$ since this only happens for $L(\lambda_i) \cong Ber^{\ldots} \otimes S^1$. For $d_i \geq 3$ we obtain $2 d_i < d_i^2 - 1$, hence again $2 \dim(V_{\lambda_i}) < \dim(G_{\lambda_i})$. Clearly this estimates also hold if we have more than $n/2$ equivalence classes of weights or if we have $SO(V_{\lambda_i})$ or $Sp(V_{\lambda_i})$ in case of $SL(W_i)$.

\medskip\noindent
Hence $\dim(V_{\lambda}) < \dim(G_{\lambda})$. If $V_{\lambda}$ is an irreducible representation of $G_{\lambda}$, it is a small representation of $G_{\lambda}$ and lemma \ref{small} applies. If it decomposes $V_{\lambda} \cong W_1 \oplus \ldots \oplus W_s$, then each $W_{\nu}$ is an irreducible small representation of $G_{\lambda}$.

\medskip\noindent
Assume first that  $V_{\lambda} \cong W_1 \oplus \ldots \oplus W_s$ with $s \geq 3$ and $\dim(W_1) \leq \frac{1}{s}\dim(V_{\lambda})$. Again the smallest rank estimate for the subgroup $G_{\lambda'}$ occurs for $n \geq 4$ if $\lambda$ is completely unnested with 2 selfdual derivatives coming from the left and right sector and we have $n/2$ equivalence classes of derivatives (we assume here $n$ even. In the odd case we can have another derivative from the middle sector. The estimate below still holds). Then \begin{align*} r(G_{\lambda}) \geq r(G_{\lambda'}) & \geq \frac{1}{2}( d_1/2 -1 + d_k/2 - 1 + \sum_{j \neq 1,k} d_j -1) \\ & = \frac{1}{2}(\dim(V_{\lambda}) - k - d_1/2 - d_k/2).\end{align*} In the completely unnested case this equals \[ \frac{1}{2}(n! - n - (n-1)!).\] We need $r(G_{\lambda}) \geq \frac{1}{2}(\dim(V_{\lambda}) - 1)$ to apply lemma \ref{verysmall}. We replace now $V_{\lambda}$ by  $W_1$ with $\dim(W_1) \leq  1/s \ \dim(V_{\lambda})$. For $n\geq 4$ and $s \geq 2$ we obtain $(n!/s) -1 \leq n! - n - (n-1)!$, hence lemma \ref{verysmall} can be applied to the irreducible representation $W_1$.

If $\lambda$ is not completely unnested, it can have at most one SD derivative coming from the middle sector for $k = 2l+1$ odd. Then we obtain \begin{align*} r(G_{\lambda}) \geq r(G_{\lambda'}) & \geq \frac{1}{2}( d_{l+1}/2 -1  + \sum_{j \neq l+1} d_j -1) \\ & = \frac{1}{2}(\dim(V_{\lambda}) - k - d_{l+1}/2).\end{align*} As above we replace $V_{\lambda}$ with $W_1$ with $\dim(W_1) \leq \frac{1}{s}V_{\lambda}$ and show $\dim(V_{\lambda}/s - 1) \leq \dim(V_{\lambda}) - k - d_{l+1}/2$. For $s=2$ this is equivalent to $\dim(V_{\lambda}) \geq d_{l+1} + 2(k-1)$. This follows easily from $\dim(V_{\lambda}) = \dim(V_{\lambda_{l+1}}) \frac{n}{r_{l+1}}$ (lemma \ref{sdim-compared}). For $s>2$ the estimates are even stronger. The cases where the SD derivative occurs and contributes $SO(V_{\lambda_{l+1}})$ or $Sp(V_{\lambda_{l+1}})$, or the case in which no SD derivative occurs, can be treated the same way.  

\medskip\noindent
We can therefore assume that either a) $V_{\lambda}$ is an irreducible representation of $G_{\lambda}$ or it splits in the form $V_{\lambda} = W \oplus W^{\vee}$. The analysis of small superdimensions in section \ref{small-superdimensions} shows that the possible superdimensions of weakly selfdual irreducible representations less than 129 are \begin{align*} 1, \ 2, \ 6, \ 12, \ 20, \ 24, \ 30, \ 42, \ 56, \ 70, \ 72, \ 80, \ 90, \ 110, \ 112.\end{align*}
Except for the numbers 20 and 56 none of the exceptional dimensions in lemma \ref{small} is equal to either the superdimension or half the superdimension of an irreducible weakly selfdual representation in $\mathcal{T}_n^+$. It is easy to exclude these two cases (see section \ref{small-superdimensions}) since in this case $V_{\lambda}$ or $W$ would be either a symmetric or alternating square of a standard representation (which would give a contradiction to the induction assumption) or the irreducible representation of minimal dimension of $E_7$ which is impossible by rank estimates.
\end{proof}

\begin{thm} \label{s=2} Either the restriction
of $V_{\lambda}$ to $H^0_{\lambda}$ and $G_{\lambda}$ is irreducible, or $G\cong SL(W)$ and $V\vert_G \cong W\oplus W^\vee$
for a vectorspace $W$ of dimension $\geq 3$. If $V\vert_G \cong W\oplus W^\vee$, then \[ V_{\lambda} \cong Ind_{H_1}^{H}( W) \] for a subgroup $H_1$ of index 2 between $H^0$ and $H$. In particular $V_{\lambda}$ is an irreducible representation of $G_{\lambda}$ if $L(\lambda)$ is not weakly selfdual.
\end{thm}

\begin{proof} As in the statement of theorem \ref{meager-applies} we can assume that $n \geq 4$ and that $L(\lambda)$ is not a Berezin twist of $S^i$ (or its dual) since these cases were already treated in section \ref{sec:ind-start}. 

We claim that the representation $V\vert_{H^0} = W_1\oplus
\ldots \oplus W_s$ is multiplicity free. Since the restriction of $V$ to $G'$ is weakly multiplicity free, at least one irreducible constituent occurs only with multiplicity 1 for some (non-exceptional) $\mu$. By Clifford theory the multiplicity of each isotypic part in the restriction of $V$ to $H^0$ is the same (since $\pi_0$ acts transitively). If the multiplicity of each isotypic part would be bigger than 1, the restriction of $V$ to $G'$ could not be weakly multiplicity free. Therefore the multiplicity of each isotypic part is 1. Any $W_{\nu}$ restricted to $G_{\lambda}$ is irreducible (restriction to the derived group). Since $G_{\lambda}$ is a normal subgroup of $H$, $H$ still operates transitively on the set $\{ W_{\nu}|_{G_{\lambda}} \}$. Fix any $W_{\nu}|_{G_{\lambda}}$. Its $H$-orbit has $s'$ elements where $s'$ divides $s$ and $s/s'$ is the multiplicity of each $W_{\nu}|_{G_{\lambda}}$ in $V_{\lambda}|_{G_{\lambda}}$. Hence the argument from Clifford theory explained preceding theorem \ref{meager-applies}
shows
$$ s' \leq | Out(G)| \ .$$
But a nontrivial outer automorphism of $G$ that does not fix the isomorphism class of the standard representation $W_1$ of $G$ exists only for the groups $G$ of the Dynkin type $A_r$ for $r\geq 2$ (note that we can ignore the $D_4$-case since no weakly selfdual irreducible representation has superdimension 8). For the special linear groups
$G=SL(\mathbb C^{r+1})$ the nontrivial representative in $Out(G)$ it is given by $g\mapsto g^{-t}$. The twist of the standard representation
by this automorphism gives the isomorphism class of the dual standard representation $W_1^\vee$.
This implies $s'=1$ or  $s'=2$. If $s' = 2$, then $V_{\lambda}|_{G_{\lambda}} \cong W \oplus W^{\vee}$ where $W$ is the standard representation of $SL$ and $G_{\lambda} \cong SL(W)$. Since $V_{\lambda}|_{G_{\lambda'}}$ is weakly multiplicity free and $G_{\lambda'} \subset G_{\lambda}$, $V_{\lambda}|_{G_{\lambda}}$ is weakly multiplicity free as well. Accordingly $s/s' = 1$ and we also obtain $s=1$ or $2$. If $s=2$, Clifford theory further implies that \[ V_{\lambda} \cong Ind_{H_1}^{H}(W) \] for a subgroup $H_1$ of index 2 between $H^0$ and $H$.
\end{proof}

\begin{remark} Since $W \oplus W^{\vee}$ is selfdual, this implies in particular that $V_{\lambda}$ can only decompose if $L(\lambda)$ is weakly selfdual. If $V_{\lambda}$ decomposes, its restriction to $G_{\lambda'}$ is of the form $\bigoplus_i W_i \oplus W_i^{\vee}$. This leads to some restrictions on SD weights $\lambda$ such that $V_{\lambda}$ decomposes in the form $W \oplus W^{\vee}$. Consider for an instance the weakly selfdual weight $[n-1,n-2,\ldots,1,0]$ for odd $n = 2l+1$. Then $V_{\lambda}$ can only decompose if the irreducible representation $V_{\lambda_l+1}$ associated to the middle derivative $L(\lambda_{l+1})$ decomposes upon restriction to $G_{\lambda'}$ in the form $W'_{l+1} \oplus (W'_{l+1})^{\vee}$.
\end{remark}




\section{The structure theorem on the full Tannaka groups} \label{sec:conjecture}

We discuss the full Tannaka groups $H_{\lambda}$ in this section. To this end we analyze the invertible elements in $Rep(H_n)$, i.e. $Pic(H_n)$, or in down-to-earth terms the character group of $H_n$.

\subsection{Invertible elements} For a rigid symmetric $k$-linear tensor category $\mathcal C$ an object $I$ of $\mathcal C$ is called invertible if $I \otimes I^\vee \cong {\bf 1}$ holds. The tensor product of two invertible objects of ${\mathcal C}$ is an invertible object of ${\mathcal C}$. Let
$Pic({\mathcal C})$ denote the set of isomorphism classes
of invertible objects of ${\mathcal C}$. The tensor product canonically turns
$(Pic({\mathcal C}, \otimes))$ into an abelian group with unit object ${\bf 1}$, the Picard group
of ${\mathcal C}$.  

\medskip\noindent
Suppose that the categorial dimension $\dim$ is an integer $\geq 0$ for all indecomposable objects of $\mathcal C$. An indecomposable object $I$ of 
$\mathcal C$ is an invertible object in $\overline{\mathcal C} = \mathcal{C}/\mathcal{N}$ if and only if
$\sdim(I) =1$ holds. In fact $\dim(I)=1$ implies $\dim(I^\vee)=1$ and hence
$\dim(I \otimes I^\vee)=1$. Hence $I \otimes I^\vee \cong \one \oplus N$ for some
negligible object $N$. Note that the evaluation morphisms $eval: I \otimes I^\vee \to \one$
splits since $\dim(I)\neq 0$.

\subsection{$Pic(\overline{\mathcal{T}}_{n})$ and its generators}

Since ${\overline{\mathcal T}}_n \sim Rep_k(H_n)$, to determine the Picard group $Pic({\overline{\mathcal T}}_n)$ is
tantamount to determine the character group of $H_n$. It coincides with the character group
of the factor commutator group $H^{ab}_n$
of $H_n$. Hence $H_n^{ab}=H_n/G_n$
is determined by  $Pic({\overline{\mathcal T}}_n)$.

The Picard group $Pic({\overline{\mathcal T}}_n)$ can be determined from the individual $Pic(H_{\lambda})$, but it is preferable to choose different generators.

For a Tannakian category $\mathcal T$ over an algebraically  closed field $k$ of characteristic zero, generated
by finitely many irreducible objects, let $H(\mathcal T)$ denote
the Tannaka group of $\mathcal T$. In our situation the Tannaka group $H_n$ is a projective limit of
certain algebraic Tannaka groups $H(\mathcal T)$ as above, so that
$ Pic(H_n)$ correspondingly is an inductive limit 
$$  Pic(H_n) =  \lim_{\to} Pic(H(\mathcal T)) $$
of the Picard groups $Pic(H(\mathcal T))$. Let
us consider generators of this inductive limit.
Such generators are the 
Picard groups $Pic(H(\overline{\mathcal{T}}_\lambda))$, attached
to the Tannakian categories $\overline{\mathcal{T}}_\lambda$ that are generated
by $X_\lambda$ and the normalized Berezin $B$. 
By definition, the Tannakian category $\overline{\mathcal{T}}_\lambda$ only depends on the equivalence class $\lambda/\!\sim$ of $\lambda$ since $\overline{\mathcal{T}}_{\lambda'} =
\overline{\mathcal{T}}_{\lambda}$ for $X_{\lambda'} = X_{\lambda} \otimes B$.

\medskip\noindent
In the limit, the passage from $H_\lambda$ to $H(\overline{\mathcal{T}}_\lambda)$ allows a slicker description of the structure of 
the projective limit  $H_n$:
Obviously there exists
a canonical splitting
$$ Pic(H(\overline{\mathcal{T}}_\lambda))  \cong  Pic^0(H(\overline{\mathcal{T}}_\lambda)) \times \mathbb Z \ ,$$
compatible with the splitting $ Pic(H_n)  \cong  Pic^0(H_n) \times \mathbb Z $ given in \ref{picardgroup}, in such a way that $Pic^0(H_n)$ is generated by the
images of the groups  $Pic^0(H(\overline{\mathcal{T}}_\lambda))$ for $\lambda$
ranging over the equivalence classes $\lambda/\!\sim$.

\subsection{$Pic(\overline{\mathcal{T}}_{n})$ and the determinant} 


The elements  of $Pic({\overline{\mathcal T}}_n)$ are
represented by indecomposable objects $I \in \mathcal{T}_n^+$ with the property
\[ I \otimes I^{\vee} \cong \one \oplus \text{ negligible}\ .\]

Since $\sdim(X_\lambda) \geq 0$, we can define $det(X_{\lambda}) = \Lambda^{\sdim(X_\lambda)}(X_\lambda)
$. Notice
$$    \det(X_\lambda)\ = \ I_\lambda \ \oplus \ \text{negligible} $$
is the sum of a unique indecomposable module $I_\lambda$ in $\mathcal{T}_n^+$ and a direct sum
of negligible indecomposable modules in $\mathcal{T}_n^+$. Furthermore 
$  I_\lambda^* \cong I_\lambda $
and $ \sdim(I_\lambda) = 1$ holds, and if $X_\lambda$ is selfdual, 
then $I_\lambda$ is selfdual. In particular, $det(X_\lambda)$ in ${ \mathcal{T}_n^+}$
has superdimension one, hence its image defines an invertible 
object of the representation category ${\overline{\mathcal T}}_n \sim Rep_k(H_n)$. By abuse of notation 
we also write $det(X_\lambda)  \ \in \ Rep_k(H_n)$.

\subsection{The NSD-case}
According to the structure theorem theorem \ref{s=2}
 in the (NSD)-cases
the group $H_\lambda$ satisfies
$$  SL(V_\lambda) \subseteq H_\lambda \subseteq GL(V_\lambda) \ .$$
Hence to determine $H_\lambda$ it suffices to show that
the restriction of the determinant $\det: GL(V_\lambda)\to k^*$
to the $H_\lambda$ is either trivial or surjective. In theorem \ref{thm1.4} we later show
that $\det: H_\lambda \to {\mathbb G}_m$ a a representation is represented
by a power $B^{\ell(\lambda)}$ of the twisted Berezin object $B$ in ${\mathcal T_n}^+$. 
Hence $H_\lambda\cong SL(V_\lambda)$ if and only if the integer $\ell(\lambda)$ is zero, and
$H_\lambda\cong GL(V_\lambda)$ holds otherwise. In particular $H_\lambda^{ab}$
is the Tannaka group of the Tannaka category generated by $B^{\ell(\lambda)}$
in the (NSD)-cases.

\subsection{The SD-cases} We distinguish between the exceptional cases
where $V_{\lambda}|_{G_{\lambda}} \cong W \oplus W^{\vee}$ and the
remaining cases where this does not happen. We call these the regular (SD)-cases since we conjecture that the exceptional (SD)-case do not occur.

\medskip
In all (SD)-cases the group $H_\lambda$ 
is a genuine subgroup of the similitude group $G(V_\lambda,\langle .,.\rangle)$. 
As such it inherits the similitude character $\mu$ resp. the determinant character
$\det$ of $G(V_\lambda,\langle .,.\rangle)$ that are represented by $B^r$ resp. $B^{\ell(\lambda)}$. Recall from
section \ref{derived-group} that $sign= \det/\mu^n$ is a character of order two
of the similitude group $G(V_\lambda,\langle .,.\rangle)$. Since the restriction of $sign$
to $H_\lambda$ is represented by $B^{\ell(\lambda)- r}$, and the latter is either trivial
or non-torsion. This implies that $sign$ is trivial on $H_\lambda$, hence
$H_\lambda \subset GSO(V_\lambda)$ resp. $H_\lambda \subset GSp(V_\lambda)$
holds according to the parity $\varepsilon(X_\lambda)$ of the pairing $\langle .,.\rangle$.

\subsection{The regular SD-cases}
For the \textit{regular (SD)-cases} the group $H_\lambda$ satisfies $H_\lambda \subseteq GO(V_\lambda)$
resp. $H_\lambda \subseteq GSp(V_\lambda)$ according to the parity of the pairing
$\langle . , . \rangle$. This follows from \ref{derived-group} as well as the fact that
the kernel of the characters $\det$ and the similtude character $\mu$
are the subgroup $SO(V_\lambda)$ resp. $Sp(V_\lambda)$.
Both $\mu=B^r$ and $\det = B^{\ell(\lambda)}$ are represented by tensor powers
powers of the twisted Berezin by thm. \ref{thm1.4}. Hence in the regular (SD)-cases theorem  \ref{s=2} implies the following structure result: The group $H_\lambda$ is isomorphic to $SO(V_\lambda)$ resp.
$Sp(V_\lambda)$ if and only if $\ell(\lambda)=0$, and it is isomorphic to $GSO(V_\lambda)$ resp.
$GSp(V_\lambda)$ otherwise.

\subsection{The exceptional (SD)-cases} \label{exc-sd} 
For this exceptional situation
we recall the following facts:

\medskip
 In the exceptional case we have shown that $W$ is not isomorphic to $W^\vee$ as a representation of  $G_\lambda$ (in other words we have 
$m>2$). By Schur's lemma this implies that the restriction of the pairing  $<\cdot ,\cdot>$ 
must be trivial on the subspaces $W \subset V_\lambda$ and $W^\vee \subset V_\lambda$,
and that these two subspaces are orthogonal to each other for the nondegenerate pairing $\langle .,.\rangle$ on
$V_\lambda$. Hence $W$ and $W^\vee$ define an orthogonal pair of Lagrangian subspaces of $V_\lambda$.  
In the exceptional cases
the representation of $G_\lambda$ on the  vectorspace $$V_{\lambda}|_{G_{\lambda}} \cong W \oplus W^{\vee}$$ decomposes into two faithful
nonisomorphic irreducible representations on the subspaces $W$ and the subspace $W^\vee$  such that the image of $G_\lambda$ in $GL(W)$ contains the perfect group $SL(W)$.
Furthermore $H=H_\lambda$ (as well as $H=H({\mathcal T}_\lambda)$) preserves the unordered pair $\{W,W^\vee \}$ of disjoint Lagrangian subspaces $W$ and $W^\vee$ and
the pairing on $V_\lambda$ (up to a similitude factor). Since they fix the
Lagrangian decomposition  $V_\lambda = W \oplus W^\vee$ up to a permutation of the two subspaces, this
induces a permutation character $\chi_\lambda: H_\lambda \to \mu_2$
and similarly for $H({\mathcal T}_\lambda)$.
In the exceptional
cases, by definition there exist elements $w$ in $H$  that $\chi_\lambda(w)\neq 1$; let us fix such $w$.
For exceptional $\lambda$ we thus obtain an exact sequence
$$ 0 \to   \tilde H \to H \to
\mathbb \mu_2 \to 0 \ .$$
The kernel $ \tilde H$ is a subgroup of the group $G(W)$ of similitudes
of $V_\lambda$ that individually preserve the subspaces $W$ and $W^\vee$, and $\tilde H$ contains $SL(W)$. In terms of a basis of $W$ and a dual basis of $W^\vee$
the elements of $G(W)$ are of the blockdiagonal form $g=diag(A,\lambda\cdot A^{-t})$ for $A\in GL(W)$ and $\lambda\in k^*$.
In fact $g\mapsto (A,\lambda)$ induces a group 
isomorphism $G(W) \cong GL(W) \times \mathbb G_m$, such that
the projection onto the second factor $\mathbb G_m$ induces the similitude character $\mu: G(W)\to \mathbb G_m$.
The determinant $\det_W: GL(W)\to \mathbb G_m$ on the first factor induces an isomorphism 
$$  ({\det}_W,\mu):\   G(W)/SL(W)  \ \cong\  \mathbb G_m \times \mathbb G_m \ .$$ 
The action of $w$ by conjugation on $G(W)$  preserves
the subgroups $SL(W)$ and  $\tilde H$  of $G(W)$ for either $H=H_\lambda$ or $H=\tilde H(\overline{\mathcal{T}}_\lambda)$ such that the induced action on $\mathbb G_m \times \mathbb G_m \cong G(W)/SL(W)$ is inversion on the first factor and the identity on the second factor. This  follows from $\det_W(w g w^{-1}) = {\det}_W(g)^{-1}$ and $\mu(w g w^{-1}) = \mu(g)$. The algebraic group
$H/SL(W)$,  as a subgroup of $   G(W)/SL(W)$,
therefore is a closed subgroup of $\mathbb G_m \times \mathbb G_m$.
Since $\tilde H/SL(W)$ is contained in the first factor $\mathbb G_m$, the quotient $Q=\tilde H/SL(W)$ is either $\mathbb G_m$ or a finite
cyclic group. In the first case $Q$ is the full commutator group
of $H/SL(W)$. But in both cases, the factor commutator group $H^{ab}$ as a diagonalizable group is
isomorphic to a direct product of a finite torsion group and a torus
of rank $\leq 1$. Its rank is nonzero if and only if the similitude
$\mu: H\to \mathbb G_m$ is nontrivial and thus surjective, i.e. for the cases $H=H(\overline{\mathcal{T}}_\lambda)$ resp. $r(\lambda)\neq 0$ for $H=H_\lambda$. Recall that the subgroup $Pic^0(H)$ of $Pic(H)=X^*(H^{ab})$ is
the annihilator of a cocharacter $\mathbb G_m \to H$ that is induced
by the embedding $i: \mathbb G_m \cong H_1 \to H_n$ composed with the surjection
$H_n \twoheadrightarrow H$. Since $ \mu \circ i $ is surjective unless $\mu: H\to \mathbb G_m$ is trivial, it easily follows that $Pic^0(H)$ is the torsion subgroup  $Pic(H)_{tor}$ of $Pic(H)$, and this group $Pic^0(H)$ only depends on the equivalence $\lambda/\!\sim$ of $\lambda$
and is the same for $H=H_\lambda$ and $H(\overline{\mathcal{T}}_\lambda)$.
We make this more explicit in the next section.


\medskip
To study the equivalence class $\lambda/\!\sim$,
we consider normalised representatives $\lambda$.
Twisting $\lambda$ 
by the $a\!-\! th$ power of the Berezin $B$, in the pairing $\mu: X_\lambda \times X_\lambda \to B^r$ the character $\mu=B^r$ changes to $\mu\otimes B^{2a}=B^{r+2a}$. Hence one can choose a \textit{normalised} representative
$\lambda$ in its class so that $r=r(\lambda)$ is one or zero.
In the first case $r=1$ the Tannakian category $\overline{\mathcal{T}}_\lambda$ is generated by $X_\lambda$, hence $H_\lambda= H (\overline{\mathcal{T}}_\lambda)$ contains
the group $Z$ of scalar homotheties of $V_\lambda$.
In the second case $r=0$, the similitude character $\mu$ of $H_\lambda$ is trivial.
Then it is easy to see that $H (\overline{\mathcal{T}}_\lambda) = Z \cdot H_\lambda$,
and $Z\cap H_\lambda = \{ \pm id_{V_\lambda}\}$. Hence $H (\overline{\mathcal{T}}_\lambda) = Z \cdot H_\lambda$ holds for any exceptional $\lambda$. 
Since therefore  $H (\overline{\mathcal{T}}_\lambda)$ contains the group $Z$ of all
diagonal matrices, it is easy to see that we can modify $w$  by an 
element of $\tilde H(\overline{\mathcal{T}}_\lambda)$ (without changing its conjugation action on $S$, but possibly the sign of $\chi_\lambda(w)$) such that $w^2=id_{V_\lambda}$ holds. As a consequence, in the exceptional cases for $Q=Q_\lambda \subseteq \mathbb G_m$ this
implies
$$     H^{ab}(\overline{\mathcal{T}}_\lambda) \cong (Q/Q^2) \times \mu_2 \times \mathbb G_m \ .$$
Notice $Q/Q^2$ is trivial if $Q$ is finite of odd order or $Q\cong \mathbb G_m$, and it is isomorphic to $\mathbb Z/2\mathbb Z$ otherwise.
 
\medskip 
For $H(\langle B \rangle)= \mathbb G_m$ 
the Tannakian subcategory $\langle B^r \rangle$ generated by $B^r$
gives rise to an $r$-fold covering $H(\langle B \rangle) \to
H(\langle B^r \rangle)$, i.e. $H(\langle B^r \rangle)$
is the quotient of $H(\langle B \rangle)\cong \mathbb G_m$
by the unique cyclic subgroup of order $r$. In a similar vein
we can recover $ H_\lambda$ from  $H(\overline{\mathcal{T}}_\lambda)$. This
is seen as follows:
The inclusion $$H(\overline{\mathcal{T}}_\lambda)\ \hookrightarrow  \
 H_\lambda \times H(\langle B \rangle) \ ,$$
composed with both projections, defines
surjective maps. Furthermore $(h,t)$ for
$h\in H_\lambda$ and $t\in k^*$ is in $H(\overline{\mathcal{T}}_\lambda)$
if and only if $t^r=\mu(h)$ holds, where  
$\mu=\mu_\lambda$ is the similitude character defined on $H_\lambda$ and $r=r(\lambda)$ is defined by 
$\mu_\lambda= B^{r}$. Thus $H(\overline{\mathcal{T}}_\lambda)$ is a fibre product
of $H_\lambda$ and $\mathbb G_m$ by $\mu$ and the $r$-th power map, and $H_\lambda$ is obtained from
its $r$-fold covering group $H(\overline{\mathcal{T}}_\lambda)$
as the quotient by the cyclic subgroup of order $r$ that is contained in the second factor
$ H(\langle B \rangle) \cong \mathbb G_m$.

\medskip\noindent


\medskip
From the discussion above of the exceptional case we immediately derive the next two Lemmas \ref{lem:lemma1} and \ref{lem:lemma2}.

\begin{lem}\label{lem:lemma1} Suppose that the class $\lambda/\! \sim$ is exceptional. Then either for the normalized representative $\lambda$
the Tannaka group $H_\lambda$ has trivial similitude character (if $r$ is even), or (understood nonexclusively) for the normalized representative $\lambda$
the group $H(\overline{\mathcal{T}}_\lambda)$ is isomorphic
as an  algebraic group  to $(GL(\dim(V_\lambda))\cdot \mu_2) \times \mathbb G_m$, so it has  two connected components and $Pic^0(H(\overline{\mathcal{T}}_\lambda))$ is isomorphic to $\mu_2$.
\end{lem}

\begin{lem} \label{lem:lemma2} For exceptional classes $\lambda/\!\sim$ the group $Pic^0(H(\overline{\mathcal{T}}_\lambda)) = Pic^0(H_{\lambda})$ is a nontrivial two-torsion group of rank $\leq 2$ .
\end{lem}

\begin{remark} The object $I$ realizing the surjective projection
$\chi_\lambda: H_\lambda \to \mu_2 =\langle w\rangle$ corresponds to an element with 
the properties in lemma \ref{the-module-I}. Indeed, $I$ appears as an indecomposable constituent
of  superdimension 1 in $L(\lambda)\otimes L(\lambda)^{\vee}$ that is
not isomorphic to the trivial representation. This follows immediately
from the description of $V_{\lambda} \cong Ind_{H_1}^{H}(W)$ as an 
induced representation. \end{remark}
 

We summarize the results of this section in the following theorem.

\begin{thm} \label{ThisIsMain} In the (NSD) resp. in the regular (SD)-cases the
groups $H_\lambda$ are isomorphic to $GL(V_\lambda)$,
resp. to $Kern(sign: G(V_\lambda,\langle .,.\rangle) \to \mu_2)$, i.e. $GSO(V_\lambda)$
or $GSp(V_\lambda)$, if $\ell(\lambda)\neq 0$, and to
 $SL(V_\lambda)$,
resp. to $SO(V_\lambda)$
or $Sp(V_\lambda)$ if $\ell(\lambda)\neq 0$. So in these cases $H_\lambda$ is connected and $H_\lambda/G_\lambda \cong H_\lambda^{ab}$.
In the exceptional (SD)-cases the groups $H_\lambda$ are not connected and
there exists a nontrivial homomorphism $\chi_\lambda: H_\lambda \to \mu_2$.
The kernel $\tilde H_\lambda$ is a subgroup that contains $SL(W)$ such that 
$S=\tilde H_\lambda/SL(W)$ becomes an algebraic  subgroup of $\mathbb{G}_m^2$
of dimension one or two, depending on whether $\ell(\lambda)=0$ or 
$\ell(\lambda)\neq 0$ (or, equivalently, $r = 0$ or $r \neq 0$).
\end{thm}

\subsection{The full Tannaka group}


Now consider the full tannakian category $\bar{\mathcal{T}}_n$. Note that \[ H_n \hookrightarrow \prod_{\lambda/\sim} H_{\lambda}.\] As in the introduction we fix an isomorphism $\mu_B$ between the Tannaka group $H_B:=H_\lambda$ of the twisted Berezin 
and the multiplicative group ${\mathbb G}_m$. Recall further that $\mu_{\lambda}$ was given by $\det_{\lambda}$ in the (NSD)-case and the similitude character in the (SD)-case. 
Then $H_n$ is a 
subgroup of the infinite fibre product defined by the elements 
$h=(h_\lambda)_{\lambda/\sim}$ in $ \prod_{\lambda/ \sim\ \in\, Y^+_0(n)} H_\lambda$  that satisfy 
$   \mu_\lambda(h_\lambda) = \mu_B(h_B) $ for all $\lambda$.
The induced fibre homomorphism $\mu: H_n \to \mathbb G_m$ defined by $\mu_B$ is surjective. Its kernel
$\tilde H_n= Ker(\mu: H_n \to \mathbb G_m)$ contains the projective limit $G_n$ of the derived groups  of the connected
component of the $H_\lambda$ as a normal subgroup. For $H_{n}^{ab}$ determined by $X^*(H_n^{ab})= Pic(H_n)  = Pic^0(H_n) \times \mathbb Z $ and $X^*(\tilde H_n^{ab}) \cong  Pic^0(H_n)$, as explained above, our computations imply 

\begin{cor}   $H_n^{ab}$
is the factor commutator group of $H_n/G_n$, and the commutator subgroup $M$ of 
$H_n/G_n$ is a pro-diagonalizable group. 
\end{cor}

\begin{proof}
Although only left exact, projective limits preserve exactness if the so called Mittag-Leffler conditions are satisfied. For a projective limit of algebraic groups as in our case, these conditions hold. Thus the projective limit $H_n$ contains the
projective limit $G_n$ as a normal subgroup, and the quotient  group $H_n/G_n$ admits the projective limit $H_n^{ab}$ as quotient group such that the remaining kernel $M$ in $H_n/G_n$ is the projective limit of the
diagonalizable algebraic groups $\prod_{\lambda/\!\sim\in I} (Q_\lambda)^2$
extended over all finite subsets $I$ of the set of exceptional classes $\lambda /\!\sim$. Here the $Q_\lambda$ are the diagonalizable groups from section \ref{exc-sd}.   
\end{proof}



\textbf{Remark}. If we knew that the characters $\chi_\lambda$
of $H_\lambda$ (attached to the exceptional classes $\lambda/\!\sim$)
were linear independent on $H_n$, this would easily imply
that $Pic^0(H_n)$ is a two torsion group.
To prove linear independency of the $\chi_\lambda$, by Schur's lemma amounts to show that every finite tensor product $\bigotimes_\lambda X_\lambda$ of simple objects 
attached to exceptional classes $\lambda/\!\sim$ 
is a simple object. However, in absence of such stronger results, we have to be content with the following weaker statement.

%
%

\begin{cor} $Pic^0(H_n)$ is a 2-power torsion group (possibly infinite), and this group is nontrivial if and only if exceptional classes $\lambda/\!\sim$ exist.
\end{cor}

\begin{proof} By our previous computations it suffices to consider the contribution of the exceptional classes.
Let $\lambda_i, i\in I$, denote a finite set $I$
of exceptional weights $\lambda_i$. Let $H_I$ denote
the Tannaka group attached to the Tannakian category $\overline{\mathcal{T}}_I$
that is generated by $B$ and the objects $X_{\lambda_i}$ for $i\in I$.
Notice $H_I \subset \prod_{i\in I} H(\overline{\mathcal{T}}_{\lambda_i})$,
and $\prod_{i\in T}\mu_i $ defines a character $\mu_I$ on $H_I$,
factorizing over the quotient group $H_I^{ab}$.
To prove our assertion it is enough to show that the kernel of $\mu_I: H_I^{ab}\to \mathbb G_m$ is a finite torsion group annihilated by $2^{\# I}$.
For each exceptional class $\lambda/\!\sim$,
the group $H(\overline{\mathcal{T}}_\lambda)$ contains in its commutator group the perfect group
$SL(W_{\lambda_i})$. Hence we may replace $H_I$ by its image $D$ in
$\prod_{i\in I} D_i $, for
$D_i =Kern(\mu_i: H_{\lambda_i} \to \mathbb G_m)/SL(W_{\lambda_i})$,  as is easy
to see. Each group $D_i$ for $i\in I$ is a generalized dihedral  group, i.e. a semidirect product
$$      D_i =  S_i \cdot \langle w_i \rangle \quad , \quad w_i^2 = 1 $$
for an algebraic group $S_i$ of multiplicative type, such that the involution $w_i$ acts on $S_i$ by conjugation: $w_i s_i w_i = s_i^{-1}$.
As a subgroup of $\prod_{i\in I} D_i$ the group $D$ admits
an exact sequence 
$$   0 \to S  \to D \to \prod_{i\in I}  \ \langle w_i \rangle \ ,$$
where $S$ is a subgroup of $\prod_{i\in I} S_i$. The image $W$
of $D$   is a
subvectorspace 
of $\prod_{i\in I}  \langle w_i \rangle$, considered as a vector space over the prime field of characteristic 2.
The latter  may be identified with the power set of $I$, 
so that the elements of $W$ may be considered as subsets $J\subseteq I$ of $I$.
Let $\chi_J: I \to \mathbb Z$ thus denote the characteristic function of $J\subset I$ in this sense.
Any $\mathbb Z$-valued function $f$ on $I$ can be considered
as an endomorphism of $\prod_{i\in I} S_i$ via 
$f((s_i)_{i\in I})= (f(i)\cdot s_i)_{i\in I}$.
For $J\in W$ with representative $w\in D$, the commutator of $w$ with elements of $S$ induces an endomorphism  of $S$. Considered as a subgroup of $\prod_{i\in I} S_i$, this is the endomorphism attached to the integer-valued function $ f= 2 \chi_J$ on $I$. By definition this endomorphism of $\prod_{i\in I} S_i$ preserves the subgroup $S$. Therefore any  element of $S$ in the image of the endomorphism 
$\sum_{J\in W} 2 \chi_J$ of $S$ is contained in the commutator group of $D$. By Lemma \ref{lem:lemma3}, the factor commutator group of $D$ thus is a finite group annihilated by $2^{\dim(W)}$, so also by $2^{\# I}$. Notice, Lemma \ref{lem:lemma3} can be applied since
$\bigcup_{J\in W} = I$ is satisfied, as follows from the fact that all projections $H_I \mapsto H_{\lambda_i}$, $i\in I$ are surjective.
\end{proof}

\begin{lem} \label{lem:lemma3} If $\ \bigcup_{J\in W} J = I$ holds, then as endomorphism of $S$ we obtain:
$$    \sum_{J\in W} 2 \chi_J \ =\  2^{\dim(W)} \cdot id_S  \ .$$
\end{lem}

\begin{proof} For a basis $J_1,...,J_d$ of the ${\mathbb F}_2$-vectorspace $W$ any $J\in W$ can be uniquely written in the form $J= \sum_{\nu=1}^d a_\nu J_\nu$ for $a_\nu\in \{0,1\}$. For $i\in I$ fixed, we may reorder the basis such that $i\in J_1,...,J_k$ and $i\notin J_{k+1},...,J_d$. Notice $k=k(i)\geq 1$ holds for all $i\in I$ by $\ \bigcup_{J\in W} J = I$. Hence $\chi_J(i) =1$ if and only if $a_1 + ... + a_k$ is odd. Since $k\geq 1$, therefore there exist precisely $\frac{1}{2} 2^{k} \cdot 2^{d-k} = 2^{\dim(W)-1}$ vectors $J$ in $W$ with $\chi_J(i) =1$.  The sum of the endomorphisms $2\chi_J$ of $S$, for $J\in W$, therefore gives $2^{\dim(W)}$ times the identity of $S$. 
\end{proof}


\section{A conjectural picture} \label{sec:conj-pic}

\subsection{Equivalent conjectures} We conjecture that all SD-cases are regular:

\begin{conj} \label{Tannakagroup-conj}
$G_\lambda = SL(V_\lambda)$ resp. $G_\lambda = SO(V_\lambda)$
resp. $G_\lambda = Sp(V_\lambda)$ according to whether $X_\lambda$ satisfies (NSD) respectively
(SD) with either $X_\lambda$ being even respectively odd. 
\end{conj}

\begin{lem}\label{conj-equiv}
The following are equivalent:
\begin{enumerate}
\item Conjecture \ref{Tannakagroup-conj} holds.
\item The module $I = I_{\lambda}$ is trivial.
\item Any invertible object $I$ in $\overline{\mathcal{T}}_{n}$ is represented in $\mathcal{T}_n^+$ by a power
of the Berezin determinant.
\item Conjecture \ref{special1} holds, i.e. every special module is trivial.
\end{enumerate}
\end{lem}

\begin{proof} By theorem \ref{ThisIsMain} the Picard group is generated by Berezin powers if there are no exceptional SD-cases, i.e. $I \cong \one$.
\end{proof}

We discuss a possible approach to proving $I \cong \one$ in appendix \ref{appendix-1}.

\subsection{An element of the Picard group of ${\mathcal T}_{2\vert 2}/{\mathcal N}$}

The assertion (3) from lemma \ref{conj-equiv} cannot hold for $\mathcal{T}_n$ instead of $\mathcal{T}_n^+$. In the category ${\mathcal T}_{2}$ consider the indecomposable representation $I$, defined uniquely up to isomorphism by the nontrivial extension  \[ \xymatrix{ 0 \ar[r] &  [0,-1] \ar[r] & I \ar[r] & \one \ar[r] & 0} \] This extension is realized in the Kac module of the trivial representation \[ K(\one) = \begin{pmatrix} \one \\ [0,-1]  \\ Ber^{-2} \end{pmatrix} \] Since $\sdim(K(\one)) = 0$, $I$ has superdimension $-1$, so its parity shift has $\sdim(\Pi I) = 1$. Clearly $I\otimes I^{\vee} \cong \one \oplus N$ with $\sdim(N) = 0$. We show that $N$ is negligible. For this we can pass to the homotopy category $\mathcal{H}oT$ attached to the lower parabolic $\mathfrak{p}^-$ \cite{HW-homotopy}. The kernel of the homotopy functor $\pi: \mathcal{T}_{2} \to \mathcal{H}oT$ consists of the Kac objects, i.e. those modules with a filtration by Kac-modules. The short exact sequence for $K(\one)$ induces in the tensor triangulated category $\mathcal{H}oT$ an exact triangle \[ \xymatrix{ Ber^{-2} \ar[r] & K(\one) \ar[r] & I \ar[r] & Ber^{-2}[1].} \] Since $K(\one)$ is zero in $\mathcal{H}oT$, this gives the identification $I \cong Ber^{-2}[1]$. In particular \[ I \otimes I^{\vee} \cong Ber^{-2}[1] \otimes (Ber^{-2})^{\vee}[1] \cong \one[2].\] Now suppose $N$ would not be negligible. Then $N \cong N_1 \oplus N_2$ with $\sdim(N_1) \neq 0$. In particular $\pi(N_1) \neq 0$, a contradiction since $\pi$ is a symmetric monoidal functor and $\one[2]$ is indecomposable.

\begin{cor}\label{cor:counter} $I\otimes I^\vee ={\bf 1}\oplus N$ and $N$ is   negligible. Therefore
$I$ and its parity shift $\Pi I$ define elements of the Picard group of 
${\mathcal T}_{2}/{\mathcal N}$.
\end{cor}

Another counterexample can be found in appendix \ref{counter-2}.

\subsection{An application} The conjectural structure theorem would have the following consequences.

\begin{cor} \label{1-dimen-summands}
For given $L=L(\lambda)$ in $\calT_n$ and $r\in\mathbb Z$ there can exist at most one  
summand $M$ in $L \otimes (Ber^r \otimes L^\vee)$ with the property $\sdim(M)=\pm 1$.
If it exists then $M \cong Ber^r$.
\end{cor}

{\it Proof of the corollary}. We can assume that $L$ is maximal atypical. Then ${\bf 1}$ is a direct summand
of $L\otimes L^\vee$ and hence $Ber^r$ is a direct summand of $L \otimes (Ber^r \otimes L^\vee)$. Hence it suffices to show that ${\bf 1}$ is the unique summand $M$ of $L \otimes L^\vee$
with $\sdim(M)=\pm 1$. Equivently it suffices to show that $V_\lambda \otimes V_\lambda^\vee$
contains no one-dimensional summand except ${\bf 1}$. This now follows from conjecture \ref{Tannakagroup-conj} using the well known fact
that $st \otimes st^\vee$ for the standard representation $st$ of $SL(V),SO(V), Sp(V)$
contains only one summand of dimension 1. \qed


\subsection{The Tannaka groups $H_{\lambda}$ revisited}

The following theorem is an immediate consequence of theorem \ref{ThisIsMain}.

\begin{thm} \label{the-groups-h-L} Assuming conjecture \ref{Tannakagroup-conj}, the Tannaka groups $H_{\lambda}$ of $X_{\lambda}$ are the following: 
\begin{enumerate}
\item NSD non-basic: \  $H_{\lambda} = GL(V_{\lambda})$. 
\item NSD basic: \ $H_{\lambda} = SL(V_{\lambda})$.
\item SD, proper selfdual, $\sdim(L(\lambda_{basic})) > 0$: \  $H_{\lambda} = SO(V_{\lambda})$.
\item SD, proper selfdual, $\sdim(L(\lambda_{basic})) < 0$: \  $H_{\lambda} = Sp(V_{\lambda})$.
\item SD, weakly selfdual, $\sdim(L(\lambda_{basic})) > 0$: \  $H_{\lambda} =  GSO(V_{\lambda})$.
\item SD, weakly selfdual, $\sdim(L(\lambda_{basic})) < 0$: \  $H_{\lambda} = GSp(V_{\lambda})$.
\end{enumerate}
In each case the representation $V_{\lambda}$ of $H_{\lambda}$ coming from $X_{\lambda}$ corresponds to the standard representation. In the $GL$ case the determinant comes from a (nontrivial) Berezin power. 
\end{thm}

Note that a basic representation of SD type always satisfies $L \cong L^{\vee}$. In the (SD) case $\ell(\lambda) = 0$ if and only if $L(\lambda) \simeq L(\lambda)^{\vee}$. In the (NSD)-case $\ell(\lambda) = 0$ if and only if $\lambda$ is basic.

\subsection{Special modules} We discuss a conjecture which would show $I \cong \one$.

\begin{definition} An indecomposable module $V$ in $\mathcal{T}_n^+$ 
with $\sdim(V)=1$ will be called {\it special}, if  $V^*\cong V$ and $H^0(V)$ contains ${\bf 1}$ 
as a direct summand. 
\end{definition}

For special modules $$DS(V) \cong {\bf 1} \oplus N$$ holds for some negligible module $N$, since $\sdim(DS(V))=\sdim(V)$. This also implies $$ H_D(V) = {\bf 1} \oplus N \ .$$ 

\begin{lem}
Suppose $V\cong V^* \cong V^\vee$ and $DS(V) \cong {\bf 1} \oplus N$
holds for some negligible module $N$. Then $V$ is special.
\end{lem}

\begin{proof} The assumptions imply that there exists a unique integer $\nu$ for which
$H^\nu(V)$ is not a negligible module.
Since $H^\nu(V)^\vee \cong H^{-\nu}(V^\vee)$, the assumption $V \cong V^\vee$
implies $\nu=0$. Hence $H^0(V)={\bf 1} \oplus N$ for some negligible $N$. 
\end{proof}

\begin{conj}\label{special1}
Up to a parity shift, any special module $V$ in $\mathcal{T}_n^+$ is isomorphic to the trivial module ${\bf 1}$.
\end{conj}


\section{The Picard group of $\overline{\mathcal{T}}_{n}$}\label{picardgroup}

We study the determinant $\det(X_{\lambda})$ in this section.

\subsection{The invariant $\ell(\lambda)$} 
As one easily shows, for any object $X$ of ${\mathcal T}_n$ $$ \det(B^m \otimes X) = B^{m \cdot \sdim(X)}\otimes \det(X)  \ .$$ Hence to determine $I_\lambda$ we may assume $\lambda_n=0$. So let us assume this for the moment.  
Then, for a maximal atypical weight $\lambda$ with the property $\lambda_n=0$,
let  $S_1,...,S_k$ denote its corresponding sectors, from left to right. 
If $i=1,..,k-1$ let $d_i= dist(S_i, S_{i+1})$
denote the distances between these sectors and $r(S_i)$ denotes the rank of $S_i$,
then $\sum_{i=1}^k r(S_i)= n$. Furthermore $d=\sum_{i=1}^k d_i =0$ holds if and only if the weight $\lambda$ is
a basic weight. Recall, if we translate $S_2$ by shifting it $d_1$ times to the left, then shift $S_3$ translating it $d_1+d_2$ to left and so on,
we obtain a basic weight. This basic weight is called the {\it basic weight associated 
to $\lambda$}. 
The weighted total number of shifts necessary to obtain this associated basic weight by definition is 
the {\it integer}
$$  \ell(\lambda) := \sum_{i=1}^k \sdim(X_{\lambda_i}) \cdot (\sum_{j<i} d_j) \ $$
where $L(\lambda_i)\in \calR_{n-1}$ denote the irreducible representations
associated to the derivatives 
$S_1 .... \partial S_i .... S_k$. By \cite{Weissauer-GL} \cite[Section 16]{Heidersdorf-Weissauer-tensor} $\sdim(X_{\lambda_i}) = \frac{r_i}{n} \cdot \sdim(X_\lambda)$ holds 
for $r_\nu=r(S_\nu)$, which allows to rewrite this in the form
$$  \ell(\lambda) = n^{-1} \sdim(X_{\lambda})  D(\lambda) \ ,  $$
where $D(\lambda)$ is the {\it total number of left moves} needed to shift the support of the plot $\lambda$ into the support of the associated basic plot $\lambda_{basic}$, i.e. the integer
$$ D(\lambda) \ := \ \sum_{\nu=1}^k \ r_\nu \cdot (\sum_{\mu<\nu} d_\mu) \ .$$
Now, to remove our temporary assumption $\lambda_n=0$ and hence to make the formulas above true unconditionally, we have to introduce the additional terms $d_0= \lambda_n$ (for $\mu=0$) in the formulas above. For further details on this see \cite[Section 25]{Heidersdorf-Weissauer-tensor}. We remark that in the following
we also write $D(L)$ instead of $D(\lambda)$ for the irreducible representations $L = L(\lambda)$ and similarly $\ell(L)$ instead of $L(\lambda)$.

\subsection{$Pic^0$} We return to indecomposable objects $I \in \mathcal{T}_n^+$ 
representing invertible objects of ${\overline{\mathcal T}}_n$. 

Since $I \otimes I^\vee \cong {\bf 1} \oplus $ negligible objects, we obtain \[ \omega(I ,t) \omega(I^\vee,t) = \omega(I \otimes I^\vee,t) = 1.\] Indeed,
the functor $\omega$ annihilates negligible objects.
For the Laurent polynamial $\omega(I ,t)$ this now implies
 \[ \omega(I,t) = t^{\nu}\] for some integer $\nu \in \Z$ which defines the degree $\nu(I) = \nu$.
Obviously this degree $\nu(I)$ induces a homomorphism $Pic(\calR_n) \to \mathbb Z$ 
of groups by $I \mapsto \nu = \nu(I) \in \mathbb Z$ and gives an exact sequence
\[ \xymatrix{ 0 \ar[r] &  Pic^0({\overline{\mathcal T}}_n) \ar[r] &  Pic({\overline{\mathcal  T}}_n) \ar[r]^-\nu & \mathbb Z } \]
 with kernel $Pic^0({\overline{\mathcal T}}_n)$. Clearly $\nu(B) = n$, hence the next lemma follows. 

\begin{lem} The intersection of $Pic^0({\overline{\mathcal T}}_n)$ with the subgroup generated by the normalized Berezin $B$ is trivial. 
\end{lem}

\begin{lem} \label{det-nu} For any irreducible object  $X$ in $\mathcal{T}_n^+$ 
the invertible element $\det(X) \in {\mathcal T}_n$ has the property 
\[ \nu(\det(X)) = \sdim(X)\cdot D(X) = \ell(X) \cdot n.\] In particular, the image of the homomorphism $\nu$ contains $n \cdot \mathbb Z$.
\end{lem}

{\it Proof}. Fix some $X=X_\lambda$. We can assume $\lambda$ to be maximally atypical.
The functor $\omega: {\mathcal T}_n^+ \to gr{-}vec_k$ is a tensor functor. Hence
$\nu(\det(X)) = \nu(\det(\omega(X))$. Hence 
\[ \nu(\det(X)) = \sum_i i \cdot a_i \ \ \ \ \ \  \  \ (*) \] for $\omega(X,t) = \sum_i a_i t^i$. By \cite[Lemma 25.2]{Heidersdorf-Weissauer-tensor} $\omega(X,t^{-1}) = t^{-2D(\lambda)} \omega(X,t)$ and hence $\omega(X_{basic},t) = \omega(X_{basic},t^{-1})$, the latter because of $D(X_{basic}) = 0$. So $a_i =a_{-i}$ holds for basic $X$, and formula (*) implies $\nu(\det(X_{basic})) = 0$. From $\omega(X,t) = t^{D(X)} \omega(X_{basic},t)$ and $\sdim(X_{basic}) = \sdim(X)$, again by (*) we therefore obtain $$ \sum_i i \, a_i =   \sum_i  i \, a_{basic,i}  + D(\lambda)\cdot \sum_i  a_i   = D(\lambda) \sdim(X) \ .$$ Note  that  $X\in \mathcal{T}_n^+$ has superdimension $\geq 0$, hence $\omega(X,1) = \sum a_i $ is the superdimension of $X$ (not only up to a sign). \qed 

\medskip
Since $\omega(L(\lambda),t)t^{-D(\lambda)}$ is invariant under $t\mapsto t^{-1}$
for irreducible $L=L(\lambda)$,
we also obtain

\begin{cor} $d \ log(\omega(L,t))|_{t=1} = D(L)$.
\end{cor}

\begin{cor} \label{det-1} We have $\det(X) \otimes B^{-\ell(X)} \in Pic^0(\overline{\mathcal{T}}_n^+)$, i.e. \[ \det(X) \in Pic^0(\overline{\mathcal{T}}_n) \times B^{\Z} \] for irreducible $X \in \mathcal{T}_n^+$.
\end{cor}

\begin{example} For $GL(2|2)$ we obtained (up to parity shifts) in \cite{Heidersdorf-Weissauer-GL-2-2} the formula $S^i \otimes S^i = Ber^{i-1} \oplus M$ for some module $M$ of superdimension 3. Since $\sdim(S^i) = 2$, $\det(S^i) = Ber^{i-1} \oplus \text{ negligible}$. Indeed for $S^i$ we obtain $\ell([i,0]) = r_1 d_0 + r_2 d_1$ where $r_i$ denotes the rank of the $i$-th sector. Clearly $r_1 = r_2 = 1$ and $d_0 = 0$ and $d_1 = i-1$, hence $\ell([i,0]) = i-1$.
\end{example}


\medskip

\section{The determinant of an irreducible representation}\label{sec:determinant}

We now compute the determinant of irreducible representations. The computation uses the passage to the stable category with its triangulated structure. This determinant calculation determines in all regular cases (along with the results of section \ref{sec:conjecture}) the full Picard group of $H_{\lambda}$.

\subsection{The full even category $\mathcal T^{ev}$}

Since a svectorspace
is the sum of an even and odd subspace, we have $svec_k = vec_k \oplus \Pi(vec_k)$
as a decomposition of abelian categories. We say, an object $X$ of $\mathcal T$
is even resp. odd if  $\omega(X)$ is in $vec_k$ resp. $\Pi(vec_k)$. In terms of the Hilbert polynomial $\omega(X,t)$ defined in section \ref{sec:hilbert} even means that all $t$-powers are even. Let $\mathcal T^{ev}$ and $\mathcal{T}^{odd}$ denote the corresponding full subcategories.
In \cite[Section 24]{Heidersdorf-Weissauer-tensor} it is shown that simple objects in $\mathcal T$ are always even or odd, hence \[ \mathcal{T}^+ \subset \mathcal{T}^{ev}.\]

\medskip\noindent
Although $D^n$ is not an exact functor, exact sequences in $\mathcal T$ become exact hexagons  in ${\mathcal T}_0$.
[Sometimes it is useful that in a certain sense we need not distinguish between $D^{n-1}$ and $D^n$ since $D: \overline{\mathcal T}_1 \to {\mathcal T}_0$ is faithful. This refines the notion of even/odd if the $D^{n-1}$-image
is semisimple and even resp. odd.] 
Extensions of even (odd) objects in $\mathcal T$ are even (odd) objects in $\mathcal T$. Obviously $\mathcal T^{ev}$, as a full karoubian subcategory of $\mathcal T$,
is closed under extensions, retracts and tensor products and Tannaka duals. In particular we have stability with respect to  Schur functors. 
   
\bigskip\noindent
We consider now the semisimplification of $\mathcal{T}^{ev}$ with the same method as in section \ref{sec:tannakian-arguments}. Here we however use the Dirac tensor functor $D$ (see section \ref{sec-Dirac}) instead of $DS$ (note they agree on $\mathcal{T}^+$). 
Iterated $n$ times it factorizes over the additive quotient category ${\mathcal T}^{ev} \to {\mathcal  A}$ defined by dividing through the ideal of all morphisms that factorize over null objects (objects whose indecomposable summands have superdimension zero). i.e. $D^n: {\mathcal T}^{ev}_n \to {\mathcal T}_0^{ev}$. 
We define $$  \omega =   D^n \circ s \ $$
via a section $s$ of the semisimplification functor $\nu:   {\mathcal  A} \to  \overline{\mathcal  T}^{ev}$, i.e. $\nu \circ s = id_{\overline{\mathcal  T}}$ whose existence follows from \cite{Andre-Kahn}. The $k$-linear tensor functor $\omega$ is an exact functor since $\overline {\mathcal T}^{ev} $ is semisimple. Thus it defines a superfibre functor
$\omega: \overline{\mathcal T}^{ev} \to svec_k$ with values in $vec_k$.

\medskip\noindent
Since for indecomposable objects $X$ of $\mathcal T^{ev}$ the space $\omega(X)$ has dimension $sdim(X)>0$ (unless $sdim(X)=0$ and $X$ is negligible), the
quotient $\overline {\mathcal T}^{ev}$ defines a Tannakian category. Hence, as a tensor category
$\overline {\mathcal T}^{ev}$  is equivalent to $Rep_k(H^{ev},\varepsilon)$
for some affine (pro)reductive group scheme $H^{ev}$ over $k$. 

\begin{remark} All in all we attached to the category $\mathcal{T}$ three different semisimple supertannakian categories: \[ \mathcal{T}^+/\mathcal{N}, \ \mathcal{T}^{ev}/\mathcal{N}, \ \mathcal{T}/\mathcal{N}.\] The inclusion $\mathcal{T}^{+} \subset \mathcal{T}^{ev}$  is strict. Already for $GL(1|1)$, $\mathcal{T}^{ev}$ contains ZigZag modules of length $2m+1$ for $m \in \mathbb{N}$ \cite{Heidersdorf-semisimple} which have nonvanishing superdimension. In fact for $GL(1|1)$ the quotients $\mathcal{T}^{ev}/\mathcal{N}$ and $\mathcal{T}/\mathcal{N}$ are isomorphic. The relationship betweeen $\mathcal{T}$ and $\mathcal{T}^{ev}$ and their corresponding semisimple quotients is unclear as it is not obvious how to find indecomposable objects of nonvanishing superdimension which are mixed (i.e. neither even nor odd).   
\end{remark}


\subsection{The stable category} Recall that $\mathcal{T}_n = {\mathcal T}$ is a $k$-linear tensor category and as an abelian category it is a Frobenius category. 
Associated to a Frobenius category ${\mathcal T}$ one defines
its stable category ${\mathcal K}$ as a quotient category \cite{Happel}. 
For the quotient functor $$\alpha: {\mathcal T}\to {\mathcal K}$$ 
the objects of  ${\mathcal K}$ are those of 
${\mathcal T}$, but morphisms are equivalence classes of
morphisms in ${\mathcal T}$. Two morphisms become  equivalent if their difference
is a morphism that factorizes over a projective module in ${\mathcal T}$. 
${\mathcal K}$ is a triangulated category with a suspension functor $S(X)=X[1]$
such that $Ext^i_{\mathcal T}(X,Y)\cong Hom_{\mathcal K}(X,Y[i])$ holds for all
$i\geq 0$ and $\alpha$ is a tensor functor. The $\alpha$-image of an exact sequence in $\mathcal T$  induces a distinguished triangle in $\mathcal K$. Any distinguished triangle in $\mathcal K$ is isomorphic  as a triangle to the $\alpha$-image of an exact sequence in  ${\mathcal T}$ \cite{Happel}.

\medskip\noindent 
Let $\mathcal K^{ev}, \mathcal{K}^{odd}$ denote the corresponding full subcategories of $\mathcal{K}$ corresponding to $\mathcal{T}^{ev}, \mathcal{T}^{odd}$.
Similarly to the $\mathcal{T}$-case under the Dirac functor $D$ exact triangles in $\mathcal K$ become exact triangles in $svec_k$. In $\mathcal K$  the following holds: If $X\to Y\to Z\to $ is a distiguished triangle
and $X$ and $Z$ are in $\mathcal K^{ev}$, then also $Y$ . This follows since $\omega=D^n$ induces an exact hexagon from each exact sequence in $\mathcal T$ as in \cite[Lemma 2.1]{Heidersdorf-Weissauer-tensor} via $\omega=\omega^+\oplus \omega^-$.

\subsection{Determinants} 
Recall $\sdim(X) \geq 0$ for $X \in \mathcal{T}^{ev}$. Hence
$\det(X) = \Lambda^{\sdim(X)}(X)$ is defined so that
$ \Lambda^{\sdim(X)+1}(X)$ is negligible. Here $L=\det(X)$ by definition is ${\mathbf 1}$ if $sdim(X)=0$. The image of $\det(X)$ under the functor $\omega=D^n$ defines an invertible object in the Tannakian subcategory generated by $\omega(X)$ in $\overline{\mathcal T}^{ev}$.

\bigskip\noindent
For exact sequences $0\to X\to Y\to Z \to 0$ in ${\mathcal T}^{ev}$
 we have $\det(Y)\cong \det(X)\otimes \det(Z)$ in ${\mathcal T}^{ev}$.
The analogous assertion holds for a distinguished triangle $0\to X\to Y\to Z \to$
with $X,Y,Z\in {\mathcal K}^{ev}$, simply by lifting this to an exact sequence $0\to X'\to Y'\to Z'\to 0$
in $\mathcal T^{ev}$. For this notice, if $X,X'$ in $\mathcal T^{ev}$ become isomorphic in
$\mathcal K^{ev}$, then $\det(X)$ and $\det(X')$ become isomorphic on $\mathcal K^{ev}$.
In particular, their images in the Tannakian representation category $\overline{\mathcal T}^{ev}$ of the reductive group $H^{ev}$
become isomorphic.  

\bigskip\noindent
Now consider the following special situation, where  
$I$ in $\mathcal T$ has a filtration of length 3  whose graded pieces are a submodule $S$, a middle layer $M$ and the quotient $T=I/M$ on top.

\begin{lem} \label{lem-lem} Suppose $S,T$ are in $\mathcal T^{odd}$ and
$I$ is in $\mathcal T^{ev}$. Then $M, S[1]$ and $T[-1]$ are in $\mathcal K^{ev}$ and the following holds in $\mathcal K$ up to negligible summands: 
 $$   \det(M) \cong \det(I)\otimes \det(S[1])\otimes \det(T[-1]) \ . $$
\end{lem}
 
\begin{proof} If $S$ is odd/even in the stable category, then $S[1]$ is even/odd   
in the stable category, and conversely. Indeed $S[1]$ is represented by a quotient  $P/S$ in the representation category for a suitable
projective/injective module $P$. To the exact
sequence $0\to S \to P\to P/S\to 0$ the functor $\omega=D^n$ attaches an exact hexagon
for $\omega= \omega^+ \oplus \omega^-$ which immediately implies $\omega^\pm(S)\cong
\omega^{\mp}(P/S)$ in $\mathcal T$ and then $\mathcal K$. From this the above assertion follows.   
We now have two distinguished triangles 
 $I \to Y \to S[1] \to $ for suitable $Y$ and $T[-1] \to M \to Y \to$  in  $\mathcal K$. Here
 $S[1]$, $T[-1]$ and $I$ are in ${\mathcal K}^{ev}$ by our assumptions. Therefore $Y$ and then
 also $M$  are in ${\mathcal K}^{ev}$. All objects $S[1]$, $T[-1]$, $M$, $I$ and $Y$  are represented by objects  $U$, $V$, $M'$, $I'$, and $Y''$  in $\mathcal T^{ev}$
 so that there are exact sequences $0 \to I' \to Y' \to U\to 0$ and
 $0\to V \to M' \to Y'' \to 0$ in $\mathcal T$. Since all the superdimensions are $\geq 0$,
 we conclude $\det(Y') \cong \det(I')\otimes \det(U)$ and $\det(M')\cong \det(V)\otimes \det(Y'')$.
 Since $\det(Y')\cong \det(Y'')$, $\det(I')\cong \det(I)$ and $\det(M')\cong \det(M)$
 hold in $\mathcal K^{ev}$, in the stable category $\mathcal K^{ev}$ we obtain 
 $   \det(M) \cong \det(I)\otimes \det(S[1])\otimes \det(T[-1])$. 
\end{proof}

\textit{A symbolic way of writing}. 

For  $Y\in {\mathcal K}^{ev}$ define $ Y \langle\pm 1 \rangle = \Pi(Y)[\pm 1]$
in $ {\mathcal K}^{ev}$. Then $Y\langle \pm 1 \rangle \cong Y \otimes {\mathbf 1}\langle\pm \rangle$. For the iterated tensor powers $({\mathbf 1}\langle \pm m \rangle)$ of  $({\mathbf 1}\langle \pm \rangle)$  one has ${\mathbf 1}\langle n_1 \rangle
\otimes {\mathbf 1}\langle n_2 \rangle \cong {\mathbf 1}\langle n_1 + n_2 \rangle $ for all
$n_1,n_2\in \mathbb Z$.
For $X\in {\mathcal K}^{odd}$ put $\det(X):=\det(\Pi(X))^\vee \in {\mathcal K}^{ev} $.
Using this definition,  up to negligible objects the formula in the lemma above becomes
$$ \det(I)\cong \det(S)\otimes \det(M)\otimes \det(T)\langle \sdim(T)-\sdim(S) \rangle \ .$$ 
For this notice $\det(\Pi(X)\langle \pm 1\rangle)\cong \det(\Pi(X)) \langle
\pm  \sdim(X)\rangle$. This formula follows from \cite[Proposition 1.11]{Deligne-tensorielles} \[ \det(X \otimes Y) \cong det(X)^{\sdim(Y)} \otimes \det(Y)^{\sdim(X)} \]applied for $Y = \one$ using $\sdim(\one\langle 1 \rangle) =1$ and $X \langle 1 \rangle \cong X \otimes \one \langle 1 \rangle$.

\subsection{Calculation of determinants}


\begin{thm} \label{thm1.4} \label{cor-det-formula} For any maximal atypical weight $\lambda$  defining $X_\lambda$ in $\mathcal{T}_n^+$, for $\lambda_n=0$
the module $\det(X_\lambda)$ satisfies $$ \det(X_\lambda) \ = \ B^{\ell(\lambda)} \ \oplus \ \text{negligible}  .$$ 
In particular, for $\lambda_n=0$ we have $\det(X_\lambda)={\bf 1}$ if (and only if) the maximal atypical weight weight $\lambda$ is a basic weight. 
\end{thm}

\begin{proof} We prove this 
claim in $\mathcal K$ by a kind of induction, using the method of \cite{Heidersdorf-Weissauer-tensor}. This requires a certain ordering of the maximal
atypical simple representations, described in the section on the algorithms I,II, and III in \cite{Weissauer-GL}\cite[Section 20]{Heidersdorf-Weissauer-tensor}. For that we define an order on the set of cup diagrams for a fixed block such that the representations with completely nested cup diagrams (in our case the Ber powers) are the minimal elements. 

\bigskip\noindent
In \cite{Weissauer-GL} \cite[Section 20]{Heidersdorf-Weissauer-tensor} it is also shown that for every maximal atypical irreducible module $X$
there exists a negligible indecomposable object $I$ of Lowey length 3 in $\mathcal T$
such that the socle $S\cong A$ and cosocle $T\cong A$ are isomorphic 
simple objects $A$, and the middle layer $M= X \oplus M'$ is a direct sum of simple objects 
such that $A$ and all simple summands of $M'$ are smaller than $X$ with respect to the ordering. More precisely, this indecomposable module is one of the translation functors $F_i L^{\times \circ}$ of \cite[Section 18]{Heidersdorf-Weissauer-tensor}.
Furthermore, it was shown that all simple objects in $M$ have the same parity \cite[Section 20]{Heidersdorf-Weissauer-tensor}. Without restriction of generality we may therefore assume that $M\in \mathcal T^{ev}$ and $A\in \mathcal T^{odd}$ holds. 
Hence lemma \ref{lem-lem} implies that $\omega(\det(X))$ is a power of $
\omega(B)$, by induction on $X$ with respect to the mentioned ordering. 

\bigskip\noindent
Concerning the start of this induction:
the claim holds for groundstates $X$ in the sense of  \cite{Weissauer-GL}\cite{Heidersdorf-Weissauer-tensor}. In the present situation for $GL(n\vert n)$  these are the powers $B^k$ of $B$ for $k\leq 0$. Since the groundstates are the start of the induction  above, the determinant is a Ber-power. The specific power $\ell(\lambda)$ follows then from corollary \ref{det-1}. 

\end{proof}

\begin{remark} In \cite{Heidersdorf-Weissauer-tensor} we showed that if $i$ is chosen correctly, one can find for given maximal atypical $L$ an irreducible module $L^{\times \circ}$ such that the translation functor $F_i (L^{\times \circ})$ satisfies the conditions of the proof. In particular it contains $L$ in the middle layer such that all other composition factors of $F_i (L^{\times \circ})$ are of lower order. For $L$ with more then one segment we can choose $i$ and $L^{\times \circ}$ in such a way that all composition factors have one segment less then $L$. We can now apply the same procedure to all the composition factors of $F_i(L^{\times \circ})$ with more then one segment. Iterating this we finally end up with a finite number of indecomposable modules where all composition factors have weight diagrams with only one segment. This procedure is called Algorithm I. In Algorithm II we decrease the number of sectors in the same way. Iterating we finally relate $L$ to a finite number of maximal atypical representations with only one sector. Hence after finitely many iterations we have reduced everything to irreducible modules with one segment and one sector. This sector might not be completely nested.
In this case we can apply Algorithm II to the internal cup diagram having one segment enclosed by the outer cup. If we iterate this procedure we will finally end up in a collection of Berezin powers.
\end{remark}


\section{The conformal group and low rank cases}\label{sec:physics}

In view of the relation with the conformal group $G$ of the Lorentz metric
we discuss ceretain cases of rank $\leq 4$.  The complexified
Lie algebra $Lie(G)\otimes_{\mathbb R} {\mathbb C}$ of the conformal group is isomorphic to
the complex Lie algebra $\mathfrak{sl}(4)$. So the Lie superalgebras $\mathfrak{gl}(4\vert N)$
are of potential interest as supersymmetry algebras of conformal field theories
and the finite dimensional representations $L$ of these Lie superalgebras
may serve as targets of fields $\psi: M\to L$
on certain supermanifolds $M$ related to
Minkowski space such that $Lie(G)$ acts on $M$ by supervector fields.
A covering of the Poincare group can be embedded into $G$, 
and in particular the universal covering $SL(2,{\mathbb C})$ of the Lorentz group
$SO(1,3)$. The restriction of the representation $L$ to the Lie subalgebra
$Lie(SL(2,{\mathbb C}))$ decomposes into irreducible representations of the complex
Lie algebra $\mathfrak{sl}(2)$ and their highest weights defines the underlying classical spin
values of the $L$-valued fields. For physical reasons it seems relevant
that these spins $s$ are contained in the set $\{0,\frac{1}{2},1,\frac{3}{2},2\}$.
In other words, the highest weights should not exceed $5$. We refer to this as the spin condition.

\medskip\noindent
The structure of tensor products of irreducible representation of $GL(4\vert N)$ resp.
$SL(4\vert N)$ is controlled by the number $m= min(N,4)$. For $m < 4$
the information is encoded in the tensor products of irreducible representations
of the reduced group $GL(m\vert m) \times GL(4-m)$. For $m=4$, this reduced
group has to be replaced by $GL(4\vert 4) \times GL(N-4)$.
So these leads us to consider $GL(n\vert n)$ for $n=4$ and $3$. The case
$n=2$ was completely discussed in section \ref{sec:ind-start}.
In the following we therefore list some interesting candidates for irreducible
superrepresentations $L$ where the spin condition is satisfied. In fact
there exist only finitely many isomorphism classes of irreducible representation
where the above spin condition is satisfied. The most prominent example
is given by $L=S^1$ where only spin $s=0$ and $s=\frac{1}{2}$ shows up
in the restriction to the Lie algebra of the Lorentz group of this 
representation of dimension $\dim(L)= n^2-2$. In the case $n=4$
the largest and most interesting example we give is probably  the irreducible
representation $L=[3,2,1,0]$ of $\mathfrak{gl}(4\vert 4)$ of dimension $\dim(L)= 11\, 163\, 160$.  
Here all spins $s$ of the restriction are in $\{0,\frac{1}{2},1,\frac{3}{2},2\}$ and
all these numbers occur. As already explained, the Tannaka groups $H_\lambda$
related to the irreducible representations $L=L(\lambda)$ may perhaps show up in such theories as hidden approximate symmetry groups. $L(\lambda)=[3,2,1,0]$ defines a symmetric (SD)-case. The underlying group $H_\lambda$
should be the group $SO(24)$ if this case is regular (if not $G_\lambda$
would be $SL(12)$, but we could not exclude this). 
This case is the only case of our example where we could not exclude
exceptional (SD)-case. 

\medskip\noindent
One remark for this section. For the convenience of physicists we 
replace here the groups $H_\lambda$ by their compact inner forms $H_\lambda^c$.
So we write $U(1)$ instead of $\mathbb G_m$ and
$SU(k)$ instead of $SL(k)$, $Sp^c(2k)$ of $Sp(2k)$ etc.
In fact, the tensor categories ${\mathcal T}_\lambda$
of the complex algebraic groups $H_\lambda$ are isomorphic to the tensor
categories of their compact inner forms $H_\lambda^c$.

\medskip\noindent
The expected behaviour of the groups $H_{\lambda}$ was summarized in theorem \ref{the-groups-h-L} (all of which is proven except for the exceptional SD-case!). Here we discuss the $GL(3|3)$ and $GL(4|4)$-case.

\begin{example} \label{exn3} The $GL(3|3)$-case. For $n=3$ the structure theorem on the $G_{\lambda}$ holds unconditionally and therefore also the results on the $H_{\lambda}$. Here is a list of the nontrivial basic representations and their Tannaka groups. We automatically consider the possible parity shifted representation with positive superdimension here.
\begin{enumerate}
\item $[2,1,0]$, \  $\sdim = 6$, \ $H_{\lambda} = Sp^c(6)$.

%
%

\item $[1,1,0]$, \ $\sdim = 3$, \ $H_{\lambda} = SU(3)$.
\item $[2,0,0]$, \   $\sdim = 3$, \ $H_{\lambda} = SU(3)$.
\item $[1,0,0]$, \  $\sdim = 2$, \ $H_{\lambda} = SU(2)$.
\end{enumerate}

Twisting any of these with a nontrivial Berezin power gives the $GL$, $GSO$ or $GSp$ version. The appearing groups exhaust all possible Tannaka groups arising from an $L(\lambda)$.
\end{example}

\begin{example} The $GL(4|4)$-case. Here the structure theorem for $G_{\lambda}$ (and therefore the determination of $H_{\lambda}$) holds unconditionally for basic weights except for the case where $L(\lambda)$ is weakly selfdual with $[\lambda] \neq [3,2,1,0]$ by the following lemma:

\begin{lem} \label{lem-gl-4-4} The basic representations of (SD) type \[ [3,1,1,0], \ [2,1,0,0], \ [2,2,0,0] \] are regular (i.e. $I \cong \one$).
\end{lem}

\begin{proof} For $[2,2,0,0]$ this follows from appendix \ref{appendix-1} and example \ref{2-2-0-0}. It is enough to verify that $DS([2,2,0,0])$ does not contain a summand $L(\lambda_i)$ with $(\lambda_i)_{basic} = [2,1,0]$. The irreducible representations $[3,1,1,0]$ and $[2,1,0,0]$ have $k=3$ sectors each. However $V_{\lambda}$ can only decompose under the restriction to $G_{\lambda}$ if $k$ is even. Alternatively note that we have embedded subgroups $Sp(6)$ and $Sp(6) \times SL(3)$ in $G_{[2,1,0,0]}$ and $G_{[3,1,1,0]}$ respectively which implies that $G_{\lambda}$ cannot be $SL(3)$ or $SL(6)$.
\end{proof}

For $n=4$ there are 14 maximal atypical basic irreducible
representations in $\calR_4$, the self dual representations
$$ \one= [0,0,0,0] , S^1=[1,0,0,0] , [2,1,0,0] , [2,2,0,0] , [3,1,1,0] , [3,2,1,0] $$
of superdimension $1,-2,-6,6,-12 ,24$ and the representations 
$$ S^2=[2,0,0,0] , S^3=[3,0,0,0] , [3,1,0,0] , [3,2,0,0] $$
of superdimension $3,-4, 8, -12$ and their duals 
$$ [1,1,0,0] , [1,1,1,0] , [2,1,1,0] , [2,2,1,0] \ .$$

Here is a list of the nontrivial basic representations and their Tannaka groups. We automatically consider the possible parity shifted representation with positive superdimension here. Note that the result for the first example $[3,2,1,0]$ assumes that $G_{\lambda} \cong SO(24)$ (a consequence of the conjectural structure theorem \ref{Tannakagroup-conj}).
\begin{enumerate}

\item $[3,2,1,0]$, $\sdim = 24$, \ $H_{\lambda} = SO(24)$ (conjecturally).
\item $[3,2,0,0]$, $\sdim = 12$, \ $H_{\lambda} = SU(12)$.
\item $[3,1,1,0]$, $\sdim = 12$, \ $H_{\lambda} = Sp^c(12)$.
\item $[3,1,0,0]$, $\sdim = 8$, \ $H_{\lambda} = SU(8)$.
\item $[3,0,0,0]$, $\sdim = 4$, \ $H_{\lambda} = SU(4)$.
\item $[2,2,1,0]$, $\sdim = 12$, \ $H_{\lambda} = SU(12)$.
\item $[2,2,0,0]$, $\sdim = 6$, \ $H_{\lambda} = SO(6)$.
\item $[2,1,1,0]$, $\sdim = 8$, \ $H_{\lambda} = SU(8)$.
\item $[2,1,0,0]$, $\sdim = 6$, \ $H_{\lambda} = Sp^c(6)$.
\item $[2,0,0,0]$, $\sdim = 3$, \ $H_{\lambda} = SU(3)$.
\item $[1,1,1,0]$, $\sdim = 4$, \ $H_{\lambda} = SU(4)$.
\item $[1,1,0,0]$, $\sdim = 3$, \ $H_{\lambda} = SU(3)$.
\item $[1,0,0,0]$, $\sdim = 2$, \ $H_{\lambda} = SU(2)$.
\end{enumerate}
In addition there is the normalised Berezin representation $B$,
with $ [1,1,1,1]$ and  $\sdim = 1$ and in the notation above $$  H_{\lambda} = U(1)\ .$$ 
Twisting any of the basic representations above with a nontrivial Berezin power gives the $GL$, $GSO$ or $GSp$ versions. For $n=4$ the appearing groups exhaust all possible Tannaka groups arising from an $L(\lambda)$.

\medskip
Theorem \ref{mainthm} implies the following branching rules (the lower index indicates
the superdimensions up to a sign):
\begin{enumerate}
\item $DS([3,2,1,0]_{24}) \cong \ [3,2,1]_6 \oplus [1,0,-1]_6 \oplus [3,0,-1]_6 \oplus [3,2,-1]_6$ 
\item $DS([3,2,0,0]_{12}) \cong \ [3,2,0]_6 \oplus [1,-1,-1 ]_3 \oplus [3,-1,-1]_3$
\item $DS([3,1,1,0]_{12}) \cong \ [3,1,1]_3 \oplus [3,1,-1]_6 \oplus [0,0,-1]_3$
\item $DS([3,1,0,0]_8) \ \cong \ [3,1,0]_6 \oplus [0,-1,-1]_2$
\item $DS([3,0,0,0]_4) \ \cong \ [3,0,0]_3 \oplus [-1,-1,-1]_1$ 
\item $DS([2,2,0,0]_6) \ \cong \ [2,2,0]_3 \oplus [2,-1,-1]_3$ 
\item $DS([2,1,0,0]_6) \ \cong \ [2,1,0]_6$ 
\item $DS([2,0,0,0]_3) \ \cong \ [2,0,0]_3$ 
\item $DS([1,0,0,0]_2) \ \cong\ [1,0,0]_2$
\item $DS([1,1,1,1]_1) \ \cong\ [1,1,1]_1$ 
\end{enumerate}
and $DS([n,0,0,0]_4) \cong [n,0,0]_3 \oplus [-1,-1,-1]_1$ for all $n\geq 4$.  We also have to consider the dual representations in the cases (2), (4), (5) and (8). We remark that even while most of the derivatives are not basic, they also give examples for $n=3$ of representations which satisfy the spin condition.
\end{example}

\begin{example} Consider $L(\lambda) =  [6,6,1,1]$. It is weakly selfdual with the dual representation $[1,1,-4,-4] = Ber^{-5} \otimes [6,6,1,1,]$. Its superdimension is 6. Since $\ell(\lambda) \neq 0$ and its basic weight $[2,2,0,0]$ carries an even pairing, the associated Tannakagroup is therefore $H_{\lambda} = GSO(V_{\lambda}) \simeq GSO(6)$. This does not depend on the conjecture $I \simeq \one$. Indeed $DS([6,6,1,1])$ does not contain an irreducible summand $L(\lambda_i)$  with $(\lambda_i)_{basic} = [2,1,0]$ and one can argue as in lemma \ref{lem-gl-4-4}. 
\end{example}


\begin{appendix}

\section{Equivalences and derivatives}\label{equivalences} 

Recall that two weights $\lambda, \ \mu$ are equivalent $\lambda \sim \mu$ if there exists $r\in \Z$ such that 
$L(\lambda) \cong Ber^r \otimes L(\mu)$ or $L(\lambda)^\vee \cong Ber^r \otimes L(\mu)$ holds. We denote the equivalence classes of maximal atypical weights by $Y_0^+(n)$. The embedding $H_{n-1} \to H_n$ induces an embedding $G_{n-1} \to G_{n}$. Since inductively $G_{n-1} = \prod_{\lambda \in Y_0^+(n)} G_{\lambda}$, we need to understand the equivalence classes of weights and their behaviour under $DS$.

\subsection{Plots} We use the notion of plots from \cite[Section 13]{Heidersdorf-Weissauer-tensor} to describe weight diagrams and their sectors. A plot $\lambda$ is a map
$$ \lambda: \mathbb Z \to \{\boxplus,\boxminus\}\ $$ such that  
the cardinality $r$ of the fiber $\lambda^{-1}(\boxplus)$ is finite. Then by definition $r=r(\lambda)$ is the degree and $\lambda^{-1}(\boxplus)$ is the support of $\lambda$. The fiber $\lambda^{-1}(\boxplus)$ corresponds to those vertices of the weight diagram which are labeled by a $\vee$.
An interval $I=[a,b]$ of even cardinality $2r$ and a subset $K$ of cardinality
of rank $r$ defines a plot $\lambda$ of rank $r$ with support $K$. We consider formal finite linear combinations $\sum_i n_i\cdot \lambda_i$ of plots with integer coefficients. This defines an abelian group $R = \bigoplus_{r=0}^\infty R_r$
(gradn uu7887u87u.  ation by rank $r$). In \cite{Heidersdorf-Weissauer-tensor} we defined a derivation on $R$ called derivative. Any plot can be written as a product of prime plots and we use the formula $ \partial (\prod_i \lambda_i) = \sum_i \partial \lambda_i \cdot \prod_{j\neq i} \lambda_j$ to reduce the definition to the case of a prime plot $\lambda$. For prime $\lambda$ let $(I,K)$ be its associated sector. Then $I=[a,b]$.  Then for prime plots $\lambda$ of rank $n$ with sector $(I,K)$ we define $\partial \lambda$ in $R$ by $ \partial \lambda =  \partial(I,K), \ I=[a,b]$ with $\partial(I,K) = (I,K)' = (I',K')$ for $I'=[a+1,b-1]$ and $K'=I'\cap K$. The importance of $\partial$ is that it describes the effect of $DS$ on irreducible representations according to theorem \ref{mainthm}: If $L(\lambda)$ has sector structure $S_1 \ldots S_k$, $L(\lambda_i)$ has sector structure $S_1 \ldots \partial S_i \ldots S_k$. For a segment $(I,K)$ with $I=[a,b]$ put
$$ \int(I,K) =  ([a-1,b+1],K\cup\{a-1\}) \ $$
increasing the rank by 1. We call this {\em integrating}.
Observe that 
$([a-1,b+1],K\cup\{a-1\})$ always defines a sector.

\subsection{Duality} \label{sec:dual} If $L=L(\lambda)$ is an irreducible maximal atypical representation in $\calR_n$, its weight $\lambda$ is uniquely determined by its plot.
Let $S_1...S_2... S_k$ denote the segments of this plot. Each segment $S_\nu$ has
even cadinality $2r(S_\nu)$, and can be identified up to a translation with a unique basic
weight of rank $r(S_\nu)$ and a partition in the sense of \cite[Lemma 21.4]{Heidersdorf-Weissauer-tensor}.
For the rest of this section we denote the segment of rank $r(S_\nu)$
attached to the dual partition by $S_\nu^*$, hoping that this will not be confused 
with the contravariant functor $*$. Using this notation, Tannaka duality
maps the plot $S_1 .. S_2 ... S_k$ to the plot $S_k^* ... S^*_2 .. S_1^*$
so that the distances $d_i$ between $S_i$ and $S_{i+1}$ coincide
with the distances between $S^*_{i+1}$ and $S^*_i$. 
 This follows from
\cite[proposition 21.5]{Heidersdorf-Weissauer-tensor} and determines
the Tannaka dual $L^\vee$ of $L$ up to a Berezin twist.

\medskip\noindent
If we identify the basic plots with rooted trees $S_i \leftrightarrow \calT_i$, we can describe a weight by a \textit{spaced} forest \[ \calF = (d_0, \calT_1, d_1, \calT_2, \ldots, d_{k-1}, \calT_k).\]

\begin{lem} \cite[Lemma 21.5]{Heidersdorf-Weissauer-tensor} The weight of the dual representation corresponds to the spaced forest \[ \calF^{\vee} = (d_0^*, \calT_k^*, d_1^*, \calT_{k-1}^*, d_2^*, \ldots, d_{k-1}^*, \calT_1^*)\] where $d_i^* := d_{k-i}$ for $i=1,\ldots,k-1$ and $d_0^*  =   - d_0 - d_1 - \ldots - d_{k-1}$ and $\calT_i^{*}$ denotes the mirror image (along the root axis) of the planar tree $\calT_i$.
\end{lem}

 In the following we will use $S_1 ... S_i ... S_k$ to denote the sectors of $\lambda$ since the effect of $DS$ on $L(\lambda)$ can be described conveniently in this setup.
It is however important to note that the description of the Tannaka dual does not require the $S_i$ to be sectors (segments, i.e. unions of adjacent sectors, is enough). In particular if $\partial S_i$ is not a sector, the dual of $S_1 ... \partial S_i ... S_k$ is still $S_k^* ... (\partial S_i)^* .... S_1^*$ (up to a shift).

\begin{example} Consider the irreducible representation $[11,9,9,5,3,3,3]$ in $\mathcal{T}_7$. It can be either described by the spaced forest \[ \xymatrix@R-0.3cm@C-0.3cm{ d_0 = 3 & & & \bullet \ar@{-}[dl] \ar@{-}[dr] & & d_1 = 2 & \bullet \ar@{-}[d] & d_2 = 0 & \bullet \\ & & \bullet \ar@{-}[dl] & & \bullet & &  \bullet & & \\ & \bullet  & & & & & & & } \] or by the cup diagram

\begin{center}
\medskip
 
 \scalebox{0.7}{
\begin{tikzpicture}
\foreach \x in {11,8,7,2,-1,-2,-3} 
     \draw[very thick] (\x-.1, .1) -- (\x,-0.1) -- (\x +.1, .1);
\foreach \x in {12,10,9,6,5,4,3,1,0} 
     \draw[very thick] (\x-.1, -.1) -- (\x,0.1) -- (\x +.1, -.1);
%


\draw[very thick] [-,black,out=90, in=90](11,+0.2) to (12,+0.2);
\draw[very thick] [-,black,out=90, in=90](8,+0.2) to (9,+0.2);
\draw[very thick] [-,black,out=90, in=90](7,+0.2) to (10,+0.2);
\draw[very thick] [-,black,out=90, in=90](2,+0.2) to (3,+0.2);
\draw[very thick] [-,black,out=90, in=90](-1,+0.2) to (0,+0.2);
\draw[very thick] [-,black,out=90, in=90](-2,+0.2) to (1,+0.2);
\draw[very thick] [-,black,out=90, in=90](-3,+0.2) to (4,+0.2);


\end{tikzpicture} }
\smallskip

\end{center}

The dual is the representation $[1,1,0,0,-4,-4,-5]$ with spaced forest \[ \xymatrix@R-0.3cm@C-0.3cm{ d_0^* = -5 & \bullet & d_1^* = 0 & \bullet \ar@{-}[d] & d_2^* = 2 & & \bullet \ar@{-}[dl]  \ar@{-}[dr] & & \\  & & & \bullet & &  \bullet & & \bullet \ar@{-}[dr] &  \\  & & & & & & & & \bullet } \] and cup diagram

\begin{center}
\medskip
 
 \scalebox{0.7}{
\begin{tikzpicture}
\foreach \x in {-5,-3,-2,3,4,6,7} 
     \draw[very thick] (\x-.1, .1) -- (\x,-0.1) -- (\x +.1, .1);
\foreach \x in {-4,-1,0,1,2,5,8,9,10} 
     \draw[very thick] (\x-.1, -.1) -- (\x,0.1) -- (\x +.1, -.1);
%


\draw[very thick] [-,black,out=90, in=90](-5,+0.2) to (-4,+0.2);
\draw[very thick] [-,black,out=90, in=90](-3,+0.2) to (0,+0.2);
\draw[very thick] [-,black,out=90, in=90](-2,+0.2) to (-1,+0.2);
\draw[very thick] [-,black,out=90, in=90](3,+0.2) to (10,+0.2);
\draw[very thick] [-,black,out=90, in=90](4,+0.2) to (5,+0.2);
\draw[very thick] [-,black,out=90, in=90](6,+0.2) to (9,+0.2);
\draw[very thick] [-,black,out=90, in=90](7,+0.2) to (8,+0.2);


\end{tikzpicture} }
\smallskip

\text{The dual weight}
\end{center}

\end{example}

We will switch between the language of sectors, cup diagrams and forests as is convenient.

\subsection{Equivalent weights} Let $\lambda$ be a maximal atypical highest weight in $X^+(n)$ 
with the sectors $S_1.... S_k$. The constituents $\lambda_i$ (for $i=1,..,k$)
of the derivative have the sector-structure $S_1... \partial S_i ... S_k$. Recall that two irreducible representations $M,N$ in $\mathcal{T}_n$ are 
equivalent $M \sim N$, if either $M \cong Ber^r \otimes N$ or $M^\vee \cong Ber^r \otimes N$ holds for some $r\in \mathbb Z$.
Assume that $\lambda_i$ and $\lambda_j$ are equivalent for $i\neq j$. Then
$$S_1... (\partial S_i) ...  S_j ... S_k \ \sim \ S_1... S_i ...  (\partial S_j) ... S_k$$
define equivalent weights of $\mathcal{T}_{n-1}^+$. Passing from $L(\lambda)$ to $Ber^i \otimes L(\lambda)$ involves a shift of the vertices in the weight diagram by $i$. We refer to this as the translation case. Applying the duality functor $L(\lambda) \mapsto L(\lambda)^{\vee}$ is described in terms of the cup diagram as a kind of reflection. We refer to this as the reflection case.

\begin{lem}\label{ds-equivalence}
For a maximal atypical weight $\lambda$ assume that there exists
an equivalence $\lambda_i \sim \lambda_j$ for some $i\neq j$ between two constituents $\lambda_i, \lambda_j$
of the derivative of $\lambda$. Then $S_\nu \equiv S_{k+1-\nu}^*$ holds for all $\nu=1,..,k$ 
and $d_{k-\nu}=d_\nu$ 
holds for all $\nu=1,..,k$.
\end{lem}

{\it Proof}. 
{\it 1) Translation case}.  We first discuss whether this equivalence can be achieved by a translation
and show that this implies $$ \text{ $r(S_\nu)=1$ for all $\nu$, $i=1$ and $j=k$.} $$
To prove this we first exclude $1 < i,j$. Indeed then the starting sector is $S_1$ in both cases
and a translation equivalence different from the identity 
is impossible. Now assume $i=1$. Again an equivalence is not possible
unless $\partial S_1 = \emptyset$, since otherwise $S_1$ and $\partial S_1$ would be 
starting sectors of different cardinality and hence they can not be identified by a translation. 
So the only possibility could be $i=1$ and $r(S_1)=1$ (so that $\partial S_1 = \emptyset$). 
The equivalence of $S_2 .... S_k$ with $S_1 ... \partial S_j ... S_k$ then implies 
$\partial S_j=\emptyset$, since both plots must have the same number of sectors. But then
the only equivalence comes from a left shift by two. Hence it is not hard to see
that this implies $r(S_\nu)=1$ for all $\nu$, $i=1$ and $j=k$.  Furthermore $d_1=...=d_k$
must hold. But then we see that this translation equivalence is also induced by a reflection
equivalence. 

\medskip 
{\it 2) Reflection case}. Let us consider equivalences between $S_1... (\partial S_i) ...  S_j ... S_k$ and $S_1... S_i ...  (\partial S_j) ... S_k$
involving duality as in section \ref{sec:dual}.

\medskip
{\it The case $r(S_i)>1$}.  Notice that $r(S_i)>1$ 
is equivalent to $\partial S_i \neq \emptyset$. Furthermore notice that $r(S_i)>1$ implies  $r(S_j)>1$, since 
equivalent plots need to have the same number of sectors. 
To proceed let us temporarily ignore the distances between the 
different sectors $S_\nu$; we write $\equiv$ 
to indicate this. Then  for all $\nu\neq i,j,k+1-i,k+1-j$ we get
$$S_\nu^* \equiv S_{k+1-\nu}   $$ (equality up to a shift). 
Let us assume that $\partial S_i$ is one sector (which implies the same for $\partial S_j$). The easy case now is $j = k+1 -i$, where we get the further condition (*)
$$S_i^* \equiv S_{k+1-i} \text{ and hence } \partial S_i^* \equiv \partial S_{k+1-i}  \ .$$
We also then conclude 
$$  d_\nu = d_{k-\nu} \ \ \text{ for all } \ \nu=1,..,k \ .$$ If $\partial S_i$ consists of several sectors dualizing still yields for $j = k+1 -i$ that $\partial S_i \equiv \partial S_j^*$. This in turn implies S$_i \equiv S_j^*$. 

\medskip
We now show that the more complicated looking case $i \neq j$ and $i \neq k+1-j$,
where we also have $j\neq k+1 -i$, can not occur. We again assume that $\partial S_i$ and $\partial S_j$ consist of one sector. In this case \cite[proposition 20.1]{Heidersdorf-Weissauer-tensor}, implies,
from comparing 
$$  ...\quad  \partial S_i \quad ... \quad S_{k+1-j} \quad ... \quad S_j \quad ... \quad S_{k+1-i} \quad ... $$
and the reflection of 
$$  ...\quad  S_i \quad ... \quad S_{k+1-j} \quad ... \quad \partial S_j \quad ... \quad S_{k+1-i} \quad ... $$
the following assertions
$$\partial S_i \equiv S_{k+1-i}^* \quad , \quad \partial S_j \equiv S_{k+1-j}^* \ ,$$
$$ S_i \equiv S_{k+1-i}^* \quad , \quad S_j \equiv S_{k+1-j}^* \ .$$
However this is absurd, since it would imply $r(S_i) = r(S_{k+1-i}^*) = r(S_i)-1$. If $\partial S_i$ consists of $r > 1$ sectors, so does $\partial S_j$ (since we assume that the weights are equivalent). The same matching $S_1 \equiv S_k^*, ... $ as above of the $k-1$ other sectors forces again $\partial S_i \equiv \partial S_j^*$ and therefore $S_i \equiv S_j^*$.

\medskip
{\it So now $r(S_i)=1$}. Then $\partial S_i=\emptyset$ and hence
also $\partial S_j=\emptyset$ since the cardinality of sectors of equivalent plots coincide. First assume $j=k+1-i$. In the case of a reflection symmetry this implies
$$  S_\nu \equiv S_{k+1-\nu}^*   \ \ \text{ for all } \ \ \nu \neq i, k+1-i \ .$$
Furthermore it implies
$$  d_{k-\nu} = d_\nu  \quad , \quad \nu= 1,...,k  \ .$$
This follows by comparing
$$ S_1 \quad ... d ... \quad \partial S_i \quad d_i \quad
 S_{i+1} \quad .... \quad S_{k+1-i} \quad ... d ... \quad S_k $$
 with the reflection of
$$ S_1 \quad ... d ... \quad S_i \quad d_i \quad
 S_{i+1} \quad .... S_{k-i} \quad \partial S_{k+1-i} \quad ... d ... \quad S_k \ .$$
Then $d_1=d_{k+1-i},...,d_{i-1}=d_{k-i+1}$, by a comparison of the lower left side and the upper right side, and then also $d_i= d_{k-i}$
and so on till $d_{k-i-1}= d_{i+1}$, but then also $d_{k-i}+d= d_i+d$
for $d=d_1+...+d_{i-1}$. Hence we conclude that $d_\nu = d_{k-\nu}$ holds for all
$\nu=1,...,k$. Similarly we see $  S_\nu \equiv S_{k+1-\nu}^*$ 
for $\nu\neq i, k+1-i$. But taking into account $r(S_i)=r(S_{k+1-i})$ the assertion $  S_\nu \equiv S_{k+1-\nu}^*$
also holds for $\nu=i, k+1-i$. 

\medskip
Finally we want to show that we have now covered all case. 
This means that again for $r(S_i)=1$ the case $j\neq k+1-j$ is impossible.  
To show this we can assume
$min(i,k+1-i) < min(j,k+1-j)$ by reverting the role of $i$ and $j$ and we can then assume $i < k+1-i$ by left-right reflection.  Then we have to compare the reflection of
$$  S_1 \quad ... d  ... \quad \partial S_i S_{i+1}\quad ... \quad S_{k+1-j} \quad ... \quad S_j \quad ... \quad S_{k+1-i} \quad ... d .... \quad S_k $$
with
$$  S_1 \quad... d ... \quad S_i S_{i+1} \quad ... \quad S_{k+1-j} \quad ... \quad \partial S_j \quad ... \quad S_{k+1-i} \quad ... d ... \quad S_k \ .$$
We claim that an equivalence is not possible by a reflection!
(We could easily reduce to the case where $i=1$ by the way).
In fact, by comparing the left side of the second plot with the right side of the first plot, then $S_i \equiv S_{k+1-i}^*$ and the distance $d=d_1+...+d_{i-1}$ between
$S_1$ and $S_i$ must be the same as the distance $d_{k+1-i}+... +d_{k-1}$ 
between $S_{k+1-i}$ and $S_k$.
However, by comparing the left side of the first plot with the right side of the second plot, then $S_{i+1} \equiv S_{k+1-i}^*$ and the distance $d+2+d_i$ between
$S_1$ and $S_{i+1}$ must be the same as the distance $d$ between $S_{k+1-i}$ and $S_k$.
In fact this follows from the fact $\partial S_i =\emptyset$ and $\# S_i=2r(S_i)=2$.
This implies $2+d_i=0$. A contradiction! \qed

\medskip
From lemma \ref{ds-equivalence} we easily get 

\begin{prop}\label{duality}
Suppose for the $k$ irreducible constituents $L(\lambda_i)$ of $DS(\lambda)$  there are two different integers $i,j\in \{1,...,k\}$ such that
$\lambda_i \sim \lambda_j$. 
Then there exists an integer $r$ such that $L(\lambda)^\vee \cong Ber^r \otimes L(\lambda)$ holds. If conversely $L(\lambda)$ is weakly selfdual with sector structure $S_1 ... S_i ... S_k$, then $\lambda_i \sim \lambda_{k+1-i}$ for all $i$.
\end{prop}

\medskip
{\it Proof}. By the last lemma we conclude $S_\nu \equiv S_{k+1-\nu}^*$  and $d_{k-\nu}=d_\nu$ 
for all sectors $S_\nu, \nu=1,..,k$ of $\lambda$. By proposition \cite[Proposition 20.1]{Heidersdorf-Weissauer-tensor} or section \ref{sec:dual} this implies
$L(\lambda)^\vee \cong Ber^r \otimes L(\lambda)$ for some integer $r$. The converse statement is obvious from the description of the dual and the $DS$ rule.\qed 

\medskip
Another conclusion of the considerations above is

\begin{lem}
For fixed $i$ between $1$ and $k$ the plot 
$S_1... \partial S_i ...   S_k$ can only be equivalent to at most
one of the plot $S_1... S_i ...  (\partial S_j) ... S_k$ for $j\neq i$.
\end{lem} 

\begin{cor} \label{size-equiv} Every equivalence class of the constituents $\lambda_i$ of the derivative of $\lambda$ can contain at most $s=2$ representatives.
\end{cor}

%

\subsection{Selfdual derivatives} 

We discuss the question under what circumstances a weight $\lambda$ can have weakly selfdual derivatives. Before the characterize these, we discuss the special case of {\em ladder types} first.

\begin{definition} We call $\lambda$ to be of ladder type if $\lambda$ is an equidimensional union of sectors of minimal length (i.e. $r_i = 2$ for all $i=1,\ldots,k$ and all distances $d_1,\ldots,d_{k-1}$ are the same).
\end{definition}

\begin{example} Here is a weight of ladder type along with its derivatives.

\begin{center}
\medskip
 
 \scalebox{0.7}{
\begin{tikzpicture}
\foreach \x in {-6,-4,-2,0,2} 
     \draw[very thick] (\x-.1, .1) -- (\x,-0.1) -- (\x +.1, .1);
\foreach \x in {-5,-3,-1,1,3} 
     \draw[very thick] (\x-.1, -.1) -- (\x,0.1) -- (\x +.1, -.1);
%


\draw[very thick] [-,black,out=90, in=90](-6,+0.2) to (-5,+0.2);
\draw[very thick] [-,black,out=90, in=90](-4,+0.2) to (-3,+0.2);
\draw[very thick] [-,black,out=90, in=90](-2,+0.2) to (-1,+0.2);
\draw[very thick] [-,black,out=90, in=90](0,+0.2) to (1,+0.2);
\draw[very thick] [-,black,out=90, in=90](2,+0.2) to (3,+0.2);


\end{tikzpicture} }
\smallskip

\text{A ladder type weight}
\end{center}

\begin{center}
\medskip
 
 \scalebox{0.7}{
\begin{tikzpicture}
\foreach \x in {-4,-2,0,2} 
     \draw[very thick] (\x-.1, .1) -- (\x,-0.1) -- (\x +.1, .1);
\foreach \x in {-6,-5,-3,-1,1,3} 
     \draw[very thick] (\x-.1, -.1) -- (\x,0.1) -- (\x +.1, -.1);
%


\draw[very thick] [-,black,out=90, in=90](-4,+0.2) to (-3,+0.2);
\draw[very thick] [-,black,out=90, in=90](-2,+0.2) to (-1,+0.2);
\draw[very thick] [-,black,out=90, in=90](0,+0.2) to (1,+0.2);
\draw[very thick] [-,black,out=90, in=90](2,+0.2) to (3,+0.2);


\end{tikzpicture} }
\smallskip

\text{The derivative $\lambda_1$ (SD)}
\end{center}

\begin{center}
\medskip
 
 \scalebox{0.7}{
\begin{tikzpicture}
\foreach \x in {-6,-4,-2,0} 
     \draw[very thick] (\x-.1, .1) -- (\x,-0.1) -- (\x +.1, .1);
\foreach \x in {-5,-3,-1,1,2,3} 
     \draw[very thick] (\x-.1, -.1) -- (\x,0.1) -- (\x +.1, -.1);
%


\draw[very thick] [-,black,out=90, in=90](-6,+0.2) to (-5,+0.2);
\draw[very thick] [-,black,out=90, in=90](-4,+0.2) to (-3,+0.2);
\draw[very thick] [-,black,out=90, in=90](-2,+0.2) to (-1,+0.2);
\draw[very thick] [-,black,out=90, in=90](0,+0.2) to (1,+0.2);


\end{tikzpicture} }
\smallskip

\text{The derivative $\lambda_5$ (SD)}
\end{center}

\begin{center}
\medskip
 
 \scalebox{0.7}{
\begin{tikzpicture}
\foreach \x in {-6,-2,0,2} 
     \draw[very thick] (\x-.1, .1) -- (\x,-0.1) -- (\x +.1, .1);
\foreach \x in {-5,-4,-3,-1,1,3} 
     \draw[very thick] (\x-.1, -.1) -- (\x,0.1) -- (\x +.1, -.1);
%


\draw[very thick] [-,black,out=90, in=90](-6,+0.2) to (-5,+0.2);
\draw[very thick] [-,black,out=90, in=90](-2,+0.2) to (-1,+0.2);
\draw[very thick] [-,black,out=90, in=90](0,+0.2) to (1,+0.2);
\draw[very thick] [-,black,out=90, in=90](2,+0.2) to (3,+0.2);


\end{tikzpicture} }
\smallskip

\text{The derivative $\lambda_2$ (NSD)}
\end{center}

\begin{center}
\medskip
 
 \scalebox{0.7}{
\begin{tikzpicture}
\foreach \x in {-6,-4,-2,0} 
     \draw[very thick] (\x-.1, .1) -- (\x,-0.1) -- (\x +.1, .1);
\foreach \x in {-5,-3,-1,1,2,3} 
     \draw[very thick] (\x-.1, -.1) -- (\x,0.1) -- (\x +.1, -.1);
%


\draw[very thick] [-,black,out=90, in=90](-6,+0.2) to (-5,+0.2);
\draw[very thick] [-,black,out=90, in=90](-4,+0.2) to (-3,+0.2);
\draw[very thick] [-,black,out=90, in=90](-2,+0.2) to (-1,+0.2);
\draw[very thick] [-,black,out=90, in=90](2,+0.2) to (3,+0.2);


\end{tikzpicture} }
\smallskip

\text{The derivative $\lambda_4$}
\end{center}

\begin{center}
\medskip
 
 \scalebox{0.7}{
\begin{tikzpicture}
\foreach \x in {-6,-4,0,2} 
     \draw[very thick] (\x-.1, .1) -- (\x,-0.1) -- (\x +.1, .1);
\foreach \x in {-5,-3,-2,-1,1,3} 
     \draw[very thick] (\x-.1, -.1) -- (\x,0.1) -- (\x +.1, -.1);
%


\draw[very thick] [-,black,out=90, in=90](-6,+0.2) to (-5,+0.2);
\draw[very thick] [-,black,out=90, in=90](-4,+0.2) to (-3,+0.2);
\draw[very thick] [-,black,out=90, in=90](0,+0.2) to (1,+0.2);
\draw[very thick] [-,black,out=90, in=90](2,+0.2) to (3,+0.2);


\end{tikzpicture} }
\smallskip

\text{The derivative $\lambda_3$}
\end{center}

Here $\lambda_1$ and $\lambda_5$ are weakly selfdual (in fact one is a $Ber^{2}$-shift from the other) and $\lambda_2$ and $\lambda_4$ are dual to each other. The derivative $\lambda_3$ is weakly selfdual. Hence we have three derivatives of (SD)-type.

\end{example}

Obviously we have the following general observation about ladder type representations with at least two sectors:

\begin{enumerate} 
\item If $k=2m+1$ is odd, then $\lambda$ has three weakly selfdual derivatives: $\lambda_1$, $\lambda_k$ and $\lambda_{m+1}$.
\item If $k=2m$ is even, $\lambda$ has two weakly selfdual derivatives: $\lambda_1$ and $\lambda_k$.
\end{enumerate}

\begin{lem} \label{selfdual-derivative}
Suppose the maximal atypical weight $\lambda$ has a weakly selfdual
derivative $\lambda_i$ for some $i=1,...,k$. Then $\lambda_i$ is the unique weakly selfdual derivative except in the case where $\lambda$ is weakly selfdual and has equidistant
sectors all of cardinality two or $n=3$.
\end{lem}

\begin{example} We show three examples of such weights. They correspond to the cases 1), 2) and 3) in the proof of the lemma.

\begin{center}
\medskip
 
 \scalebox{0.7}{
\begin{tikzpicture}
\foreach \x in {-6,-5,-2,0,2,4,5} 
     \draw[very thick] (\x-.1, .1) -- (\x,-0.1) -- (\x +.1, .1);
\foreach \x in {-4,-3,-1,1,3,6,7} 
     \draw[very thick] (\x-.1, -.1) -- (\x,0.1) -- (\x +.1, -.1);
%


\draw[very thick] [-,black,out=90, in=90](-6,+0.2) to (-3,+0.2);
\draw[very thick] [-,black,out=90, in=90](-5,+0.2) to (-4,+0.2);
\draw[very thick] [-,black,out=90, in=90](-2,+0.2) to (-1,+0.2);
\draw[very thick] [-,black,out=90, in=90](0,+0.2) to (1,+0.2);
\draw[very thick] [-,black,out=90, in=90](2,+0.2) to (3,+0.2);
\draw[very thick] [-,black,out=90, in=90](4,+0.2) to (7,+0.2);
\draw[very thick] [-,black,out=90, in=90](5,+0.2) to (6,+0.2);


\end{tikzpicture} }
\smallskip

\text{A weight of type (1) with a selfdual derivative}
\end{center}

\begin{center}
\medskip
 
 \scalebox{0.7}{
\begin{tikzpicture}
\foreach \x in {-7,-5,-4,-1,0} 
     \draw[very thick] (\x-.1, .1) -- (\x,-0.1) -- (\x +.1, .1);
\foreach \x in {-6,-3,-2,1,2,3} 
     \draw[very thick] (\x-.1, -.1) -- (\x,0.1) -- (\x +.1, -.1);
%


\draw[very thick] [-,black,out=90, in=90](-7,+0.2) to (-6,+0.2);
\draw[very thick] [-,black,out=90, in=90](-5,+0.2) to (-2,+0.2);
\draw[very thick] [-,black,out=90, in=90](-4,+0.2) to (-3,+0.2);
\draw[very thick] [-,black,out=90, in=90](-1,+0.2) to (2,+0.2);
\draw[very thick] [-,black,out=90, in=90](0,+0.2) to (1,+0.2);


\end{tikzpicture} }
\smallskip

\text{A weight of type (2) with a selfdual derivative}
\end{center}

\begin{center}
\medskip
 
 \scalebox{0.7}{
\begin{tikzpicture}
\foreach \x in {-6,-5,-3,0,1,2,4} 
     \draw[very thick] (\x-.1, .1) -- (\x,-0.1) -- (\x +.1, .1);
\foreach \x in {-4,-2,-1,3,5,7,6} 
     \draw[very thick] (\x-.1, -.1) -- (\x,0.1) -- (\x +.1, -.1);
%


\draw[very thick] [-,black,out=90, in=90](-6,+0.2) to (-1,+0.2);
\draw[very thick] [-,black,out=90, in=90](-5,+0.2) to (-4,+0.2);
\draw[very thick] [-,black,out=90, in=90](-3,+0.2) to (-2,+0.2);
\draw[very thick] [-,black,out=90, in=90](0,+0.2) to (7,+0.2);
\draw[very thick] [-,black,out=90, in=90](1,+0.2) to (6,+0.2);
\draw[very thick] [-,black,out=90, in=90](2,+0.2) to (3,+0.2);
\draw[very thick] [-,black,out=90, in=90](4,+0.2) to (5,+0.2);


\end{tikzpicture} }
\smallskip

\text{A weight of type (3) with a selfdual derivative}
\end{center}

\end{example}

\begin{proof}
Suppose $\lambda$ is a maximal atypical weight such that one of its derivatives
$\lambda_i$ is weakly selfdual. Let $S_1,...,S_k$ denote the sectors of $\lambda$.
Then $\lambda$ belongs to one of the following cases mutually exclusive cases:
\begin{enumerate}
\item \label{deri-1} $k=2m+1$ is odd and $S^1...\partial S_{m+1} ...., S_k$ is weakly selfdual.
\item  \label{deri-2} $S_1...\partial S_{\nu} ...., S_k$ is weakly selfdual such that $\partial S_\nu =\emptyset$
and not of type (1).
\item \label{deri-3} $S_1...\partial S_{\nu} .... S_k$ is weakly selfdual and we are in neither of the two cases above.
\end{enumerate}

In the first case either $\lambda$ is of ladder type (in which case it is weakly selfdual) or $\lambda_{m+1}=S^1,...,\partial S_{m+1} ...., S_k$ 
is the unique selfdual derivative of $\lambda$. This immediately follows from lemma \ref{ds-equivalence}. Furthermore, if $S^1,...,\partial S_{m+1} ...., S_k$ is weakly selfdual,
$S^1,...,S_{m+1} ...., S_k$ is weakly selfdual in this first case.

\medskip\noindent
In the second case we change notation and we can suppose
that  \[ \lambda = S_1,S_2,.....,S_{\nu-1},[a,a+1], S_{\nu +1},......,S_2^*,S_1^* \]
where we also allow the the sector $[a,a+1]$ to be at the left or right.  If $n=3$ with $k=3$ sectors, $\lambda_1$ and $\lambda_3$ are always weakly selfdual (but $\lambda$ is in general not weakly selfdual).

\begin{center}
\medskip
 
 \scalebox{0.7}{
\begin{tikzpicture}
\foreach \x in {-5,1,3} 
     \draw[very thick] (\x-.1, .1) -- (\x,-0.1) -- (\x +.1, .1);
\foreach \x in {-6,-4,-3,-2,-1,0,2,4,5} 
     \draw[very thick] (\x-.1, -.1) -- (\x,0.1) -- (\x +.1, -.1);
%


\draw[very thick] [-,black,out=90, in=90](-5,+0.2) to (-4,+0.2);
\draw[very thick] [-,black,out=90, in=90](1,+0.2) to (2,+0.2);
\draw[very thick] [-,black,out=90, in=90](3,+0.2) to (4,+0.2);


\end{tikzpicture} }
\smallskip

\text{The derivatives $\lambda_1$ and $\lambda_2$ are weakly selfdual}
\end{center}

Let us ignore this case now. We show that $\lambda$ is of ladder type if there are two weakly selfdual derivatives in case (2). It is immediately clear
that there does not exist a weakly selfdual derivative $\lambda_j$ different from $\lambda_i$ except if $[a,a+1]$ is the rightmost or leftmost sector of $\lambda$ (since $\lambda$ is not of type (1)).  Without restriction of generality we may 
assume it is the leftmost sector of $\lambda$, i.e.
$\lambda = S_0,S_1,S_2,...,S_2^*,S_1^*$ for $S_0=[a,a+1]$ holds, i.e. $S_\nu^* = S_{k+1-\nu}$.
If $\lambda_j$ is obtained by $S_j \mapsto \partial S_j$, we distinguish two cases:
The first is where $j=\min(j,k+1-j) $ and the second is where $j=\max(j,k+1-j) $.
In the first case we obtain $S_0^*=S_k= S_1^*, S_1^* = S_{k-1}= S_2^*,..., S_{j-1}^* = S_j^*$
and it is immediately clear that $\partial S_j=\emptyset$. It is then clear, that there is a conflict with the symmetry of distances at $j$ unless $j=k$. In this case $j=\max(j,k+1-j)$. So let us turn to this case
now, where it follows in a similar way that the only possible case is $j=k$. Then we can easily
show that all distances of $\lambda$ are the same and all sectors have length two, hence $\lambda$ is of ladder type and again $\lambda$ is weakly selfdual.

\medskip\noindent
In the third case  by lemma \ref{ds-equivalence} and its proof any translation equivalence between two derivatives implies $r(S_{\nu}) = 1$ for all $\nu$, hence $\partial S_{\nu} = \emptyset$. Therefore $\lambda_i$ is the unique weakly selfdual derivative of $\lambda$.
\end{proof}

\begin{cor}
Suppose a maximal atypical weight $\lambda$ that is not weakly selfdual admits a weakly selfdual derivative $\lambda_i$ for some $i=1,...,k$. 
Then $\lambda_i$ is unique with this property and we are in case \ref{deri-3} above.
\end{cor}

\subsection{Multiplicities} \label{sec:mult} We reformulate some of the previous results in the language of Section \ref{proof-derived}. The same group $G_{\mu}$ can appear for different derivatives $\mu_1,\ldots,\mu_k$ ($G_{\mu}$ is also called the type of $\mu$). Different derivatives lead to the same $G_{\mu}$ if they are equivalent. By Corollary \ref{size-equiv} the number of derivatives that lead to the same $\mu$ can be at most 2. We write this number as $m(G_{\mu})$. Assuming $n \geq 4$, the results of this appendix can be summarized as follows. If $m(G_{\mu}) = 2$, then either

\begin{itemize}

\item  $G_{\mu} = G_{\nu} = SL(W)$ and $W_{\nu} = W_{\mu}^{\vee}$ as a representation of $H_{\mu} = H_{\nu}$ 
\item or $\lambda$ is of ladder type and $\{\nu,\mu\} = \{1,k\}$ such that $G_{\mu} = G_{\nu}$ is of $(SD)$-type, either symplectic, orthogonal (this leads to case (2)(iii) in Lemma \ref{meager-lemma}) or exceptional (this leads to case (2)(ii) in Lemma \ref{meager-lemma}).
\end{itemize}

\begin{cor} If $m(G_{\mu}) \neq 1$, then the representation $W_{\mu}|_{G_{\mu}} = W_{\mu} \oplus W_{\mu}^{\vee} \ncong 2 W_{\mu}$ unless $\dim(W_{\mu}) = 2$. This can only happen for $n=3$.
\end{cor}

\section{Equivalences and separated weights }

\subsection{Equivalent and separated weights}

Suppose that $\lambda, \tilde \lambda$ are maximal atypical weights in $X^+_0$ and $X = L(\lambda)$ and $Y = L(\tilde{\lambda})$. Recall the equivalence relation $\sim$ on classes of such representations
and the relation $\equiv$ (meaning \lq{up to translation}\rq).

\begin{definition}  We say that $X$ is not separated from $Y$ if  
\begin{enumerate}
\item $X$ is not equivalent to Y
\item For all derivatives $\partial_i X$ of $X$ there exists an equivalent derivative $\partial_j Y$
of $Y$  (where $j$ depends on $i$).
\end{enumerate}
\end{definition}

Conversely we say for inequivalent weights $X$ and $Y$ that $X$ is separated from $Y$
if the second condition fails. This definition only depends on the equivalence class of $X$ resp. $Y$, but
corollary \ref{asymmetry} shows
 that this is an asymmetric relation on the set of equivalence classes.

\begin{thm} \label{resume} If $X$ is not separated from $Y$, then $X$ and $Y$ have at most two sectors and the possible cases are described in section \ref{example:separated}.
\end{thm}

\begin{cor} \label{asymmetry} Suppose $X$ and $Y$ are inequivalent of rank $>2$.
If $X$ is not separated from $Y$, then $Y$ is separated from $X$.
\end{cor}

\begin{proof} This can be checked easily for all the examples in section \ref{example:separated}.
\end{proof}

The corollary is false for  rank 2  since then $S^i$ is not separated from $S^j$ for all $j\neq i$. 

\subsection{Anchoring} 
If $X$ is not  separated from $Y$, we can replace $X$ and $Y$ by other representatives in their equivalence class such that \[ \partial_i X \equiv \partial_j Y \text{ for some pair } (i,j).\] Note $ i \neq j$ holds since $X$ and $Y$ are inequivalent. By duality we may therefore suppose $i < j$. \textit{Hence we make this assumption once and for all}.

\subsubsection{Sector structure} Let $\sigma_i$ denote the segment of $\partial_i X$ obtained as the derivative of  $S_i$; here we have to allow that $\sigma_i$ is an \lq{empty segment}\rq\ , i.e. an empty interval of length 2 if $S_i$ is of rank 1. Similarly denote the derivative of the $j$-th sector of $Y$ by $\sigma_j$. If $S_k$ denotes the last sector of $Y$, then the above assertion $\partial_i X \equiv \partial_j Y$ for $i<j$ defines 
a chain of sectors and segments $(*)$

\begin{align*} X \ & = \ \ S_1 \ d_1 \ S_2 \ \ldots  d_{i-1}\!-\! 1 \ S_i \ d_{i}-1 \ldots \ d_{j-1} \  \ \sigma_j \ \ \ d_j \ \ \ldots \ S_k \\ Y \ & =  \ \ S_1 \ d_1 \ S_2 \ \ldots \ \ d_{i-1} \ \ \ \sigma_i \ \ \ d_i \ldots \ d_{j-1}\!-\! 1 \ S_j  \ d_{j}-1 \ \ \ldots\ S_k \end{align*} 
where $S_\nu$ for $\nu\neq j$ are sectors of $X$ and $S_i, S_j$ are sectors such that
\[ S_i = \int \sigma_i \ \ (\text{for X}) \ \  \text{and} \ \ S_j = \int \sigma_j \ \ (\text{for Y}).\]
Obviously the $i$-th sector of $Y$ is the first sector of the segment $\sigma_i$, and the $j$-th sector of $X$ is the first sector of the segment $\sigma_j$.  

\subsubsection{Inductive and non-inductive cases}
It is helpful to distinguish the following cases: 
The \textit{non-inductive cases}
\begin{itemize}  
\item where $i=1$ is and $j=k$.
\item where $i=1$,  but $1<j< k$.
\end{itemize}
and the \textit{inductive cases},
\begin{itemize}
\item where $1<i<j< k$. 
\end{itemize}

\subsubsection{Overview of the proof} 
The crucial lemma \ref{dual-cond} shows $\partial_{s} X \equiv \, \partial_t (Y^{\vee})$ for all $s \neq i$ and $t = t(s)$. This gives too many conditions to hold as soon as $X$ has more than three sectors, thereby proving theorem \ref{asymmetry}. The case of at most 2 sectors is discussed in section \ref{example:separated}. If $X$ has at least three sectors, we partition $X$ into three intervals $L \ M \ R$ (depending on whether we are in the inductive or non-inductive case) and check what happens if condition (**) holds with respect to $\partial_s$ corresponding to a sector in $L$, $M$ or $R$. Case distinctions are necessary depending on whether the derivative of the last sector is $\emptyset$ or not. 

\subsection{Condition $(**)$}

Suppose that $X$ has $r$ sectors and $Y$ $r'$ sectors. If $X$ is not separated from $Y$, then 
for all $s\in \{1,...,r\}\setminus \{i \}$ either $\partial_{s} X \equiv \partial_m Y$, or  $\partial_{s}X \equiv \partial_{t} Y^{\vee}$ (i.e. equal up to a translation) holds for some integers $m=m(s)$ resp. $t = t(s)$
in $\{1,...,r'\}$.

\begin{lem} \label{dual-cond} If $X$ is not separated from $Y$, then for all $s\in \{1,...,r\}\setminus \{i \}$ there exists an integer $t\in \{1,...,r'\}$ such that  \[ (**)\quad \quad \partial_{s} X \equiv \, \partial_t (Y^{\vee}).\]
\end{lem}

\begin{proof} \textit{First case}. $\partial_{s} X \equiv \partial_m Y$ cannot hold for $s > i$. Indeed, by $(*)$ and $i<j$ the $i$-th sector of $X$ is larger than the $i$-th sector of $Y$ (since the latter is the left sector of $\sigma_i$). 

\smallskip
\textit{Second case}. If $s < i$, then by $(*)$ (the sectors of $X$ and $Y$ match until position $i$) the partner $m$ of $s$ has to be equal to $s$ except possibly for the cases where either (i) $s = 1$ and $|S_1| = 1$ or (ii) $m=1$ and $|S_1| = 1$. In case $s = m$,  since $\sigma_i$ is smaller than $S_i = \int \sigma_i$, we can argue as in the case $s > i$.

\smallskip
\textit{Case 2(i)}. Suppose $s=1$ and $m \neq 1$.  Since $X$ and $Y$ agree up to the $i$-th sector in case $m < i$ we get a contradiction by $(*)$. If $m \geq i$, then $\partial S_1 = \emptyset$ and $(*)$ forces $S_1 = S_2 = \ldots = S_{i-1} = S_i$. Since $|S_1| = 1$, this implies $|S_i| = 1$. But the cardinality of $S_i$ is larger than 1 since $S_i = \int \sigma_i$, a contradiction, unless $\sigma_i = \emptyset$, i.e. $|S_i| = 1$. Condition (*)  implies that $S_{i+1}$ equals the $i$-th sector of $Y$.

Suppose first that $m=i$. But then the derivative of the $i$-th sector of $Y$ is equal to $S_{i+1} \cap S_{i+2} \cap \ldots \cap S_{i+l}$ for some $l \geq 1$, a contradiction. Similarly, if $i < m < j$, then $S_{i + l}$ equals the $(i+l-1)$-th sector of $Y$, but $\partial_1 X \equiv \partial_m Y$ implies that the derivative of the $m$-th sector of $Y$ is $S_{m+1} \cap \text{other sectors}$, a contradiction. If $m > j$, then in order for $\partial_1 X \equiv \partial_m Y$ and $\partial_i X \equiv \partial_j Y$ to hold, the distances between $S_i$ and $S_{i+1}$ and the $(i-1)$-th and $i$-th sector in Y would have to be different in both cases.

\smallskip
\textit{Case 2(ii)}. Suppose $m=1$ and $s > 1$. Since we also have $s < i$, we can argue similarly as for 2(i) and conclude that $S_i$ is equal to the first sector of $\sigma_i$, a contradiction as seen before.  
\end{proof}

\textit{Notation}. Since $\sigma_i$ as a segment may consist of several sectors, from now on we are changing the notation: Let denote $X_1,....,X_r$ the sectors of $X$ and $Y_1,...,Y_{r'}$ the sectors of $Y$, in consequtive ordering from left to right.

\begin{lem}\label{lastl} Suppose the number of sectors $X_\nu$ of $X$ is $r\geq 2$.
Suppose $\partial_s X \equiv
\partial_u X$ holds for two different sectors $X_s$ and $X_u$ of $X$.
Then $X$ is of ladder type and $s=1$ and $u=r$,
or in reversed order $s=r$ and $u=1$.
\end{lem}

\begin{proof} This is part 1 (translation case) of the proof  of lemma \ref{ds-equivalence}. 
\end{proof}

\subsection{Non-inductive cases}
All cases, where $X$ has 1 or 2 sectors will be determined in section \ref{example:separated}.
Hence we suppose that $X$ has at least 3 sectors.   
Instead of the more complicated formula $(*)$ we now use short symbolic diagrams. 

\medskip
We first
consider the \textit{non-inductive} case (where $i=1$). Put $L=\sigma_1$ in both non-inductive cases
and $R= \sigma_j$ respectively $R=S_j$ depending on whether $j=k$
or $j<k$.  
$X$ can be partitioned into three intervals.  
Depending on whether $i=1$ and $j=k$ (\textit{left diagram}) or $j< k$ (\textit{right diagram}) the decomposition 
$(*)$ will be written symbolically
 \[ \begin{pmatrix} X \\ Y \end{pmatrix} = \begin{pmatrix} \int\! L & M & R \\ L & M & \int\! R \end{pmatrix} \ \text{ or } \ \begin{pmatrix} \int\! L & M & R \\  L & \int\! M & R \end{pmatrix}, \] 
 Then $S_1=\int L$ is the first sector of $X$. In the right diagram $R$ is the last sector of $X$. In the left diagram $\int R$ is the last sector of $Y$. Notice
\[ \begin{pmatrix} X \\ Y^{\vee} \end{pmatrix} = \begin{pmatrix} \int\! L & M & R \\ \int\! R^{\vee} & M^{\vee} & L^{\vee} \end{pmatrix} \ \text{ or } \ \begin{pmatrix} \int\! L & M & R \\ R^{\vee} & (\int\! M)^{\vee} & L^{\vee} \end{pmatrix}, \] 
For an interval $I$ we write $I^\vee$
for the interval (of the same length) that is obtained by the reflection at its middle point. Note that this notation is in conflict with the previously used $S^*$ for a sector $S$. It is convenient in this section to treat if on equal footing with $X^{\vee}$ (the dual of the irreducible representation $X$) and we hope that this notation doesn't lead to any confusion.

\begin{prop}\label{prop-nondeg}
Suppose $X$ has $r\geq 3$ sectors and suppose that $X$ is not separated from $Y$.
Then for the first non-inductive case (left diagrams)   $$L=R^\vee$$ follows, and
for the second non-inductive case (right diagrams) $$\int\! L=R^\vee \ .$$
\end{prop}

\begin{proof}
In the non-inductive cases, for $Z=Y^\vee$, the first sectors of $X$ and $Z$ are
$X_1=\int\! L$, $Z_1=\int R^\vee$ (left diagram) resp. $X_1=\int\! L$, $Z_1=R^\vee$
(right diagram). Here it is important to recall that $R$ is always a sector in the 
second non-inductive case (right diagram). 
$X$ has at least three sectors by assumption. Hence in condition $(**)$ we have at
least two different choices for $s\neq i=1$. For at least  one of these choices
$t(s)\neq 1$ must hold. Otherwise lemma \ref{lastl} would imply that
$X$ is of ladder type and the two choices $s,u$ belong to $u=1$ and $s=r$.
But $s,u\neq i=1$ would yield a contradiction. This shows the existence of $s$
such that $t(s)\neq 1$.  Hence, by construction we have $s>1$ and $t(s)>1$,  
and $\partial_s X= X_1... \partial(X_s)... \equiv \partial_{t(s)} Z  = Z_1 ... \partial(Z_{t(s)}) ...
$ finally  implies $X_1\cong Z_1$. 
\end{proof}

\subsubsection{The special subcase of the non-inductive cases}
For the right diagram of the non-inductive type  $R$ as a sector can not be empty.
For the left type diagram,  $R$ as a segment could be empty.
Then $X=(\int\! L M \emptyset)$ and $Y^\vee = (C M^\vee L^\vee)$ for a cap $C$ of rank 1.
 Now there two cases: In the first,  the partner $t=t(s)$ is different from $1$ for all $s>1$.
 Then $(**)$ implies $\int\! L (\partial_s M) \equiv C \partial_{t(s)} M^\vee$ for all $s>1$.
 This implies $\int\! L =C$ and $L=\emptyset$, and hence  $(\partial_s M) =  \partial_{t(s)} M^\vee$
 as true equalities for all $s>1$. It is easy to see that this implies $t(s)=s$, and hence $M=M^\vee$.
 Thus $Y^\vee = X$ and $X \sim Y$, a contradiction! Hence there
 exist $s>1$ such that $t(s)=1$ holds, and by lemma \ref{lastl}  
 there exists at least one $s'$ such that $t(s')\neq 1$ holds. For $s'$ condition $(**)$ implies 
 $\int\! L^\vee =C$, whereas for $s$ condition $(**)$ implies $\partial_s(X)=(C\partial_s(M)\emptyset)
 \equiv \partial_t(C M^\vee\emptyset) = M^\vee$.
 Therefore we obtain the identity
 $$  M^\vee = C \partial_s(M) .$$
This case where $L=R^\vee$ is empty and $  M^\vee = C \partial_s(M) $ holds,
will be  called the \textit{special case}. 

\begin{lem} \label{special-case-non-inductive} In the special subcase of the non-inductive case $X$ is of ladder type
and $X\sim Y$.  
\end{lem}

\textit{Notation}. In the following distances will play a crucial role. We use the notation $IaJ$ to denote a distance $a$ between the intervals $I$ and $J$.

\smallskip
Before the proof of the lemma we first prove an auxiliary technical lemma.

\begin{lem} \label{critical-equation}
Let $C$ be a cap and let $M$ be an interval
such that $C\partial_s(M) \equiv M^\vee$ holds. Then $CM=CaIaC$
for selfdual $I=I^\vee$ and $s$ is the last derivate.
\end{lem}


\begin{proof}
The left \lq{critical}\rq\ equation implies by duality $M=(\partial_sM)^\vee C
= b Ia C$ for $b I a :=(\partial_s M)^\vee$ and minimal
interval $I$. Then the left critical equation 
becomes $C\partial_s(b I a C) = C a I^\vee b$.
If $s$ is not the last derivative, then $C\partial_s(b I a C)$ has true length
$2+b + \ell(I) + a + 2$, whereas  $C a I^\vee b$ has true length
$2+a+\ell(I)$. This leads to a contradiction and shows
$Cb I = C a I^\vee $. This implies $a=b$ and $I=I^\vee$.
Hence $CM= C a I a C = C a I^\vee a C =  C a I a C =M^\vee C$. 
\end{proof}

Notice that the left critical equation $C\partial_s(M) \equiv M^\vee$
for $M$ is equivalent to the right critical equation $\partial_s(N)C \equiv N^\vee$
for $N=M^\vee$.

\medskip
We now proof lemma \ref{special-case-non-inductive}.

\begin{proof} In the \textit{special cases}
we have  $X=CM$ and $Y=MC$
where $C$ is as in lemma \ref{critical-equation} and $M$
is an interval with at least two sectors. Then $Y^\vee= CM^\vee$ 
and we had an
 equation of the form
$C\partial_sM \equiv M^\vee$ for a nonempty interval $C$
of rank 1. Concerning  solutions of this equation, lemma 
\ref{critical-equation} shows $M=aIaC$ for a selfdual interval $I=I^\vee$. Hence  $X=CM=CaIaC$ is symmetric and $Y^\vee = CCa Ia$ holds.

We assume now by induction that $X$ and $Y$ are of the form \[ X = C^r I C^r, \ \ Y = C^{r+1} I C^{r-1}\] for some selfdual interval $I = I^{\vee}$ and some $r$. We can assume $I \neq \emptyset$. Consider now $s=r+1$ (i.e. the first derivative of $I$) with $t(s) \neq 1$. Then \begin{align*} \partial_s (C^r I C^r) & = C^r \partial_s(I)C^r = \partial_t(Y^{\vee}) \\ & = \partial_t(C^{r+1}I C^{r-1}) =  C \partial_t(C^r I C^{r-1}).\end{align*} This implies $t \geq r+1$, hence the last entry becomes $C^r \partial_t(CIC^{r-1})$. The comparison gives $\partial_{r+1}(I)C^r = \partial_t(CIC^{r-1})$. Put $I = aI_1 J$ for the first sector $I_1$ of $I$. The left side $\partial_{r+1}(I)C^r$ is $(a+2)I_1C^r$ if $\partial(I_1) = \emptyset$ and $(a+1)\partial{I_1} JC^r$ if $\partial(I_1) \neq \emptyset$. The right side $\partial_t(CIC^{r-1})$ is $(a+2)I_1 J C^{r-1}$ if $t = r+1$ and $C a \partial_t(I_1 JC^{r-1})$ if $t > r+1$. Hence $t=r+1$ and $I_1 = C$. This implies $I = aCJ$ and by comparison of both sides $JC^r = C J C^{r-1}$ and then $JC = CJ$. This in turn implies $J =  J_1 C$ for some $J_1$ and therefore $I = a C J_1 C$. Since $I = I^{\vee}$ we get $a = 0$ and $J_1 = J_1^{\vee}$. Since $r$ was arbitrary, this implies that $X$ is of ladder type $X = CC\ldots CC$ and $Y^{\vee} = CC \ldots CC$ as well. 

\end{proof}

Granting this, in the proof of the following proposition we 
can  assume $R\neq \emptyset$ not only for the right type diagrams, but
also for the left
type diagrams of the non-inductive situation.

\begin{prop} \label{no-non-inductives}
Assume that $X$ has at least 3 sectors and $X$ is not separated from $Y$.
Then the non-inductive cases do not occur.
\end{prop}

\begin{proof}
\textit{Left type of diagrams}. There proposition \ref{prop-nondeg} gives $L=R^\vee$ 
 and where we can suppose $R\neq \emptyset$ (since we already discussed the special cases).
 Thus
$X=(\int\! L\, M\,  L^\vee)$ and $Y^\vee = (\int\! L\,  M^\vee\, L^\vee)$. Take $s=r$ to be the derivative with respext to the last sector of $X$.
Then $1< t(s) < k$ can not be occur, since the total length of $\partial_s X$ would then be smaller than the
total length of $\partial_{t(s)}(Y^\vee)$. This is not possible by $(**)$. Hence $t(r)\in \{1,r \}$. If $t(r)=1
$ then $\int L = \partial \int\! L =L$. A contradiction. If $t(r)=r$, then $M=M^\vee$ and hence $X \sim Y$. 

\medskip
\textit{Right type of diagrams}. Here proposition \ref{prop-nondeg} gives $\int L=R^\vee$ and $R$ is a sector
and $L$ are sectors. Thus
$X=(\int\! L \, M \, \int\! L^\vee)$ and $Y^\vee = (\int\! L\, (\int\! M)^\vee L^\vee)$. Take $s=r$ to be the derivative with respect to the last sector of $X$.
As before $1< t(s) < k$ can not be occur, since the total length of $\partial_s X$ would then be smaller than the
total length of $\partial_{t(s)}(Y^\vee)$.  Hence $t(r)\in \{1,r \}$ and hence  $t(r)=r$.
Notice $M=S_2,...,S_{r-1}$,  
Therefore $t(r)=r$ together with lemma \ref{lastl} implies that $s\mapsto t(s)$ defines an injection $\{2,...,r-1\} \to \{2,...,r-1\}$
[since $(\int_j M)^\vee$ also has $r-2$ sectors and since $t(s)=1$ is not possible by reasons of total length]. 
Hence $(**)$ implies $\int\! L^\vee = L^\vee$ by a comparison of the last sectors of the derivatives $\partial_s X $ and $\partial_{t(s)} (Y^\vee)$. A contradiction.   
\end{proof}

\subsection{The inductive case}
The \emph {inductive situation} $1<i< j<k$  
and $(*)$ leads to the following diagrams   \[ \begin{pmatrix} X \\  Y^{\vee} \end{pmatrix} =
 \begin{pmatrix} L & \int_i\! M & R \\ R^{\vee} & (\int_j\! M)^{\vee} & L^{\vee} \end{pmatrix}  = \begin{pmatrix} L & \tilde X & R \\  \ R^{\vee} & \ \tilde Y^{\vee} & \ L^{\vee} \end{pmatrix}\]
 where $\tilde X = \int_i\! M$ and $\tilde Y= \int_j\! M$. 
Note that $L$ denotes the first sector of $X$ and $R$  the last sector of $X$, such that $M$ is the remaining  middle interval defined by the decomposition $(*)$. We may as in the non-inductive case assume that $X$ has at least 3 sectors.

\begin{prop}\label{prop-deg}
Suppose $X$ has $r\geq 3$ sectors and is not separated from $Y$.
Then for the inductive case  we have an equality of sectors $L=R^\vee$.
\end{prop}

\begin{proof}
Use lemma \ref{lastl} to find
 $s>1$ such that $t(s)\neq 1$ and then us $(**)$  as in the proof 
 of proposition \ref{prop-nondeg}.  
\end{proof}

By proposition \ref{no-non-inductives} non-inductive cases do not contribute
if $X$ has more than 3 sectors. Hence the next proposition
proves theorem \ref{resume}.

\begin{prop}\label{induct}
Suppose $X$ has $r\geq 3$ sectors and is not separated from $Y$.
Then the inductive case can not occur.
\end{prop}

\begin{proof} 
Consider condition $(**)$ for the last derivative $\partial_k$ of $X$, i.e. $s=k$. This implies 
$ \partial_{t} (Y^\vee) = \partial_s(X) =(L\tilde X \partial R)$. Since $R=L^\vee$ by proposition \ref{prop-deg},
$\partial(R)$ can not contain the sector $L^\vee$. Hence the only  possibilities are case 1) where  $t=t(k)=k$ 
or case 2) where $t(k)=1$ [consider the total length] and   
$\partial R=\emptyset$ (then also $\partial L^\vee =\emptyset$ by $R=L^\vee$).

 \medskip
Case 1) \textit{Suppose $L$ and $R$ do not have rank 1}. Then $t(k)=k$
and $(**)$ implies  $\tilde X = \tilde Y^\vee$. Hence $Y^\vee = (R^\vee \tilde Y^\vee L^\vee)$ by proposition \ref{prop-deg} is equal to
$(L\tilde X R ) = X$. This shows $X\sim Y^\vee$, hence a contradiction.   

\medskip
Case 2) \textit{Now suppose $R$ and $L$ have rank 1}.
Then in $(**)$ we can derive $X$ with respect to $s=1$ and $s=k$ since $i\neq 1,k$.
But $t(s)\not\in\{1,k\}$ for $s\in\{1,k\}$ is impossible since otherwise the total length (i.e. the distance from the beginning of the first sector to the end of the last sector) of $\partial_s X$ would be strictly smaller than that of $\partial_t(Y^\vee)$. 
This however contradicts $\partial_s X \equiv \partial_t(Y^\vee)$. Thus $t(\{1,k\}) \subseteq \{1,k\}$.
Then $t(1)=k$ and $t(k)=1$ since otherwise $\tilde X = \tilde Y^\vee$, and therefore $X = Y^\vee$ would
contradict $X\not\sim Y$. This implies 
$(\tilde X R) \equiv( L \tilde Y^\vee)$ and also $(L\tilde X) \equiv( \tilde Y^\vee R)$
for the sectors $R$, $L$ of rank 1. 
Therefore the first and last sector of $\tilde X$ and $\tilde Y$ must have rank 1.
If we can exclude this, our proof of proposition \ref{induct} is complete.

\medskip
For this we first observe 
$t(s)\not\in\{1,k\}$ for $s\not\in\{1,k\}$. 
This is seen by arguing with the total length as before.
We now  apply induction with respect to the number of sectors. By the induction  
we can therefore either assume  $\tilde X \sim \tilde Y$, or $\tilde X$ and $\tilde Y$ only have at most
two sectors.  
In the latter case $\tilde X$ and $\tilde Y$ have two sectors of rank 1 by the last observation of case 2). 
The first case is impossible because we could assume
$\tilde X \equiv \tilde Y^\vee$ [otherwise replace $Y$ by $Y^\vee$]. Then $t(s) = s$ for all $s\not\in \{1,k\}$.
Furthermore $(**)$ would imply $\partial_s X = \partial_s Y^\vee$ for all $s\notin \{1,k \}$ (the equality sign
comes from the fact that the sectors $R$ and $L$ rigidify the situation). This shows $\tilde X = \tilde Y^\vee$
and hence $X = Y^\vee$, a contradiction. Notice, this rigidified version of $(**)$ 
now implies $X = Y^\vee$ also in the last remaining case where
$\tilde X, \tilde Y$ both consist of two rank 1 sectors.  
\end{proof}

\subsection{The examples where $X$ has $\leq 2$ sectors}  \label{example:separated}

\textit{Case where $X$ has only one sector}.
Consider a sector $X$ and $Y= (\partial(X) C)$ 
for a cap diagram $C$ of rank 1 with an arbitrary 
distance between the segment $\partial(X)$ and $C$.
Then $X$ is not separated from $Y$. It is easy to see
that these are the only cases.

\medskip
\textit{Case of two sectors}.  
If $X$ has only two sectors, suppose $X$ can not be separated from $Y\not\sim X$. 
So without restriction of generality we can assume $\partial_1 X = (\partial X_1) X_2 \equiv \partial_2 Y = Y_1 (\partial Y_2)$ as in condition $(*)$.
If $\partial X_1 = \emptyset$, then also $\partial Y_2= \emptyset$ and $X_2 = Y_1$.
Hence
$    X = C R$ and $Y = R C$ 
for $C$ of rank 1. We can assume $R\neq R^\vee$. 

\medskip
1) Example: If $R$ has rank 1, then we
obtain examples where $X$ can not be separated from $X$ in the form
$$  X = C a C   \ \text{ and } \  Y = C b C \ \text{ for } a\neq  b \ .$$
In this case $X$ has rank 2.
If $R$ has rank $>1$ no further examples arise: Indeed, $\partial_2(X)=C \partial R$ can not be $\partial_1(Y)= \partial(R) C$
nor  $\partial_2(Y)= R$ nor  $\partial_1(Y^\vee)= R^\vee$
nor  $\partial_2(Y^\vee)= C \partial R^\vee$.

\medskip
2) It remains to discuss the case $\partial X_1 \neq \emptyset$. Then $\partial_1 X \equiv \partial_2 Y$ implies $\partial X_1 =Y_1$ and $X_2=\partial Y_2$.
Hence  $X = \int\!L \,  R $ and $Y = L  \int\! R$. Hence we are in the left type non-inductive case
with $M=\emptyset$ and $Y^\vee = \int\! R^\vee\, L^\vee$. But now show $R=\int\! L^\vee$ (in contrast to
the case with more than 3 sectors discussed in proposition \ref{prop-nondeg}).
In fact, condition $(**)$ for $s=2$ implies 
$\int\! L \partial(R) \equiv \partial_t(\int\! R^\vee L^\vee)$. For $t=1$ this implies $\int\! L = R^\vee$,
for $t=2$ it implies $L=R^\vee$ and hence $X=Y^\vee$ and hence $X\sim Y$. This gives
the second examples $Y= (L \, \int\!\int\! L^\vee)$ with
\[ \begin{pmatrix} X \\ Y^\vee \end{pmatrix}  \sim \begin{pmatrix} \int \! L &  \int \! L^\vee \\  \int \!\int \! L & L^\vee  
 \end{pmatrix}\]
Obviously, $X$ and $Y$ thus defined for an arbitrary sector $L$ 
have the property that $X$ can not be separated from $Y$ since $\partial_1 X = \partial_2 Y$ and $\partial_2X = \partial_1 Y^\vee$.

\medskip
So if $X$ is not separated from $Y$, $X$ is selfdual. For rank $n>2$, however 
$Y$ is not selfdual. Hence proposition \ref{asymmetry}  follows.


\bigskip\noindent

\section{Pairings} \label{sec:pairings}

Selfdual objects $L(\lambda)$ will give rise to groups of type $B, \ C, D$ according to section \ref{derived-group}. In order to distinguish between the orthogonal and the symplectic case we check whether these representations are even or odd in the sense defined below.

\subsection{Strong selfduality} We say that an object $M$ is strongly selfdual, if there exists an isomorphism $\rho: M \to M^\vee$ such that $\rho^\vee = \pm\rho$ holds and call it even or odd depending on the sign. Here $\rho^\vee: M \to M^\vee$ is the dual morphism of $\rho$. Here we use the canonical identification $M = (M^\vee)^\vee$, since a priori we only have $\rho^\vee: (M^\vee)^\vee \to M^\vee$. Note that any selfdual irreducible object
is strongly selfdual in this sense. Slightly more general: If $L$ is an invertible object
in a tannakian category and $\rho: M \cong M^\vee \otimes L$, then
$(\rho^\vee \otimes id_L) \circ (id_M \otimes coeval_L) = \pm \rho$. 
Furthermore any multiplicity one retract of a strongly selfdual
object is strongly selfdual. Finally, if $F$ is a tensor functor between rigid symmetric tensor categories, then $F(M)$ is strongly selfdual if $M$ is strongly selfdual.
We remark that we can define the similar notion of strong selfduality for $*$-duality. 

\medskip
By \cite[(4.30)]{Scheunert} a supersymmetric invariant bilinear form on a representation $(V,\rho)$
in $T$ defines a skew-supersymmetric invariant bilinear form on the representation $\Pi(V,\rho)$.

\medskip
Suppose $L \cong L^\vee $ in $\calR$ is a maximal atypical self dual
representation. We consider now irreducible representations of the form $[\lambda] = [\lambda_1,\ldots,\lambda_{n-1},0]$. We call these positive. For general $\lambda$ we can twist with an appropriate Berezin power to get this form. We will induct on the degree $\sum \lambda_i$, hence we start with the case $S^1$.

\begin{lem} $S^1$ is an even selfdual representation.
\end{lem}

{\it Proof.} Obviously $S^1 \cong (S^1)^\vee$, and therefore there exists a nondegenerate super bilinear form
$$  B: S^1 \otimes S^1 \to  {\bf 1} = k \ .$$
Note that the adjoint representation of $G_n$ on $\A:= \g_n$ carries
the nondegenerate invariant Killing form
$$  K: \g_n \otimes \g_n  \to \one = k \ .$$
This bilinear form is supersymmetric: $K(S(x \otimes y)) = K(x \otimes y)$ for the symmetry constraint $S: \g_n \otimes \g_n  \cong \g_n \otimes \g_n $, or $K(x,y)=(-1)^{\vert x\vert\vert y\vert}
K(y,x)$. Let $\g_n^0$ denote the kernel of the supertrace $\g_n \to \one$.
Then $S^1=\g_n^0/z$, where $z$ is the center of $G_n$. The Killing form $K$ restricts
to a supersymmetric form on $\g_n^0$ which becomes nondegenerate on $S^1=\g_n^0/z$.
Hence $S^1$ carries a nondegenerate supersymmetric bilinear form. \qed

\medskip\noindent
We now treat the general $[\lambda] = [\lambda_1,\ldots,\lambda_{n-1},0]$-case. Recall that the direct summands of $V^{\otimes r} \otimes (V^{\vee})^{\otimes s}$ are called mixed tensors. The maximal atypical mixed tensors are parametrized by partitions $\lambda$ satisfying $k(\lambda) \leq n$ for an integer $k(\lambda)$ defined in \cite[6.17]{Brundan-Stroppel-5} \cite[Section 4]{Heidersdorf-mixed-tensors}.  We furthermore recall from \cite[Theorem 12.3]{Heidersdorf-mixed-tensors}: For every such $[\lambda]$ the mixed tensor $R(\lambda)$ contains $[\lambda]$ with multiplicity 1 in the middle Loewy layer. $[\lambda]$ is the constituent of highest weight of $R(\lambda)$. If we define $deg \ [\lambda] = \sum_{i=1}^n \lambda_i$, then $[\lambda]$ has larger degree then all other constituents. We denote the degree of a partition by $|\lambda|$. We recall further: If $\lambda$ and $\mu$ are two partitions of length $\leq n$, the tensor product $R(\lambda) \otimes R(\mu)$ splits in $\calR_n$ as \begin{align*}  R(\lambda) & \otimes R(\mu) = \\ & \bigoplus_{|\nu| = |\lambda| + |\mu|, k(\nu) \leq n} (c_{\lambda \mu}^{\nu})^2 R(\nu) \oplus \bigoplus_{|\nu| < |\lambda| + |\mu|, k(\nu) \leq n} d_{\lambda \mu}^{\nu} R(\nu)\end{align*} for some coefficients $d_{\lambda \mu}^{\nu} \in \N$, the Littlewood-Richardson coefficients $c_{\lambda \mu}^{\nu}$  and the invariant $k(\lambda)$ \cite[Lemma 14.4]{Heidersdorf-mixed-tensors}.

\begin{prop} Let $[\lambda]$ be positive of degree $r$. Then $R(\lambda)$ occurs as a direct summand with multiplicity 1 in a tensor product $\A \otimes R(\lambda_i)$ where $l(\lambda_i) \leq n$, $|\lambda_i| = r-1$ and $R(\lambda_i)$ is a direct summand in $\A^{\otimes r-1}$. The constituent $[\lambda]$ occurs with multiplicity 1 as a composition factor in the tensor product $\A \otimes R(\lambda_i)$.
\end{prop}

\begin{proof} For $\lambda, \mu$ of length $\leq n $ we know that \[ R(\lambda) \otimes R(\mu) = \bigoplus_{|\nu| = |\lambda| + |\mu|, k(\nu) \leq n} (c_{\lambda \mu}^{\nu})^2 R(\nu) \oplus \tilde{R} \] where $\tilde{R}$ are the terms of lower degree. We apply this for $\lambda = \mu = (1)$ (i.e. $\A \otimes \A$) and then to tensor products of the form $R(\lambda) \otimes \A$. Since every summand in a tensor product of the standard representation of $SL(n)$ with any other irreducible module has multiplicity 1, $\A^{\otimes r}$ decomposes as the standard representation of $SL(n)$ modulo contributions of lower degree and contributions of length $l(\nu) > n$. Since every irreducible $SL(n)$-representation with highest weight $\lambda$ of degree $deg(\lambda) = \sum \lambda_i =r$ occurs as a summand in $st^{\otimes r}$, every mixed tensor $R(\lambda)$ with $l(\lambda) \leq n$ and $deg(\lambda) = r$ occurs as a direct summand in $\A^{\otimes r}$. Hence there exists in $\A^{\otimes r-1}$ a mixed tensor $R(\lambda_i)$ of length $\leq n$ and degree $deg(\lambda_i) = r-1$ with \[ \A \otimes R(\lambda_i) = R(\lambda) \oplus \bigoplus R(\nu_i) \] and $\nu_i \neq \lambda$ for all i. $R(\lambda)$ contains the composition factor $[\lambda]$ with multiplicity 1 and no other mixed tensor in this decomposition contains $[\lambda]$. Indeed if $deg(\nu_i) < r$, its constituent of highest weight has degree $<r$. If $R(\nu_i)$ has degree $r$ and $l(\nu_i) \leq n$, its constituent of highest weight is $[\nu_i] \neq [\lambda]$ and if $R(\nu_i)$ has degree $r$ and $l(\nu_i) >n$, its constituent of highest weight has degree $<r$ by \cite[Section 14]{Heidersdorf-mixed-tensors}.
\end{proof}

This applies in particular to positive $[\lambda]$ which are (Tannaka) selfdual. Every such $[\lambda]$ occurs as a multiplicity 1 constituent in a multiplicity 1 summand in a tensor product $\A \otimes R(\lambda_i)$ for $|\lambda_i| = r-1$ which in turn appears as a multiplicity 1 summand in a tensor product $\A \otimes R(\lambda_{i_2})$ with $|\lambda_{i_2}| = r -2$ etc.

\begin{cor} \label{duality-type-old} The selfdual representation $[\lambda] = [\lambda_1,\ldots,\lambda_{n_1},0]$ is even. Its parity shift $\Pi [\lambda]$ is odd.
\end{cor}

\begin{proof} The parity is inherited to super tensor products (look at the even parts) and to multiplicity 1 summands.
\end{proof}

Since for basic $L(\lambda)$ a weakly selfdual weight $\lambda$ is selfdual, we obtain the next corollary.

\begin{cor} \label{basic-duality type} If $\lambda$ is basic of type (SD), its parity is even.
\end{cor}




\subsection{Combinatorics of selfdual weights}
If the representation $L(\lambda)$ is only selfdual up to a Berezin twist, the argument breaks down. If one simply restricts $L(\lambda)$ to $SL(n|n)$ to get a selfdual weight, we lose the multiplicity one assertions from above. 

\begin{definition} The total height of a forest is defined recursively as follows: \begin{align*} ht({\mathcal F}) & = \sum_i ht({\mathcal T}_i), \\ 
ht(\mathcal T) & = \sum_{x\in\mathcal T} ht(x)\end{align*}
where $ht(x)$ is the distance to the root $x_0$ of the tree (so $ht(x_0)=0$ etc).
\end{definition}

\medskip\noindent
Recall from section \ref{sec:dual} that in the (SD)-case we have $r_1 = r_k, \ldots$ and $d_1 = d_k-1, d_2 = d_{k-2},\ldots$ for the $k$ sectors of rank $r_1,\ldots,r_k$ and the distances $d_1,\ldots, d_{k-1}$. From this we get \[ D(\lambda) = \sum_{i=0}^k ( \sum_{1 \leq \mu < i} d_{\mu}) = (\sum_{i=1}^{[k/2]} d_i) \cdot n - d_{middle} \cdot \frac{n}{2}\] using $r_1 + \ldots r_k = n$ where $d_{middle} = 0$ unless $2|k$.

\begin{lem} \label{lemma-1} For weakly selfdual forests $\mathcal F_\lambda$ with $k$ trees we have \[ D(\lambda)=
(d_0+...+ d_{[k/2]})n - \frac{n d_{middle} }{2}\] where $d_{middle}$ is zero by definition
if the number of trees  is odd.
\end{lem}


\begin{lem} \label{lemma-2} For weakly selfdual $L=L(\lambda)$
with $L^\vee \cong L \otimes Ber^{-r}$ we have  $r= 2(d_0+...+ d_{[k/2]}) + d_{middle}$
and $rn=2D(\lambda)$, so that $rn$ is even.
\end{lem}

\begin{proof} The isomorphism $L^{\vee} \cong L \otimes Ber^{-r}$ implies \[\det(L)^{-1} \cong \det(L) Ber^{-nr dim(V_{\lambda})}.\] We apply $\nu: Pic(\mathcal{T}_n^+) \to \mathbb{Z}$ and $\det(L)^{\vee} \cong det(L)^{-1}$ and use results of section \ref{picardgroup}. Then $\nu(det(L)) = \sdim(L)D(L)$ and $\nu(Ber)=n$ give the result.
\end{proof}

\begin{lem}\label{lemma-3} We have $p(\lambda)-p(\lambda_{basic})=D(\lambda)$
and furthermore $p(\lambda_{basic})= n(n-1)/2 - ht({\mathcal F}(\lambda))$ for the
height $ht({\mathcal F}(\lambda))$ of the forest ${\mathcal F}(\lambda)$.
\end{lem}

\begin{proof} Induction on $n$ using $ht(\partial {\mathcal T}) = ht({\mathcal T})
- \# {\mathcal T} + 1$. For the first assertion see \cite[Corollary 25.4]{Heidersdorf-Weissauer-tensor}.
\end{proof}

We obtain the following corollary. 

\begin{cor} \label{Corollary-1} If $L$ is weakly selfdual and the number of trees of 
${\mathcal F}(L)$ is odd, let $\partial_{middle}{\mathcal F}(L)$ denote
the forest obtained by the middle derivative of ${\mathcal F}(L)$. 
The associated weight $\tilde\lambda$ is weakly selfdual such that $p(\tilde\lambda_{basic})
= p(\lambda_{basic})$ mod 2.
\end{cor}

\begin{proof} By lemma \ref{lemma-3} and its proof the congruence is equivalent to $n-1 - (r_{middle}- 1) = 0$ mod 2,
hence follows from $n \equiv r_{middle}$ mod 2. 
\end{proof}

\subsection{Weakly selfdual cases} Let $L$ be weakly selfdual as above, then there exists a nondegenerate
pairing
$$  L \times L \to Ber^r $$
which is either symmetric or antisymmetric. Let $\varepsilon_{\lambda}$ denote
this parity of the pairing. It induces a pairing on the Dirac cohomology 
$V = \omega(L)$ of the same parity. Since we work mainly in $\mathcal{T}_n^+$,
notice that $B^r=Ber^r$ is always in $\mathcal{T}_n^+$ by lemma \ref{lemma-2}  because $nr$ is even in the (SD)-cases.
But the twist $X_\lambda= \Pi^{p(\lambda)} L(\lambda)$ shifts the parity
of the induced pairing on $\varepsilon(X_\lambda)$ which has parity $(-1)^{p(\lambda)}
\varepsilon_\lambda$. 

We now analyze the parity $\varepsilon_{\lambda}$. For this we make a case distinction into Case 1 (easier case) and Case 2 (more  difficult case). Our proof is vaguely reminiscent of the classical proof in Bourbaki \cite[Chapter IX, Section 7.2, Proposition 1]{Bourbaki-7-9} where one constructs a Lie subalgebra isomorphic to $\mathfrak{sl}(2)$ and considers the restriction of the bilinear form to a certain irreducible $\mathfrak{sl}(2)$-module. The sign can then be read off from the one for this restriction. In the difficult second case we reduce the computation of the sign to the restriction of the form to one direct summand in $DS_{n,2} (L(\lambda))$ (where we know the sign by \cite{Heidersdorf-Weissauer-GL-2-2}). Actually a similar argument could also be applied in the case 1 where $n$ is odd, $r$ is even or
$r$ is odd, $n$ is even. However our previous approach in these cases is constructive, whereas we do not know an explicit description of the pairing in Case 2.

\bigskip\noindent
{\bf Case 1}: {\it Let ${\mathcal F}(\lambda)$ be a weakly selfdual forest such that it either has an odd number of trees, or the middle distance $d_{middle}$ is even.}

\bigskip\noindent
Then middle integration gives a chain of weakly selfdual weights \[ \lambda= \lambda^0,
\lambda^1,...,\lambda^s = \Lambda\] such that $\Lambda$ is basic selfdual up to
a Berezin shift. Conversely we obtain $\lambda$ by iterated middle derivatives from
$\Lambda$.

\begin{center}
\medskip
 
 \scalebox{0.7}{
\begin{tikzpicture}
\foreach \x in {-5,-2,2,5} 
     \draw[very thick] (\x-.1, .1) -- (\x,-0.1) -- (\x +.1, .1);
\foreach \x in {-4,-3,-1,0,1,3,4,6} 
     \draw[very thick] (\x-.1, -.1) -- (\x,0.1) -- (\x +.1, -.1);
%


\draw[very thick] [-,black,out=90, in=90](-5,+0.2) to (-4,+0.2);
\draw[very thick] [-,black,out=90, in=90](-2,+0.2) to (-1,+0.2);
\draw[very thick] [-,black,out=90, in=90](2,+0.2) to (3,+0.2);
\draw[very thick] [-,black,out=90, in=90](5,+0.2) to (6,+0.2);


\end{tikzpicture} }
\smallskip

\text{The cup diagram of $\lambda$}
\end{center}

\medskip
\begin{center}
\medskip
 
 \scalebox{0.7}{
\begin{tikzpicture}
\foreach \x in {-5,-2,2,5} 
     \draw[very thick] (\x-.1, .1) -- (\x,-0.1) -- (\x +.1, .1);
\foreach \x in {-4,-3,-1,0,1,3,4,6} 
     \draw[very thick] (\x-.1, -.1) -- (\x,0.1) -- (\x +.1, -.1);
%


\draw[very thick] [-,black,out=90, in=90](-5,+0.2) to (-4,+0.2);
\draw[very thick] [-,black,out=90, in=90](-2,+0.2) to (-1,+0.2);
\draw[very thick] [-,black,out=90, in=90](2,+0.2) to (3,+0.2);
\draw[very thick] [-,black,out=90, in=90](5,+0.2) to (6,+0.2);
\draw[very thick] [-,black,out=90, in=90](0,+0.2) to (1,+0.2);
\draw[very thick] [-,black,out=90, in=90](-3,+0.2) to (4,+0.2);


\end{tikzpicture} }
\smallskip

\text{The basic (up to Ber shift) weight $\Lambda$ obtained from $\lambda$}
\end{center}

\begin{lem} \label{lemma-4} In the situation of our assumption the parity of the pairing $\varepsilon(X_{\lambda})$ is equal to $(-1)^{p(\lambda_{basic})}=\varepsilon(X_{\lambda_{basic}})$.
\end{lem}

\begin{proof} By corollary \ref{Corollary-1} it suffices to show this for basic  $\Lambda$ selfdual up to a Berezin twist. This case has been settled in corollary \ref{basic-duality type}. Indeed by corollary \ref{Corollary-1}, for the middle derivative $\tilde \lambda$ of $\lambda$ the parity $\varepsilon_\lambda$ coincides with $\varepsilon_{\tilde\lambda}$ since $\tilde\lambda$ is an irreducible multiplicity one summand of $H_D^\bullet(L)$.
\end{proof}
 
\subsection{Case 2} Now we turn to cases were the previous assumption is not satisfied.
The integer $d_{middle}$ then is odd. By lemma \ref{lemma-1} then $D(L)$ is an odd
multiple of $n/2$. Hence  $n$ must be even. By lemma \ref{lemma-2} furthermore $nr$ is not divisible
by 4. Hence $r$ must be odd. 

\begin{lem} \label{lemma-5} Under our new assumption $n$ is even and $r$ is odd.
\end{lem} 

In the situation of the previous assumption notice that $nr$  is always divisible by 4 and $r$ is even. Hence $L Ber^{-r/2}$ is selfdual
and $Ber^{-r/2}$ was in $\mathcal{T}_n^+$. If we normalize now and replace $L$ by the selfdual ${\mathcal L} := L Ber^{-r/2}$, we need to be more careful. First we have to define an extended representation tensor category of $\mathcal{T}_n^+$ that contains $Ber^{1/2}$ (we sometimes write simply $B$ instead of $Ber$). Second, the parity of  $B^{1/2}$ and hence also of $Ber^{r/2}$ will be odd, so it needs a parity shift so that $\Pi B^{1/2}$ is in
the extended tensor category $\mathcal{T}_n^{ext}$. To do this replace $GL(n\vert n)$ by the real
supersubgroup generated by replacing $GL(n\vert n)_{\overline 0}= GL(n,\mathbb{C}) \times GL(n,\mathbb{C})$ 
by the subgroup $G$ generated by $diag(E,-E)$ and by the matrices $diag(A,D)$ in $GL(n,\mathbb{R})\times GL(n,\mathbb{R})$ with $\det(A), \det(D)\in \mathbb{R}^*_{>0}$. Then $Ber^{1/2}$ is well-defined
as a complex representation of the real super Liegroup $G$. The super parity decomposition of the space of a representation $\rho$ is defined, as before,
by the eigenvalues of $\rho(diag(E,-E))$. We define ${\mathcal T}_n^{ext}$ as the category of all finite dimensional representations on which $\rho(diag(E,-E))$ acts by the parity endomorphism as well as their parity shifts (analogously to $\mathcal{T}_n = \mathcal{R}_n \oplus \Pi \mathcal{R}_n$). 

 Hence the parity of  $Ber^{1/2}$
becomes $(-1)^{nr}=-1$. Notice that $Lie(G)\otimes_{\mathbb{R}} {\mathbb{C}} = \mathfrak{gl}(n\vert n)$. Hence
we can view our new category again as a tensor category ${\mathcal T}_n^{ext}$ of representations  of $\mathfrak{gl}(n\vert n)$ containing $\mathcal{T}_n$.

\medskip\noindent    
Recall ${\mathcal L}^\vee \cong {\mathcal L}$. 
Let $J$, $J^2=id$ be the antidiagonal unit matrix in $GL(2n)$,
then $X\mapsto \phi(X)=X^J=JXJ$ defines an automorphism of
the Lie superalgebra $\mathfrak{gl}(n\vert n)$ of order two. Notice that this automorphism
exchanges the highest and lowest vectors
$v_+$ and $v_-$ of $L$. This follows from the fact that $J$ interchanges
the ideals $\mathfrak{p}_+$ and $\mathfrak{p}_-$ and the fact that irreducible modules are cyclic modules.

\begin{lem}
For maximal atypical irreducible representations $L$ the twisted representation $\rho_{\mathcal L}^J(X) :=\rho_{\mathcal L}(X^J)$
is isomorphic to the dual representation given by $\rho_{{\mathcal L}^\vee}(X)$.
\end{lem}

\begin{proof} This uses $L^*\cong L$, for
$L^*$ defined on the representation space of $L^\vee$ by $X \mapsto \rho^\vee(-X^T)$
and the automorphism $(\begin{smallmatrix} A & B \cr C & D \end{smallmatrix})=X\mapsto -X^{T}= (\begin{smallmatrix} -A' & C' \cr -B' & -D' \end{smallmatrix})$ of $\mathfrak{gl}(n\vert n)$.
This automorphism, composed with conjugation by $J$, becomes the automorphism
$X \mapsto (\begin{smallmatrix} -wD'w & -wB'w \cr wC'w & -wA'w \end{smallmatrix})$ of $\mathfrak{gl}(n\vert n)$ for the antidiagonal unit matrix $w$ in $GL(n)$. Clearly, if we twist a representation with this composed automorphism, this preserves highest weight vectors
and their weights $(\lambda_1,...,\lambda_n\vert \lambda_{n+1},...,\lambda_{2n})$ become
$(-\lambda_{2n},...,-\lambda_{n+1}\vert -\lambda_{\lambda_n},...,-\lambda_{1})$. Recall that the highest weight of a maximal atypical irreducible representation is of the form $(\lambda_1,\ldots,\lambda_n \ | \ - \lambda_n, - \lambda_{n-1},\ldots,-\lambda_1)$. 
Since $J^2=id$, this implies ${\mathcal L}^\vee \cong ({\mathcal L}^\vee)^* \cong {\mathcal L}^J$ for maximal atypical
irreducible objects and proves our claim.
\end{proof}

 Hence
for selfdual ${\mathcal L}$ we obtain from $\mathcal{L}^J \cong \mathcal{L}^{\vee}$ a composite isomorphism $${\mathcal L}^J \cong {\mathcal L}^\vee \cong
{\mathcal L}\ .$$ Note that ${\mathcal L}^J$ and ${\mathcal L}$ share the same
underlying representation space. By the isomorphism ${\mathcal L}^J \cong {\mathcal L}$  there exists an automorphism $\phi$ of this vectorspace
such that $\rho_{\mathcal L}^J(X)=\phi \circ \rho_{\mathcal L}(X) \circ \phi$ holds. There is a unique choice for $\pm \phi$ if we normalize $\phi$ by a scalar such that $\phi^2 = id$ holds. Fix such $\phi$ (unique up to $\pm 1)$. Then this extends
the representation ${\mathcal L}$ to a representation of the "semidirect product"
$$   \mathfrak{gl}(n\vert n) \cdot \langle \phi \rangle \ $$

resp. $GL(n\vert n) \cdot \rangle \phi \rangle$ such that $(g_1,\phi^i)(g_2,\phi^j)=
(g_1 g_2^{J^i}, \phi^{i+j})$.
Notice that $J$ acts on $H^\bullet_D(L)$ (reversing the degrees) since $J$ 
commutes with the Dirac operator $D=\partial + \overline\partial$ (see section \ref{sec-Dirac}). 
The pairing $(.,.)_{\mathcal L}: {\mathcal L} \times  {\mathcal L}  \to {\bf 1}$ descends to $H_D({\mathcal L})$
since $(closed,exact)_{\mathcal L}=0$. However this make sense only 
for $D=D_s: \mathcal{T}_n^> \to \mathcal{T}_{n-s}^>$ for even integers $s$,
although it makes sense for the pairing of $L\times L \to B^r$ for all $s$.
This fact will be important below.
Furthermore,
for the $G$-invariant pairing $(.,.)_{\mathcal L}: {\mathcal L}\times {\mathcal L} \to {\bf 1}$ 
on ${\mathcal L}$, $(\phi(.),\phi(.))_{\mathcal L}$ is a $G$-invariant pairing as well. 
Again $\phi$ only makes sense for $D=D_s$ and even $s$.
Hence by Schur's lemma $(\phi(v),\phi(w))_{\mathcal L} = \varepsilon_\phi
\cdot (v,w)_{\mathcal L}$ for some constant $\varepsilon_\phi$. Notice
$\varepsilon_\phi^2=1$ holds since $\phi^2=1$.

\begin{lem} $\varepsilon_{\mathcal L} = - \varepsilon(X_\lambda)$.
\end{lem}

\begin{proof} Since $\varepsilon_{\mathcal L} = (-1)^r \varepsilon(X_\lambda)$ and 
$r$ now is odd, the assertion follows.
\end{proof}

\subsection{Case 2 continued: Derivatives} We consider the spaced forest of $L$ (we prefer to write $L$ instead of ${\mathcal L}$ since the spaces are the same, only the parity $\varepsilon_L$ changes
to $\varepsilon_{\mathcal L}$). This spaced forest  ${\mathcal F}$ has the structure ${\mathcal F}
= {\mathcal F}_L ... d_{middle} ... {\mathcal F}_L^{dual}$.  Its first
derivative (symbolically) is 
$$   \partial{\mathcal F}
= \partial {\mathcal F}_L ... d_{middle} ... {\mathcal F}_L^{dual} \quad \cup
\quad {\mathcal F}
= {\mathcal F}_L ... d_{middle} ... \partial{\mathcal F}_L^{dual} \ .$$

We now use the following symbolic notation:
The derivative $\partial {\mathcal F}_L$ is given by the union
of the tree derivatives 
$\partial_{{\mathcal T}_i}{\mathcal F}_L$ 
for the trees $1,...,k/2$
of the left subforest $\partial {\mathcal F}_L$ of 
$\partial {\mathcal F}(\lambda)$.
We always focus on the derivative of the left middle tree ${\mathcal T}_{k/2}$ 
and its associated Tannakian quotient group 
of $H_{n-1}$. The irreducible subrepresentation 
$W_L^{middle}$ in $W_L$ associated to it
has multiplicity one in $V_L$ for  all even $n\geq 4$. 
Its dual representation $(W_L^{middle})^\vee$
is associated to the right middle derivate 
$\partial_{{\mathcal T}_{k/2 +1}}{\mathcal F_L^\vee}$.
Since this dual representation is nonisomorphic, the restriction of the
pairing $V_L\times V_L \to {\bf 1}$ on $V_L$ must be trivial on $W_L^{middle}$ and
on its dual $(W_L^{middle})^\vee$. Hence both are Lagrangian subspaces
of the nondegenerate pairing on $W_L^{middle}\oplus (W_L^{middle})^\vee$. 
The Lagrangian property will be substantial for the subsequent argument.
So keep in mind, when we symbolically write $\partial {\mathcal F}_L$, we actually mean $W_L^{middle}$
and by abuse of notation ignore all other contributions.

Consider the super fibre functor  $\omega= H_D(L)$ for $D=D_n$ and $V_\lambda = V_L = \omega(L) = H_D(L)$. 
In $V_{\lambda})$ this gives an orthogonal decomposition
$V_{\lambda} = W_L \oplus^\perp W_L^\vee $ where
$W_L$ is a Lagrangian subspace.
This decomposition comes from considering $H_D(L)$ for  $D=D_1$ as
a representation in $\mathcal{T}_{n-1}$. It induces a corresponding decomposition of $\overline{\mathcal{T}}_{n-1}$, since the pairing is inherited. Notice $V_L = H_D({\mathcal L})=H_D(L)$ as a vectorspace (not as a super vectorspace
or representation). 
Hence this induces a decomposition
of $V_L$, as a representation of $H_{n-1}$ into orthogonal Lagrangian subspaces
$V_{L,left}$ and $V_{L,right}$. As a representation of $H_{n-1}$ we have
$V_{L,left}^\vee \cong V_{L,right} \otimes B^r$ such that $W_L^{middle}$ decomposes
into an orthogonal direct sum of
$W_L^{middle}\cap V_{L,left}$ and $W_L^{middle}\cap V_{L,right}$.
Both nonisomorphic summands define irreducible representations of  the middle
type factor of the group $G_{n-1}$ (the connected derived group).

\bigskip\noindent
Now we consider the second derivative, i.e the restriction of $V_L=V_{\lambda}=V_{\mathcal L}$ to the group $H_{n-2}$. 
The morphism $\phi=\phi_{n-2}$ is now defined and flips the
two Lagrangian subspaces coming from the study of the first derivative.
These two Lagrangians subspaces $V_{L,left}$ and $V_{L,right}$ remain orthogonal
subspaces for the twisted pairing $(.,.)_{{\mathcal L}_{n-2}}$ and of course 
are invariant under the action
of $H_{n-2}\subseteq H_{n-1}$. However 
now
$\phi_{n-2}$ switches the two orthogonal Lagrangian subspaces.

\begin{lem} \label{lemma-7} The involution $\phi=\phi_{n-2}$ defines
an isomorphism of vector spaces $$\phi_{n-2} : V_{L,left} \to V_{L,right}\ ,$$
i.e. flips the two orthogonal Lagrangians.
\end{lem}

\begin{proof}
 Indeed, $\phi_{n-1}$ is welldefined if we only consider the subgroups
generated by $SL(n-1\vert n-1)$ , and then $G_{n-1}$ acts
on the irreducible summands $W_L^{middle}\cap V_{L,left}$ and $W_L^{middle}\cap V_{L,right}$ nonisomorphically and these are  irreducible representations of  the middle
type factor of the group $G_{n-1}$. If $SL(n-1\vert n-1)$ acts by $\pi$ on $W_L^{middle}\cap V_{L,left} \subset  H_{D_1}(L)$, then it acts on $\phi(W_L^{middle}\cap V_{L,left})$
by the dual representation, which is not isomorphic and irreducible and hence orthogonal
to $W_L^{middle}\cap V_{L,left}$. By Schur's lemma the $H_{n-1}$-component is there
therefore uniquely determined as the $\pi^J$-isotypic component. Since
$\pi^J \cong \pi^\vee$, we obtain 
for the induced action of $G_{n-1}$ on $V_L$ that $\phi_{n-1}$ (defined on the level
of $SL(n-1\vert n-1)$) switches 
the two irreducible summands $W_L^{middle}\cap V_{L,left}$ and $W_L^{middle}\cap V_{L,right}$ of $G_{n-1}$
[The ladder type representations, where $H_{D_1}(L)$ may not be
multiplicity free, do not cause a problem unless $n=2$ since we consider only the middle derivative components].
\end{proof}

\subsection{Case 2: Second derivatives} Now let us consider the second derivatives.
Besides $\partial^2{\mathcal F}_L ... {\mathcal F}_L^{dual}$
and ${\mathcal F}_L ... \partial^2{\mathcal F}_L^{dual}$ (again symbol writing) it 
produces spaced forests that appear with multiplicity two of type   
$$ \partial{\mathcal F}_L ... \partial{\mathcal F}_L^{dual}$$
(again symbolic writing with focus only on the second middle derivative). One of them,
the right middle derivative of the left middle derivative of ${\mathcal F}_L$, defines an irreducible  constituent subspace $W_{n-2}$
of $V_{L,left}$, the other one, the left middle derivative of the right middle derivative of ${\mathcal F}_L$, defines an irreducible constituent subspace $W_{n-2}^\vee$ of $V_{{\mathcal L},right}$. 

\begin{example}\phantom{ddd}
  \bigskip

\begin{center}

 \scalebox{0.6}{
\begin{tikzpicture}
\foreach \x in {-9,-8,-5,-4,2,3,6,7} 
     \draw[very thick] (\x-.1, .1) -- (\x,-0.1) -- (\x +.1, .1);
\foreach \x in {-7,-6,-3,-2,-1,0,1,4,5,8,9} 
     \draw[very thick] (\x-.1, -.1) -- (\x,0.1) -- (\x +.1, -.1);
%


\draw[very thick] [-,black,out=90, in=90](-9,+0.2) to (-6,+0.2);
\draw[very thick] [-,black,out=90, in=90](-8,+0.2) to (-7,+0.2);
\draw[very thick] [-,black,out=90, in=90](-5,+0.2) to (-2,+0.2);
\draw[very thick] [-,black,out=90, in=90](-4,+0.2) to (-3,+0.2);
\draw[very thick] [-,black,out=90, in=90](3,+0.2) to (4,+0.2);
\draw[very thick] [-,black,out=90, in=90](2,+0.2) to (5,+0.2);
\draw[very thick] [-,black,out=90, in=90](6,+0.2) to (9,+0.2);
\draw[very thick] [-,black,out=90, in=90](7,+0.2) to (8,+0.2);


\end{tikzpicture} }
\smallskip

\text{The cup diagram of $\lambda$}
\end{center}

After deriving the second respectively third tree in this diagram, one obtains the cup diagrams

  \bigskip

\begin{center}

 \scalebox{0.6}{
\begin{tikzpicture}
\foreach \x in {-9,-8,-4,2,3,6,7} 
     \draw[very thick] (\x-.1, .1) -- (\x,-0.1) -- (\x +.1, .1);
\foreach \x in {-7,-6,-5,-3,-2,-1,0,1,4,5,8,9} 
     \draw[very thick] (\x-.1, -.1) -- (\x,0.1) -- (\x +.1, -.1);
%


\draw[very thick] [-,black,out=90, in=90](-9,+0.2) to (-6,+0.2);
\draw[very thick] [-,black,out=90, in=90](-8,+0.2) to (-7,+0.2);
\draw[very thick] [-,black,out=90, in=90](-4,+0.2) to (-3,+0.2);
\draw[very thick] [-,black,out=90, in=90](3,+0.2) to (4,+0.2);
\draw[very thick] [-,black,out=90, in=90](2,+0.2) to (5,+0.2);
\draw[very thick] [-,black,out=90, in=90](6,+0.2) to (9,+0.2);
\draw[very thick] [-,black,out=90, in=90](7,+0.2) to (8,+0.2);


\end{tikzpicture} }

\text{Derivative of the 2nd sector of $\lambda$}
\end{center}

  \bigskip

\begin{center}

 \scalebox{0.6}{
\begin{tikzpicture}
\foreach \x in {-9,-8,-5,-4,3,6,7} 
     \draw[very thick] (\x-.1, .1) -- (\x,-0.1) -- (\x +.1, .1);
\foreach \x in {-7,-6,-3,-2,-1,0,1,2,4,5,8,9} 
     \draw[very thick] (\x-.1, -.1) -- (\x,0.1) -- (\x +.1, -.1);
%


\draw[very thick] [-,black,out=90, in=90](-9,+0.2) to (-6,+0.2);
\draw[very thick] [-,black,out=90, in=90](-8,+0.2) to (-7,+0.2);
\draw[very thick] [-,black,out=90, in=90](-5,+0.2) to (-2,+0.2);
\draw[very thick] [-,black,out=90, in=90](-4,+0.2) to (-3,+0.2);
\draw[very thick] [-,black,out=90, in=90](3,+0.2) to (4,+0.2);
\draw[very thick] [-,black,out=90, in=90](6,+0.2) to (9,+0.2);
\draw[very thick] [-,black,out=90, in=90](7,+0.2) to (8,+0.2);


\end{tikzpicture} }

\text{Derivative of the third sector of $\lambda$}
\end{center}

Taking now either the derivative of the 2nd sector of the  derivative of the 3rd sector or the derivative of the 3rd sector of the derivative of the 2nd sector gives the cup diagram

  \bigskip

\begin{center}

 \scalebox{0.6}{
\begin{tikzpicture}
\foreach \x in {-9,-8,-4,3,6,7} 
     \draw[very thick] (\x-.1, .1) -- (\x,-0.1) -- (\x +.1, .1);
\foreach \x in {-7,-6,-5,-3,-2,-1,0,1,2,4,8,9} 
     \draw[very thick] (\x-.1, -.1) -- (\x,0.1) -- (\x +.1, -.1);
%


\draw[very thick] [-,black,out=90, in=90](-9,+0.2) to (-6,+0.2);
\draw[very thick] [-,black,out=90, in=90](-8,+0.2) to (-7,+0.2);
\draw[very thick] [-,black,out=90, in=90](-4,+0.2) to (-3,+0.2);
\draw[very thick] [-,black,out=90, in=90](3,+0.2) to (4,+0.2);
\draw[very thick] [-,black,out=90, in=90](6,+0.2) to (9,+0.2);
\draw[very thick] [-,black,out=90, in=90](7,+0.2) to (8,+0.2);


\end{tikzpicture} }
\smallskip

\text{The cup diagram of $W_{n-2} \cong W_{n-2}^{\vee}$}
\end{center}

\end{example}

Notice that $W_{n-2}\cong W_{n-2}^\vee$ as irreducible representations
of $G_{n-2}$ on $V_{\mathcal L}=H_{D_{n-2}}({\mathcal L})$. Furthermore
the isotypic component of this irreducible representation in $V_{\mathcal L}$
consist precisely of these two irreducible orthogonal constituents. 
By the last Lemma
$\phi=\phi_{n-2}$ hence switches
these two {\it Langragian subspaces} $W_{n-2}$ and $W_{n-2}^\vee$
of the nondegenerate restriction of the pairing to $W_{n-2}\oplus W_{n-2}^\vee$. 

Considered as a module under the semidirect product of $G_{n-2}$ and $\phi_{n-2}$ this isotypic subspace remains stable and defines a representation of the semidirect
product $G_{n-2} . \langle \phi_{n-2}\rangle$. 

\begin{lem} \label{lemma-8} As module of $G^*_{n-2}=G_{n-2} . \langle \phi_{n-2}\rangle$
the isotypic subspace   $Y=W_{n-2}\oplus W_{n-2}^\vee$
contains each irreducible constituent with multiplicity 1.
\end{lem}

\begin{proof} We may assume that $Y$ is not irreducible for $G_{n-2}^*$.
Then $Y\cong W' \oplus W''$ with $Res(W')\cong Res(W'')\cong W_{n-2}$
as a representation of $G_{n-2}$. We have to  show $W'\not\cong W''$.
Otherwise $Y\cong W' \oplus W' = 2W'$ and for every pair of constants
$(\alpha,\beta)\neq 0$ and $0\neq w= (w',0) \in W$ the representation space $W_{\alpha,\beta}$ spanned  
by $(\alpha w', \beta w')$ is isomorphic to $W'$ and hence invariant under $\phi_{n-2}$.
Since $Hom_{G_{n-2}}(Res(W'), Res(W))=2$, for suitable choice of $(\alpha,\beta)\neq 0$ the subspace $W_{\alpha,\beta}$
must coincide with the subspace $W_{n-2}$ as a vectorspace and representation space
of $G_{n-2}$ (Schur's lemma). This would imply that also the subspace $W_{n-2}$, obtained for some special choice of $(\alpha,\beta)$, is  stable
under $\phi_{n-2}$. Since $W_{n-2} \subset L_{L,left}$, $\phi_{n-2}$ flips into the orthogonal subspace $W_{n-2}^\vee$ of $Y$ by lemma \ref{lemma-7}, this gives a contradiction and hence implies
$W'\not\cong W''$. 
\end{proof}

\begin{cor}\label{corollary-2}
The parity of the pairing
$(.,.)_{{\mathcal L}_n}$, or equivalently of $(.,.)_{{\mathcal L}_{n-2}}$,
is therefore inherited to each of the non-isomorphic irreducible constituents
(by multiplicity one).
\end{cor}


\begin{cor} \label{corollary-3} In the situation where $d_{middle}$ is odd, we have 
$ \varepsilon_{W_{n-2}} = \varepsilon_{\mathcal L}$ for the irreducible
module $W_{n-2}$ of $GL(n-2\vert n-2)$ attached to the second 
middle derivative $\partial^2_{middle}\lambda$ of $\lambda$, i.e 
$W_{n-2}:= {\mathcal L}''$ for the normalized representation ${\mathcal L}_{n-2}={\mathcal L}''$ in ${\mathcal{T}}_{n-2}^{ext}$ that is attached to $L(\partial^2_{middle}\lambda)$ in $\mathcal{T}_{n-2}$.
\end{cor}

So for short, let $\varepsilon_{\mathcal L}$ and $\varepsilon_{{\mathcal L}_{n-2}}$ denote the parities of the parings $(.,.)_{\mathcal L}$ and of $(.,.)_{{\mathcal L}''}$, then
we obtain from corollary \ref{corollary-3}
$$ \varepsilon_{\mathcal L}  = \varepsilon_{{\mathcal L}_{n-2}}.$$
Similarly let $\varepsilon_\lambda$ and $\varepsilon_{\lambda_{n-2}}$  denote the parities
of the pairings on $L(\lambda)$ and $L(\lambda_{n-2})$ for 
$\lambda_{n-2}=\partial^2_{middle}\lambda$.
Then we conclude  \begin{align*} \varepsilon_{\lambda}  = \pm 1 & \Longrightarrow
\varepsilon_{\mathcal L} = \pm (-1)^{n/2} \Longrightarrow
\varepsilon_{{\mathcal L}_{n-2}} =  \pm (-1)^{n/2}  \\ & \Longrightarrow  
\varepsilon_{\lambda_{n-2}} =  \pm (-1)^{n/2} (-1)^{(n-2)/2} = \mp 1\ .\end{align*}
The first shift comes from the normalization factor
$(-1)^{nr/2} = (-1)^{n/2}$ that arises from the passage from $L$ to ${\mathcal L}$
since $r$ is odd and $n$ is even (see \cite{Scheunert}) by twisting with the invertible object $Ber^{-r/2}$ of parity $(-1)^{nr/2}$. The second assertion comes from corollary \ref{corollary-3}. The last
assertion comes from the passage from ${\mathcal L}_{n-2}$ to $\tilde\lambda= \lambda_{n-2}$ again by twisting with a Berezin power.

\subsection{Case 2: Final step}
We now reformulate this in terms of the parities for the objects $X_{\lambda}\in \mathcal{T}_n^+$
and $X_{\tilde\lambda} \in \mathcal{T}_{n-2}^+$. 
This passage produces an extra
factor $$(-1)^{p(\lambda)} = (-1)^{p(\lambda_{basic}) + D(\lambda)}\ .$$
Now $D(\lambda) = nr/2$ is congruent to 1 resp. 0 mod 2 for $n= 2$ mod 4
resp. $n=4$ mod 4. On the other hand $p(\lambda_{basic})=n(n-1)/2 - ht({\mathcal F}_\lambda) = n(n-1)/2$ mod 2 since $ht({\mathcal F}_\lambda)=2 ht({\mathcal F}_L)$
is even. For even $n$ we have $n(n-1)/2 = 1$ mod 2 resp. 0 mod 2
if $n= 2$ mod 4
resp. $n=4$ mod 4. Hence $  2 \vert  p(\lambda)$ and
$$   \varepsilon(X_\lambda) = \varepsilon_\lambda \ .$$
Finally we have $\varepsilon_{\lambda_{basic}} = (-1)^{p(\lambda_{basic})} = (-1)^{n(n-1)/2}$
from lemma \ref{lemma-3}. Similarly $\varepsilon_{\tilde\lambda_{basic}} = (-1)^{(n-2)(n-3)/2}$.
Thus
$  \varepsilon_{\lambda_{basic}} = - \varepsilon_{\tilde \lambda_{basic}} $ implies

\begin{cor}\label{corollary-4} The quotients $\varepsilon(X_\lambda)/\varepsilon(X_{\lambda_{basic}})$
and $\varepsilon(X_{\tilde\lambda})/\varepsilon(X_{\tilde\lambda_{basic}})$ coincide
in the situation of our second assumption ($r$ odd, $n$ even).
\end{cor}

This being said, we immediately observe that the property
$2\vert d_{middle}(\lambda)$ of $\lambda$ is inherited by $\tilde\lambda=\lambda_{n-2}$. 
Also $r$ remains the same. Hence by decending induction on even $n$ we conclude

\begin{thm}\label{parity-thm} For all irreducible objects $X_\lambda$ in $\mathcal{T}_n^+$ we have
$$  \varepsilon(X_\lambda) = \varepsilon(X_{\lambda_{basic}})  $$
and $ \varepsilon(X_\lambda) = (-1)^{p(\lambda_{basic})}$.
\end{thm}
   
\begin{proof} This holds for $n=2$ by \cite{Heidersdorf-Weissauer-GL-2-2} (see section \ref{sec:ind-start}), hence follows inductively 
from corollary \ref{corollary-2} for all even $n$ with odd $r$. The remaining cases, where
$d_{middle}$ is even or ${\mathcal F}_{middle}\neq \emptyset$, are covered by
lemma \ref{lemma-4}. 
\end{proof}

\goodbreak



\section{Technical lemmas on derivatives and superdimensions} \label{technical-lemmas}

\subsection{Derivatives}

\begin{lem} \label{trivial-rep} Suppose $L$ is
a simple module and suppose the trivial module ${\bf 1}$ is a 
constituent in $H^0_D(L)$, then $L\cong {\bf 1}$.
\end{lem}

\medskip
{\it Proof}.  Suppose
$H^0_D(L)$ contains ${\bf 1}$ and suppose $L\not\cong {\bf 1}$.
Then theorem \ref{mainthm} implies that $L$ has two sectors with sector structure
$[-n+2,...,0,1,...,n-1] S_1  $ and $r(S_1)=1$, hence $$ L \cong  Ber \otimes S^i\ $$
for some $i\geq n-1$,  or  has sector structure
$S_2 [-n+2,...,0,1,...,n-1]$ with $r(S_2)=1$ and hence $$ L \cong (Ber \otimes S^i)^\vee \ $$
for some $i\geq n-1$. 
However $H^\nu_D(Ber \otimes S^i) \cong Ber \otimes H^{\nu-1}_D(S^i) $
vanishes unless $\nu-1=0$ with $H^{0}_D(S^i) =S^i$ or $\nu-1= i - (n-1) \geq 0$, as follows 
from the next lemma \ref{co}.
 Hence this implies
$H^0_D(Ber \otimes S^i)=0$. Similarly then also $H^0_D((Ber \otimes S^i)^\vee)=0$ holds
by  duality. This contradiction proves our claim. \qed

\begin{lem} \label{co} Suppose $i\geq 1$. Then for $S^i$ in $\calR_n$ the cohomology is 
$H^\nu(S^i) =S^i$ for $\nu=0$ and $H^\nu(S^i) = Ber^{-1}$ for $\nu= max(0,i-n+1)$, and 
$H^\nu(S^i)$ is zero 
otherwise. 
\end{lem}

\medskip
{\it Proof}. An easy consequence of theorem \ref{mainthm} and \cite[Proposition 22.1]{Heidersdorf-Weissauer-tensor}. \qed

\medskip\noindent

The following lemma is an immediate consequence of theorem \ref{mainthm} or lemma \ref{trivial-rep}.

\begin{lem} \label{trivial-occurence} $DS(L(\lambda))$ has a summand of superdimension 1 only if $L(\lambda) \cong Ber^r \otimes S^i$ for some $r,i$.
\end{lem}

\medskip\noindent
Recall that an irreduble representation is weakly selfdual (or of type (SD)) if $L(\lambda)^{\vee} \cong Ber^r \otimes L(\lambda)$ for some $r \in \Z$.

\begin{lem} \label{weird1}
A (weakly) selfdual irreducible object $L=L(\lambda)$ with odd superdimension $\sdim(L)$ is  a power of the Berezin determinant.
\end{lem}

\medskip
{\it Proof}. For (weakly) selfdual maximal atypical irreducible objects $L=L(\lambda)$ of odd dimension their plot has sectors $S_1,\cdots S_k$ from left to right 
of lengths say $2r_1,...,2r_k$ that must satisfy
$$r_{k+1-i} = r_i \ $$ and hence in particular $r_1=r_k$. By \cite{Weissauer-GL}\cite{Heidersdorf-Weissauer-tensor} the superdimension is divisible by the multinomial coefficient $n!/(\prod_i r_i!)$ for $n=\sum_i r_i$. Hence, in case $k\geq 2$, the superdimension is divisible by the integer $(r_1+r_k)!/(r_1! r_k!)$, which is  $(2r_1)!/(r_1)!(r_1)!$ and hence {\it even}. Therefore $\sdim(L)\notin 2\mathbb Z$ implies $k=1$, i.e.
the associated plot only has a single sector. For this sector, we may continue with the same argument using the recursion formula for the superdimension given in \cite{Weissauer-GL}\cite{Heidersdorf-Weissauer-tensor}. \qed

\begin{lem} \label{sdim-compared} Let $L(\lambda)$ be a maximal atypical weight with $k$ sectors of rank $r_1,\ldots, r_k$ and derivatives $L(\lambda_j), \ j =1,\ldots,k$. Then for all $j=1,..,k$
\[ \sdim(L(\lambda))  = \sdim(L(\lambda_j)) \cdot \frac{n}{r_j} \ .\]
\end{lem}

\medskip
{\it Proof}. By the superdimension formula \cite{Heidersdorf-Weissauer-tensor}
$$  \sdim(L(\lambda)) = {n \choose r_1,...,r_k} \cdot T(S_1,...,S_k) $$
for a term $T(S_1,...,S_k)$ that only depends on the sektors $S_j$ such that
\[   T(S_1,...,S_k) = T(S_1,...,\partial S_j , ..., S_k) \ .\]
Since
\[  \sdim(V_j) = {n-1 \choose r_1,..., r_j-1,...,r_k} T(S_1,...,\partial S_i , ..., S_k) \ ,\]
this implies for all $j=1,..,k$
$$  \sdim(L(\lambda))  = \sdim(L(\lambda_j)) \cdot \frac{n}{r_j} \ .$$\qed


\subsection{Small superdimensions} \label{small-superdimensions}

According to lemma \ref{small} a small representation belongs to one of four infinite families of regular cases or to a finite list of exceptional cases. The largest dimension occuring in the exceptional cases is 64 (the spin representations of $D_7$). Assume that $V_{\lambda}$ restricted to $G_{\lambda}$ splits as $V_{\lambda} = W_1 \oplus \ldots \oplus W_s$. We may assume $\dim(W_1)\leq \frac{1}{s}\dim(V_{\lambda})$. The rank estimates in section  \ref{mackey-clifford} show that $W_1$ belongs to the regular cases of lemma \ref{small} if $s \geq 3$. We therefore consider here the case where $V_{\lambda}$ restricted to $G_{\lambda}$ splits into at most two representations $V_{\lambda} = W \oplus W^{\vee}$. We want to rule out that $W$ or $W^{\vee}$ is one of the exceptional cases. The dimension of $W$ is $\dim(V_{\lambda})/2$. Therefore we compute all superdimensions of irreducible weakly selfdual representations up to superdimension 128. Except for the numbers 20 and 56 none of the exceptional dimensions is equal to either the superdimension or half the superdimension of an irreducible weakly selfdual representation in $\mathcal{T}_n^+$.

\begin{lem} \label{basic-lift} If $[\lambda]$ is a basic representation of $\mathcal{T}_n^+$, then $[\lambda,0]$ is a basic representation of $\mathcal{T}_{n+1}^+$ of the same superdimension. Every basic representation of $\mathcal{T}_{n+1}^+$ with one sector is of this form.
\end{lem}

Therefore we can always assume that the irreducible representations have at least two sectors. Note also that a weakly selfdual representation cannot have an even number of sectors if $n$ is odd. For a list of the basic representations in the case $n=3$ and $n=4$ we refer to the examples in section \ref{sec:physics}.

\subsubsection{Basic selfdual weights for $n=5$} 

\begin{align*}
& [4,3,2,1,0], \ &\sdim 120; \hspace{1.5cm} &[3,3,2,0,0], &\sdim \ 30 \\[-0.2ex]
&[4,1,1,1,0], &\sdim \ 20; \hspace{1.5cm}  &[1,0,0,0,0], &\sdim \ 2 \\[-0.2ex]
&[2,1,0,0,0], &\sdim \ 6; \hspace{1.5cm}  &[3,2,1,0,0], &\sdim \ 24 \\[-0.2ex]
&[2,2,0,0,0], &\sdim \ 6;  \hspace{1.5cm} &[3,1,1,0,0], &\sdim \ 12
\end{align*}

\subsubsection{Basic selfdual weights for $n=6$}

By lemma \ref{basic-lift} we can focus on the case of two or more sectors. These basic weights are listed below.
\begin{align*}
&[5,4,3,2,1,0], \ &\sdim \ 720; \hspace{1.5cm}  &[3,3,3,0,0,0], &\sdim \ 20 \\[-0.2ex]
&[4,3,3,1,0,0], &\sdim \ 80; \hspace{1.5cm} &[5,1,1,1,1,0], &\sdim \ 30 \\[-0.2ex]
&[4,4,2,2,0,0], &\sdim \ 90; \hspace{1.5cm} &[5,4,2,2,1,0], &\sdim \ 360 \\[-0.2ex]
&[5,3,3,1,1,0], &\sdim \ 180;  \hspace{1.5cm} &[4,3,2,2,1,0], &\sdim \ 180 \\
\end{align*}

\subsubsection{Basic selfdual weights for $n=7$}

By lemma \ref{basic-lift} we can focus on the case of two or more sectors. Since $n$ is odd, a weakly selfdual weight cannot have an even number of sectors. If the weight has $\geq 5$ sectors, its superdimension exceeds 128. Therefore we list the basic SD weights with 3 sectors. 
\begin{align*}
&[4,4,4,3,0,0,0], \  &\sdim \ 140; \qquad &[5,5,2,2,2,0,0], & \sdim \ 210 \\[-0.2ex]
&[6,1,1,1,1,1,0],  &\sdim \  30; \qquad  &[6,3,3,1,1,1,0], & \sdim \ 252 \\[-0.2ex]
&[6,4,3,2,1,1,0],   &\sdim \ 1008 \qquad & & 
\end{align*}

\subsubsection{Basic selfdual weights for $n \geq 8$}

If the weight has 2  sectors for $n \geq 8$, then the smallest possible superdimension is $\geq n!/ ((n/2)! (n/2)!)$. This equals the case $[\lambda] = [n/2,n/2,\ldots,n/2,0,0,\ldots,0]$ (each $n/2$ times). For $n=8$ the superdimension is then $70$, for $n=9$ it is already 252. All other weights with 2 sectors have superdimension $> 128$.

\medskip\noindent

If the weight has 3 sectors and $n \geq 9$, the smallest superdimension is given by the hook weight $[n-1,1,\ldots,1,0]$ of superdimension $n(n-1)$. The next smallest superdimension is given by the irreducible representation $[n-1,2,1,\ldots,1,0]$ of superdimension $2 \cdot n(n-1)$. For $n=8$ these superdimensions are 56 and 112. For $n\geq 9$ the second case has superdimension larger than $128$. In the first case the superdimensions are 72 ($n=9$), 90 ($n=10$), 110 ($n=11$) and exceed 128 otherwise.

\medskip\noindent
If $n \geq 8$ and the weight has $\geq 4$ sectors, its superdimension exceeds 128. 

\subsubsection{Comparison with the exceptional cases}

We compare the superdimensions above with the dimensions of the exceptional cases in lemma \ref{small}. Except for the cases where the superdimension is 20 or 56 the dimensions are different. If the dimension is 20, then the irreducible representation is $\Lambda^3(st)$ for $SL(6)$. If the dimension is 56, then the irreducible representation is either $\Lambda^3(st)$ for $SL(8)$ or the irreducible representation of minimal dimension of $E_7$. 

\medskip\noindent
If $V_{\lambda}$ or $W$ is of the form $\Lambda^3(st)$, then so is its restriction $Res(W)$ to $G_{\lambda'}$ since $\Lambda^3$ commutes with $Res$, in contradiction to the induction assumption. 

\medskip\noindent 
In the $\dim = 56$-case with $V_{\lambda}$ irreducible upon restriction to $G_{\lambda}$, the corresponding $L(\lambda)$ is the hook weight $[n-1,2,1,\ldots,1,0,\ldots,0]$ for $n \geq 8$. For $n = 8$ \[ DS(L(\lambda)) \cong Ber \otimes S^6 \oplus (Ber \otimes S^6)^{\vee} \oplus [7,1,1,1,1,1,0].\] The connected derived Tannaka group of $[7,1,1,1,1,1,0]$ is either $SO(42)$,  $Sp(42)$ or $SL(24)$ and doesn't embed into $E_7$. If $n \geq 9$, the hook weight $[n-1,2,1,\ldots,1,0,\ldots,0]$ has one sector and therefore one derivative, hence the corresponding Tannaka group contains either $SO(42)$, $Sp(42)$ or $SL(24)$.

\medskip\noindent
If $V_{\lambda}$ decomposes as $W \oplus W^{\vee}$ and $\dim(W) = 56$, then $\dim(V_{\lambda}) = 112$. This happens for $L(\lambda) \cong [n-1,2,1,\ldots,1,0]$ for $n=8$. In this case \[ DS(L(\lambda)) \cong [7,2,1,\ldots,1] \oplus [7,2,1,\ldots,1]^{\vee}.\] Since this weight is NSD, its connected derived Tannaka group is $SL(56)$ which doesn't embed into $E_7$.

\subsubsection{The regular cases}

We can now assume that we are in one of the regular cases of lemma \ref{small}. If $V$ is either $S^2(st), \ S^2(st^{\vee}), \ \Lambda^2(st), \  \Lambda^2(st^{\vee})$ or the nontrivial irreducible representation of $\Lambda^2(st)$ in the $C_r$-case, we get a contradiction to the induction assumption since restriction commutes with Schur functors. Therefore the representation is a standard representation or its dual for type $A,B,C,D$.

\begin{cor} If the selfdual irreducible representation $V_{\lambda}$ is irreducible upon restriction to $G_{\lambda}$ or splits in the form $W \oplus W^{\vee}$, the group $G_{\lambda}$ is a simple group of type ABCD and $V_{\lambda}$ and $W$ are its standard representation (or its dual.) 
\end{cor}



\section{Clean decomposition} \label{appendix-1}

The ambiguity in the determination of $G_{\lambda}$ is only due to the fact that we cannot exclude special elements with $2$-torsion in $\pi_0(H_n)$. We show that $I \cong \one$ if $I \otimes I^{\vee} \simeq \one \oplus Proj$ holds. We then discuss the occurence of projective summands in tensor products of irreducible modules and show that $I \cong \one$ in some cases for $n = 4$.

\subsection{The module $I$}

The object $I$ from section \ref{exc-sd} that realizes the surjective projection
$H_\lambda = \mu_2 =\langle w\rangle$ corresponds to an element with the following properties.

\begin{lem} \label{the-module-I} The module $I$ is indecomposable with $\sdim(I)= 1$. It satisfies $I^\vee \cong I$ and $I^* \cong I$. Moreover:
\begin{enumerate}
\item There exists an irreducible object $L$ of $\mathcal{T}_n^+$.
such that $I$ occurs (with multiplicity one ) as a direct summand in $L\otimes L^\vee$.
\item $L \otimes I \cong L \oplus N$ for some negligible object $N$.
\item $DS(I) $ is $\tilde{I} \oplus \text{ negligible }$ for an indecomposable object $\tilde{I}$ concentrated in degree 1 of superdimension 0 satisfying $\tilde{I}^{\vee} = \tilde{I}$ and $\tilde{I}^* \cong \tilde{I}$. If we assume the stronger structure theorem for $n-1$ by induction, $DS(I) $ is ${\bf 1}$ plus some negligible object.
\end{enumerate} 
\end{lem}

\begin{proof} That $I$ is indecomposable of superdimension 1 is obvious. For 2) notice that by our analysis in Clifford theory in section \ref{proof-derived} we got \[ V_{\lambda} \cong Ind_{H_1}^{H} (W) \] for a subgroup $H_1$ of index 2 between $H^0$ and $H$. Here $H_1$ is simply $ker(\varepsilon)$. This implies that $I$ appears as an indecomposable constituent
of  superdimension 1 in $L(\lambda)\otimes L(\lambda)^{\vee}$ that is
not isomorphic to the trivial representation. Since $Ind_{H_1}^H (W)$ is irreducible, Clifford theory implies that $V_{\lambda} \otimes \mu \cong V_{\lambda}$ which implies 2).
Since $(L\otimes L^\vee)^{\vee} \cong L \otimes L^\vee$ and $I^{\vee}$ cannot be isomorphic to the trivial representation, this implies $I^\vee \cong I$, and similarly $(L\otimes L^\vee)^* \cong L \otimes L^\vee$ implies $I^* \cong I$.
Property 3) follows since $I$ is selfdual of superdimension 1, and so its cohomology is concentrated in degree 0. By definition $I$ is a retract of $L\otimes L^\vee$, so $DS(I)$ is a retract of $DS(L) \otimes DS(L)^\vee$.  If we assume that the stronger structure theorem holds for $n-1$, the only summands of superdimension 1 in a tensor product $L(\lambda_i) \otimes L(\lambda_i)^{\vee}$ are Berezin powers by corollary \ref{1-dimen-summands}, hence $DS(I) \cong \one \oplus N$.  
\end{proof}

\begin{conj} \label{conj-I} $I \simeq \one$.
\end{conj}

We are unable to prove this result at the moment. This conjecture immediately implies that $V_{\lambda}$ stays irreducible under restriction to $G_{\lambda}$ and therefore would prove the stronger version of the structure theorem.

\subsection{Endotrivial modules.} Our condition $I^{\otimes 2} \simeq \one \oplus N$ resembles the definition of an endotrivial representation.

\begin{lem} \label{one-proj} The following conditions are equivalent:
\begin{enumerate}
\item $I^{\otimes 2} \simeq \one \oplus Proj$.
\item $DS(I) = \one$.
\end{enumerate}
\end{lem}

\begin{proof} If $DS(I) =\one$, then $\sdim(I) = 1$, hence $I^{\otimes 2} \simeq \one \oplus negl$. But $ker(DS)$ (restricted to $\mathcal{T}_n^+$) is $Proj$.
\end{proof}
  
Modules $M$ with the property $M \otimes M^{\vee} \simeq \one \oplus Proj$ are called endotrivial. If $I$ satisfies $I^{\otimes 2} \simeq \one \oplus Proj$ or equivalently $DS(I) \simeq \one$, $I$ is endotrivial (since $I^{\vee} \simeq I$).

\begin{thm} \cite{Talian-endotrivial} \label{thm:endotrivial} All endotrivial modules for $\mathcal{T}_n$ are of the form $Ber^{j} \otimes \Omega^i(\one) \oplus Proj$ or $\Pi (Ber^j \otimes \Omega^i(\one)) \oplus Proj$ for some $i,j \in \Z$ where $\Omega^i(\one)$ denotes the $i$-th syzigy of $\one$.
\end{thm}

We remark that we can split the projective resolution defining the $\Omega^i(M)$ into exact sequences \[  1 \to \Omega^i (M) \to P \to \Omega^{i-1}(M) \to 1 \]  with some projective object $P$. It follows $\sdim(\Omega^{i}(M) = - \sdim(\Omega^{i-1}(M)$ since $\sdim(P) = 0$.

\begin{lem} If $I^{\otimes 2} \simeq \one \oplus Proj$ with $I$ as above, then $I \simeq \one$.
\end{lem} 

\begin{proof} By restricting to $SL(n|n)$ we can ignore Berezin twists. By the classification of endotrivial modules $I \simeq \Pi^j \Omega^i (\one) \oplus Proj$ for some $i,j \geq 0$. Hence according to our list of properties of $I$ \[L \otimes \Pi^j \Omega^i (\one) \oplus Proj  \cong L \oplus Proj.\] On the other hand $\Omega^i(M) \otimes N \simeq \Omega^i(M \otimes N) \oplus Proj$ holds for all $N$ and $i$. Hence for $M \simeq \one$ we would have (using $\Pi \Omega^i(M)) = \Omega^i (\Pi M)$) \[ L \otimes \Pi^i \Omega^i(\one) \simeq \Omega^i(\Pi^j L) \oplus Proj \simeq L \oplus Proj\] which is absurd since $\Omega^i(\Pi^j L) \ncong L \oplus Proj$ for $i>0$ (the latter is impossible by \cite[Theorem 11.3]{HW-homotopy}). In fact using the short exact sequences \[  1 \to \Omega^i (M) \to P \to \Omega^{i-1}(M) \to 1 \] and $DS(Proj) = 0$ we obtain \[ H^l(\Omega^i(L)) \simeq H^{l+i}(L).\] Hence $i=0$ and so $I \simeq \Omega^0(\one) \simeq \one$.
\end{proof}


\subsection{Clean decomposition}  We say a direct sum is {\it clean} if none of the summands is negligible. We say a negligible module $N$ in $\mathcal{T}_n$ is {potentially projective of degree $r$} if $DS^{n-r}(N) \in \calT_r$ is projective and $DS^{i}(N)$ is not for $i \leq n-r$. 

\medskip

Now consider the special representations $S^i$. Then we proved in \cite{Heidersdorf-Weissauer-GL-2-2} the surprising fact that the projection of $S^i \otimes S^j$ or $S^i \otimes (S^j)^{\vee}$ on the maximal atypical block is clean. To prove the result we establish the $n=2$-case by a brute force calculation. The theory of mixed tensors \cite{Heidersdorf-mixed-tensors} then shows that the Loewy length of any summand in $S^i \otimes S^j$ is less or equal to 5. This implies the result since the Loewy length of a projective cover is $2n+1$.

\begin{lem} Every maximal atypical negligible summand in a tensor product $L(\lambda) \otimes L(\mu)$ is potentially projective of degree at least 3.
\end{lem}

\begin{proof} The decomposition of $S^i \otimes S^j$ in $\calR_2$ is clean. Further $DS$ sends negligible modules in $\mathcal{T}_n^+$ to negligible modules in $\mathcal{T}_{n-1}^+$ and the kernel of $DS$ on $\mathcal{T}_n^+$ consists of the projective elements. Since $DS^{n-2} (L(\lambda) \otimes L(\mu)) \in \calT_2$ splits into a direct sum of irreducible representations of the form $Ber^{j} S^{i}$ for some $i,j \in Z$ by the main theorem of \cite{Heidersdorf-Weissauer-tensor}, $DS^{n-2}(N) = 0$.
\end{proof}

We show below that the decomposition of the tensor product $L(\lambda) \otimes L(\mu)$ is also clean in the case $n=3$ unless $\lambda_{basic} = \mu_{basic} = (2,1,0)$ and that projective summands can occur only under strong restrictions in the case $n=4$.

\medskip
\textbf{Question}. Let $L(\lambda), \ L(\mu)$ be maximal atypical. Is the projection of the decomposition of $L(\lambda) \otimes L(\mu)$ on the maximal atypical block always clean?

\medskip

\begin{lem} If all tensor product decompositions $L(\lambda) \otimes L(\mu)$ for maximal atypical $L(\lambda)$, $L(\mu)$ are clean, $I \simeq \one$. 
\end{lem}

\begin{proof} We assume by induction that the stronger structure theorem is proven for $\mathcal{T}^+_{n-1}$. Then by Lemma \ref{the-module-I} we have $DS(I) \cong \one \oplus N$ for some negligible object $N$. Note that $DS(I)$ is a direct summand in $DS(L(\lambda)) \otimes DS(L(\lambda))^{\vee}$. If this tensor product decomposition is clean, we obtain $N = 0$, hence $DS(I) \cong \one$. But then Lemma \ref{one-proj} implies $I \otimes I^{\vee} \cong \one \oplus Proj$, hence $I$ is endotrivial. However the higher syzigies of Theorem \ref{thm:endotrivial} are not $*$-invariant and non-trivial Berezin twists are not selfdual, hence $I \cong \one$.
\end{proof}  

To prove that decompositons are always clean, it would be enough to prove that the tensor product of two irreducible maximally atypical representations never contains a maximally atypical projective summand since repeated applications of $DS$ to a negligible representation results in a direct sum of projective representations. A positive answer to this question would also imply that the tensor product decomposition of two maximal atypical irreducible representations behaves classically after projection to the maximal atypical block (and not just modulo vanishing superdimension).

\begin{example} \label{2-2-0-0} If $I$ is a direct summand in $[2,2,0,0]^{\otimes 2}$ as above, then $I \cong \one$. This follows from the $S^i$-computations. Consider $L = [2,2,0,0] \in \calR_n$. Then $DS(I) \simeq \one$. In fact $DS(L) = [2,2,0] + B^{-1}S^3$. Hence $DS(L) \otimes DS(L)$ is a tensor product involving only $S^i$'s and their duals (or their Berezin twists). Their decomposition is clean according to \cite{Heidersdorf-Weissauer-GL-2-2}. Hence any negligible module in $[2,2,0,0] \otimes [2,2,0,0]$ maps to zero under $DS$. In particular $DS(I) = \one$ and hence $I \simeq \one$.
\end{example}




\section{The depth of a representation} \label{appendix-depth}

Due to the results of appendix \ref{appendix-1} it is important to know when a maximal atypical projective module $P$ occurs in a tensor product of two irreducible modules $L(\lambda),L(\mu)$. Some conditions can be obtained from a restriction of $L(\lambda)$ and $L(\mu)$ to $\mathfrak{g}_0$ as we show in this appendix.

\subsection{Depths}
The restriction of a maximally atypical module $L$ of $\g=\mathfrak{gl}(n\vert n)$ to the classical subalgebra $\g_0$ decomposes completely
into a direct sum of irreducible $\g_0$-modules. We write $\rho \boxtimes \tilde\rho$ for these.
An irreducible representation $\rho$ of $\mathfrak{gl}(n)$ is described by a highest weight vector
$(\lambda_1,...,\lambda_n)$ with $\lambda_1 \geq ... \geq \lambda_n$, and we define the degree $deg(\rho)$ of $\rho$ to be the sum $\sum_{i=1}^n \lambda_i$.  

\medskip
\begin{itemize}
\item Let $a$ be the maximal degree $deg(\rho)$ for all $\rho \boxtimes \tilde\rho$ in the restriction of $L(\lambda)$
to $\g_0$, and 
\item let $b$ be the minimal degree $deg(\rho)$ for all $\rho \boxtimes \tilde\rho$ in the restriction of $L(\lambda)$
to $\g_0$. 
\end{itemize}

\medskip
Define the {\bf depth} to be $depth(L) = a-b $ or
$$  depth(L) = deg(\mbox{highest $\g_0$-weight of } L) - deg(\mbox{lowest $\g_0$-weight of } L)$$
We often write $depth(\lambda)$ for $depth(L(\lambda))$. Rather obviously we have
$$   depth( A \otimes B) = depth( A) + depth( B)  \ .$$
and $A \hookrightarrow B $ implies
$  depth(A) \leq depth(B) $.
Two remarks are in order:
\begin{itemize}
\item For every weight $\mu$ of the Cartan algebra in the representation space of the irreducible representation space of $\mathfrak{gl}(n)$ defined by the highest weight $\rho$ the degrees $deg(\mu) = deg(\rho)$ (defined as above) coincide.
\item For any irreducible representation $\rho\boxtimes \tilde\rho$ of $\g_0$ in the restriction of 
an irreducible max. atyp. representation $L$ of $\g$ one has $deg(\tilde \rho) = - deg(\rho)$.
\end{itemize}

If we consider the restriction of $L=L(\lambda)$, the maximal degree $deg(\rho)$ for all $\rho \boxtimes \tilde\rho$ in the restriction is $a=\sum_{i=1}^n \lambda_n = deg(L)$.
One easily shows
$$ depth(L) = a - b =  deg(L) - (-deg(L^\vee)) = deg(L) + deg(L^\vee) \ .$$ 
For any highest weight submodule $W=W(\tau) \hookrightarrow L\otimes L^\vee$ we therefore get
$$    deg(\tau)  \leq    depth(L)  \ .$$
Indeed $deg(\tau) \leq deg(L) + deg(L^\vee)$; here $deg(L)$ denotes the degree of the highest
weight of $L$. We also conclude
$$   depth(L(\tau))  \leq   2 \cdot depth(L)  \ .$$

\medskip\noindent
If we consider $W=L \otimes L^\vee$ for $L=L(\lambda)$, then the highest weight
in $W$ has degree $deg(\lambda) + deg(\lambda^\vee) = depth(\lambda)= depth(L)$.
Since $depth(W)=depth(L) + depth(L^\vee) = 2 depth(L)$, therefore all weights in
$W$ have degrees within 
$$[-depth(L),depth(L)] \ .$$

\medskip
{\it The weights $\lambda^0$ and $\lambda^c$}. We recall the definition of the weight $\lambda^0$ attached to $\lambda$ from \cite{Brundan-Stroppel-2}. In the weight diagram of $\lambda$ add $n$
cups to the diagram by repeatedly connecting $\wedge \vee$-pairs of vertices that are
neighbours in the sense that there are no vertices in between
not yet connected
to cups. Then $\lambda^0$ is the weight whose associated cup diagram is the cup diagram just constructed.

\begin{example} If $depth(L)=n(n-1)$, then $L(\lambda^0) = L(\lambda) \otimes Ber^{-1}$.
\end{example}

\begin{example} If $depth(L)=0$, then $L(\lambda^0) = L(\lambda) \otimes Ber^{-n}$.
\end{example}

\medskip
The assignment of weights
$$ \lambda \mapsto \lambda^0 $$
 has a unique inverse
$$ \mu \mapsto \mu^c \ ,$$
where $\mu^c$ is obtained from $\mu$ by a total left move (in the language of cup diagrams). Hence $\mu^c$ is  the weight attached to the complementary plot of the plot corresponding to $\mu$ (for the notion complementary plot we refer to \cite{Heidersdorf-Weissauer-tensor}; i.e. one passes from the plot to the complements in each sector of the plot). 

\begin{lem}\label{depth-calc} $depth(L) = deg(L) + deg(L^\vee) = depth(L_{basic}) = 2 deg(L_{basic})$ holds for irreducible $\g$-modules $L$.
\end{lem}

\medskip
{\it Proof}. 
Consider the Kac module $V(\rho_\lambda)$ or $V(\lambda)$ for short.  
Its irreducible socle (as a $\g$-module)
is $L(\lambda^0)$; see \cite[Theorem 6.6]{Brundan-Stroppel-2}. 
The restriction of $L(\lambda)$ to $\g_0$ contains the weight $\lambda = (\lambda_1,\ldots,\lambda_n)$ of maximal degree, whereas the restriction of $\lambda^0$ to $\g_0$ contains the weight \[ \tau= \lambda\otimes det^{-n} \]

where $\tau$ is the minimal highest weight of $\g_0$ in the restriction
of $L(\lambda^0)$. We also write $(\lambda^0)_{min}$ instead of $\tau=\lambda \otimes det^{-n}$.

Indeed the lowest $\g_0$-representation $\tau$  in $V(\lambda)\vert_{\g_0} \cong
\rho_\lambda \otimes \Lambda^\bullet(g/p) $ is $\rho_\lambda \otimes \Lambda^{n^2}(g/p)$, that corresponds in our notation
to the representation $\rho_\lambda \otimes det^{-n}$. 
We conclude $  deg(\lambda) - deg(\lambda^0) + depth(\lambda^0) = deg(det^n) = n^2$. In terms of $\mu = \lambda^0$, we conclude from the above arguments

\begin{lem} {\it The maximal degree  resp. minimal degree for the highest $\g_0$-weights of the restriction of the irreducible $\g$-module $L(\mu)$
are the degrees of the extremal highest weights $\mu$ resp. $\mu_{min}= \mu^c - (n,...,n)$ and 
$$  depth(\lambda) = deg(\lambda) - deg(\lambda^c - (n,...,n)) = deg(\lambda) - deg(\lambda^c) + n^2 \ .$$
Furthermore $\mu_{min}$ is the unique irreducible $\g_0$-constituent in $L(\mu)$ of minimal degree.}    
\end{lem}

Note that $\deg(\lambda)$ is the same as $S(\lambda) + n(n-1)/2$ where 
$S(\lambda)$ is the sum of the points $x$ in the support of the plot $\lambda(x)$.
Since $S(\lambda) - S(\lambda^c)$ only depends on the associated basic weight $\lambda_{basic}$,
we find that
$$  depth(\lambda) = depth(\lambda_{basic}) \ .$$

Since $\lambda^\vee = \lambda^*$,  we have 
$depth(\lambda) = deg(\lambda) + deg(\lambda^*)$ for basic weights $\lambda=\lambda_{basic}$.
Since $deg(\lambda^*)= deg(\lambda)$, we obtain for basis weights $\lambda=\lambda_{basic}$
the formula
$$  depth(\lambda_{basic}) =  2 \cdot deg(\lambda_{basic})  \ .$$
This proves lemma \ref{depth-calc}. \qed

\begin{cor} {\it For all weights $\lambda=(\lambda_1,...,\lambda_n)$ we have
$$   0 \leq depth(\lambda) \leq   n(n-1)  \ .$$}
\end{cor}

\medskip
{\it Proof}. Obvious from $deg(\lambda_{basic}) \leq deg((n-1,...,1,0)) = n(n-1)/2$. \qed

\bigskip
Note that $V(\lambda)^\vee$ again is a Kac-module whose
cosocle now is the dual $L(\lambda^0)^\vee$ of $L(\lambda^0)$.
Hence $V(\lambda)^\vee = V((\lambda^0)^\vee)$. Its highest weight is 
the highest weight $(\lambda^0)^\vee$ of $L((\lambda^0)^\vee)$.

\medskip
We also remark that $Ber^{-n} V(\lambda^\vee)^\vee = V(\lambda_{min})$ for $L(\lambda^\vee):= L(\lambda)^\vee$ and $\lambda_{min}= \lambda^c det^{-n}$. Indeed we can replace $\lambda$ by $\lambda^0$. Then $ V(\lambda^\vee)^\vee $ becomes $V(\lambda)$ and $\lambda_{min}$
becomes $\tau= \lambda det^{-n}$.



\subsection{Projectives in tensor products} Let $L'$ and $L''$ be irreducible $\g=\mathfrak{gl}(n\vert n)$-modules that are maximal atypical. 
Let us assume that $P$ is an irreducible projective maximal atypical module with the property
$$  P \subseteq L' \otimes L'' \ .$$
We assume  $L'\! =\! L(\rho')$ resp. $L''\!=\! L(\rho'')$.
In $P$ there exists a Kac module $V=V(\rho) \subset P$ whose socle $L(\tau) = L(\rho^0)$ is the socle of $P=P(\tau)$. Its precise structure will
be not important at the moment, except that the anti-Kac module $V(\tau)^*$ is also a $g$-submodule of $P$ with
cosocle $L(\tau^0)$. Indeed, $V(\tau)$ is a quotient module of $P(\tau)$ and hence $V(\tau)^*$ is a submodule of $P(\tau)^* \cong P(\tau)$. 
Consider the inclusion
$$ i: V=V(\rho) \hookrightarrow  L(\rho') \otimes L(\rho'') \ $$
and its restriction to the subalgebra $\g_0 = \gl(n) \times \gl(n)$ and similarly for $V(\tau)^*$.

\medskip
As a representation of $\g_0$ the module $V' = V(\rho')$ becomes
$$    V'\vert_{\g_0}  \cong  \rho' \otimes W^\bullet  $$
where $W^\bullet = \bigoplus_\mu \rho_{-\mu} \boxtimes \rho_{-\mu^*}$ holds for
 the irreducible representations $\rho_\mu$ of $\mathfrak{gl}(n)$ with highest weights $\mu=(\mu_1,..,\mu_n)$ running over
all $\mu$ such that $n \geq \mu_1 \geq ... \geq \mu_n\geq 0$; here
$\mu^*$ denotes the weight with the transposed Young diagram of the weight $\mu$.
In particular, the degree $deg(\mu)=\sum_{i=1}^n \mu_i$ varies
between 0 and $n^2$. In particular, in the degree grading of $W^\bullet$ we have
$deg(W^i) = -i$ and
$$  L'\vert_{\g_0} \subseteq  \rho' \otimes \bigotimes_{i=0}^{depth(L')} W^i. $$ Similarly for $L''$ instead off $L'$.

\medskip
The projective module $P=P(\tau)$ contains the irreducible $\g$-modules
$L(\rho)$, $L(\tau) =L(\rho^0)$ and $L(\tau^0) = L((\rho^0)^0)$.
$$ 
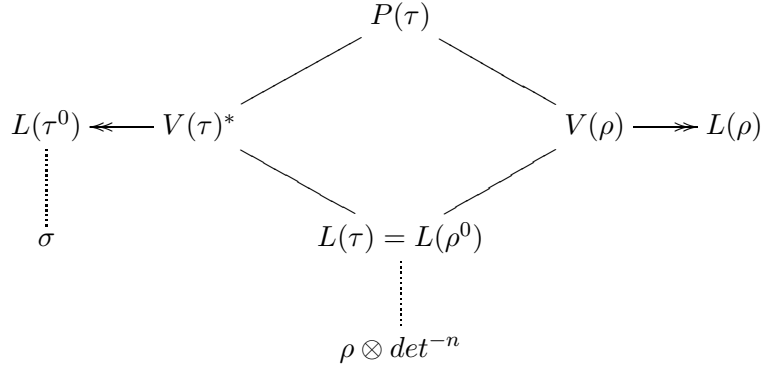
\xymatrix{  &   &  P(\tau) &  &  \cr
L(\tau^0) & V(\tau)^* \ar@{-}[ur] \ar@{->>}[l]   &              &  V(\rho)\ar@{-}[lu]  \ar@{->>}[r]  & L(\rho) \cr
\sigma \ar@{.}[u] &   &  \ar@{-}[ul]   L(\tau) = L(\rho^0) \ar@{-}[ur] &  &  \cr 
&  &   \rho \otimes det^{-n} \ar@{.}[u] &  &  \cr} $$

\medskip\noindent
The restriction of $P$ to $\g_0$ contains then the
irreducible $\g_0$-modules 
$$  \rho, \tau , \rho det^{-n} , \tau^0, \sigma $$
for  the lowest $\g_0$-weight $\sigma= \tau \otimes det^{-n}$ of $\tau^0$.
Note that $\sigma = (\tau^0)^c \otimes det^{-n} \cong \tau \otimes det^{-n}$. 
Furthermore
$$deg(\sigma) = deg(\tau^0) - depth(\tau^0) $$
$$deg(\tau) - depth(\tau) = deg(\rho) - n^2 $$
and so
$$ deg(\tau) - deg(\tau^0) = n^2 - depth(\tau^0)\ .$$
These equations imply
$$ depth(P)= deg(\rho) - deg(\sigma) = 2n^2 - depth(\tau) \geq n(n+1) \ .$$
Note that $P \hookrightarrow L \otimes L^\vee$ implies $depth(P) \leq 
depth(L) + depth(L^\vee) = 2 depth(L)$, and hence $$  2 n^2 \leq depth(L(\tau)) + 2 depth(L) \ .$$
In particular, we then obtain
$$  2n \leq depth(L(\tau))  \quad , \quad  n(n+1)/2 \leq depth(L)  \ .$$

\medskip
From the above we get

\begin{prop} If the tensor product of maximally atypical irreducible $\g$-modules
$L(\rho')$ and $L(\rho'')$ contains a maximal atypical projective module $P=P(\tau)$, then the irreducible
$\gl(n)$- representations defined by $\tau, \tau^0$ and $\tau^c=\rho$ and $\tau \otimes det^{-n}=\sigma$ 
and $\tau^c \otimes det^{-n}=\rho \otimes det^{-n}$ are constituents of
$$    \rho' \otimes \rho'' \otimes \bigoplus_{i=0}^{depth(\rho')} W^i  \otimes 
\bigoplus_{j=0}^{depth(\rho'')} W^j \ $$
and 
$$  (\rho')^c \otimes (\rho'')^c \otimes det^{-2n} \otimes \bigoplus_{i=0}^{depth(\rho')} (W^i)^{dual}  \otimes 
\bigoplus_{j=0}^{depth(\rho'')} (W^j)^{dual}  \ .$$
\end{prop}

Since the degrees in this tensor product are between $deg(\rho') + deg(\rho'')$
and $ deg(\rho') + deg(\rho'') - depth(\rho') - depth(\rho'')$, in the situation of the last proposition
the following  holds
$$ deg(\rho) - deg(\sigma) \leq depth(\rho') + depth(\rho'')  \ .$$
Hence we get

\begin{cor} $L(\sigma')\otimes L(\sigma'')$ can not contain a projective maximal
atypical $g$-module unless 
$$  deg(\rho'_{basic}) + deg({\rho''}_{basic}) \geq n(n+1)/2 \ .$$
\end{cor}



\subsection{The case $n=3$}. Let us assume first $n=3$. Here the condition $deg(\rho'_{basic}) + deg({\rho''}_{basic}) \geq 6$ may be only satisfied for $\rho'_{basic}=\rho''_{basic} = (2,1,0)$. Consider $P(\tau) \subseteq L' \otimes L''$. 
From depth comparison with projective modules
we get $depth(P(\tau))= 2n^2 - depth(\tau)\leq  depth(L') + depth(L'')$. Hence $depth(\tau)\geq 6$ and therefore $depth(\tau)=6$ resp. $\tau_{basic}=(2,1,0)$. Hence
$$depth(P(\tau))= 2n^2 - depth(\tau) = 18 -6 = 6+6 =  depth(L') + depth(L'')\ .$$
This implies that the highest weight of $L'\otimes L''$ must have the same degree as the highest weight of $P(\tau)$, i.e. \[ deg(\lambda' + \lambda'') = deg(\rho),\] where $\rho^0=\tau$. Note that $\rho$ cannot  be the highest weight $\lambda' + \lambda''$ because of the next lemma. 

\begin{lem} \label{highest-weight-const} The highest weight constituent in a tensor product of two irreducible representations $L(\lambda') \otimes L(\lambda'')$ is never contained in a projective module.
\end{lem}

{\it Proof}. We can assume that $L(\lambda') = [\lambda_1',\ldots,\lambda_{n-1}',0]$ and likewise for $L(\lambda'')$. Then the result follows from \cite[corollary 13.11]{Heidersdorf-mixed-tensors}. \qed

\begin{cor}  For $n=3$ a projective maximal atypical module $P$ can not be contained
in the tensor product $L'\otimes L''$ unless $L'_{basic} = L''_{basic} = [2,1,0]$.
\end{cor} 

\begin{remark} We note that this implies that $I \cong \one$ for $I$ as above if $I$ is a direct summand in $L(\lambda) \otimes L(\mu)$ with $\lambda_{basic} = \mu_{basic} = (2,2,0,0)$. Indeed none of the direct summands of $DS([2,2,0,0])$ is of basic type $(2,1,0)$. 
\end{remark}

\begin{example} A brute force computations shows that $R(2,1)^{\otimes 2}$ only contains $R(3,2,1) = P[2,1,0]$ as a maximal atypical projective summand. Since $[2,1,0]^{\otimes 2}$ is a subquotient of $R(2,1)^{\otimes 2}$ this means that the only possible maximal atypical projective summand in $[2,1,0]^{\otimes 2}$ is $P[2,1,0]$.
\end{example}

\section{Another counterexample}\label{counter-2}
We give another counterexample to show that not every invertible element in $\mathcal{T}_n/\mathcal{N}$ is represented by a Berezin power. In the category ${\mathcal T}_{2}$ consider the indecomposable representation $I$, defined uniquely up to isomorphism by the nontrivial extension of ${\bf 1}$ and $S^1$ such that $$0\to S^1 \to I \to {\bf 1} \to 0 \ .$$ The following lemma follows from the fusion rules between the $S^i$ (section \ref{sec:ind-start}). 

\begin{lem} $S^1$ is selfdual. Up to negligible modules $Symm^2(S^1)$ is isomorphic to ${\bf 1}$, and $\Lambda^2(S^1)$ is indecomposable of superdimension 3
with socle and top $S^1$ of  superdimension $sdim(S^1)=-2$.
\end{lem}

$I$ has sdimension $-1$, hence 
$I\otimes I^\vee = {\bf 1} \oplus N$  with  $sdim(N)=0$. We will show that $N$ is negligible so that $I$ defines an element of the Picard group. Since every morphism ${\bf 1}\to S^1$
vanishes, this implies $\dim Hom(I,I)=1 $. The following lemma is left to the reader.

\begin{lem} \label{lemma-2-app} We have $\dim Hom(I \otimes I,{\bf 1}) = \dim Hom(I, I^\vee) =1$
and we have $\dim Hom({\bf 1}, I \otimes I) = \dim Hom(I^\vee, I) =1$.
\end{lem}


Now  consider the semisimple tensor category ${\mathcal T}_{2}/{\mathcal N}$
and the tensor subcategory $Rep_{\bf C}(\overline\mu,G)$ generated by $I$. Its Tannaka group $G$
admits an isogeny $G_{class}\times G_{nil} \twoheadrightarrow G$ 
where $G_{class}$ is a classical reductive group and $G_{nil}$ is a copy of
groups $OSp(1|2m)$ whose underlying classical group is $SO(1)\times Sp(2m)$.
$I$ defines a ${\bf Z}/2$-graded module $\overline I$ of $G$ where
$\overline\mu: {\bf Z}/2 \to G$ defines the ${\bf Z}/2$-gradinging.
Considered as a representation of $G_{class}\times G_{nil}$,
the module $\overline I$ is an external tensor product $\overline I_{class}\boxtimes \overline  I_{nil}$
such that $sdim(\overline  I_{class})sdim(\overline  I_{nil})=sdim(\overline  I)=\sdim(I)=-1$.
This implies that the underlying dimension of $\overline I_{class}$ is 1. Hence $G_{class}$ is a subgroup of ${\mathbb G}_m$. Furthermore $\sdim(\overline I_{nil})=\pm 1$ (note that a priori $\overline\mu$ might not lift to the covering of $G$). 


\begin{lem} \label{lemma-3-app} $Symm^2(I)$ is indecomposable of superdimension $0$ and either $\Lambda^2(I)$ is indecomposable of superdimension $1$
or decomposes into two indecomposable summands $S^1$ resp.
$Symm^2(S^1)$ with superdimensions $-2$ resp. $3$.
\end{lem}

\begin{proof}
From the Loewy filtration of $I$ we have $\Lambda^2(S^1) \subset \Lambda^2(I)$
and  \[Symm^2(S^1) \subset Symm^2(I,)\] and furthermore
$Symm^2({\bf 1})\cong {\bf 1}$ is a quotient of $Symm^2(I)$.
In addition there only appear
single irreducible constituents $S^1$ both in 
$\Lambda^2(I)$ and $Symm^2(I)$.
Hence $\Lambda^2(I)/\Lambda^2(S^1) \cong S^1 $
and $\sdim(\Lambda^2(I))=1$. Since  $\Lambda^2(S^1)$ injects
into $\Lambda^2(I)$ and has irreducible socle, it must be isomorphic
to a submodule of one indecomposable summand of $\Lambda^2(I)$.
Hence either $\Lambda^2(I)$ is indecomposable or 
has the two indecomposable summands isomorphic to $S^1$ and $\Lambda^2(S^1)$.

\medskip\noindent
On the other hand $Symm^2(I)$ has 
superdimension 0, and the constituents are $S^1$ and twice the constituent ${\bf 1}$.
The module $Symm^2(I)$ therefore cannot contain
${\bf 1}$ as direct summand since this would imply that $Hom({\bf 1},I\otimes I)$
or $Hom(I\otimes I,{\bf 1})$ has dimension $>1$, contradicting lemma \ref{lemma-2-app}.
Since $Ext^1({\bf 1},{\bf 1})=0$, $S^1$ cannot be a direct summand
of $Symm^2(I)$ either. Hence all three
irreducible constituents must be contained  in one indecomposable constituent.
Thus $Symm^2(I)$ is indecomposable. 
\end{proof}

\begin{cor}The supergroup $G_{nil}$ is trivial, hence $G$ is a subgroup
of ${\mathbb G}_m$.
\end{cor}

\begin{proof} Since $\Lambda^2(I)$ is negligible by lemma \ref{lemma-3-app}, the module $\overline I\otimes \overline I$ contains at most $2$ indecomposable constituents and $\Lambda^2(\overline I)=0$ holds. Hence $\overline I_{nil}$ is $\epsilon$-small in the sense of  Kr\"amer-Weissauer \cite{Kr-Wei}. We remark that in loc. cit we tacitly excluded the case where the superdimension 
is $\pm 1$ since a priori the only possible supergroups $G_{nil}$ in this case are products of $OSp(1|2)$, as shown in loc. cit. But loc. cit. table 1 and 2 and the other arguments can nevertheless be applied. This shows that $G_{nil}$ could only be $OSp(1|2)$, and $\overline I_{nil}$ must then be the standard representation, unless $G_{nil}$ is trivial.  But loc. cit. table 2 then implies,  if $G_{nil}$ is nontrivial, that the tensor product $\overline I\otimes \overline I$ would have
3 different irreducible constituents and not $\leq  2$. This fact forces 
$G_{nil}$ to be trivial. 
\end{proof}

It implies that $G=G_{class}$ is contained in ${\mathbb G}_m$. Furthermore, $\overline\mu: {\bf Z}/2 \to {\mathbb G}_m$ then is an injective map since $sdim(I)=-1$. This proves

\begin{cor} $I\otimes I^\vee ={\bf 1}\oplus N$ and $N$ is   negligible. Therefore
$I$ and its parity shift $\Pi I$ define elements of the Picard group of 
${\mathcal T}_{2}/{\mathcal N}$.
\end{cor}

\end{appendix}











\end{document}